\documentclass[11pt]{amsart}

\usepackage{amsmath, amsthm, amssymb}
\numberwithin{equation}{section}

\usepackage{amsthm}
\usepackage{hyperref}
\usepackage[all]{xy}
\usepackage[usenames,dvipsnames]{xcolor}
\usepackage{mleftright}

\usepackage{cjhebrew}
\newcommand{\Fp}{{\text{\<p>}}}

\usepackage{overpic}

\newtheorem{theo}{Theorem}[section]

\newtheorem{prop}[theo]{Proposition}
\newtheorem{conj}[theo]{Conjecture}
\newtheorem{corollary}[theo]{Corollary}

\theoremstyle{definition}
\newtheorem{defn}{Definition}

\newtheorem{eg}[theo]{Example}

\newtheorem{rmk}[theo]{Remark}

\usepackage[left=2.5cm,right=2.5cm,top=3.5cm]{geometry}

\pdfoutput=1
\usepackage{amsmath}
\usepackage{amssymb}
\usepackage{wrapfig}
\usepackage{bbm}
\usepackage{arydshln}
\usepackage{multirow}
\usepackage{mathtools}
\usepackage{blkarray}
\usepackage{setspace}
\usepackage{dsfont}
\usepackage{float}
\usepackage{subcaption}
\usepackage{braket}
\usepackage[utf8]{inputenc}
\usepackage[english]{babel}
\usepackage[all]{xy}
\usepackage{tikz-cd}
\usepackage{mathtools}
\usepackage{tikz}

\makeatletter
\newcommand\mathcircled[1]{%
  \mathpalette\@mathcircled{#1}%
}
\newcommand\@mathcircled[2]{%
  \tikz[baseline=(math.base)] \node[draw,circle,inner sep=3pt] (math) {$\m@th#1#2$};%
}
\makeatother

\allowdisplaybreaks

\newcommand{\C}{{\mathbb C}}

\DeclareMathOperator{\Hom}{Hom}

\newcommand{\K}{{\bb K}}

\newcommand{\bb}{\mathbb}
\newcommand{\mc}{\mathcal}

\newcommand{\mf}{\mathfrak}

\newcommand{\MM}{{\textbf{M}}}
\newcommand{\ch}{{\textup{ch}}}
\newcommand{\nnn}{{\textbf{n}}}

\newcommand{\kk}{{\textbf{k}}}
\newcommand{\bl}{{\textbf{l}}}

\newcommand{\dd}{{\textbf{d}}}
\newcommand{\Gl}{{\textup{Gl}}}

\newcommand{\DD}{{\textup{D}}}

\newcommand{\Hilb}{\textup{Hilb}}

\newcommand{\Coh}{\textup{Coh}}

\newcommand{\op}{\textup{op}}

\newcommand{\cs}{\textup{cs}}

\newcommand{\id}{{\textup{id}}}

\newcommand{\Perv}{\textup{PervCoh}}

\newcommand{\pt}{\textup{pt}}
\newcommand{\Z}{\bb Z}

\newcommand{\red}{\textup{red}}
\newcommand{\Pic}{\textup{Pic}}

\newcommand{\End}{{\textup{End}}}

\newcommand{\Span}{{\textup{Span}}}

\newcommand{\g}{{\mf{g}}}
\newcommand{\h}{{\mf{h}}}
\newcommand{\V}{{\mathbb{V}}}
\newcommand{\one}{{\textbf{1}}}
\newcommand{\gl}{{\mf{gl}}}
\newcommand{\spl}{{\mf{sl}}}
\newcommand{\ghat}{{\hat{\mf{g}}}}
\newcommand{\Lz}{{{(\!(z)\!)}}}
\newcommand{\Pz}{{{[\![z]\!]}}}
\newcommand{\LFz}{{{[[z^{\pm 1}]]}}}
\newcommand{\prin}{{\textup{prin}}}

\newcommand{\del}{{{\partial}}}
\newcommand{\nol}{{ : }}
\newcommand{\nor}{{ : }}

\newcommand{\Vir}{{\textup{Vir}}}
\newcommand{\e}{{\varepsilon}}

\newcommand{\f}{{\textup{f}}}

\newcommand{\rk}{{\textup{rk}}}

\newcommand{\VW}{{\textup{VW}}}
\newcommand{\rr}{{{\textbf{r}}}}
\newcommand{\qq}{{{\textbf{q}}}}
\newcommand{\glh}{{{\widehat{\mathfrak{gl}}}}}

\newcommand{\sss}{{{\textup{ss}}}}

\newcommand{\lP}{{ [\![ }}
\newcommand{\rP}{{ ]\!] }}

\usepackage{ytableau}

\begin{document}

\title{Vertex algebras from divisors on Calabi-Yau threefolds}

\author{Dylan Butson}
\address{Mathematical Institute, University of Oxford, Andrew Wiles Building, Radcliffe Observatory Quarter (550), Woodstock Road, Oxford, OX2 6GG}
\email{dylan.butson@maths.ox.ac.uk}

\maketitle

\begin{abstract} We construct vertex algebras $\mathbb{V}(Y,S)$ from divisors $S$ on toric Calabi-Yau \mbox{threefolds $Y$,} satisfying conjectures of Gaiotto-Rapcak and Feigin-Gukov, as the kernel of screening operators on lattice vertex algebras determined by the GKM graph of $Y$ and a filtration on $\mathcal{O}_S$. We prove that there are representations of $\mathbb{V}(Y,S)$ on the homology groups of various moduli spaces of coherent sheaves on $Y$ supported on $S$ constructed in a companion paper with Rapcak, defined by certain Hecke modifications of these sheaves along points and curve classes in the divisor $S$. This generalizes the common mathematical formulation of a conjecture of Alday-Gaiotto-Tachikawa, the special case in which $Y=\mathbb{C}^3$ and $S=r[\mathbb{C}^2]$, to toric threefolds and divisors as proposed by Gaiotto-Rapcak. We outline an approach to the general conjecture and prove many special cases and partial results using tools developed in the companion paper, following the proof of the original conjecture by Schiffmann-Vasserot and its generalization to divisors in $\mathbb{C}^3$ by Rapcak-Soibelman-Yang-Zhao.

The vertex algebras $\mathbb{V}(Y,S)$ conjecturally include $W$-superalgebras $ W_{f_0,f_1}^\kappa(\mathfrak{gl}_{m|n})$ and genus zero class $\mc S$ chiral algebras $\mathbb{V}^{\mathcal{ S}}_{\textup{Gl}_m;f_1,...,f_k}$, each for general nilpotents $f_i$. By definition, this implies the existence of a family of compatible free field realizations of these vertex algebras, relevant to their parabolic induction and inverse quantum Hamiltonian reduction. We prove these conjectures in the examples of lowest non-trivial rank for each case, and outline the proof in general for some cases.
\end{abstract}

\begingroup
\hypersetup{hidelinks}
\tableofcontents
\endgroup

\section{Introduction}

The goal of this paper is to construct vertex algebras $\V(Y,S)$ associated to divisors $S$ on toric Calabi-Yau threefolds $Y$, and explain a conjectural correspondence between the representation theory of these vertex algebras and the enumerative geometry of the threefolds $Y$ and their divisors. The first indication of this correspondence was a certain algebraic refinement of the formula of Gottsche \cite{Gott} for the generating function for the Betti numbers of the Hilbert schemes $\Hilb_n(S)$ of zero dimensional, length $n$ subschemes of a smooth projective surface $S$:
\begin{equation}\label{Gotteqn}
	\sum_{n=0}^\infty q^n P_t(\Hilb_n(S)) = \prod_{j=1}^\infty \frac{ (1+ t^{2m-1}q^m)^{b_1(S)} (1+t^{2m+1}q^m )^{b_3(S)}}{(1- t^{2m-2}q^m)^{b_0(S)}(1- t^{2m}q^m)^{b_2(S)}(1- t^{2m+2}q^m)^{b_4(S)}} \ .
\end{equation}
This refinement was discovered independently by Grojnowski \cite{Groj} and Nakajima \cite{Nak}, and is a defining example of geometric representation theory; we now briefly describe their results.

Let $\pi(S)$ denote the Heisenberg vertex algebra generated by the homology $H_\bullet(S,\bb Z)$ with respect to the intersection pairing, and let
\begin{equation}\label{piSeqn}
	 \pi_S = \bigoplus_{n\in \bb N} H_\bullet( \Hilb_n(S), \C) \ . 
\end{equation}
It was observed in \emph{loc. cit.} that for each class $[C] \in H_\bullet(S,\C)$ there exist natural correspondences
\begin{equation}\label{Nakcorreqn}
	\vcenter{\xymatrix{ & \Hilb_{n,n+k}^{[C]}(S) \ar[dr]^p \ar[dl]_q & \\ \Hilb_n(S) && \Hilb_{n+k}(S) }}   \ ,
\end{equation}
for $k\in \bb Z$ and $n\in \bb N$ in the compatible range, inducing natural operators
\begin{equation}\label{Nakopeqn}
\alpha_{-k}^n([C])=p_*\circ q^*: H_\bullet(\Hilb_n(S)) \to H_\bullet(\Hilb_{n+k}(S)) \quad\quad \text{and}\quad\quad	\alpha_k([C]) = \sum_{n\in \bb N} \alpha_k^n([C]) 
\end{equation}
for each $k\in \bb Z$, which satisfy the relations implicit in the following theorem:

\begin{theo}\label{Nakthm} \cite{Groj,Nak} There exists a natural representation
	\[ \mc U(\pi(S)) \to \End( \pi_S)  \quad\quad \text{defined by} \quad\quad b_k^i \mapsto \alpha_k([C_i])   \ ,\]
	of the algebra of modes $\mc U(\pi(S))$ of the vertex algebra $\pi(S)$ on the vector space $\pi_S$, such that $\pi_S$ is identified with the vacuum module of the vertex algebra $\pi(S)$.
\end{theo}

\noindent In particular, this implies Equation \ref{Gotteqn} as the refined vacuum character of the Heisenberg algebra.

The first main construction of this paper gives vertex algebras $\V(Y,S)$ associated to divisors $S$ on certain toric Calabi-Yau threefolds $Y$, generalizing $\pi(S)$ above. We also recall several results of the companion paper \cite{BR1}, which provide moduli spaces $\mc M_\nnn(Y,S)$ generalizing $\Hilb_n(S)$, homology groups $\bb V_S=\bigoplus_\nnn H_\bullet(\mc M_\nnn(Y,S),\varphi_{W_S})$
generalizing $\pi_S$ of Equation \ref{piSeqn}, and certain natural correspondences inducing endomorphisms of $\bb V_S$ as in Equations \ref{Nakcorreqn} and \ref{Nakopeqn}, respectively. Moreover, these constructions satisfy the following schematic analogue of Theorem \ref{Nakthm}, and the goal of this paper is to formulate several precise variants of this and give proofs in some families of examples:

\begin{conj}\label{VWintroconj} There exists a natural representation
	\begin{equation}\label{repstrmapeqn}
		\rho: \mc U(\V(Y,S)) \to \End(\V_S )
	\end{equation}
	of the algebra of modes $\mc U(\V(Y,S))$ of the vertex algebra $\V(Y,S)$ on the vector space $\V_S$.
\end{conj}

The vertex algebras $\V(Y,S)$ are defined as the subalgebras of lattice-type vertex algebras $\Pi(Y,S)$ given by the intersection of the kernels of certain maps of vertex algebra modules:
\[ \V(Y,S) = \cap_k \ker(Q_k) \quad\quad \textup{for}\quad\quad Q_k : \Pi(Y,S) \to \Pi(Y,S)_{\lambda_k} \ . \]

The lattice-type vertex algebra $\Pi(Y,S)$ is determined by the multiplicities in the toric divisor $S$ of the reduced, irreducible components $\{S_d\}_{d\in \mf D_S}$, as follows: for a Cartier divisor $S$ satisfying
\[ [S]=\sum_{d\in \mf D_S} r_d [S_d] \quad \quad \textup{we let} \quad\quad \Pi(Y,S)= \bigotimes_{d\in \mf D_S} \Pi(Y,S_d)^{\otimes r_d} \quad\quad \text{where}\quad\quad \Pi(Y,S_d)=\bigoplus_{\lambda\in H_2(S_d,\Z)} \pi(S_d)_\lambda  \]
is a natural vertex algebra extension determined by the lattice $H_2(S_d,\Z)$ equipped with negative the intersection pairing, of the ($T$-equivariant) Grojnowski-Nakajima Heisenberg algebra $\pi(S_d)$, over $F$ the field of fractions of $H^\bullet_T(\pt)$ for $T$ the subtorus preserving the Calabi-Yau form on $Y$.

The \emph{screening operators} $Q_k$ are determined by a certain Jordan-Holder filtration on $\mc O_S$: for
\[ \mc O_S = \left[ \mc O_{S_{d_1}} < ... <\mc O_{S_{d_k}} < ... < \mc O_{S_{d_N}} \right]  \quad\quad \text{we define} \quad\quad  \V(Y,S) = \bigcap_{k=1}^{N-1} \ker (Q_k) \ \subset \Pi(Y,S) \ , \]
where $Q_k$ is defined in terms of the non-compact, toric curve classes contained in $S_{d_k}\cap S_{d_{k+1}}$. This construction satisfies the following \emph{factorization} property: for any strictly coarser composition series
\begin{equation}\label{VOAembintroeqn}
	\mc O_S = \left[ \mc O_{S_{1}} < ... <\mc O_{S_{l}} < ... < \mc O_{S_{h}} \right]\quad\quad \text{we have} \quad\quad   \bb V(Y,S) \to \bigotimes_{l=1}^h \V(Y,S_l)
\end{equation}
an embedding of vertex algebras, giving a natural system of partial free field realizations of $\V(Y,S)$.

Each such embedding determines a natural module $\bb V^\f(Y,S) = \otimes_{l=1}^h \V(Y,S_l)_0$ over $\V(Y,S)$ by restricting the vacuum module of the codomain, labelled by the partial semisimplification $\oplus_{l=1}^h \mc O_{S_l}$, which determines a \emph{framing structure} $\f$ in the sense of the companion paper \cite{BR1}, a choice required to define the desired moduli spaces $\mc M_\nnn^\f(Y,S)$ and homology groups $\V^\f_S=\oplus_\nnn H_\bullet^T(\mc M_\nnn^\f(Y,S),\varphi_{W_S^\f})$.

For example, the trivial extension $\oplus_{k=1}^N \mc O_{S_{d_k}}$ determines a framing structure $0_S$, and similarly $\mc O_S$ determines $\f_S$. In this setting, we have the following precise statement of Conjecture \ref{VWintroconj} above:
\begin{conj}\label{VOAmodintroconj} For each framing structure $\f$, there is a representation $\rho_S^\f: \mc U(\V(Y,S)) \to \End(\V_S^\f )$
such that $\V_S^\f$ identifies with the corresponding module $\V^\f(Y,S)$ above. In particular, $\V_S^{\f_S}$ and $\V_S^{0_S}$ are identified with the vacuum module $\V(Y,S)_0$ and free field module $\Pi(Y,S)$, respectively.
\end{conj}

For $Y=\C^3$, $S=r[\C^2]$, the above construction is that of \cite{FF1}, so that $\V(\C^3,r[\C^2]) \cong  W^\kappa_{f_\prin}(\gl_r)$, the principal affine $W$-algebra of $\gl_r$ over $F$ at level $\kappa=-h^\vee-\frac{\e_2}{\e_1}$, and for $\f=0_{r[\C^2]}$, Conjecture \ref{VOAmodintroconj} is precisely the mathematical formulation of a conjecture of \cite{AGT} proved in \cite{SV}, \cite{MO}, and \cite{BFN5}. Similarly, for $\f=\f_{r[\C^2]}$, Conjecture \ref{VOAmodintroconj} is the main result of \cite{CCDS}.

For $Y=Y_{m,n}$ a resolution of $X_{m,n}=\{xy-z^mw^n\}$, a pair of partitions $\mu$ and $\nu$ of $M$ and $N$, of lengths $m$ and $n$, determines a divisor $S_{\mu,\nu} \subset Y_{m,n}$ as in Figure \ref{fig:Smunufig}. We conjecture $\V(Y_{m,n},S_{\mu,\nu})\cong W_{f_\mu,f_\nu}^\kappa(\gl_{M|N})$, the affine $W$-algebra of $\gl_{M|N}$ for nilpotents $f_\mu \in \gl_M$ and $f_\nu\in \gl_N$, and we prove:

\begin{theo}[\ref{WLStheo}] There is a natural isomorphism of vertex algebras
	\begin{equation}
W_{f_\mu}^\kappa(\gl_{M}) \xrightarrow{\cong} \V(Y_{m,0},S_{\mu})
\end{equation}
identifying $\Pi(Y_{m,0},S_{\mu})$ with the generalized Wakimoto realization, and the maps of Equation \ref{VOAembintroeqn} with the parabolic induction maps of \cite{Gen2} and inverse reduction maps generalizing \cite{Sem}.
\end{theo}

In \cite{BR1} we construct geometric representations of algebras $\mc Y_S(Y)$ on $\V_S^\f$, and we conjecture
\begin{equation}\label{yangfacteqn}
	\rho_{S}^\f: \mc Y_S(Y) \twoheadrightarrow \mc U(\V(Y,S)) \to \End(\bb V^\f_S)  \ ,
\end{equation}
a factorization inducing the desired representation of $\mc U(\V(Y,S))$ geometrically by correspondences generalizing those of Equation \ref{Nakcorreqn}, where the surjection holds in an appropriate completion.

Further, we conjecture that there exist generalized coproduct maps
\[\Delta_{S_1,...,S_h} :\mc Y_S(Y) \to \bigotimes_{l=1}^h \mc Y_{S_l}(Y)   \quad\quad \text{such that the diagrams} \quad\quad   \vcenter{\xymatrix{\mc Y_S(Y)  \ar[r]\ar[d]^{} &  \bigotimes_{l=1}^h \mc Y_{S_l}(Y)  \ar[d] \\  \mc U(\V(Y,S)) \ar[r] &  \bigotimes_{l=1}^h  \mc U(\V(Y,S_l)) }} \ ,\]
commute, where the lower horizontal arrows are the vertex algebra embeddings of Equation \ref{VOAembintroeqn}.

For $S_{\mu,\nu} \subset Y_{m,n}$ as above, following \cite{CosMSRI} we conjecture in \cite{BR1} an isomorphism of algebras
\[ \mc Y_{S_{\mu,\nu}}(Y_{m,n}) \cong   \mc Y_{\sigma_{\mu,\nu}}(\glh_{m|n})  \ ,\]
with the shifted affine Yangian of $\gl_{m|n}$, for shift matrix $\sigma_{\mu,\nu}$ determined by the intersection numbers of $S_{\mu,\nu}$ with the compact curves in $Y_{m,n}$, identified with simple roots of $\gl_{m|n}$ as in Figure \ref{fig:Smunufig}.

In this case, the factorization of Equation \ref{yangfacteqn} induces a generalization to type $A$ affine superalgebras of the Brundan-Kleshchev isomorphism \cite{BrKl} between truncated, shifted Yangians and W-algebras in finite type $A$, such that the conjectural generalized coproducts on shifted affine Yangians are compatible with parabolic induction and inverse reduction maps for affine $W$-algebras.

The genus zero class $\mc S$ chiral algebras $\V^{\mc S}_{\Gl_m;f_1,...,f_k}$ defined in \cite{Ar} following \cite{BeemS} are conjecturally contained in the class of algebras $\V(Y,S)$ for general nilpotents $f_i$ and $k\leq 2$, where the threefold $Y$ and divisor $S$ are as in Figure \ref{fig:classSfig}. For general $k\in \bb N$, we also propose a variant of the construction, summarized in Figure \ref{fig:classSkfig}. This would imply a family of compatible free field realizations of these algebras, inducing analogous inverse reduction and parabolic induction maps.

\smallskip

We now give a succinct summary of the contents of each of the sections of this paper:

\begin{itemize}
 \item In Section \ref{backgroundsec} we recall some background on previous results motivating the present work, and in Section \ref{overviewsec} we give a narrative overview of our main results.

\smallskip

\item In Section \ref{prelimsec}, we recall some preliminaries from the theory of vertex algebras: in \ref{VOAsec} we recall the most basic definitions and constructions, in \ref{VOAextsec} we recall the notions of vertex algebra extensions and lattice vertex algebras, and in \ref{screeningsec} we recall the notion of screening operators.
	
	\smallskip
	
\item	 In Section \ref{screensec}, we explain the algebraic construction of $\V(Y,S)$: in \ref{ablatticesec} we define the free field algebras $\Pi(Y,S)$, in \ref{geoscreensec} we define the canonical screening operators associated to non-compact toric curve classes in $Y$, and in \ref{VOAYSsec} we define the vertex algebras $\V(Y,S)$ and prove the key factorization and locality properties.
	 
	 \smallskip
	 
\item  In Section \ref{GRsec}, we explain the conjecturally equivalent geometric construction of the various partial free field representations of the vertex algebra $\V(Y,S)$: in \ref{GRintrosec} we recall from \cite{BR1} the construction of the moduli spaces $\mc M^\f(Y,S)$, their homology groups $\V^\f_S$, and geometric representation of $\mc Y_S(Y)$ on $\V^\f_S$, in \ref{latticegeosec} we prove the main conjectures in the case that $S$ has a single reduced, irreducible component, and in \ref{generalGRsec} we outline an approach in general.

\smallskip

\item  In Section \ref{egsec}, we explain these constructions in several examples: in \ref{egabsec} we describe basic examples where $S$ is reduced and irreducible, and in \ref{Walgsec} and \ref{classSsec} we explain the conjectural applications to general type $A$ affine $W$-superalgebras and genus zero class $\mc S$ chiral algebras. 
\end{itemize}

\subsection{Background}\label{backgroundsec}
The first essential step required for the generalization of Theorem \ref{Nakthm} towards Conjecture \ref{VWintroconj} is to consider moduli spaces of sheaves of higher rank. The Hilbert scheme $\Hilb_n(S)$ viewed as a moduli space of ideal sheaves can equivalently be understood as the moduli space $\mc M_S^\zeta(1,n)$ of $\zeta$-semistable torsion free sheaves of rank 1 with trivial determinant, and it is natural to ask if there is an analogous algebraic structure governing moduli spaces of torsion free sheaves of higher rank. In fact, the analogous generating function to that of Equation \ref{Gotteqn} at $t=-1$,
\begin{equation}\label{VWeqn}
	\mc Z^\VW_{r,S}(q) = \sum_{n=0}^\infty \chi( \mc M_S^\zeta(r,n) )  \ , 
\end{equation}
where $\chi( \mc M_S^\zeta(r,n) )$ denotes the Euler characteristic of the moduli space $\mc M_S^\zeta(r,n)$, was considered in the theoretical physics literature by Vafa-Witten \cite{VW}, who predicted certain modular properties of this generating function based on considerations in string theory. However, such moduli spaces are much less well understood in higher rank and a general closed-form computation of the Vafa-Witten partition function $\mc Z^\VW_{r,S}$ analogous to that of Equation \ref{Gotteqn} does not appear to be possible.

The main idea that has been used to generalize the results of Grojnowksi and Nakajima to higher rank also came from theoretical physics, in work of Nekrasov \cite{Nek1} and Alday-Gaiotto-Tachikawa \cite{AGT}. In the context of their relation to Seiberg-Witten theory \cite{SW1,SW2}, Nekrasov proposed 
that for a class of generating functions of invariants of surfaces including that of Equation \ref{VWeqn}, there is a natural local analogue defined in this case by considering the moduli spaces $\mc M(r,n)$ of framed torsion free sheaves on $S=\C^2$, and computing the Euler characteristic of their localized equivariant homology $H_\bullet^{\tilde T}(\mc M(r,n))\otimes_{H_{\tilde T}^\bullet(\pt)} F$ with respect to the action of $\tilde T=(\bb C^\times)^2$ on $S=\C^2$.

Further, it was realized by Alday-Gaiotto-Tachikawa that the analogous vertex algebra which governs the local structure of the equivariant homology of higher rank torsion free sheaves is the principal affine W-algebra $W_{\rho_\prin}^\kappa(\gl_r)$ of $\gl_r$ at level
\begin{equation}\label{agtleveleqn}
	\kappa = -h^\vee -\frac{\e _2}{\e _1}    \ , 
\end{equation}
where we have identified $H_{\tilde T}^\bullet(\pt)=\bb K[\e_1,\e_2]$, and $h^\vee=r$ denotes the dual Coxeter number. The analogue of the result of Grojnowski and Nakajima above in this setting was one natural mathematical formulation of the predictions of \cite{AGT}, which was proved independently by Schiffmann-Vasserot \cite{SV}, Maulik-Okounkov \cite{MO}, and Braverman-Finkelberg-Nakajima \cite{BFN5}. For
\begin{equation}\label{V0rC2eqn}
	 \bb V_{r[\C^2]}^0= \bigoplus_{n\in\bb N} H_\bullet^{A}(\mc M(r,n))\otimes_{H_{A}^\bullet(\pt)} F \ ,  
\end{equation}
where $A=\tilde T\times T_\f$ for $T_\f\subset \Gl_r$ the maximal torus and $F$ is the field of fractions of $H_A^\bullet(\pt)$, we have:

\begin{theo}\label{SVAGTthm}\cite{SV,MO,BFN5} There exists a natural representation
	\[ \mc U(W^\kappa_{f_{\textup{prin}} }(\gl_r)) \to \End(\V_{r[\C^2]}^0)  \ ,\]
	such that $\V_{r[\C^2]}^0$ is identified with the universal Verma module $\bb M_r$ for $W^\kappa_{f_{\textup{prin}} }(\gl_r)$.
\end{theo}

As the notation suggests, in the special case $Y=\C^3$ and $S=r[\C^2]$ our general construction produces precisely the module $\bb V_{S}^0=\bb V_{r[\C^2]}^0$ and the vertex algebra $\V(\C^3,r[\C^2])$ is given by
\[ \bb V(\C^3,r[\C^2]) \cong \mc W^\kappa_{\rho_\prin}(\gl_r)  \ , \]
so that Theorem \ref{SVAGTthm} can be understood as a special case of one variant of Conjecture \ref{VWintroconj} above.

We now describe some aspects of the proof of Theorem \ref{SVAGTthm} given in \cite{SV}, on which our general approach is modeled. Afterwards, we will explain our proposed generalization in detail and outline its implications for the correspondence between representations of the vertex algebras $\bb V(Y,S)$ and the enumerative geometry of the threefold $Y$.

One essential benefit of the restriction to $S=\C^2$ is that the moduli spaces $\mc M(r,n)$ admit a description as the moduli spaces of stable representations of the following Nakajima quiver 
\begin{equation}\label{ADHMeqn}
	   \begin{tikzcd}\boxed{\bb C^r}  \arrow[r, shift left=0.5ex, "I"] & \arrow[l, shift left=0.5ex, "J"]
	\mathcircled{V}	\arrow[out=70,in=110,loop,swap,"B_1"]
	\arrow[out=250,in=290,loop,swap,"\normalsize{B_2}"]\end{tikzcd}
\quad\quad\text{with relations}\quad\quad  [B_1,B_2]+IJ=0   \ ,
\end{equation}
of dimension $\dim V=n$ and framing dimension $r$. This description of the moduli space of framed instantons of rank $r$ and charge $n$ on $\bb R^4=\C^2$ was discovered by Atiyah-Drinfeld-Hitchin-Manin in \cite{ADHM}, and was the motivating example of a Nakajima quiver variety \cite{Nak1}.

It was observed in \cite{SV} that the correspondences of the type in Equation \ref{Nakcorreqn} are naturally parameterized by representations of the unframed variant of the above quiver
\begin{equation}\label{C2quiveqn}
	 \begin{tikzcd}\mathcircled{V} \arrow[out=160,in=200,loop,swap,"B_1"]\arrow[out=340,in=20,loop,swap,"B_2"] \end{tikzcd}\quad\quad\text{with relations}\quad\quad  [B_1,B_2]=0   \ , 
\end{equation}
as we now explain. The stack of representations of this quiver with relations is given by
\[\mf M_n(\C^2) = \left[ C_n / \Gl_n  \right]  \quad\quad \text{where}\quad\quad  C_n=\{ (B_1,B_2) \in \gl_n^{\times 2}\ | \ [B_1,B_2]=0 \} \ ,\]
and we denote the corresponding equivariant Borel-Moore homology groups by
\[\mc H(\C^2) =  \bigoplus_{n\in \bb N}\mc H_n(\C^2) = \bigoplus_{n\in \bb N} H_\bullet^A( \mf M_n(\C^2))\otimes_{H_{A}^\bullet(\pt)} F \ . \]

There are analogous correspondences between the spaces $\mf M_n(\C^2)$, defined by the stacks of short exact sequences $\mf M_{k,l}(\C^2)$ of representations of the unframed quiver of dimension $k+l$, with a subobject of dimension $k$ and quotient object of dimension $l$,
\begin{equation}\label{cohaC2multcorreqn}
	\vcenter{\xymatrix{ & \mf M_{k,l}(\C^2) \ar[dr]^p \ar[dl]_q  & \\ \mf M_{k}(\C^2) \times \mf M_{l}(\C^2)  && \mf M_{k+l}(\C^2) }}
\end{equation}
which induce maps
\begin{equation}\label{cohaC2multeqn}
	 p_*\circ q^*: \mc H_k(\C^2) \otimes \mc H_l(\C^2) \to \mc H_{k+l}(\C^2)  \quad\quad \text{and thus}\quad\quad m: \mc H(\C^2)^{\otimes 2} \to \mc H(\C^2) \ , 
\end{equation}
which defines an associative algebra structure on $\mc H(\C^2)$; the resulting algebra $\mc H(\C^2)$ is called the \emph{preprojective cohomological Hall algebra} of $\C^2$.

Further, the stacks of short exact sequences of representations of the framed quiver with potential define analogous correspondences
\begin{equation}\label{cohacoreqn}
	\vcenter{\xymatrix{ & \mc M(r,(k,n)) \ar[dr]^p \ar[dl]_q  & \\ \mf M_{k}(\C^2) \times \mc M(r,n)  && \mc M(r,n+k) }}
\end{equation}
and induce a representation of $\mc H(\C^2)$ on $\bb V_{r[\C^2]}$, which we denote by
\begin{equation}\label{SVrepeqn}
	 \rho_{r[\C^2]} : \mc H(\C^2) \to \End_F( \V_{r[\C^2]})  \ . 
\end{equation}
In particular, for $r=1$ the equivariant analogue of the Nakajima operator $\alpha_{-1}:\V_{\C^2} \to \V_{\C^2}$ of Equation \ref{Nakopeqn} is given by the image of the fundamental class of $[\mf M_1(\C^2)] \in \mc H_1(\C^2)$ under this representation. More generally, the image of the spherical subalgebra $\mc{SH}(\C^2)\subset \mc H(\C^2)$ generated by $\mc H_1(\C^2)$ includes the equivariant Nakajima operators $\alpha_k$ for all $k> 0$, and we obtain an isomorphism
\[ \mc U(\pi^k)_+  \xrightarrow{\cong} \rho_{\C^2}( \mc{SH}(\C^2))  \ , \]
between the positive half $\mc U(\pi^k)_+$ of the algebra of modes $\mc U(\pi^k)$ of the vertex algebra $\pi^k$, the Heisenberg algebra over $F$ at level
\[k = - \langle [\C^2] , [\C^2]\rangle = - \e_A(T_0\C^2)^{-1} =- \frac{1}{\e_1\e_2} \ \in F  \ ,\]
and the image of the spherical subalgebra $\mc{SH}(\C^2)$ under the above representation $\rho_{\C^2}$.  The above Theorem \ref{SVAGTthm} of \cite{SV} establishes that $\rho_{r[\C^2]}$ of Equation \ref{SVrepeqn} also induces an isomorphism
\begin{equation}\label{SVequiveqn}
	  \mc U(W^\kappa_{f_{\textup{prin}} }(\gl_r))_+ \xrightarrow{\cong} \rho_{r[\C^2]}^0( \mc{SH}(\C^2))   \ .
\end{equation}
Similarly, the action of the negative half of the algebra of modes is defined by taking adjoints of the endomorphisms in $\rho_{r[\C^2]}^0( \mc{SH}(\C^2))$, with respect to the equivariant intersection pairing on $\bb V_{r[\C^2]}$.

The proof of \emph{loc. cit.} proceeds as follows: the tensor product structure on the framing equivariant cohomology of Nakajima quiver varieties provides natural isomorphisms
\begin{equation}\label{C2facteqn}
	 \bb V_{r[\C^2]}^0 \cong \bb V_{\C^2}^{\otimes r}  \quad\quad \text{defining} \quad\quad \mc U((\pi^k)^{\otimes r}) \to \End_F(\bb V_{\C^2}^{\otimes r})\cong \End_F( \bb V^0_{r[\C^2]}) \ , 
\end{equation}
and the action of $\mc{SH}(Y)$, generated by correspondences of the type in Equation \ref{cohacoreqn}, factors through a subalgebra of $\mc U((\pi^k)^{\otimes r})$. Moreover, it is proved that this subalgebra is precisely the algebra of modes of a vertex subalgebra $\bb V(\C^3,r[\C^2])\subset (\pi^k)^{\otimes r}$ defined by
\begin{equation}\label{screeneqn}
	\bb V(\C^3,r[\C^2]) = \bigcap_{l=1}^{r-1} \ker(Q_l) \quad\quad \text{where}\quad\quad Q_l=\int V^l(z) dz \ : (\pi^k)^{\otimes r} \to  (\pi^k)^{\otimes r}_{\lambda_l}   \ ,
\end{equation}
and $V^l(z)$ denotes a certain lattice vertex operator with coefficients in $F$; we recall the formalism of lattice vertex algebras and screening currents in Section \ref{prelimsec} below. In particular, we recall the famous results of Feigin-Frenkel \cite{FF1}, which imply that the resulting vertex algebra is given by
\begin{equation}\label{}
	 \bb V(\C^3,r[\C^2]) \cong W^\kappa_{f_{\textup{prin}} }(\gl_r)   \quad\quad \text{for} \quad\quad \kappa= - h^\vee - \frac{\e_2}{\e_1}\ , 
\end{equation}
as claimed, completing the proof of Theorem \ref{SVAGTthm} from \cite{SV}. The goal of this paper is to generalize the constructions of Equations \ref{screeneqn} and \ref{SVrepeqn}, and the proof of their equivalence outlined above.

\subsection{Overview of results}\label{overviewsec} We now give a narrative overview of the new results of this paper. A succinct summary of the contents of each of the sections was provided just before Section \ref{backgroundsec} above.

In Sections \ref{screensec} and \ref{GRsec}, respectively, we give the two main constructions of this paper, establish some of their corresponding properties, and prove several variants of Conjecture \ref{VWintroconj} :
\begin{enumerate}
	\item We construct vertex algebras $\bb V(Y,S)$, as the kernel of certain screening operators on lattice vertex algebras determined by the GKM graph of $Y$ and a Jordan-Holder filtration on $\mc O_S$, generalizing the Feigin-Frenkel realization \cite{FF1} of $ W^\kappa_{f_{\textup{prin}} }(\gl_r)$ as in Equation \ref{screeneqn}.
	\item We construct analogous moduli spaces $\mc M^0(Y,S)$, homology groups $\bb V_S^0$, algebras $\mc H(Y)$, representations $\rho_S:\mc H(Y)\to \End_F(\bb V_S^0)$, and extensions thereof to vertex algebras, generalizing those in the construction described above of the representation in Theorem \ref{SVAGTthm} of \cite{SV}.
\end{enumerate}

\noindent We begin by recalling the details of the latter \emph{geometric} construction, following the results of the companion paper \cite{BR1}, and afterwards describe the former \emph{algebraic} construction motivated by this and outline some partial results towards the proof of their equivalence.

Let $Y\to X$ be a toric Calabi-Yau threefold resolution of the class considered in \emph{loc. cit.} and let $S$ be a toric divisor on $Y$. We define algebraic stacks
\[ \mf M(Y)=\mf M_{\Perv_\cs(Y)} \quad\quad \text{and}\quad\quad \mf M^0(Y,S)=\mf M^{0_S}(Y,\mc O^\sss_{S^\red}[1]) \ , \]
parameterizing certain complexes of coherent sheaves on $Y$. The stack $\mf M^0(Y,S)$ is defined so that for a suitable choice of stability condition $\zeta=\zeta^\VW$, the moduli space of $\zeta$-stable objects
\begin{equation}\label{M0eqn}
	\mc M^0(Y,S)=\mf M^{0,\zeta}(Y,S)
\end{equation}
provides a model in algebraic geometry for the moduli space of instantons on the divisor $S$, generalizing the role of $\mc M(r,n)$ in the usual AGT conjecture. The stack $\mf M(Y)$ parameterizes complexes of coherent sheaves with compactly supported cohomology, contained in the heart of a certan exotic t-structure of $\DD^b\Coh(Y)$, which parameterize the natural correspondences generalizing those of Equation \ref{cohacoreqn}, analogously to the spaces $\mf M_n(\C^2)$ in the argument above.

The first of the two main Theorems in \cite{BR1} establishes an equivalence between the stacks $\mf M(Y)$ and $\mf M^0(Y,S)$ and stacks of representations of a quiver with potential $(Q_Y,W_Y)$ and a framed variant thereof $(Q^0_S,W^0_S)$, respectively. For example, in the case $Y=\C^3$ and $S=r[\C^2]$ from \cite{RSYZ}:
\begin{equation}\label{C3quiveqn}\hspace*{-1cm}
	Q_{\C^3} = \begin{tikzcd}
		\mathcircled{V}\arrow[out=340,in=20,loop,swap,"B_3"]
		\arrow[out=220,in=260,loop,swap,"B_2"]
		\arrow[out=100,in=140,loop,swap,"\normalsize{B_1}"] \end{tikzcd}
	\quad\quad \quad Q_{r[\C^2]}^0 =\quad \begin{tikzcd} \boxed{\C^r} \arrow[r, shift left=0.5ex, "I"] & \arrow[l, shift left=0.5ex, "J"]
		\mathcircled{V}\arrow[out=340,in=20,loop,swap,"B_3"]
		\arrow[out=220,in=260,loop,swap,"B_2"]
		\arrow[out=100,in=140,loop,swap,"\normalsize{B_1}"] \end{tikzcd}  \quad\quad\quad   \begin{cases} W_Y &  = B_1[B_2,B_3]  \\  W_S^0 &  = B_1[B_2,B_3] + B_3I J  \end{cases}  \ .  
\end{equation}

In this setting, the natural homology theory to consider is the Borel-Moore homology with coefficients in the sheaf of vanishing cycles determined by the potential, so that the vector spaces over $F$ underlying the algebra $\mc H(Y)$ and module $\bb V_S^0$ are given by
\begin{equation}\label{cohaandmoddefn}
	 \mc H(Y) = \bigoplus_{\nnn \in \bb N^{V_{Q_Y}}} H_\bullet^A( \mf M_\nnn(Y), \varphi_{W_Y})\otimes_{H_{A}^\bullet(\pt)} F  \quad\quad \text{and}\quad\quad \bb V_S^0 =\bigoplus_{\nnn \in \bb N^{V_{Q_Y}}} H_\bullet^A( \mc M^0_\nnn(Y,S), \varphi_{W_S^0}) \otimes_{H_{A}^\bullet(\pt)} F   \ .
\end{equation}

The second main Theorem in \cite{BR1} constructs the analogous representation
\begin{equation}\label{rhoSeqn}
	 \rho_S:\mc H(Y) \to  \End_F(\bb V_S^0) \ , 
\end{equation}
which conjecturally induces a representation of the algebra of modes of the vertex algebra $\bb V(Y,S)$, by the same mechanism described above in the proof of Theorem \ref{SVAGTthm} from \cite{SV}.

In examples such as $Y=\C^3$ and $S=r[\C^2]$, the homology groups $\bb V_S^0$ are related by \emph{dimensional reduction}, in the sense of Appendix A of \cite{Dav1} for example, to the ordinary Borel-Moore homology groups defining $\mc H(\C^2)$ and $\bb V^0_{r[\C^2]}$ in the preceding section: there is a natural isomorphism
\begin{equation}\label{dimredeqn}
	H_\bullet( \mf M_\dd(Q,W),\varphi_W ) \xrightarrow{\cong} H_\bullet(\mf M_\dd(\tilde Q,R) )  \ ,
\end{equation}
between the homology of the stack of representations of the quiver with potential $(Q,W)$, with coefficients in the sheaf of vanishing cycles determined by $W$, and the ordinary Borel-Moore homology of the corresponding dimensionally reduced quiver with relations $(\tilde Q,R)$. In particular, we have
\[  H_\bullet^A( \mc M^0_n(\C^3,r[\C^2]), \varphi_{W_{r[\C^2]}^0})\xrightarrow{\cong } H_\bullet^A(\mc M(r,n)) \ ,  \]
so that the preceding definition of $\bb V_S^0$ in this case agrees with that of $\bb V_{r[\C^2]}^0$ in Equation \ref{V0rC2eqn} above:
\begin{equation}\label{dimredrC2eqn}
	  \bb V^0_{r[\C^2]} = \bigoplus_{n\in \bb N}  H_\bullet^A( \mc M^0_n(\C^3,r[\C^2]), \varphi_{W_{r[\C^2]}^0}) \otimes_{H_A^\bullet(\pt)} F  \xrightarrow{\cong } \bigoplus_{n\in \bb N} H_\bullet^A(\mc M(r,n)) \otimes_{H_A^\bullet(\pt)} F   \ . 
\end{equation}

\noindent Similarly, this extends to an ismorphism of representations under an analogous isomorphism of associative algebras $\mc H(\C^3)\cong \mc H(\C^2)$, which follows from the results of \cite{YZ1} and the appendix to \cite{RenS}, so that our general construction of the representation in Equation \ref{rhoSeqn}, in the case $Y=\C^3$ and $S=r[\C^2]$, is equivalent to that of Equation \ref{SVrepeqn}, from the proof of Theorem \ref{SVAGTthm} of \cite{SV}.

We now outline the definition of the vertex algebras $\V(Y,S)$ for general divisors $S$ in toric Calabi-Yau threefolds $Y$, and outline our approach to the proof of Conjecture \ref{VWintroconj} in this case. We begin with the case that $S$ has a single smooth, reduced, irreducible component, for which the corresponding moduli space $\mc M^0(Y,S)$ parameterizes rank 1 sheaves, as in the classical setting of Grojnowski and Nakajima, though our construction differs from \emph{loc. cit.} whenever $H_2(S;\bb Z)\neq 0$.

An argument is outlined in the final Chapter of \cite{NakLec} that for moduli spaces of rank 1 torsion free sheaves of arbitrary first Chern class on a surface $S$, the Heisenberg algebra action on the homology defined in \emph{loc. cit.} should naturally extend to an action of
\[ \Pi(S):=  \pi_{H_0(S,\C)} \otimes V_{H_2(S,\Z)}   \ ,  \]
the tensor product of the usual Heisenberg algebra $\pi_{H_0(S,\C)}=\pi$ on $H_0(S,\C)=\C$ with the lattice vertex algebra $V_{H_2(S,\C)}$ extending $\pi_{H_2(S,\C)}$. In the localized, $A$-equivariant setting, as in the definitions of Equation \ref{cohaandmoddefn}, for $\mf F_S=S^T$ the localization theorem implies
\[ H_\bullet^A(S,\C)  \otimes_{H_A^\bullet(\pt)} F \xrightarrow{\cong}  \bigoplus_{y\in \mf F_S} F \quad\quad \text{and thus} \quad\quad \pi_{H_\bullet^A(S,\C)} \xrightarrow{\cong} \bigotimes_{y\in \mf F_S} \pi_{k_y} \ ,\]
where $\pi_{k_y}$ denotes the Heisenberg algebra over the base field $F$ at level $k_y= -\frac{1}{\e_{A}(T_yS)}$ . From this perspective, the analogous lattice extension to $\Pi(S)$ is equivalent to extending $ \pi_{H_\bullet^A(S,\C)}$ by 
\[ V_{i}(z) = \nol \exp \left( \e_T(N_{C_i,0_i} S) \phi^{0_i}(z) + \e_T(N_{C_i,\infty_i}S) \phi^{\infty_i}(z) \right) \nor  \quad  \in \Hom( \pi_{m}, \pi_{m+[i]})\LFz \ , \]
a lattice vertex operator for each toric curve class $C_i\subset S$ between two fixed points $0_i,\infty_i\in \mf F_S$. In Section \ref{ablatticesec} we explain this presentation of the lattice vertex algebra over $F$, which we denote $\Pi(Y,S)$, towards giving the general definition of $\V(Y,S)$ for $S$ not necessarily reduced or irreducible.

The general definition of $\V(Y,S)$ is modeled on the Feigin-Frenkel realization of the principal affine W-algebra at generic level $\kappa = - h^\vee - \frac{\e_2}{\e_1}$, by which we defined $\V(\C^3,r[\C^2])$ in Equation \ref{screeneqn}. Given a divisor $S$, we write
\[ S^\red = \bigcup_{d\in \mf D_S} S_d \quad\quad \text{so that}\quad\quad S = \sum_{d\in \mf D_S} r_d [S_d] \]
for some positive integers $r_d\in \bb N$. Then, generalizing Equation \ref{C2facteqn} above, we expect that the vector space $\bb V_S^0$ admits a factorization
\[  \bb V_S^0 = \bigotimes_{d\in \mf D_S} (\bb V_{S_d}^0)^{\otimes r_d} \quad\quad \text{and thus}\quad\quad \mc U(\bigotimes_{d\in \mf D_S} \Pi_d^{\otimes r_d} ) \to \End_F( \V_S^0 ) \ , \]
a representation of the algebra of modes of the vertex algebra
\[ \Pi(Y,S) = \bigotimes_{d\in \mf D_S} \Pi_d^{\otimes r_d} \quad\quad \text{where}\quad\quad  \Pi_d=\Pi(Y,S_d) \ , \]
the lattice vertex algebra over $F$ corresponding to the reduced, irreducible divisor $S_d$, as above.

The vertex algebra $\V(Y,S)$ is then defined as the kernel of a certain collection of screening operators on $\Pi(Y,S)$ constructed as follows: we fix a presentation of $\mc O_S$ as an iterated extension
\[ \mc O_S = \left[ \mc O_{S_{d_1}} < ... <\mc O_{S_{d_k}} < ... < \mc O_{S_{d_N}} \right] \ , \]
for some ordered list of elements $d_k\in \mf D$, in which each $d\in \mf D$ occurs $r_d$ times, and define
\[ \V(Y,S) = \bigcap_{k=1}^{N-1}  \bigcap_{s_k} \ker(Q_{s_k})\quad\quad \text{where}\quad\quad \begin{cases}
	Q_{s_k} =& \int Q_{s_k}(z)dz \ : \Pi(Y,S) \to \Pi(Y,S)_{\lambda_{s_k}} \\  Q_{s_k}(z) =&  \nol \exp \left( \e_T(N_{C_s,y_s} S_{d_+}) \phi^{j_{d_+}}_{y_s}(z) - \e_T(N_{C_s,y_s}S_d) \phi^{j_{d_-}}_{y_s}(z) \right)\nor  \end{cases}  \]
is a canonical screening operator defined for each non-compact curve class $C_{s_k} \subset S_{d_k} \cap S_{d_{k+1}}$.

As we explain in Section \ref{egsec}, the vertex algebras $\V(Y,S)$ conjecturally include $W$-superalgebras $ W_{f_0,f_1}^\kappa(\gl_{m|n})$ and genus zero class $\mc S$ chiral algebras $\V^{\mc S}_{\Gl_m;f_1,...,f_k}$ with $k\leq 2$ marked points, each for general nilpotents $f_i$, and we give a proof of these conjectures in several low rank examples. By definition, this implies the existence of free field realizations of these vertex algebras, which for $W$-algebras $W_f^\kappa(\gl_N)$ appears to give a certain canonical bosonized presentation of the generalized Wakimoto resolutions defined by Genra \cite{Gen1}. More generally, these free field realizations satisfy a natural compatibility as the divisor $S$ varies within a fixed threefold $Y$, which implies natural parabolic induction and inverse quantum Hamiltonian reduction maps for this class of algebras.

In particular, if $S=S_1+S_2$ then we evidently have $\Pi(Y,S)=\Pi(Y,S_1)\otimes \Pi(Y,S_2)$, and for composition series of $\mc O_{S}$, $\mc O_{S_1}$ and $\mc O_{S_2}$ in terms of the sheaves $\mc O_{S_d}$, compatible in the sense that
\[ \mc O_{S_1} \to \mc O_S \to  \mc O_{S_2}   \  , \]
defines a short exact sequence compatible with the filtrations, we obtain
\begin{equation}\label{voaembeqn}
	\bb V(Y,S) \to \V(Y,S_1)\otimes \V(Y,S_2)  
\end{equation}
an embedding of vertex algebras, with image given by the kernel of a single screening operator.

For example, in the case that $Y=|\mc O_{\bb P^1}\oplus \mc O_{\bb P^1}(-2)|$ and $S=S_{2,1,0,0}=2 [\bb A^2_{xy}] +[|\mc O_{\bb P^1}|]$, we explicitly compute $\bb V(Y,S)=V^\kappa(\spl_2)\otimes \pi$, identify the resulting free field realizations with a bosonisation of the Wakimoto realization, and deduce the $\spl_2$ inverse reduction result of Semikhatov \cite{Sem}; this is the content of Theorem \ref{affinesl2theo}, and is summarized in Figure \ref{fig:sl2factfig}. The four natural modules appearing in this calculation correspond under Conjecture \ref{VOAmodintroconj} to the four framed quivers of Figure \ref{fig:Walgframingsfig}.

\section{Preliminaries}\label{prelimsec}

\subsection{Vertex algebras and modules}\label{VOAsec}

\begin{defn} An element $a(z) \in \End(V)[[z^{\pm 1}]]$ is called a field if $a(z)v \in V((z))$ for each $v\in V$.
\end{defn}

\begin{rmk} More explicitly, $A=\sum_{n\in\Z} a_{n}z^{-n-1} \in \End(V)[[z^{\pm 1}]]$ is a field if for each $v\in V$ there exists $N\in \Z$ such that $a_n(v)=0$ for all $n>N$.
	
\end{rmk}

\begin{defn}\label{voadefn} A vertex algebra is a tuple $(V,\O,T,Y)$ consisting of
	\begin{itemize}
		\item a vector space $V$,
		\item an element $\one \in V$, called the vacuum,
		\item a linear map $T\in\End(V)$, called the translation operator, and
		\item a linear map $Y(\cdot,z):V\to \End(V)[[z^{\pm 1}]]$, called the vertex operator,
	\end{itemize}
	such that:
	\begin{itemize}
		\item $Y(\one,z)=\id_V$,
		
		\item $Y(a,z)=\sum_{n\in\Z} a_{n}z^{-n-1} \in \End(V)[[z^{\pm 1}]]$ is a field for each $a\in V$ and $v\in V$,
		
		\item $Y(a,z)(\one)\in V[[z]]\subset V((z))$ for each $a\in V$, and the evaluation satisfies $Y(a,z)(\one)|_{z=0}=a$,
		
		\item $[T,Y(a,z)]=\del_z Y(a,z)$ for each $a\in V$,
		
		\item $T\one=0$, and
		
		\item for each $a,b\in V$, the fields $a(z)=Y(a,z)\in \End(V)[[z^{\pm 1}]]$ and $b(w)=Y(b,w)\in \End(V)[[w^{\pm 1}]]$ are mutually local with respect to one another, that is, there exists $N\in \bb N$ such that
		\[ (z-w)^N [a(z),b(w)] = 0 \ , \]
		as elements of $\End(V)[[z^{\pm 1},w^{\pm 1}]]$.
	\end{itemize}
\end{defn}

The mutual locality of the fields $a(z)$ and $b(w)$ implies that
\begin{equation}\label{localeqn}
	 [a(z),b(w)] = \sum_{j=0}^{N-1} \frac{1}{j!} \gamma_j(w) \del^j_w \delta(z-w) \ ,
\end{equation}
for some collection of fields $\gamma_j(w)\in \End(V)[[w^{\pm 1}]]$ defined for $j=0,...,N-1$. In fact, these fields can be explicitly identified with the vertex operators $\gamma_j(w)=Y(a_jb,w)$ and the above expression implies
\[ a(z)b(w) =  \sum_{j=0}^{N-1} \frac{Y(a_jb,w)}{(z-w)^{j+1}} + \nol a(z) b(w) \nor \quad\quad \text{which we abbreviate} \quad\quad a(z)b(w) \sim \sum_{j=0}^{N-1} \frac{Y(a_jb,w)}{(z-w)^{j+1}} \ ,  \]
where $ \nol a(z) b(w) \nor$ denotes the normally ordered product of the fields $a(z)$ and $b(w)$, which is in particular regular at $z-w=0$.

Let $\g$ a simple Lie algebra over $\bb C$ equipped with a non-degenerate, invariant bilinear pairing $\kappa:\g\otimes \g \to \bb C$. We denote by $\ghat$ the canonical central extension
\[ 0 \to \bb C_\kappa \to \ghat \to \g\Lz \to 0 \ , \]
defined by the 2-cocycle $c:\g(\mc K)\otimes \g(\mc K) \to \bb C$, where $\g\Lz=\g\otimes_{\bb C} \bb C\Lz$ denotes the loop algebra, defined by $c(f\otimes a,g\otimes b)=-\kappa(a,b)\textup{Res}(f\cdot dg)$ for $f,g\in \bb C\Lz$ and $a,b\in \g$, where $\textup{Res}:\Omega^1_{\mc K} \to \bb C$ denotes the residue pairing.

Let $U(\ghat)$ denote the universal enveloping algebra of $\ghat$, $\mc U(\ghat)$ its completion with respect to the usual topology, and define
\[V_k(\g)= U(\ghat) \otimes_{U(\g\Pz \oplus \bb C_\kappa )} \bb C_k  \ , \]
where $\bb C_k$ for $k\in \bb C$ is the one dimensional module with basis vector $v_k$ on which $\g(\mc O)$ acts trivially and $\bb C_\kappa$ acts by $k$.
For elements $a\in \g$, there are corresponding fields
\[ J^a(z) = \sum_{n\in \bb Z} J^a_n z^{-n-1} \ \in \End(V_k(\g))\LFz \quad\quad \text{satisfying} \quad\quad   J^a(z) J^b(w)\sim   \frac{J^{[a,b]}(z)}{z-w} + \frac{k \kappa(a,b)}{(z-w)^2} \ , \]
which generate a vertex algebra structure on $V_k(\g)$ such that
\[ V_k(\g)=\Span_\C\{\  J^{a_1}_{n_1}  \cdots J^{a_j}_{n_j} v_k \ |  a_i \in \mc B_\g, \ n_1 \leq \cdots \leq n_j \leq -1 , \ j \in \bb N \  \}    \ , \]
where $\mc B_\g$ denotes some choice of basis for the Lie algebra $\g$.

Moreover, the topological associative algebra $\mc U(V_k(\g))$ associated to the vertex algebra $V_k(\g)$ is isomorphic to $\mc U_k(\ghat)$, the completion with respect to the standard topology of the quotient $U_k(\ghat)$ of $U(\ghat)$ by the two sided ideal generated by $1_\kappa - k$ where $1_\kappa\in \bb C_\kappa$ and $k\in \C$.

Consider the case that $\g=\h$ is an abelian Lie algebra of dimension $r$ and define
\[\pi^k=V_k(\h) = U(\hat \h) \otimes_{ U(\h\Pz \oplus \bb C_\kappa)} \bb C_{k} \ . \]

The vertex algebra structure on $\pi^k(\h)$ is generated by fields
\[J^a(z) = \sum_{n\in \bb Z} b_n^a z^{-n-1} \ \in \End(\pi^k)\LFz  \quad\quad \text{satisfying} \quad\quad   J^a(z)J^b(w)\sim \frac{k \kappa(a,b)}{(z-w)^2} \ \ , \]
and we denote the corresponding basis
\begin{equation}\label{fockmodeqn}
	 \pi^k = \Span_{\bb C}\{\  b_{n_1}^{a_1}  \cdots b_{n_j}^{a_j} \one^k \ |  a_i \in \mc B_\h, \ n_1 \leq \cdots \leq n_j \leq -1 , \ j \in \bb N \  \} \cong \bb C[ b_n^a]_{n\leq 1}^{ a\in \mc B_\h}  \one^k   \ .  
\end{equation}

The topological associative algebra $\mc U(\pi^k)=\mc U_k(\hat\h)$ contains a natural commutative subalgebra $\mc U_k(\hat\h)_+\cong \bb C[b_n^a]_{n\geq 0}^{a\in \mc B_\h}$ such that
\[ \pi^k = \mc U_k(\hat\h) \otimes_{\mc U_k(\hat\h)_+} \bb C_0\cong \bb C[ b_n^a]_{n\leq 1}^{ a\in \mc B_\h}  \one^k  \ ,\]
where $\bb C_0$ denotes the one dimensional $\mc U_k(\hat\gl_1)_+$ module on which all $b_n^a$ act by zero.

More generally, for $\lambda \in \h^\vee$, we define
\begin{equation}\label{heismodeqn}
	\pi_\lambda^k  = \mc U_k(\hat\h) \otimes_{\mc U_k(\hat\h)_+} \bb C_\lambda \cong \bb C[ b_n^a]_{n\leq 1}^{ a\in \mc B_\h}  \one^k_\lambda  \ , 
\end{equation}
where $\bb C_\lambda$ denotes the one dimensional $\mc U_1(\h)_+$ module on which $b_0^a$ acts by $k \lambda(a)$, and the remaining generators $b_{n}^a$ act by zero for $n\geq 1$. For each $\lambda \in \bb C$ there is a canonical vertex algebra module structure on $\pi^k_\lambda$ over the vertex algebra $\pi^k$ described above.

\subsection{Vertex operators and extensions}\label{VOAextsec}

Let $\pi=\pi^k(\h)$ denote the Heisenberg vertex algebra on $\h$ at level $k$, and for each $\lambda \in \h^\vee$ let $\pi_\lambda=\pi_\lambda^k$ denote the corresponding module as in Equation \ref{heismodeqn} above. Further, fix an orthogonal basis $\mc B=\mc B_\h$ for $\h$.

For each $\lambda \in \h^\vee$, we define the induced `exponential vertex operator'
\[ V_\lambda(z) =  \ \nol  \prod_{a\in \mc B} \exp(-\lambda(a) \phi^a(z) ) \nor    \   \in \Hom(\pi_\mu , \pi_{\lambda+\mu})\LFz \]
where the above expression is a shorthand notation for the field defined by
\begin{equation}\label{vertexopeqn}
	V_\lambda(z) = S_\lambda  \prod_{a\in \mc B} z^{\lambda(a) b_0^a} \exp\left( - \lambda(a) \sum_{n<0} \frac{b_n^a}{n} z^{-n} \right)  \exp\left(- \lambda(a) \sum_{n>0} \frac{b_n^a}{n} z^{-n} \right) \ , 
\end{equation}
where $S_\lambda:\pi_\mu \to  \pi_{\lambda+\mu}$ is the standard shift operator defined by the relations
 \begin{equation}\label{shiftopeqn}
 	S_\lambda \one_\mu = \one_{\lambda+\mu}  \quad\quad \text{and} \quad\quad [S_\lambda, b_n]=0 \quad\quad \text{for $n\neq 0$,} 
 \end{equation}
and the expontentials are defined in terms of their power series expansions, noting the sum is locally finite on $\pi_\mu$ at each order in $z$. In particular, we can write
\[ V_\lambda(z) = \sum_{n\in \bb Z} V_\lambda[n] z^{-n-\frac{\lambda^2}{2}}  \ \in \Hom(\pi_\mu, \pi_{\lambda+\mu})\LFz  \quad\quad \text{with}\quad\quad V_\lambda[n] \in \Hom(\pi_\mu, \pi_{\lambda+\mu}) \]
for each $n\in \bb N$, where we have assumed for simplicity $\lambda^2\in 2 \bb Z$. The fields $V_\lambda$ extend to define vertex algebra module intertwiners $Y:\pi_\lambda \to \Hom(\pi_\mu,\pi_{\lambda+\mu})\LFz$ for each $\mu$ by the formula
\begin{equation}\label{expintereqn}
	 Y(b_{n_1}^{a_1} \cdots b_{n_j}^{a_j} \one_\lambda,z) =  \nol  \left( \prod_{i=1}^j \frac{1}{(-n_i-1)!}\del_z^{-n_i-1}b^{a_i}(z)  \right) V_\lambda(z)\  \nor \quad \in \Hom(\pi_\mu , \pi_{\lambda+\mu})\LFz \ .  
\end{equation}

We recall that the translation operator acts on these fields according to the simple expression
\[ \del_z V_\lambda(z) = Y(T\cdot \one_\lambda, z) = \sum_{a\in \mc B} \lambda(a) Y( b_{-1}^a \one_\lambda, z) \ , \]
and moreover that the standard relations between the fields $V_\lambda, V_\nu$ and $J$ are given by
\[J^a(z) V_\lambda(w) \sim  \frac{k \lambda(a) V_\lambda(w)}{z-w} \quad\quad \text{and}\quad\quad V_\lambda(z) V_\nu(w) =(z-w)^{k \kappa^{-1}(\lambda,\nu)} \nol V_{\lambda}(z) V_{\nu}(w) \nor  \ , \]
where the field $\nol  V_{\lambda}(z) V_{\nu}(w) \nor \in \Hom(\pi_\mu , \pi_{\lambda+\nu+\mu})\LFz$ is defined by
\begin{equation}\label{expNOeqn}
	\hspace*{-1cm} \nol V_{\lambda}(z) V_{\nu}(w) \nor = S_\lambda S_\nu   \prod_{a\in \mc B} z^{\lambda(a) b_0^a} w^{\nu(a) b_0^a}\exp\left( - \sum_{n<0} \lambda(a)\frac{b_n^a}{n} z^{-n} + \nu(a) \frac{b_n^a}{n} w^{-n} \right)  \exp\left(- \sum_{n>0} \lambda(a) \frac{b_n^a}{n} z^{-n}+ \nu(a) \frac{b_n^a}{n} w^{-n} \right) \ . 
\end{equation}

Let $V_0$ be a vertex algebra, and $\{V_\lambda\}_{\lambda \in \Lambda}$ a collection of $V_0$ modules parameterized by a countable abelian group $\Lambda$, such that the module corresponding to the identity element $0\in \Lambda$ is the rank 1 free module $V_0$.

\begin{defn} A vertex algebra extension of $V_0$ by the collection of vertex modules $\{V_\lambda\}_{\lambda \in \Lambda}$ is a collection of vertex algebra intertwiners
\[ Y_\lambda^\mu: V_\lambda \to \Hom(V_\mu, V_{\mu+\lambda})\LFz  \quad\quad \text{defined for each $\lambda,\mu \in \Lambda$} \]
such that their direct sum defines a vertex algebra structure
\[ Y= \bigoplus_{\lambda,\mu \in \Lambda} Y_\lambda^\mu : V \to \Hom(V,V)\LFz \quad\quad \text{where} \quad\quad V= \bigoplus_{\lambda\in \Lambda} V_\lambda   \  . \]
\end{defn}

Note that for any subgroup $\tilde \Lambda\subset \Lambda$, the restricted direct sum $\tilde V = \oplus_{\lambda \in \tilde \Lambda} V_\lambda$ is a sub vertex algebra of the vertex algebra extension $V$. In particular, $V_0$ is canonically a sub vertex algebra of any vertex algebra extension $V$.

\begin{eg}\label{sqzeroeg} Let $V$ be a vertex algebra and $M$ a vertex module over $V$. The square zero extension of $V$ by $M$ is the vertex algebra extension $V_M=V\oplus M$ indexed by $\Lambda=\bb Z/ 2 \bb Z$ defined by the vertex module intertwiners
\begin{align*}
Y_0^0=Y_V & : V \to \Hom(V,V)\LFz  & Y_0^1 = Y_M & : V \to \Hom(M,M)\LFz \\
Y_1^0= \tilde Y_M & : M \to \Hom(V,M)\LFz & Y^1_1=0 &: M \to \Hom(M,V)\LFz  & , 
\end{align*}
where $Y_V$ denotes the vertex algebra structure map for $V$, $Y_M$ the vertex module structure map for $M$, and $\tilde Y_M$ is defined by
\[ \tilde Y_M(m,z)v = e^{z T} Y_M(v,-z)m \ . \]
\end{eg}

\begin{eg}
Let $L$ be a lattice and let $\chi:L\times L \to \bb Z$ denote the symmetric bilinear pairing. Further, let $\h=L \otimes_{\bb Z} \bb C$ with the induced pairing, and $\pi=\pi(\h)$ the corresponding Heisenberg vertex algebra at level $1$. Note that each point $l\in L$ defines $\lambda_l=\chi(l,\cdot)\in \h^\vee$ and thus a vertex algebra module $\pi_l=\pi_{\lambda_l}$ over the vertex algebra $\pi$, as in Equation \ref{heismodeqn}.

The lattice vertex algebra $\Pi_L$ corresponding to $L$, when it exists, is the vertex algebra extension of $\pi$ given by
\[ \Pi_L = \bigoplus_{l\in L} \pi_l  \]
together with the intertwiners induced by the fields $V_l(z)=V_{\lambda_l}(z) \in \Hom(\pi_m, \pi_{l+m})\LFz$ for each $l,m\in L$, as in Equation \ref{expintereqn}. Thus, the lattice vertex algebra is generated by fields $J^a$ for $a\in \mc B_\h$ some basis of $\h$ as well as fields $V_l$ for each $l\in L$, with the following operator product expansions
\[ J^a(z) J^b(w) \sim \frac{\chi(a,b)}{(z-w)^2} \quad\quad  J^a(z) V_{l}(w) \sim  \frac{ \chi(a,l) V_{l}(w)}{(z-w)}\quad\quad V_l(z) V_m(w) = (z-w)^{\chi(l,m)}\nol V_l(z)V_m(w)\nor \ .  \]
\end{eg}

\subsection{Screening operators}\label{screeningsec}

Let $V$ be a vertex algebra, $v\in V$ an element of the underlying vector space, and denote the corresponding field $A(z)=Y(v,z)\in \End(V)\LFz$. The formal residue $A_0$ of the field $A(z)$ is defined by
\[ \int A(z) dz = \int \sum_{n\in \Z} A_n z^{-n-1} = A_0 \quad \in \End(V) \ . \]
The endomorphism $A_0\in \End(V)$ is a vertex algebra derivation, in the sense that it satisfies
\[ [A_0, Y(v,z)] = Y(A_0 v, z) \ ,\]
for each element $v\in V$; this follows immediately from the formula of Equation \ref{localeqn} above. In particular, this clearly implies:

\begin{corollary} The subspace $W=\ker A_0 \subset V$ is a vertex subalgebra.
\end{corollary}

Now, let $M$ be a module over a vertex algebra $V$, $m\in M$ an element of the underlying vector space, and denote the corresponding field
\[Q(z)=\tilde Y_M(m,z) \in \Hom(V,M)\LFz  \ , \]
induced by the vertex algebra structure on the square zero extension of $V$ by $M$, as in Example \ref{sqzeroeg}. Then the formal residue of the field $Q(z)$ defines a linear map
\[ Q = \int Q(z) dz \quad \in \Hom(V,M) \]
such that $W=\ker(Q)\subset V$ is a vertex subalgebra of $V$. In this context, the field $j$ is called the screening current, its formal residue $Q$ is called the screening charge, and we say that the vertex subalgebra $W$ is screened by the current $Q(z)$.

\begin{eg}\label{FFsl2eg} Let $\g=\spl_2$ so that the principal affine $W$-algebra $W_\kappa(\spl_2)\cong \Vir_{c(\kappa)}$ is isomorphic to the Virasoro vertex algebra of central charge
\[ c(\kappa) = 1- 6 \frac{(\kappa+1)^2}{(\kappa+2)} \ . \]

The Heisenberg algebra $\pi=\pi^k$ at level $k$ admits Virasoro currents
\[ T^d(z)= \frac{1}{2 k} \nol b(z) b(z) \nor + \frac{d}{\sqrt{k}} \del_z b(z) \]
of central charge $c(d)=1-12d^2$ for each $d\in \bb K$, defining a Virasoro subalgebra $\Vir_{c(d)} \to \pi$. In fact, this subalgebra is the kernel of a screening operator
\[ Q_\beta=\int V_\beta(z)dz : \pi \to \pi_\beta  \ ,\]
for some appropriately chosen $\beta$, conditions on which we determine below. Indeed, the operator product expansions against each term are given by
\[  \frac{1}{2 k} \nol b(z) b(z) \nor V_\beta(w) \sim \frac{k \beta^2}{2} \frac{V_\beta(w)}{(z-w)^2} + \frac{\del_w V_\beta(w)}{z-w} \quad\quad \text{and}\quad\quad  \frac{d}{\sqrt{k}} \del_z b(z) V_\beta(w) \sim - d \sqrt{k} \beta \frac{V_\beta(w)}{(z-w)^2} \ , \]
so that after integrating by parts we have
\[ T^d(z) V_\beta(w) \sim \del_w\left( \frac{V_\beta(w)}{z-w}\right) +  \left( \frac{ k \beta^2}{2} - d\sqrt{k}\beta  - 1\right)\frac{V_\beta(w)}{(z-w)^2}  \ .\]
Thus, we find that $\Vir_{c(d)}\subset \ker Q_\beta$ if and only if $ \frac{ k \beta^2}{2} - d\sqrt{k}\beta - 1=0$, which uniquely determines
\[ d = \frac{\sqrt{k}\beta}{2}-\frac{1}{\sqrt{k}\beta}  \quad\quad \textup{and thus} \quad\quad T^\beta(z) = \frac{1}{2 k} \nol b(z) b(z) \nor + \left( \frac{\beta}{2} - \frac{1}{k\beta}   \right) \del_z b(z)  \  \]
is the induced form of $T^d(z)$. Note that this expression is evidently invariant under $\beta\mapsto -\frac{2}{k\beta}$.

In particular, the Feigin-Frenkel screening operator for $\spl_2$ at level $\kappa$ identifies with the residue $Q_\beta$ of the exponential vertex operator $V_\beta(z)$ for $\beta=-\sqrt{ \frac{2}{k(\kappa+2)}}$, which indeed screens a Virasoro subalgebra of central charge $c(\kappa)$ in $\pi^k$.

\end{eg}

Generalizing the preceding example, we have the following seminal result of Feigin-Frenkel \cite{FF1}, which provides a free field realization of the principal affine $W$-algebra of $\gl_r$. The screening operator constructions for the family of vertex algebras $\V(Y,S)$ from divisors $S$ on toric Calabi-Yau threefolds $Y$ defined in Section \ref{screensec} below should be understood as a broad generalization of this result, which corresponds to the case $Y=\C^3$ and $S=r[\C^2]$, as we will soon explain:

\begin{theo}\cite{FF1}\label{FFtheo} There is a canonical embedding of vertex algebras
\[ W^\kappa_{f_\prin}(\g) \to \pi_{\mf h} \]
with image characterized as the intersection of kernels
\[W^\kappa_{f_\prin}(\g) = \bigcap_{i=1}^r \ker(Q_{\alpha_i}) \quad\quad \textup{for}\quad\quad Q_{\alpha_i}= \int V_{\alpha_i}(z) \ dz \ : \pi_{\mf h} \to \pi_{\mf h, \alpha_i} \]
where $\pi_{\mf h, \alpha_i}$ and $V_{\alpha_i}(z)$ are the Heisenberg algebra module and exponential vertex operator determined by the positive simple roots $\alpha_i$ for $i=1,...,r=\rk(\g)$, as in Equations \ref{heismodeqn} and \ref{vertexopeqn}, respectively.
\end{theo}

The geometric interpretation of this result provided by the proof of the AGT conjecture in \cite{SV}, \cite{MO}, and \cite{BFN5} is the primary motivation for the construction of $\V(Y,S)$ explained in the succeeding section below. We provide an analogous interpretation for the generalized Wakimoto realization of arbitrary affine $W$-algebras in Section \ref{Walgsec} below.

\section{The screening operator construction of $\V(Y,S)$}\label{screensec}

Let $Y$ be a resolution of an affine, toric Calabi-Yau threefold $X$ satisfying the hypotheses of the companion paper \cite{BR1}. Denote by $(\bb C^\times)^3$ the torus acting on $Y$, $ T=(\bb C^\times)^2 \subset (\bb C^\times)^3$ the maximal subtorus that preserves the Calabi-Yau structure, and define $R=H_T^\bullet(\pt)$, $K=\tilde H^\bullet_T(\pt) $ its field of fractions, so that we have
\begin{equation}\label{baseringeqn}
	 R = \bb K[\e_1,\e_2,\e_3]/I_T  \quad\quad \text{and}\quad\quad  K = \bb K(\e_1,\e_2,\e_3)/ I_T  \ , 
\end{equation}
where $I_T$ denotes the ideal generated by weights of $(\bb C^\times)^3$ which vanish on $T$. We choose conventions such that $I_T=(\e_1+\e_2+\e_3)$ and thus $\e_3=-\e_1-\e_2$, and further we let
\begin{itemize}
	\item $\mf D$ be the set of irreducible toric surfaces $S_d$ in $Y$,
	\item $\mf C$ the set of compact, irreducible toric curves $C_i$ in $Y$,
	\item $\mf S$ the set of non-compact, irreducible toric curves $C_s$ in $Y$, and
	\item $\mf F$ the set of torus fixed points $y$ in $Y$.
\end{itemize}

\noindent 
For each irreducible toric surface $S_d$ corresponding to $d\in \mf D$, denote by $\mf C_d \subset \mf C$ the set of compact, irreducible toric curves $C_i$ in $S_d$, and similarly $\mf S_d\subset \mf S$ and $\mf F_d\subset \mf F$ the sets of non-compact, irreducible toric curves and torus fixed points in $S_d$. Similarly, for each toric curve $C_i$ or $C_s$ corresponding to $i\in\mf C$ or $s \in \mf S$, let $\mf F_i$ or $\mf F_s\subset \mf F$ denote the set of torus fixed points in $C_i$ or $C_s$. In particular, note that for $i\in \mf C$ we have that $\mf F_i=\{0_i,\infty_i\}$ is given by two points corresponding to $0$ and $\infty$ under the identification $C_i\cong \bb P^1$.

Let $S$ be a toric effective Cartier divisor, with decomposition into irreducible components given by
\begin{equation}\label{divisoreqn}
	 S = \sum_{d\in \mf D} r_d\  [ S_d ] \ , 
\end{equation}
for some tuple of non-negative integers $\rr=(r_d) \in \bb N^{\mf D}$.

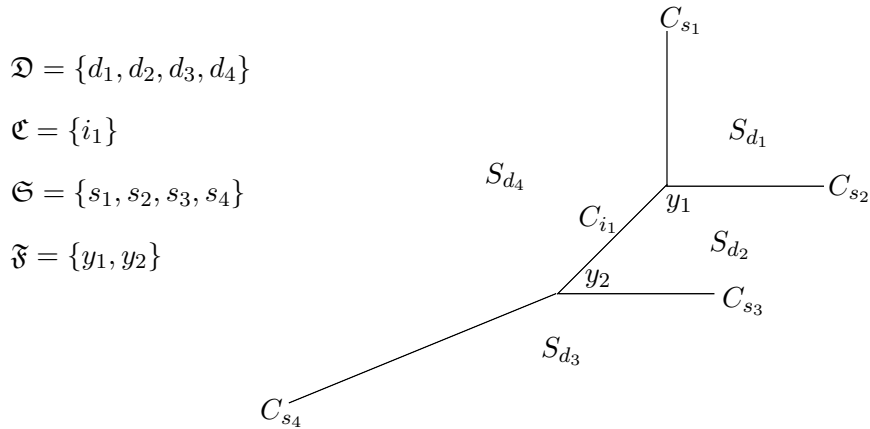
\begin{figure}[b]
		\caption{The resolution $ Y_{2,0}\to X_{2,0}=\{xy-z^2\} \times \bb C $ and its toric subvarieties}
	\begin{overpic}[width=.5\textwidth]{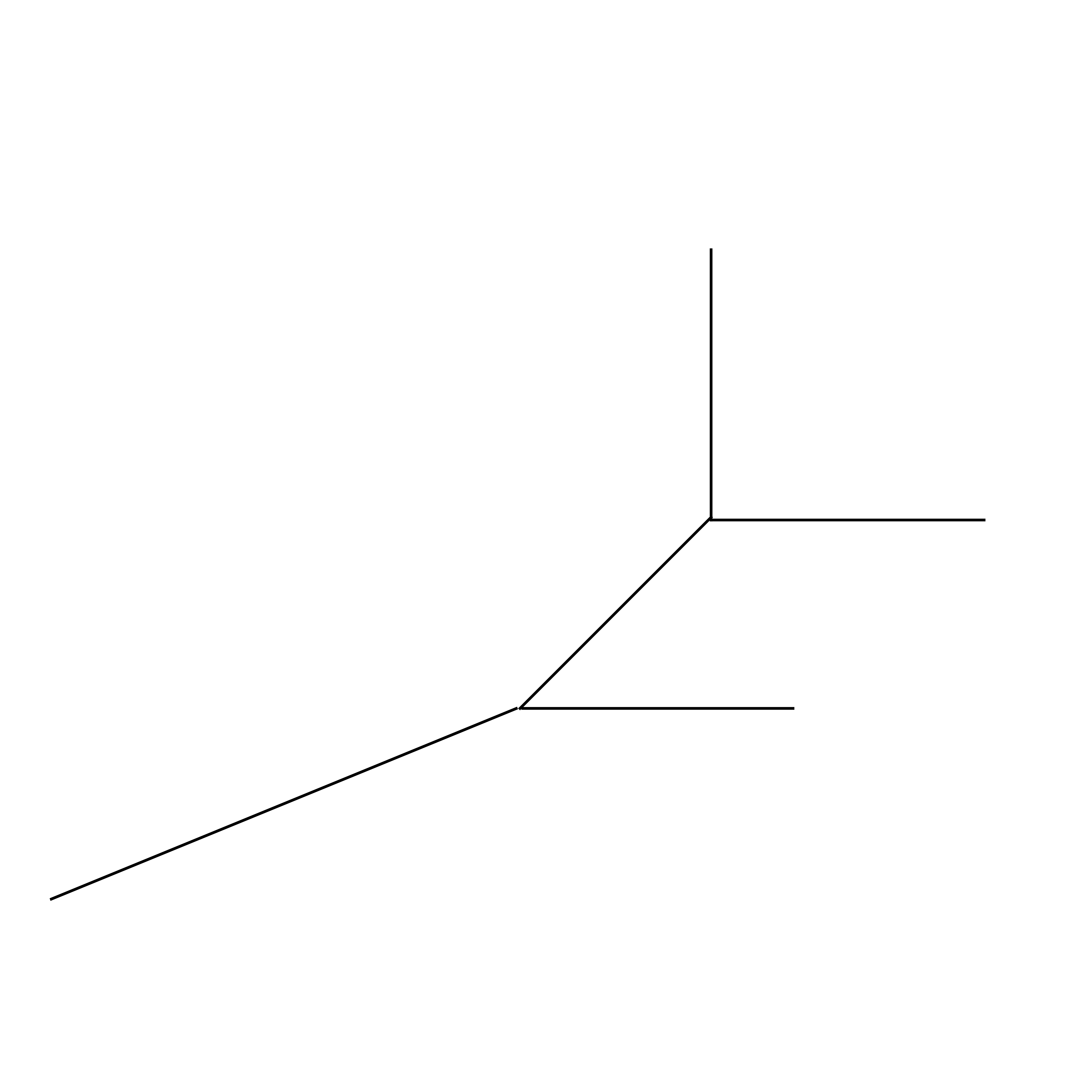}
		\put(75,60) {$S_{d_1} $}
		\put(72,42) {$S_{d_2} $}
		\put(45,25) {$S_{d_3} $}
		\put(36,53) {$S_{d_4} $}
		\put(51,46) {$ C_{i_1}$}
		\put(65,49) {$y_1$}
		\put(52,37) {$y_2$}
		\put(91,51) {$C_{s_2}$}
		\put(74,33) {$C_{s_3}$}
		\put(0,15) {$C_{s_4}$}
		\put(64,78) {$C_{s_1}$}

		\put(-40,70) {$ \mf D = \{ d_1,d_2,d_3,d_4\} $}
		\put(-40,60) {$ \mf C = \{ i_1\} $}
		\put(-40,50) {$ \mf S = \{ s_1,s_2,s_3,s_4\} $}
		\put(-40,40) {$ \mf F = \{ y_1,y_2\} $}
		
	\end{overpic}
	\label{fig:toriclabelsfig}
\end{figure}

\subsection{Lattice vertex algebras associated to irreducible divisors}\label{ablatticesec}

For each irreducible toric surface $S_d$ in $Y$, we define the Heisenberg algebra
\[\pi^d = \bigotimes_{y\in \mf F_d} \pi^y \quad\quad \text{where each} \quad\quad \pi^y = \pi^{k_y}=\mc U_{k_y}(\hat\gl_1) \otimes_{\mc U_{k_y}(\hat\gl_1)_+} K \]
denotes a standard one dimensional Heisenberg algebra over the base field $K$, at level
\begin{equation}\label{Heisleveleqn}
	 k_y=-\frac{1}{\e_T(T_yS_d)} \in K 
\end{equation}
given by the inverse of the $T$-equivariant Euler class $\e_T$ of the tangent space $T_yS_d$ to $S_d$ at $y$. We also define the abelian Lie algebra $\h_d$ over $K$ by
\[ \h_d = \tilde H_T^\bullet(S_d^T)= \bigoplus_{y \in \mf F_d} K_y \quad\quad \text{where}\quad\quad  K_y= \tilde H^\bullet_T(y)  \]
denotes the $T$ equivariant cohomology of the point $y$, equipped with the pairing
\[ (\cdot,\cdot): \h_d\otimes_K \h_d \to K \quad\quad \text{defined by} \quad\quad (1_y,1_{y'}) = \begin{cases} k_y & \textup{if $y=y'$} \\ 0 & \textup{otherwise}\end{cases} \ , \]
and we have the equivalent definition
\[ \pi^d =  V_1(\h_d) = \mc U_1(\hat\h_d) \otimes_{\mc U_1(\hat\h_d)_+}  K \cong K [ b_n^y]_{n\leq -1}^{y\in \mf F_d} \one \ ,\]
where we recall that ${\mc U_1(\hat\h_d)_+}\subset {\mc U_1(\hat\h_d)}$ denotes the commutative subalgebra
\[  {\mc U_1(\hat\h_d)_+} = K[b_n^y]_{n\geq 0 }^{y\in\mf F_d} \ . \]

For each compact, irreducible toric curve $C_i$ in $S_d$ corresponding to $i\in \mf C_d$, let $N_{C_i}S_d$ denote the normal bundle to $C_i$ in $S_d$, and define
\[ \lambda_i : \h_d \to K \quad\quad \text{by}\quad\quad \lambda_i(1_y) = \begin{cases} \e_T(N_{C_i,y} S_d) & \textup{if $y\in \mf F_i$} \\
	0 & \textup{otherwise} \end{cases}	 \ ,   \]
where we recall that $\mf F_i = \{ 0_i,\infty_i\} \subset \mf F$ denotes the set of fixed points $y\in Y^T$ contained in the curve class $C_i$ corresponding to $i \in \mf C$. Thus, for each $l\in \bb Z^{\mf C_d}$ there is a vertex algebra module over $\pi^d$ defined by
\[ \pi^d_{l} = \mc U_1(\hat \h_d) \otimes_{\mc U_1(\hat\h_d)_+}  K_{l} \cong K [ b_n^y]_{n\leq -1}^{y\in \mf F_d} \one_{l} \]
where $K_{l}$ denotes the one dimensional $\mc U_1(\hat\h_d)_+$ module on which $b_0^y$ acts by $\lambda_l(1_y)=\sum_i l_i\lambda_i(1_y)$ and $b_n^y$ acts by zero for all $n\geq 1$.

We define for each $i\in \mf C_d$ and $m\in \bb Z^{\mf C_d}$, the corresponding fields
\begin{equation}\label{curveextopeqn}
	 V^d_{i}(z) = \nol \exp \left( \e_T(N_{C_i,0_i} S_d) \phi^{0_i}(z) + \e_T(N_{C_i,\infty_i}S_d) \phi^{\infty_i}(z) \right) \nor  \quad  \in \Hom( \pi^d_{m}, \pi^d_{m+[i]})\LFz \ , 
\end{equation}
and more generally, for $l=\sum_i l_i [i] \in \bb Z^{\mf C_d}$ where $[i]\in \bb Z^{\mf C_d}$ denotes the generator corresponding to $i\in \mf C_d$, the fields
\[V^d_{l}(z)  = \nol  \prod_{i\in \mf C_d} V^d_i(z)^{l_i} \nor = \nol \exp( \sum_{y\in \mf F_d} \lambda_l(1_y) \phi^y(z) ) \nor   \in \Hom( \pi^d_{m}, \pi^d_{m+l})\LFz \ , \]
where $\phi^y$ denotes the formal bosonic fields corresponding to the Heisenberg fields $J^y$ which generates the subalgebras $\pi^y$ of $\pi^d$ for each torus fixed point $y\in \mf F_d$, and we recall that $\mf F_i=\{0_i,\infty_i\} \subset \mf F_d$ denotes the subset of the fixed points which are contained in the compact curve class $C_i$ for each $i\in \mf C_d$; again, we remind the reader that the expression in each of the preceding equations is a shorthand notation for the fields defined as in Equations \ref{vertexopeqn} and \ref{expNOeqn}.

In summary, we obtain fields $J^y(z)$ for each $y\in \mf F_d$ and $V_i(z)$ for each $i\in \mf C_d$ satisfying
\begin{equation}\label{geoheisopeeqn}
 J^y(z) J^{y'}(w) \sim -\frac{\delta_{y=y'}}{\e_T(T_yS_d)} \frac{1}{(z-w)^2} \quad\quad  J^y(z) V_{i}^d(w) \sim  \left( \frac{\e_T(N_{C_i,y} S_d) }{\e_T(T_yS_d)}\delta_{y=0_i} +\frac{\e_T(N_{C_i,y} S_d) }{\e_T(T_yS_d)}\delta_{y=\infty_i} \right) \frac{ V_{i}(w)}{z-w}
\end{equation}
as well as for each pair $i,j\in \mf C_d$ the relation
\[ V_{i}^d(z) V_{j}^d(w) = (z-w)^{\chi(i,j)} \nol  V_{i}^d(z) V_{j}^d(w) \nor  \quad\quad \text{where}\quad\quad \chi(i,j) = \sum_{y\in \mf F_i, y'\in \mf F_j} \delta_{y=y'} \frac{\e_T(N_{C_i,y}S_d)\e_T(N_{C_j,y}S_d)}{\e_T(T_yS_d)} \ . \]

In fact, we have the following elementary identification, which motivates the above conventions:
\begin{prop}\label{pairingprop} Let $i,j\in \mf C_d$ corresponding to compact, irreducible toric curves $C_i,C_j$ in $S_d$. Then
\[\chi(i,j) = - \langle [C_i] , [C_j] \rangle_{S_d} \ \in \bb Z\]
where $[C_i]\in H_2(S_d;\Z)$ denotes the fundamental class of $C_i$ and $\langle\cdot,\cdot\rangle_{S_d}:H_2(S_d;\Z)^{\times 2}\to \Z$ denotes the intersection pairing.
\end{prop}

For each irreducible divisor $S_d$ for $d\in \mf D$ we define the total Heisenberg subalgebra $\pi_d$ by
\begin{equation}\label{Heiscompeqn}
	 \pi_d = \pi^d_0 = \otimes_{y \in \mf F_d}  \pi^y  \ . 
\end{equation}
Moreover, we introduce the $\pi_d$ module
\[ \Pi_d = \bigoplus_{l\in \bb Z^{\mf C_d}} \pi^d_l \cong  \bigoplus_{l\in \bb Z^{\mf C_d}} K [ b_n^y]_{n\leq -1}^{y\in \mf F_d} \one_{l}  \]

\noindent and extend the above fields to define vertex algebra intertwiners
\[ Y(b_{n_1}^{y_1} \cdots b_{n_j}^{y_j} \one_{l},z) =  \nol  \left( \prod_{k=1}^j \frac{1}{(-n_k-1)!}\del_z^{-n_k-1}b^{y_k}(z)  \right) V_l^d(z)\  \nor \quad \in \Hom(\pi^d_m , \pi^d_{l+m})\LFz \ ,\]
which define a lattice type vertex algebra on $\Pi_d$, noting the required integrality follows from Proposition \ref{pairingprop}, and moreover by \emph{loc. cit.} we have:

\begin{corollary} There is a natural isomorphism
\[	\Pi_d \cong \pi_K \otimes_K \left( V_{H_2(S_d;\bb Z),\chi}\otimes_{\bb C} K\right) \ . \]
\end{corollary}

\begin{defn} Let $S$ be a toric divisor in $Y$ given by the decomposition of Equation \ref{divisoreqn}. We define the free field vertex algebra $\Pi(Y,S)$ associated to $Y$ and $S$ to be
	\[ \Pi(Y,S) = \bigotimes_{d\in \mf D} \Pi_d^{\otimes r_d}  = \bigotimes_{d\in \mf D,j_d=1,...,r_d}\Pi_d^{j_d} \ ,\]
and similarly define the total Heisenberg subalgebra $\pi_S$ to be
\[ \pi_S = \bigotimes_{d\in \mf D} \pi_d^{\otimes r_d} = \bigotimes_{d\in \mf D, \ j_d=1,...,r_d}  \pi_d^{j_d} \  . \]
\end{defn}

It will also be useful to introduce the following more detailed notation to describe the vertex algebra $\Pi(Y,S)$: define the abelian Lie algebra
\[ \mf h_S = \bigoplus_{d\in \mf D} \tilde H^\bullet_T(S_d^T)^{\oplus r_d} =  \bigoplus_{d\in \mf D, y\in \mf F_d} K_y^{\oplus r_d} = \bigoplus_{d\in \mf D, y\in \mf F_d,   j_d=1,...,r_d} K_y^{j_d}  \]
equipped with the pairing
\[ (\cdot,\cdot): \h_S\otimes_K \h_S \to K \quad\quad \text{defined by} \quad\quad (1_y^{j_d},1_{y'}^{j_{d'}} ) = \begin{cases} k_y & \textup{if $d=d'$, $y=y'$ and $j_d=j_{d'}$} \\ 0 & \textup{otherwise}\end{cases} \ , \]
and we have the equivalent definition
\begin{equation}\label{HeisSeqn}
	 \pi_S =  V_1(\h_S) = \mc U_1(\hat\h_S) \otimes_{\mc U_1(\hat\h_S)_+}  K \cong K [ b_n^{y,j_d}]_{n\leq -1}^{d \in \mf D, y\in \mf F_d,j_d=1,...,r_d} \one \ . 
\end{equation}

To simplify this notation, we introduce the index set
\[ \mf F_S = \bigsqcup_{d\in \mf D} \bigsqcup_{j=1,...,r_d} \mf F_d \ , \]
for which the elements $y\in \mf F_S$ can be interpreted as $T$-fixed points in the disjoint union of the irreducible components of the divisor $S$, in terms of which we have simply
\[\mf h_S = \bigoplus_{y\in \mf F_S} K_y \quad\quad \text{and}\quad\quad \pi_S = \bigotimes_{y\in \mf F_S} \pi_y=\bigotimes_{y\in \mf F_S} K[b_n]_{n\leq 1}\one_y  = K[b_n^y]_{n\leq 1}^{y\in \mf F_S} \one  \ . \]

\begin{figure}[b]
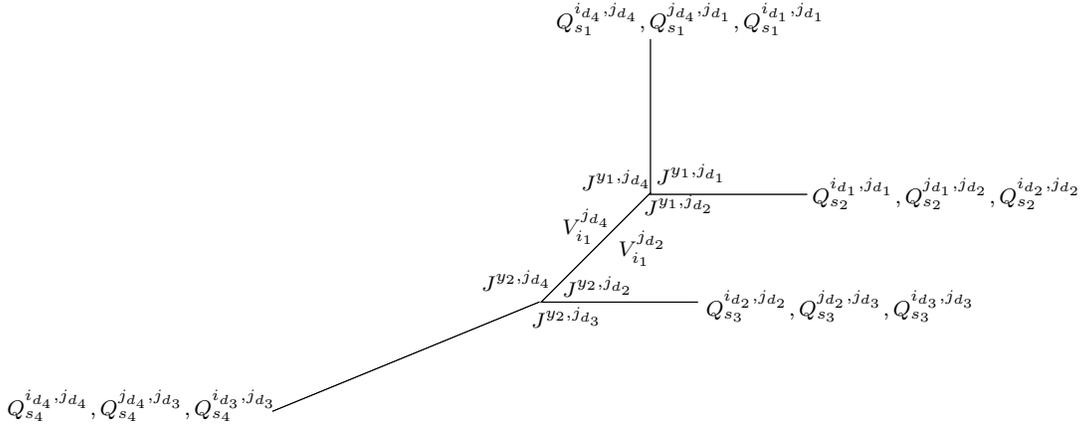

	\caption{The resolution $ Y_{2,0}\to X_{2,0}=\{xy-z^2\} \times \bb C $ and its vertex algebra data}
	\begin{overpic}[width=.5\textwidth]{picOm2}
		\put(51,46) {\scriptsize{$ V_{i_1}^{j_{d_4}}$}}
		\put(60,42.5) {\scriptsize{$ V_{i_1}^{j_{d_2}}$}}
		\put(66,54) {\scriptsize{${J^{y_1,j_{d_1}}}$}}
		\put(64,49) {\scriptsize{${J^{y_1,j_{d_2}}}$}}
		\put(54,53) {\scriptsize{${J^{y_1,j_{d_4}}}$}}
		\put(51,36) {\scriptsize{$J^{y_2,j_{d_2}}$}}
		\put(46,31) {\scriptsize{$J^{y_2,j_{d_3}}$}}
		\put(38,37) {\scriptsize{$J^{y_2,j_{d_4}}$}}
		\put(91,51) {\scriptsize{$Q_{s_2}^{i_{d_1},j_{d_1}},Q_{s_2}^{j_{d_1},j_{d_2}},Q_{s_2}^{i_{d_2},j_{d_2}}$}}
		\put(74,33) {\scriptsize{$Q_{s_3}^{i_{d_2},j_{d_2}},Q_{s_3}^{j_{d_2},j_{d_3}},Q_{s_3}^{i_{d_3},j_{d_3}}$}}
		\put(-38,17) {\scriptsize{$Q_{s_4}^{i_{d_4},j_{d_4}},Q_{s_4}^{j_{d_4},j_{d_3}},Q_{s_4}^{i_{d_3},j_{d_3}}$}}
		\put(50,79) {\scriptsize{$Q^{i_{d_4},j_{d_4}}_{s_1},Q^{j_{d_4},j_{d_1}}_{s_1},Q^{i_{d_1},j_{d_1}}_{s_1}$}}
	\end{overpic}
	\label{fig:vertexlabelsfig}
\end{figure}

Further, for each $d \in \mf D$, $j_d\in\{1,...,r_d\}$, and $i\in \mf C_d$ corresponding to a compact, irreducible curve $C_i$ contained in (the $j_d^{\textup{th}}$ copy of) $S_d$, we define
\[ \lambda_i^{j_d} : \mf h_S \to K \quad\quad \text{by}\quad\quad \lambda_i^{j_d} ( \one_y^{j_{d'}}) =  \begin{cases} \e_T(N_{C_i,y} S_d) & \textup{if $d=d'$, $y\in \mf F_i$ and $j_d=j_{d'}$} \\
	0 & \textup{otherwise} \end{cases}	 \ ,   \]
where we recall that $\mf F_i = \{ 0_i,\infty_i\} \subset \mf F$ denotes the set of fixed points $y\in Y^T$ contained in the curve class $C_i$, and that $\e_T(N_{C_i,y} S_d)$ denotes the equivariant Euler class of the fibre of the normal bundle to $C_i$ in $S_d$ at the point $y$.

To simplify this notation, we introduce the index set
\[ \mf C_S = \bigsqcup_{d\in \mf D} \bigsqcup_{j=1,...,r_d} \mf C_d \ , \]
for which the elements $i\in \mf C_S$ can be interpreted as compact, irreducible toric curves $C_i$ in the disjoint union of the irreducible components of the divisor $S$, and note that for each $i \in \mf C_S$ we have the corresponding subset $\mf F_i = \{ 0_i, \infty_i\} \subset \mf F_S$. In terms of this notation, for each $i\in \mf C_S$ we can equivalently define 
\begin{equation}\label{Cscochareqn}
	 \lambda_i : \mf h_S \to K \quad\quad \text{by} \quad\quad \lambda_i ( \one_y) =  \begin{cases} \e_T(N_{C_i,y} S_d) & \textup{$y\in \mf F_i$} \\
	0 & \textup{otherwise} \end{cases}	 \ ,   
\end{equation}
and moreover for each $l\in \bb Z^{\mf C_S}$ we can define the corresponding module
\[ \pi_{S,l} = \mc U_1(\hat{\mf h}_S)\otimes_{\mc U_1(\hat{\mf h}_S)_+} K_l \cong K[b_n^y]_{n\leq -1}^{y\in \mf F_S} \one_l  \]
where $K_l$ denotes the one dimensional $\mc U_1(\hat{\mf h}_S)_+$ module on which $b_0^y$ acts by $\lambda_l(\one_y)=\sum_{i\in \mf C_S} l_i \lambda_i(\one_y)$ for each $y\in \mf F_S$. In terms of these modules, we can decompose $\Pi(Y,S)$ as a module over $\pi_S$ by
\[ \Pi(Y,S) = \bigoplus_{l\in \bb Z^{\mf C_S}} \pi_{S,l} \ .  \]

\subsection{Screening operators associated to curve classes}\label{geoscreensec}

For each non-compact, irreducible toric curve $C_s$ in $Y$ for $s\in \mf S$, and each pair $S_{d_+}$ and $S_{d_-}$ of irreducible components of the divisor $S$ corresponding to $d_+,d_-\in \mf D$, $j_{d_+}\in\{1,...,r_{d_+}\}$ and $j_{d_-}'\in\{1,...,r_{d_-}\}$, we also define
\[ \lambda_s^{j_{d_+},j_{d_-}}: \mf h_S \to K \quad\quad \textup{by}\quad\quad  \lambda_s^{j_{d_+},j'_{d_-}}( \one_y^{j_e}) =  \begin{cases} \e_T(N_{C_i,y} S_{d_+}) & \textup{if $d_+=e$, $j_{d_+}=j_e$, and $y\in \mf F_i \cap \mf F_{d_+}$} \\
-\e_T(N_{C_i,y} S_{d_-}) & \textup{if $d_-=e$, $j_{d_-}=j_e$, and $y\in \mf F_i\cap \mf F_{d_-}$} \\
	0 & \textup{otherwise} \end{cases}	 \ .  \]
We also write simply $\lambda_s=\lambda_s^{j_{d_+},j_{d_-}}$ when there is no risk of confusion. We define the analogous module shifted by $\lambda_s$ for each $l\in \bb Z^{\mf F_S}$ by
\[ \pi_{S,l+\lambda_s} = \mc U_1(\hat{\mf h}_S)\otimes_{\mc U_1(\hat{\mf h}_S)_+} K_{l+\lambda_s} \cong K[b_n^y]_{n\leq -1}^{y\in \mf F_S} \one_{l+\lambda_s} \]
where $K_{l+\lambda_s}$ denotes the one dimensional $\mc U_1(\hat{\mf h}_S)_+$ module on which $b_0^y$ acts by $\lambda_l(\one_y)+\lambda_s^{j_{d_+},j_{d_-}}(\one_y)$, and introduce the corresponding field $Q_s^{j_{d_+},j_{d_-}}(z) \in \Hom( \pi_{S,l}, \pi_{S,l+\lambda_s})\LFz$ defined by
\[Q_s^{j_{d_+},j_{d_-}}(z) =\nol \exp ( \sum_{y\in\mf F_S} \lambda_s^{j_{d_+},j_{d_-}}(\one_y) \phi_y(z) ) = \nol \exp \left( \e_T(N_{C_s,y_s} S_{d_+}) \phi^{j_{d_+}}_{y_s}(z) - \e_T(N_{C_s,y_s}S_d) \phi^{j_{d_-}}_{y_s}(z) \right)\nor  \quad  \ , \] 
where we note that for each $s\in \mf S$ the non-compact irreducible toric curve $C_s$ contains at most a unique fixed point which we have denoted $y_s$. This field defines a map of $\pi_S$ modules, and extends to define a map of vertex algebra modules
\[W_s^{j_{d_+},j_{d_-}}(z) \ \in \Hom( \Pi(Y,S) , \Pi(Y,S)_{\lambda_s})\LFz \quad\quad \text{where}\quad\quad \Pi(Y,S)_{\lambda}= \bigoplus_{l\in\bb Z^{\mf C_S}} \pi_{S,l+\lambda} \]
is the $\Pi(Y,S)$ module defined for each $\lambda \in \mf h_S^\vee$, and we note that $\Pi(Y,S)_0=\Pi(Y,S)$ is the usual vacuum module for the lattice vertex algebra. Thus, we obtain a putative screening operator
\[ Q_s^{j_{d_+},j_{d_-}} = \int Q_s^{j_{d_+},j_{d_-}}(z) dz \ : \Pi(Y,S) \to \Pi(Y,S)_{\lambda_s} \ , \]
for each non-compact, irreducible toric curve $C_s$ in $Y$ for $s\in \mf S$, and each pair $S_{d_+}$ and $S_{d_-}$ of irreducible components of the divisor $S$.

\subsection{The vertex algebras $\V(Y,S)$, factorization, and locality}\label{VOAYSsec}

Let $S$ be a effective, toric Cartier divisor, with decomposition into irreducible, reduced components given by
\begin{equation}
 S^\red = \bigsqcup_{d\in \mf D_S} S_d \quad\quad \text{so that}\quad\quad 	[S] = \sum_{d\in \mf D} r_d\  [ S_d ] \ , 
\end{equation}
for some tuple of non-negative integers $\rr_S=(r_d) \in \bb N^{\mf D}$. Further, fix a filtration $\mc F$ of the structure sheaf $\mc O_S$ of the divisor $S$, with subquotients $\mc F^k/\mc F^{k-1}$ given by structure sheaves of the irreducible components $\mc O_{S_{d_k}}$, and such that each induced extension $\mc O_{S_{d_k}} \to E \to \mc O_{S_{d_{k+1}}}$ is non-trivial, denoted
\[ \mc O_S = \left[ \mc O_{S_{d_1}} < ... <\mc O_{S_{d_k}} < ... < \mc O_{S_{d_N}} \right] \ , \]
for some ordered list of elements $d_k\in \mf D$ in which each element $d\in \mf D$ occurs $r_{d}$ times; we interpret the index $k$ as an element of the ordered set $\{1,..., N\}$ where $N=\sum_{d \in \mf D} r_d$ is the total number of irreducible components.

Now, for each $k=1,...,N-1$, and each non-compact, irreducible toric curve $C_s$ contained in $S_{d_k}\cap S_{d_{k+1}}$, which corresponds to a point $s_k \in \mf S_{d_k}\cap \mf S_{d_{k+1}}$, let $y\in \mf F_s$ denote the unique toric fixed point contained in $C_s$, and note that it defines two distinct elements $y_k$ and $y_{k+1} \in \mf F_S$ coming from the fixed points in $S_{d_k}$ and $S_{d_{k+1}}$, respectively. We let $\lambda_{s_k}:\mf h_S \to K$ denote the corresponding cocharacter and define the field $Q_{s_k} \in \Hom( \Pi(Y,S), \Pi(Y,S)_{\lambda_{s_k}})\LFz$ by
\[ Q_{s_k}(z) =\nol \exp ( \sum_{y\in\mf F_S} \lambda_{s_k}(\one_y) \phi_y(z) ) = \nol \exp \left( \e_T(N_{C_s,y} S_{d_k}) \phi_{y_k}(z) - \e_T(N_{C_s,y}S_{d_{k+1}}) \phi_{y_{k+1}}(z) \right)\nor  \ , \]
as above, and similarly the corresponding screening operator
\[ Q_{s_k}=\int Q_{s_k}(z)dz \ : \Pi(Y,S) \to \Pi(Y,S)_{\lambda_{s_k}} \ . \]
In summary, we define the total screening module and operator
\[ \Pi_1(Y,S) = \bigoplus_{k=1}^{N-1} \bigoplus_{s_k \in \mf S_{d_k}\cap \mf S_{d_{k+1}}} \Pi(Y,S)_{\lambda_{s_k}} \quad\quad \text{and}\quad\quad Q= \sum_{k=1}^{N-1}\sum_{s_k \in \mf S_{d_k}\cap \mf S_{d_{k+1}}}  Q_{s_k} : \Pi(Y,S) \to \Pi_1(Y,S) \]
and we can now state the main definition of this section, inspired by \cite{GaiR}, \cite{PrR} and \cite{GaiR2}:
\begin{defn} Let $Y$ be a toric Calabi-Yau threefold and $S$ be an effective, toric Cartier divisor as above. The vertex algebra $\V(Y,S)$ associated to $Y$ and $S$ is defined by
	\[ \V(Y,S) = \ker(Q) \subset \Pi(Y,S) \ .\]
\end{defn}

\noindent We now explain two key properties of this construction, which we call factorization and locality:

The factorization property is a simple consequence of the definition, which holds formally in any free field realization type construction, but it will play a crucial role in computations of examples and in the comparison with the geometric definition of $\V(Y,S)$.

To begin, we fix $k_0 \in \{ 1,..., N-1\}$ and consider the expression for $\mc O_S$ as an extension
\[ 0\to \mc O_{R} \to \mc O_S \to  \mc O_{T} \to 0 \]
induced by the filtration $\mc F$ by defining the subsheaf and $\mc O_{R}=\mc F^{k_0}$ and quotient sheaf $\mc O_{T}=\mc O_S / \mc F^{k_0}$. Then $\mc O_{R}$ and $\mc O_{T}$ are themselves structure sheaves of effective toric Cartier divisors, with induced filtrations 
\[ \mc O_R = \left[ \mc O_{S_{d_1}} < ... <\mc O_{S_{d_{k_0}}} \right] \quad\quad \text{and}\quad\quad \mc O_T  = \left[ \mc O_{S_{d_{k_0+1}}} < ...  < \mc O_{S_{d_N}} \right]  \ ,  \]
and thus define vertex algebras $\V(Y,R)$ and $\V(Y,S)$, respectively. The factorization property is given by the following proposition:

\begin{prop}\label{factprop} There is a canonical embedding of vertex algebras
\[ \V(Y,S) \to \V(Y,R) \otimes_K \V(Y,T) \ , \]
with image equal to the kernel of a screening operator $Q_{s_{k_0}}:\V(Y,R) \otimes_K \V(Y,T) \to \Pi(Y,S)_{\lambda_{k_0}} $.
\end{prop}
\begin{proof} Let $\V(Y,S)=\ker(Q_S) \subset \Pi(Y,S)$ be as above, and similarly for $R$ and $T$, and note there is a canonical isomorphism
\[ \Pi(Y,S) = \Pi(Y,R) \otimes_K \Pi(Y,T)  \ .\]
	It is a tautological property of the free field realization construction that there is an embedding
\begin{equation}\label{FFRmaptauteqn}
		 \ker(Q_S) = \bigcap_{k=1}^{n-1} \bigcap_{s_k \in \mf S_{d_k}\cap \mf S_{d_{k+1}}} \ker(Q_{s_k}) \to \bigcap_{k\neq k_0}  \ker(Q_{s_k}) = \ker( \sum_{k\neq k_0} Q_{s_k} ) \ . 
\end{equation}

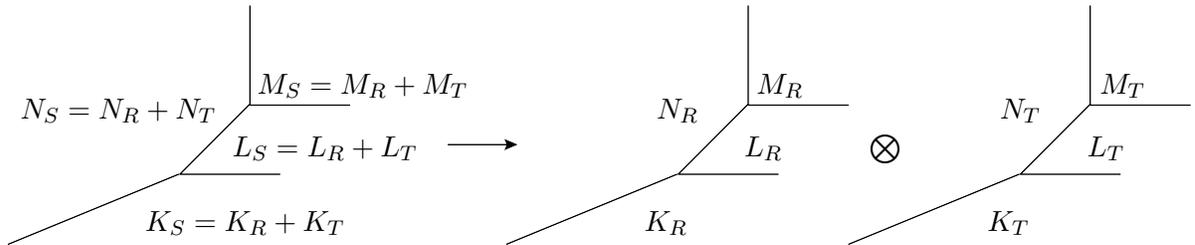
\begin{figure}[b]
	\caption{The factorization property on the resolution $ Y_{2,0}\to X_{2,0}=\{xy-z^2\} \times \bb C $}
	\begin{overpic}[width=1\textwidth]{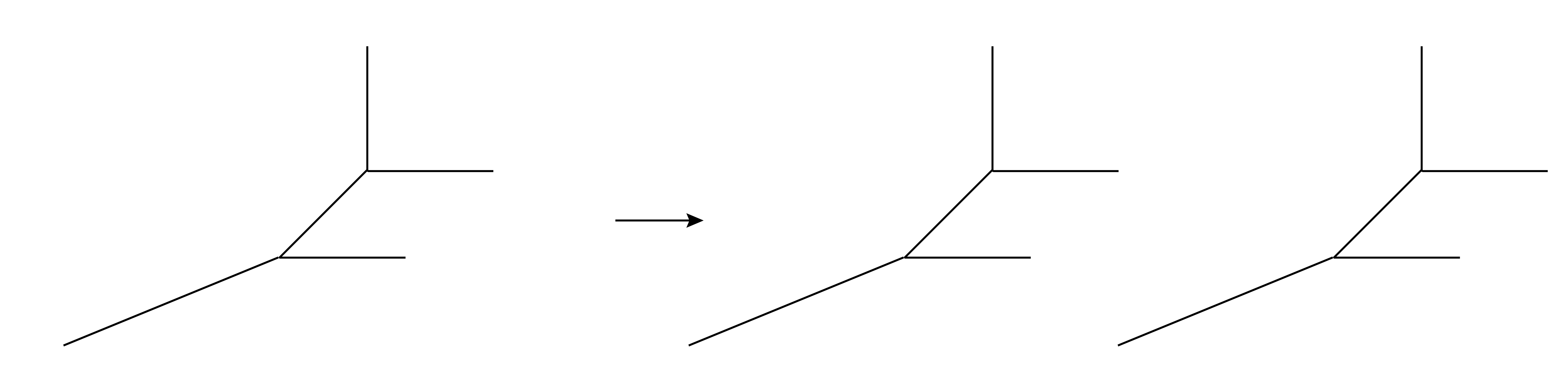}
		\put(24,15) {$M_S= M_R + M_T$}
		\put(22,10) {$L_S= L_R + L_T$}
		\put(5,13) {$N_S= N_R + N_T $}
		\put(15,4) {$K_S = K_R + K_T$}
		
		\put(64,15) {$M_R$}
		\put(63,10) {$L_R$}
		\put(56,13) {$N_R$}
		\put(55,4) {$K_R$}
		
		\put(73,10) {$\bigotimes$}
		
		\put(91.5,15) {$M_T$}
		\put(90.5,10) {$L_T$}
		\put(83.5,13) {$N_T $}
		\put(82.5,4) {$K_T$}
	\end{overpic}
	\label{fig:factfig}
\end{figure}

Similarly, we note that the decomposition
	\[ \tilde Q_R = \sum_{k=1}^{k_0-1} \sum_{\in \mf S_{d_k}\cap \mf S_{d_{k+1}}}  Q_{s_k} \quad\quad \textup{and} \quad\quad \tilde Q_T = \sum_{k=k_0+1}^N \sum_{\in \mf S_{d_k}\cap \mf S_{d_{k+1}}} Q_{s_k}\]
	 satisfies the analogous property
	\[ \ker(\tilde Q_R)\cap \ker(\tilde Q_T) = \ker( \sum_{k\neq k_0} Q_{s_k} )  \]
	and moreover we have that
	\[ \tilde Q_R|_{\Pi(Y,T)} = \tilde Q_T|_{\Pi(Y,R)} = 0 \quad\quad \text{and}\quad\quad  \tilde Q_R|_{\Pi(Y,R)} = Q_R \quad\quad \tilde Q_T|_{\Pi(Y,T)} = Q_T  \ , \]
	so that we obtain the desired map composing that of Equation \ref{FFRmaptauteqn} with the natural isomorphisms
	\[\ker(\tilde Q_R)\cap \ker(\tilde Q_T)  \cong  (\ker(Q_R)\otimes \Pi(Y,T) ) \cap (\Pi(Y,R)\otimes \ker(Q_T)) \cong \ker(Q_R) \otimes \ker(Q_T) \ . \]
	Evidently the image of $\V(Y,S)=\ker(Q_S)$ is given by the kernel of the induced screening operator
	\[ Q_{s_{k_0}}|_{\ker(\tilde Q_R)\cap \ker(\tilde Q_T)} :  \V(Y,R) \otimes_K \V(Y,T) \to  \Pi(Y,S)_{\lambda_{k_0}}   \ .  \]
\end{proof}

We now formulate and prove the locality principle: Suppose $Y$ admits a cover by two open toric subvarieties $Y_1$ and $Y_2$ each satisfying the same hypotheses required of $Y$ throughout, and such that the intersection
\[ Y_1\cap Y_2\cong  (\bb A^1\setminus \{0\}) \times \bb A^2 \]
and the unique toric curve class $C=(\bb A^1\setminus \{0\})\times \{(0,0)\}\subset Y_1\cap Y_2$ has compact closure $\overline{C}$ in $Y$. Define
\[S_1=Y_1\cap S \quad\quad\text{and}\quad\quad  S_2 = Y_2\cap S \ , \]
and consider the corresponding vertex algebras $V(Y_1,S_1)$ and $V(Y_2,S_2)$.
The locality property of the vertex algebras $\V(Y,S)$ is given by the following proposition:

\begin{prop}\label{localityprop} There is a canonical embedding of vertex algebras
	\[ \V(Y_1,S_1)\otimes \V(Y_2,S_2) \to \V(Y,S)  \]
	such that $\V(Y,S)$ is a vertex algebra extension of $ \V(Y_1,S_1)\otimes \V(Y_2,S_2)$, defined by a decomposition
	\[ \V(Y,S) = \bigoplus_{\lambda \in P_+} \V(Y_1,S_1)_\lambda \otimes \V(Y_2,S_2)_{\lambda}   \ , \]
	where $P_+ \subset \bb Z^{r}$ is defined in Equation \ref{Ppeqn} below.
\end{prop}

\begin{proof}

Note that $Y_1\cap Y_2$ contains no torus fixed points so that $\mf F_Y= \mf F_{Y_1} \sqcup \mf F_{Y_2}$, and similarly $\mf F_S = \mf F_{S_1} \sqcup \mf F_{S_2}$ so that we have
\[ \mf h_S = \mf h_{S_1} \oplus \mf h_{S_2} \quad\quad \text{and}\quad\quad  \pi_{S} = \bigotimes_{y \in \mf F_S} \pi_y = (\bigotimes_{y\in \mf F_{S_1}} \pi_y ) \otimes ( \bigotimes_{y\in \mf F_{S_2}} \pi_y) = \pi_{S_1}\otimes \pi_{S_2} \ , \]
where $\pi_{S_1}\subset \Pi(Y_1,S_1)$ and $\pi_{S_2} \subset \Pi(Y_2,S_2)$ denote total Heisenberg subalgebras.

There are only two toric divisors in $Y_1\cap Y_2$ given by
 \[ \mathring{S}_+\cong (\bb A^1\setminus \{0\})\times \{0\} \times \bb A^1 \quad\quad \text{and}\quad\quad  \mathring{S}_-\cong (\bb A^1\setminus \{0\})\times \bb A^1 \times \{0\}  \ ,\]
and their closures $S_+=S_{d_+}$ and $S_-=S_{d_-}$ correspond to some elements $d_+,d_- \in \mf D$. The compact curve class $\overline{C}$ is given by $C_i$ for some $i\in \mf C$, and corresponds to $r=r_{d_+}+r_{d_-}$ distinct points
\[ i \in   \mf C_{S,\cap}:=  (\bigsqcup_{j_{d_+}=1,...,r_{d_+}} \{ i_{d_+} \}  )  \sqcup ( \bigsqcup_{j_{d_-}=1,...,r_{d_-}} \{i_{d_-}\}) \quad \subset  \quad \mf C_S = \bigsqcup_{d\in \mf D,j_d=1,...,r_d}  \mf C_d \ \ , \]
each of which defines $\lambda_i:\mf h_S \to K$ as in Equation \ref{Cscochareqn}, where $i_{d_\pm} \in \mf C_{d_{\pm}}$ denotes the element corresponding to $\overline{C}$. Note that all of the remaining compact curve classes $C_i$ in $Y$ are contained completely in either $Y_1$ or $Y_2$, so that we have
\[ \mf C_Y = \{y\} \sqcup \mf C_{Y_1} \sqcup \mf C_{Y_2} \quad\quad \textup{and similarly} \quad\quad \mf C_S = \mf C_{S,\cap}\sqcup \mf C_{S_1} \sqcup \mf C_{S_2}  \ .\]

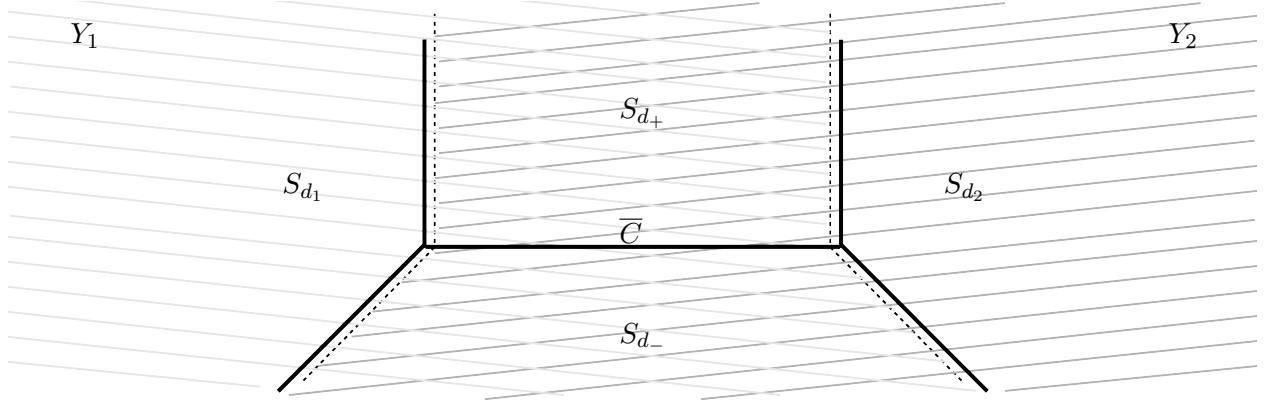
\begin{figure}[t]
	\begin{overpic}[width=1\textwidth]{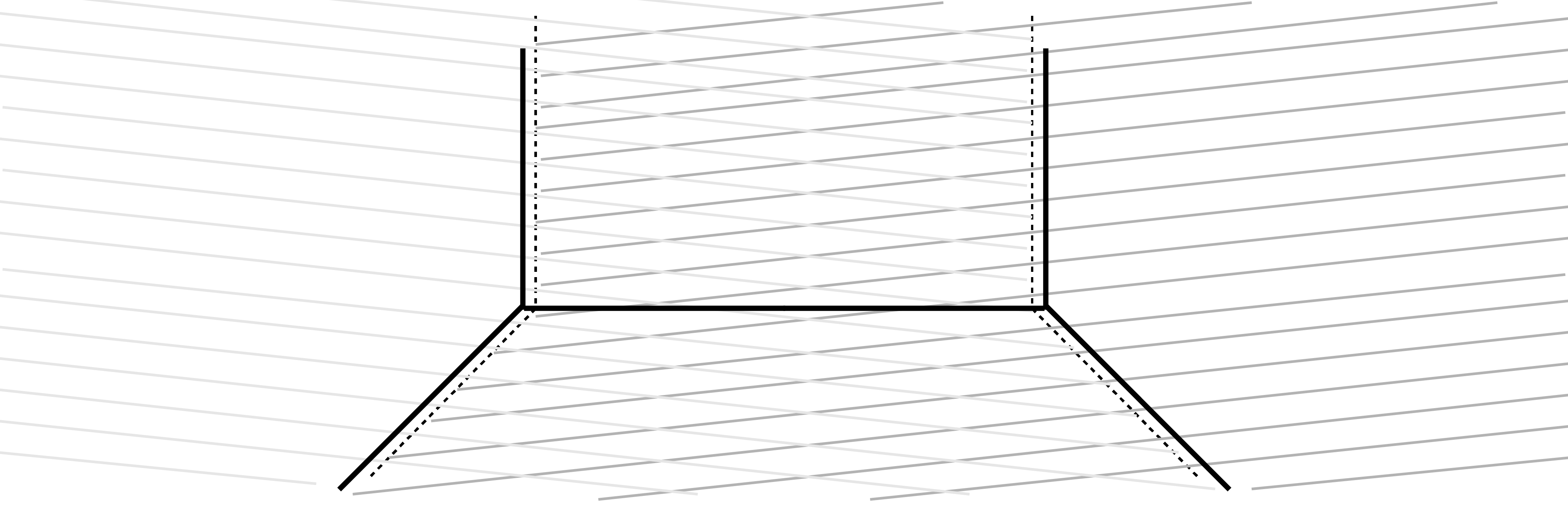}
		
		\put(5,30) {$Y_1$}
		\put(93,30) {$Y_2$}
		\put(49,24) {$ S_{d_+}$}
		\put(49,6) {$ S_{d_-}$}
		\put(22,18) {$S_{d_1}$}
		\put(75,18) {$S_{d_2}$}
		\put(49,14) {$\overline{C}$}
	\end{overpic}
	\label{fig:conlocalityfig}
		\caption{The locality principle for $Y_{2,0}\to X_{2,0}=\{xy-z^2\}\times \bb A^1 $}
\end{figure}

\noindent Thus, we can decompose
\[ \Pi(Y,S) = \bigoplus_{l\in \bb Z^{\mf C_S}} \pi_{S,l}  = \bigoplus_{ \lambda \in \bb Z^{\mf C_{S,\cap}}} \Pi(Y,S_1,S_2)_\lambda \quad\quad \text{where}\quad\quad \Pi(Y,S_1,S_2)_\lambda = \bigoplus_{ l_1 \in \mf C_{S_1}, l_2 \in \mf C_{S_2}} \pi_{S,(\lambda,l_1,l_2)}  \]
and moreover, we have a canonical identification of modules over $\pi_S =\pi_{S_1} \otimes \pi_{S_2}$ given by
\[ \Pi(Y,S_1,S_2)_0 =  \bigoplus_{ l_1 \in \mf C_{S_1}, l_2 \in \mf C_{S_2}} \pi_{S,(0,l_1,l_2)} =  \bigoplus_{ l_1 \in \mf C_{S_1}} \bigoplus_{l_2 \in \mf C_{S_2}} \pi_{S_1,l_1}\otimes \pi_{S_2,l_2}  = \Pi(Y_1,S_1) \otimes \Pi(Y_2,S_2) \ , \]
noting that $\lambda_i:\mf h_S \to K$ for $i\in \mf C_{S_1}$ vanishes on the subalgebra $\mf h_{S_2}$ and vice versa. More generally, we have the canonical identifications
\[ \Pi(Y,S_1,S_2)_\lambda = \bigoplus_{ l_1 \in \mf C_{S_1}} \bigoplus_{l_2 \in \mf C_{S_2}} \pi_{S_1,l_1+\lambda}\otimes \pi_{S_2,l_2+\lambda}  = \Pi(Y_1,S_1)_\lambda \otimes \Pi(Y_2,S_2)_\lambda \]
where we have denoted both the restrictions $\lambda|_{\mf h_{S_1}}:\mf h_{S_1}\to K$ and  $\lambda|_{\mf h_{S_2}}:\mf h_{S_2}\to K$ by $\lambda$.

Finally, since none of the non-compact irreducible toric curve classes $C_s$ for $s\in \mf S$ intersect $Y_1\cap Y_2$, we have the decomposition $\mf S_Y = \mf S_{Y_1} \sqcup \mf S_{Y_2}$ and we have the corresponding decomposition of the screening operator
\[ Q_S = \tilde Q_{S_1} +  \tilde Q_{S_2} \quad\quad \textup{where} \quad\quad \tilde Q_{S_1} = \sum_{k_1=1}^{N_1} \sum_{s_{k_1} \in \mf S_{d_{k_1}}\cap \mf S_{d_{{k_1}+1}}} Q_{s_{k_1}} \quad\quad \tilde Q_{S_2} = \sum_{k_2=1}^{N_2} \sum_{s_{k_2} \in \mf S_{d_{k_2}}\cap \mf S_{d_{{k_2}+1}}}  Q_{s_k}   \ .\]
Moreover, it follows that we have
\[ \tilde Q_{S_1} |_{\Pi(Y_2,S_2)} = \tilde Q_{S_2} |_{\Pi(Y_1,S_1)} = 0 \quad\quad\text{and}\quad\quad \tilde Q_{S_1}|_{\Pi(Y_1,S_1)} = Q_{S_1} \quad\quad  \tilde Q_{S_2}|_{\Pi(Y_2,S_2)} = Q_{S_2}  \ , \]
so that the preceding decomposition yields the desired identification
\[ \ker(Q_S|_{\Pi(Y,S_1,S_2)_0}) \cong ( \ker(Q_{S_1}) \otimes \Pi(Y_2,S_2)) \cap ( \Pi(Y_1,S_1)\otimes \ker(Q_{S_2})) \cong  \V(Y_1,S_1) \otimes \V(Y_2,S_2)  \]
and in turn we obtain the claimed vertex algebra embedding
\[ \V(Y_1,S_1) \otimes \V(Y_2,S_2) \cong  \ker(Q_S|_{\Pi(Y,S_1,S_2)_0})  \to  \ker(Q_S)  = \V(Y,S)  \ .\]

More generally, for each $\lambda \in \bb Z^{\mf C_{S,\cap}}$ we have
\[  (\ker Q_S |_{\Pi(Y,S_1,S_2)_\lambda})\cong ( \ker(\tilde Q_{S_1}|_{\Pi(Y,S_1)_\lambda}) \otimes \Pi(Y_2,S_2)) \cap ( \Pi(Y_1,S_1)\otimes  \ker(\tilde Q_{S_2}|_{\Pi(Y_2,S_2)_\lambda})) \cong  \V(Y_1,S_1)_\lambda \otimes \V(Y_2,S_2)_\lambda \ , \]
where we have defined
\[  \V(Y_1,S_1)_\lambda  = \ker(\tilde Q_{S_1}|_{\Pi(Y,S_1)_\lambda}) \quad\quad \text{and}\quad\quad  \V(Y_2,S_2)_\lambda = \ker(\tilde Q_{S_2}|_{\Pi(Y_2,S_2)_\lambda})  \ , \]
so that the direct sum is parameterized by the subset
\begin{equation}\label{Ppeqn}
	 P_+ = \{ \lambda \in \bb Z^{\mf C_{S,\cap}}\cong \bb Z^r \ | \ \V(Y_1,S_1)_\lambda \neq \{0\} ,\  \V(Y_2,S_2)_\lambda \neq \{0\}   \} \ \subset \bb Z^{\mf C_{S,\cap}}\cong \bb Z^r , 
\end{equation}
and we obtain the desired result.
\end{proof}

\begin{rmk}\label{localityprop}

This shows that the vertex algebras $\V(Y,S)$ satisfy one of key properties predicted in \cite{GaiR} and in more detail in \cite{PrR} in this setting. Similarly, in the case that $S$ consists of a single reduced, irreducible component $S_0$, the vertex algebra $\V(Y,r[S_0])$ appears to provide a construction of the vertex algebra $\textup{VOA}[M_4;\g]$ associated to a choice of four-manifold $M_4$ and ADE type $\g$ conjectured to exist in \cite{FeiG}, in type $A_{r-1}$ and for the four-manifold $M_4=S_0^{\textup{an}}$ underlying the analytification of $S_0$, and it was explained in \emph{loc. cit.} that for a four-manifold given by the union along a common boundary $M_3$ of two four-manifolds with boundary $M_4^\pm$,
\[M_4=M_4^+ \cup_{M_3} M_4^- \ , \quad\quad \textup{one expects} \quad\quad \textup{VOA}[M_4] = \bigoplus_{\lambda \in \Lambda(M_3,r)}   \textup{VOA}[M_4^+]_{\lambda_+} \otimes \textup{VOA}[M_4^-]_{\lambda_-}     \ , \]
the associated vertex algebra is given by an extension of the tensor product $\textup{VOA}[M_4^+] \otimes \textup{VOA}[M_4^-]$, parameterized by $\Lambda(M_3,r)\subset H_1(M_3;\Z)^{\oplus r}$. Thus, the preceding proposition restricted to the case $[S]=r[S_0]$ gives a proof that this property holds for $\V(Y,r[S_0])$, as illustrated in Figure \ref{fig:FGfig} below. 

\end{rmk}

\begin{figure}[b]
	\caption{Differential geometric interpretation of the locality property}
	\vspace*{1cm}
	\begin{overpic}[width=1\textwidth]{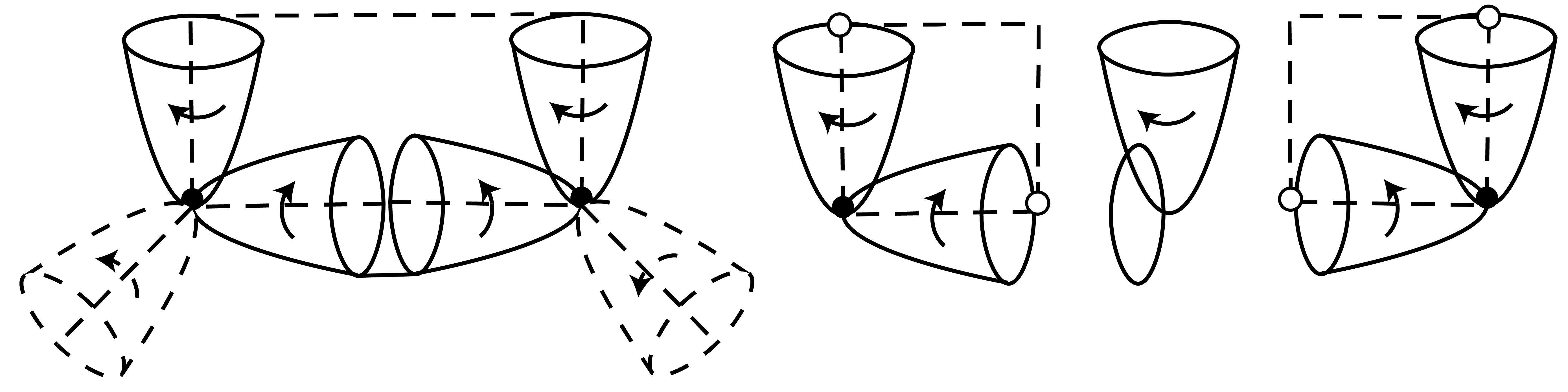}
		\put(18,0) {$M_4=\C\bb P^1\times \C$}
		\put(52,0) {$M_4^+=\C^2$}		
		\put(68,0) {$M_3=S^1\times \C$}		
		\put(87,0) {$M_4^-=\C^2$}	
	\end{overpic}
	\label{fig:FGfig}
	\vspace*{-1cm}
\end{figure}
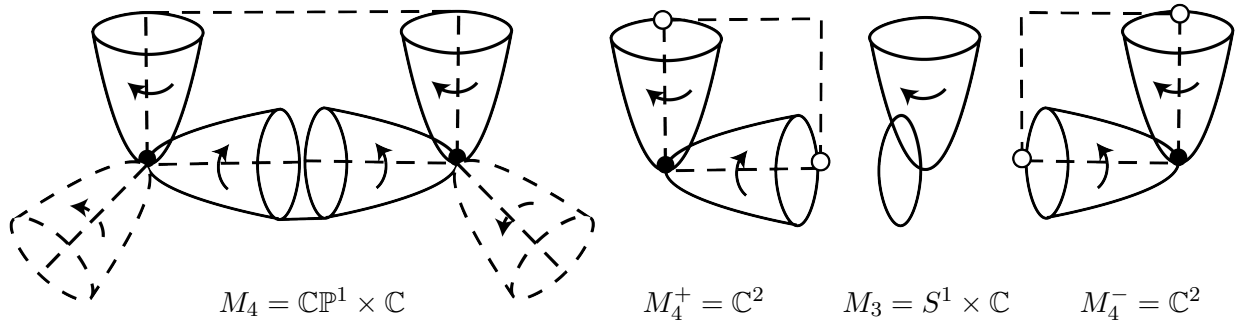

\section{The Gaiotto-Rapcak conjecture}\label{GRsec}

\subsection{The geometric construction of $\V(Y,S)$ and its representations}\label{GRintrosec}

In this section, we explain a conjecturally equivalent geometric construction of the vertex algebras $\V(Y,S)$ , following the constructions of Grojnowksi \cite{Groj} and Nakajima \cite{Nak}, the mathematical formulation of the AGT conjecture \cite{AGT}, its proof by Schiffmann-Vasserot \cite{SV}, and the generalization to divisors in $\C^3$ of Rapcak-Soibelman-Yang-Zhao \cite{RSYZ}, as outlined in the introduction.

Our proposed generalization of these constructions to divisors $S$ in toric Calabi-Yau threefolds $Y$ relies on the results of the companion paper \cite{BR1}, which we now recall: for $M\in \DD^b\Coh(Y)^T$ satisfying some hypotheses, we define a stack $\mf M(Y,M)$ parameterizing \emph{perverse coherent extensions} of $M$, iterated extensions of $M$ and the compactly supported perverse coherent sheaves of Bridgeland \cite{Bdg1}. We define framed variants $\mf M^\f(Y,M)$, and prove that they are equivalent to stacks of representations of framed quivers with potential $(Q^\f_M,W^\f_M)$,so that we obtain natural cohomological enumerative invariants as vanishing cycles homology groups of $\zeta$-stable points
\[ \bb V^{\f,\zeta}(Y,M) =  \bigoplus_{\nnn \in  \bb N^{V_{Q^\f_M}}} H_\bullet(\mf M^{\f,\zeta}_\nnn(Y,M),\varphi_{W^\f_\nnn}) \quad\quad \text{and}\quad\quad \mc Z^\zeta_M(\qq) =  \sum_{\nnn \in \bb N^{Q_M^\f}} \qq^\nnn \chi(H_\bullet( \mf M_\nnn^{\f,\zeta}(Y,M),\varphi_{W^{\f}_\nnn}))    \]
corresponding generating functions for these invariants, under some additional hypotheses.
In sum:
\begin{theo}\label{Athm}\cite{BR1} Let $M\in \DD^b\Coh(Y)^T$ as above and let $\f$ be a framing structure for $M$.
	There is a canonical framed quiver with potential $(Q_M^\f,W_M^\f)$ and an equivalence of algebraic stacks
	\begin{equation}
		\mf M(Q_M^\f,W_M^\f) \xrightarrow{\cong} \mf M^{\f}(Y,M) \ ,
	\end{equation}
	such that the equivalence of groupoids of $\bb K$-points is defined on objects by a monad presentation.
\end{theo}

In fact, we must consider the localized, $A$-equivariant analogue of the homology groups $\V^{\f,\zeta}(Y,M)$, where $A=T_\f\times T$ for $T_\f$ the maximal torus of the group of automorphisms of the framing structure $\f$ and localized with respect to the field of fractions $K$ of $H_T^\bullet(\pt)$, in the notation of Equation \ref{baseringeqn}.

For $S$ an effective, toric Cartier divisor, with decomposition into irreducible components
\begin{equation}
	S=\bigsqcup_{d\in \mf D_S} S_d \quad\quad \textup{so that}\quad\quad [S] = \sum_{d\in \mf D_S} r_d\  [ S_d ] \ , 
\end{equation}
for some tuple of non-negative integers $\rr_S=(r_d) \in \bb N^{\mf D_S}$, we let $M=\mc O^\sss_{S^\red}[1]=\oplus_{d\in \mf D_S}\mc O_{S_d}[1]$ be the sum of structure sheaves of irreducible components $S_d\subset S^\red$. The space
\begin{equation}\label{MS0defeqn}
	 \mc M^0(Y,S)=\mf M^{0_S,\zeta_\VW}(Y,\mc O^\sss_{S^\red}[1]) \ , 
\end{equation}
of (rank $\rr_S$, trivially framed, $\zeta_\VW$-stable) perverse coherent extensions of $\mc O^\sss_{S^\red}[1]$ provides the desired model in algebraic geometry for the space of rank $\rr_S$ instantons on $S^\red$, and correspondingly
\begin{equation}\label{VS0defeqn}
	 \V_S^0=\V^{0_S,\zeta_\VW}(Y,\mc O^\sss_{S^\red}[1])=  \bigoplus_{\nnn \in \bb N^{V_{Q^{0_S}}} } H_\bullet^A( \mc M_\nnn^0(Y,S),\varphi_{W^{0_S}_\nnn}) \otimes_{H_T^\bullet(\pt)} K \ , 
\end{equation}
provides the vector space on which we will define the geometric representation conjecturally identified with the representation of $\V(Y,S)$ on the module $\Pi(Y,S)$, and similarly
\[ \mc Z^\VW_S(\qq) := \mc Z^{0_S,\zeta_\VW}_{\oplus_d\mc O_{S_d}[1]}(\qq)= \sum_{\nnn \in \bb N^{Q^{0_S}}} \qq^\nnn \chi(H_\bullet( \mc M_\nnn^0(Y,S),\varphi_{W^{0_S}_\nnn}))   \]
defines a local Vafa-Witten-type invariant \cite{VW}, generalizing that of Equation \ref{Gotteqn} to describe torsion free sheaves of higher rank $r$ and moreover with arbitrary first Chern class $c_1 \in H^2(S^\red,\Z)$.

The second main result of \cite{BR1} is the general construction of a representation on the cohomological invariant $\V^{\f,\zeta}(Y,M)$ introduced above, of a certain universal algebra $\mc H(Y)$ called the \emph{Kontsevich-Soibelman cohomological Hall algebra}. These representations induce the desired representations of the algebra of modes $\mc U(\V(Y,S))$ of the vertex algebra $\V(Y,S)$.

The algebra $\mc H(Y)$ has underlying vector space given by the vanishing cycles homology groups of the full stack of representations of the quiver with potential $(Q_Y,W_Y)$ corresponding to $Y$,
\[  \mc H(Y) = \bigoplus_{\nnn \in \bb N^{V_{Q_Y}}}  H_\bullet(\mf M_\nnn(Y),\varphi_{W_{Y,\nnn}})  \quad \quad \text{with}\quad\quad \mf M_\nnn(Y) = \mf M_\nnn(Q_Y,W_Y)\ , \] 
where $(Q_Y,W_Y)$ is the unframed variant of the quivers with potential in Theorem \ref{Athm} above, corresponding to the trivial object $M=0$, and $\mf M(Q_Y,W_Y)$ its stack of representations. For example, for the threefold $Y=| \mc O_{\bb P^1}(-1)\oplus\mc O_{\bb P^1}(-1) |$ and divisor $S=r|\mc O_{\bb P^1}(-1) |$, we have
\begin{equation}\label{quiveregeqn}
 \begin{cases}  Q_Y & =  \begin{tikzcd}
		\mathcircled{V_0} \arrow[r, bend left=25 ] \arrow[r, bend left=40 ,  "A\ C"]  & \arrow[l, bend left=25 ] \arrow[l, bend left=40 ,  "B\ D"]  \mathcircled{V_1}
	\end{tikzcd} \\
	W_Y &  = ABCD-ADBC   \end{cases} \quad\quad \text{and}\quad\quad 	\begin{cases}
	Q_{\mc O_{S^\red}[1] }^{0_S} & =  \begin{tikzcd}
		\boxed{\C^{r}}\arrow[d,shift left=0.5ex, "I"] \\ 
		\mathcircled{V_0} \arrow[r, bend left=25 ] \arrow[r, bend left=40 ,  "A\ C"]  & \arrow[l, bend left=25 ] \arrow[l, bend left=40 ,  "B\ D"] \mathcircled{V_1}\ar[ul, bend right=40, swap,"J"]    
	\end{tikzcd}\\    W_{\mc O_{S^\red}[1] }^{0_S} & = ABCD-ADBC + IJC
	\end{cases}
	\ .
\end{equation}

The algebra structure on $\mc H(Y)$ is constructed in terms of natural correspondences
\[ \vcenter{\xymatrix{ & \mf M_{\kk,\bl}(Y) \ar[dl] \ar[dr] \\ \mf M_{\kk}(Y)\times \mf M_\bl(Y) & & \mf M_{\kk+\bl}(Y) }} \quad\quad \text{inducing}\quad \quad m: \mc H(Y)^{\otimes 2} \to \mc H(Y)  \ ,\]
generalizing the correspondences and induced multiplication of Equations \ref{cohaC2multcorreqn} and \ref{cohaC2multeqn}, respectively. Similarly, following results of Soibelman \cite{Soi} we define analogous correspondences
\[ \vcenter{\xymatrix{ & \mf M^{\f,\zeta}_{\kk,\bl}(Y,M) \ar[dl] \ar[dr] \\ \mf M_{\kk}(Y)\times \mf M^{\f,\zeta}_\bl(Y,M) & & \mf M^{\f,\zeta}_{\kk+\bl}(Y,M) }}  \]
and prove the following theorem constructing representations of $\mc H(Y)$ on $\V^{\f,\zeta}(Y,M)$:
\begin{theo} \cite{BR1} \label{Bthm}
	There exists a natural representation
	\[ \rho_M: \mc H(Y) \to \End( \V^{\f,\zeta}(Y,M) )  \]
	of the Kontsevich-Soibelman cohomological Hall algebra $\mc H(Y)$ on the invariant $\V^{\f,\zeta}(Y,M) $.
\end{theo}

In particular, in the example $M=\mc O^\sss_{S^\red}[1]=\oplus_{d\in \mf D_S}\mc O_{S_d}[1]$ and $\f=0_S$ described above, we have:

\begin{corollary}\label{Nakopcoro} There exists a natural representation
	\[ \rho_S^0:\mc H(Y) \to \End_F(\V_S^0)  \ , \]
of the Kontsevich-Soibelman cohomological Hall algebra $\mc H(Y)$ on $\V_S^0=\V^{0_S,\zeta_\VW}(Y,\mc O^\sss_{S^\red}[1])$.
\end{corollary}

We define the subspace of spherical generators $\mc H_1(Y)\subset \mc H(Y)$ by
\[ \mc H_1(Y) = \bigoplus_{|\nnn|=1} \mc H_\nnn(Y)  = \bigoplus_{i\in V_{Q_Y}} \mc H_{\one_i}(Y) \]
where $\one_i\in \bb N^{V_{Q_Y}}$ denotes the $i^{th}$ standard basis vector, and the spherical subalgebra
\[ \mc{SH}(Y)= \langle \mc H_1(Y) \rangle  \subset \mc H(Y) \ ,\]
as the subalgebra generated by this subspace over the base ring, which in the present setting is given by $F$, the field of fractions of $H^\bullet_A(\pt)$.

Let $\mf N \subset \mf M_{\one_i}(Y)\subset \mf M^1(Y)=\bigsqcup_{i \in V_{Q_Y}} \mf M_{\one_i}(Y)$ be a closed substack admitting fundamental class 
\[ [\mf N] \ \in H_\bullet(\mf M_{\one_i}(Y),\varphi_{W_{Y,\one_i}})  \subset \mc H_1(Y)=\bigoplus_{i\in V_{Q_Y}} H_\bullet(\mf M_{\one_i}(Y),\varphi_{W_{Y,\one_i}})  \ , \]
for example given by $\mf N=\mf M_{\one_i}(Y)$, and consider the endomorphism $\rho([\mf N]) \in \End_F(\V_S^0) $ induced by the representation of Corollary \ref{Nakopcoro}. There are natural correspondences
\begin{equation}\label{Nakcorrmaineqn}
	\vcenter{\xymatrix{ & \mc M_{\nnn,\nnn+\one_i}^{[\mf N]}(Y,S) \ar[dr]^p \ar[dl]_q & \\ \mc M_\nnn^0(Y,S) && \mc M_{\nnn+\one_i}^0(Y,S) }}  \ ,
\end{equation}
where $\mc M_{\nnn,\nnn+\one_i}^{[\mf N]}(Y,S)$ is the moduli space of short exact sequences of $\zeta_\VW$-stable representations
\begin{equation}\label{seseqn}
	 0 \to F \to \mc E \to \mc F \to 0  \ ,
\end{equation}
such that $\mc E\in \mc M_{\nnn+\one_i}^0(Y,S)$, $\mc F \in\mc M_{\nnn}^0(Y,S) $, and such that the subobject $F$, which thus necessarily determines a $\K$-point of $\mc M_{\one_i}(Y)$, is an element of the substack $\mf N$. This is the desired generalization of the correspondences of Nakajima recalled in Equation \ref{Nakcorreqn} for $k=1$, and we have:

\begin{prop}\label{Nakopprop} The endomorphism $\rho([\mf N]) \in \End_F(\V_S^0) $ is given by the direct sum of the maps
\begin{equation}\label{Nakendmainkeqn}
		\alpha^n([\mf N]) =p_*\circ q^*: H_\bullet^A( \mc M_\nnn^0(Y,S),\varphi_{W^{0_S}_\nnn}) \otimes_{H_T^\bullet(\pt)} K \to H_\bullet^A( \mc M_{\nnn+\kk}^0(Y,S),\varphi_{W^{0_S}_{\nnn+\kk}}) \otimes_{H_T^\bullet(\pt)} K  ,
\end{equation}
over $\nnn \in \bb N^{V_{Q_Y}}$, induced by the correspondences of Equation \ref{Nakcorrmaineqn}.
\end{prop}
\begin{proof} This follows from a straightforward argument generalizing the proof of Theorem 5.6 in \cite{YZ2}.
\end{proof}
These are the desired generalizations of the endomorphisms of Equation \ref{Nakopeqn} for $k=1$. More generally, for $k>0$ the analogous elements given by products of fundamental classes
\[  [\mf N_1] \cdot \hdots  \cdot [\mf N_k]  \ \in \mc{SH}_k(Y) = \mc H_1(Y) \cdot \hdots \cdot \mc H_1(Y) \subset \mc H(Y) \]
define endomorphisms induced by the analogous correspondences
\begin{equation}\label{Nakcorrmainkeqn}
	\vcenter{\xymatrix{ & \mc M_{\nnn,\nnn+\kk}^{ [\mf N_1] \cdot \hdots  \cdot [\mf N_k] }(Y,S) \ar[dr]^p \ar[dl]_q & \\ \mc M_\nnn^0(Y,S) && \mc M_{\nnn+\kk}^0(Y,S) }}  \ ,
\end{equation}
where $ \mc M_{\nnn,\nnn+\kk}^{ [\mf N_1] \cdot \hdots  \cdot [\mf N_k] }(Y,S)$ is the moduli space of short exact sequences as in Equation \ref{seseqn}, such that $\mc E\in \mc M_{\nnn+\kk}^0(Y,S)$, $\mc F \in\mc M_{\nnn}^0(Y,S) $, and such that the subobject $F$, which thus necessarily determines a $\K$-point of $\mf M_{\kk}(Y)$, admits a composition series
\[F = \left[ F_0 < F_1 < \hdots <  F_k \right]  \]
with each object $F_i$ determining a $\bb K$-point in $\mf N_i \subset \mf M_{\one_i}(Y) $.

To define the analogous Nakajima operators for $k<0$, we need to make the following additional assumption, which appears to hold in all examples of interest: For the remainder of this section, we assume that the $A$-fixed subvariety $ \mc M_\nnn^{0}(Y,S)^A$ is given by a set $\Fp_\nnn$ of isolated fixed points. In particular, the pullback map on Borel-Moore homology groups, with coefficients in the sheaf of vanishing cycles, gives an isomorphism
\[\bb V_{S,\nnn}^0 \cong H_\bullet^{A}(\mc M_\nnn^{0}(Y,S), \varphi_{\overline{W}^{0_S}_\nnn})\otimes_{H^\bullet_A(\pt)}F  \xrightarrow{\iota^*}  H_\bullet^A(\Fp_\nnn,\varphi_{\overline{W}^{0_S}_\nnn\circ \iota}) \otimes_{H^\bullet_A(\pt)}F = \bigoplus_{\lambda\in \Fp_\nnn} H^A_\bullet(\pt_\lambda)\otimes_{H^\bullet_A(\pt)}F \ , \]
for each $\nnn\in \bb N^{V_{Q_Y}}$, and letting $F_\lambda$ denote a copy of the base field $F$, we have an identification
\[  \bigoplus_{\lambda\in \Fp_\nnn}  F_\lambda   \xrightarrow{\cong}  \bigoplus_{\lambda\in \Fp_\nnn}   H^A_\bullet(\pt_\lambda)\otimes_{H^\bullet_A(\pt)}F \quad\quad \text{defined by} \quad\quad P \mapsto P\cap [ \pt_\lambda ]  \ ,\]
for each $P\in F_\lambda$ so that we obtain a natural basis for the module $\bb V^\zeta_\dd$ given by
\[ \bb V_{S,\nnn}^0 =\bigoplus_{\lambda\in \Fp_\nnn}  F_\lambda  \quad\quad \text{and thus}\quad\quad \bb V_{S}^0  = \bigoplus_{\lambda \in \Fp} F_\lambda  \quad\quad\text{for}\quad\quad  \Fp = \bigsqcup_{\nnn\in \bb N^{V_{Q_Y}}} \Fp_\nnn\ . \] 
In particular, there is a natural pairing
\begin{equation}\label{pairingeqn}
	(\cdot,\cdot):\V^\zeta \otimes_F \V^\zeta \to F \quad\quad \text{defined by} \quad\quad  ([\pt_\lambda],[\pt_\mu]) = \delta_{\lambda,\mu}\textup{Eu}_A(T_\lambda)  \ , 
\end{equation}
where $T_\lambda$ denotes the tangent space to $ \mf M^\zeta_\dd(Q_M^\f, W_M^\f)$ at the fixed point $\lambda \in \Fp$ and $\textup{Eu}_A$ denotes the $A$-equivariant Euler class. Let $\mc{SH}^\op(Y)$ denote the opposite algebra of $\mc{SH}(Y)$ and we have:

\begin{prop}\label{opactprop} There exists a natural action right action
	\begin{equation}\label{shoprepmapeqn}
		(\rho_S^0)^*:\mc{SH}^\op(Y) \to \End_F(\V^\zeta) \quad\quad \text{defined by} \quad \quad f=e^\op \mapsto \rho(e)^*  
	\end{equation}
	for each $f=e^\op \in \mc{SH}^\op$, where $\rho:\mc{SH}\to \End_F(\V^\zeta)$ is the representation of Theorem \ref{Bthm} and $(\cdot)^*:\End_F(\V^\zeta)\to \End_F(\V^\zeta)$ denotes the adjoint with respect to the pairing of Equation \ref{pairingeqn} above.
\end{prop}
\begin{proof}
	This follows from the proof of Proposition 3.6 of \cite{SV}, \emph{mutatis mutandis}.
\end{proof}

\noindent Analogously to Corollary \ref{Nakopcoro}, the endomorphisms $(\rho_S^0)^*([ \mf N ]) \in \End_F(\V^\zeta) $ are induced by analogous correspondences to those of Equation \ref{Nakcorrmaineqn}.

Following the proof of the AGT conjecture from \cite{SV} and its generalization to divisors $S
\subset \C^3$ in \cite{RSYZ}, in \cite{BR1} we prove that the representations of Corollary \ref{Nakopcoro} and Proposition \ref{opactprop}, which are the analogues of the raising and lowering operators in the Heisenberg algebra construction of Nakajima, respectively,  extend to define a representation of a larger algebra $\mc Y_S(Y)$ where
\begin{equation}\label{QYeqn}
	\mc Y_S(Y)=\mc Y_S(Y)_- \otimes \mc Y_S(Y)_0 \otimes \mc Y_S(Y)_+ \quad\quad \mc Y_S(Y)_+ \cong \mc{SH}(Y) \quad\quad \mc Y_S(Y)_-\cong \mc{SH}(Y)^\op  \ , 
\end{equation}
that is, there exists a natural representation $\rho_S^0: \mc Y_S(Y)\to \End_F(\bb V_S^0)$ such that the restrictions to $\mc Y_S(Y)_+$ and $\mc Y_S(Y)_- $ identify with the representations of $\mc{SH}(Y)$ and $\mc{SH}(Y)^\op$ above.

We can now state the generalization to this setting of the standard mathematical formulation of the AGT conjecture \cite{AGT}, proved in \cite{SV} and \cite{MO}. We will call this statement the Gaiotto-Rapcak conjecture, as although the authors did not explicitly conjecture this statement, the study of this family of vertex algebras associated to divisors in toric Calabi-Yau threefolds, understood to be generalizing those featuring in the original AGT conjecture, was initiated in \cite{GaiR} and many important aspects of the correspondence were studied in the subsequent work of Gaiotto, Rapcak and their collaborators (see for example \cite{RSYZ}, \cite{GaiR2}, \cite{Rap}, \cite{PrR}, \cite{CrG}, and \cite{GRZ}):

\begin{conj}[Gaiotto-Rapcak]\label{GRconj} There exists a natural representation
	\[\rho: \mc U(\V(Y,S)) \to \End_F(\V_S^0 )\]
	of the algebra of modes $\mc U(\V(Y,S))$ of the vertex algebra $\V(Y,S)$ on $\V_S^0$, inducing an isomorphism
	\[ \mc U(\V(Y,S)) \xrightarrow{\cong} \rho( \mc U(\V(Y,S))) \xrightarrow{\cong} \rho_S^0(\mc Y_S(Y)) \]
and such that $\V_S^0$ identifies with $\Pi(Y,S)$, the free field module for $\V(Y,S)$.
\end{conj}

Following the observations above, the algebra $\rho_S^0(\mc Y_S(Y))\subset \End_F(\V_S^0)$ is by definition the subalgebra generated over $F$ by endomorphisms induced by correspondences of the type in Equation \ref{Nakcorrmaineqn} above. Thus, this conjecture asserts that the algebra of modes $\mc U(\V(Y,S)) $ of the vertex algebras $\V(Y,S)$ defined in Section \ref{screensec} is equivalent to the algebra generated by the Nakajima-type operators introduced above, generalizing the proof of the AGT conjecture in \cite{SV}. In the example $Y=\C^3$, this is the generalization proved in \cite{RSYZ}, on which our general approach is modelled.

We now describe the analogous conjectural geometric construction of more general vertex algebra modules over $\V(Y,S)$, as outlined in the introduction. Recall that in the definition of $\V_S^0$ in Equation \ref{VS0defeqn}, we fixed the choice of trivial framing structure $\f=0_S$ of rank $\rr_S$ determined by the divisor $S$. In fact, for each choice of framing structure $\f$ of rank $\rr_S$ satisfying some appropriate hypotheses, there exists a moduli space of $\f$-framed, $\zeta_\VW$-stable, perverse coherent extensions of $M=\mc O_{S^\red}^\sss[1]$,
\begin{equation}\label{MSfdefeqn}
	 \mc M^\f(Y,S)=\mf M^{\f,\zeta_\VW}(Y,\mc O_{S^\red}^\sss[1]) 
\end{equation}
generalizing those of Equation \ref{MS0defeqn}, and their corresponding homology groups
\begin{equation}\label{VSfdefeqn}
	\V_S^\f=\V^{\f,\zeta_\VW}(Y,\mc O_{S^\red}^\sss[1])=  \bigoplus_{\nnn \in \bb N^{V_{Q^{\f}_S}} } H_\bullet^A( \mc M_\nnn^\f(Y,S),\varphi_{W^{\f}_\nnn}) \otimes_{H_T^\bullet(\pt)} K \ , 
\end{equation}
where $(Q^{\f}_S,W^{\f})$ denotes the corresponding framed quiver with potential as in Theorem \ref{Athm}.

These satisfy the following analogue of Corollary \ref{Nakopcoro} of Theorem \ref{Bthm}:

\begin{corollary}\label{Nakopgencoro} There exists a natural representation
	\[ \rho:\mc H(Y) \to \End_F(\V_S^\f)  \ , \]
	of the Kontsevich-Soibelman cohomological Hall algebra $\mc H(Y)$ on $\V_S^\f=\V^{\f,\zeta_\VW}(Y,\mc O_{S^\red}^\sss[1])$.
\end{corollary}

\noindent Further, Proposition \ref{opactprop} again applies in this case, so that we obtain the putative geometric construction of the conjectural representation, precisely as above:

\begin{conj}\label{voamodconj} There exists a natural representation
	\[\rho_S^\f: \mc U(\V(Y,S)) \to \End_F(\V_S^\f )\]
	of the algebra of modes $\mc U(\V(Y,S))$ of the vertex algebra $\V(Y,S)$ on $\V_S^\f$, inducing an isomorphism
	\[ \rho_S^\f ( \mc U(\V(Y,S))) \xrightarrow{\cong} \rho(\mc Y_S(Y))  \ .\]
\end{conj}

\begin{figure}[t]
	\hspace*{-1cm}
	\begin{overpic}[width=.4\textwidth]{picOm2}
		\put(0,90) {$ Y_{2,0}\to X_{2,0}=\{xy-z^2\} \times \bb C $}
		\put(0,84) {\rotatebox[origin=c]{90}{$\subset$}}
		\put(0,78) {$S_{\mu}$}
		\put(67,57) {$M_0=\mu_1+\mu_2$}
		\put(62,38) {$M_1=\mu_2$}
	\end{overpic}\quad\quad
	\begin{overpic}[width=.55\textwidth]{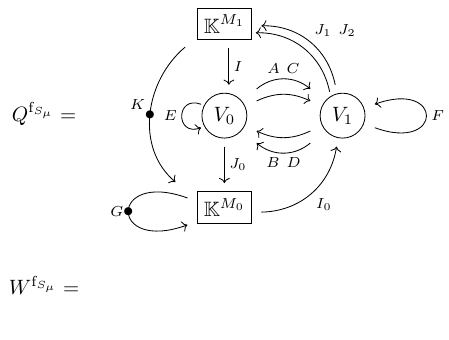}
		\put(22,10) { $  E(BC-DA)+ F(AD-CB) + IJ_1A  $} 
		\put(42,5) {$ + IJ_2C + I_0J_0D+I_0 K J_1 + I_0 G J_0 $}
	\end{overpic}
	\caption{The toric divisor $S_\mu$ in $Y_{2,0}$ and framed quiver with potential $(Q^{\f_{S_\mu}},W^{\f_{S_\mu}})$}
	\label{fig:Walgquivfig}
\end{figure}
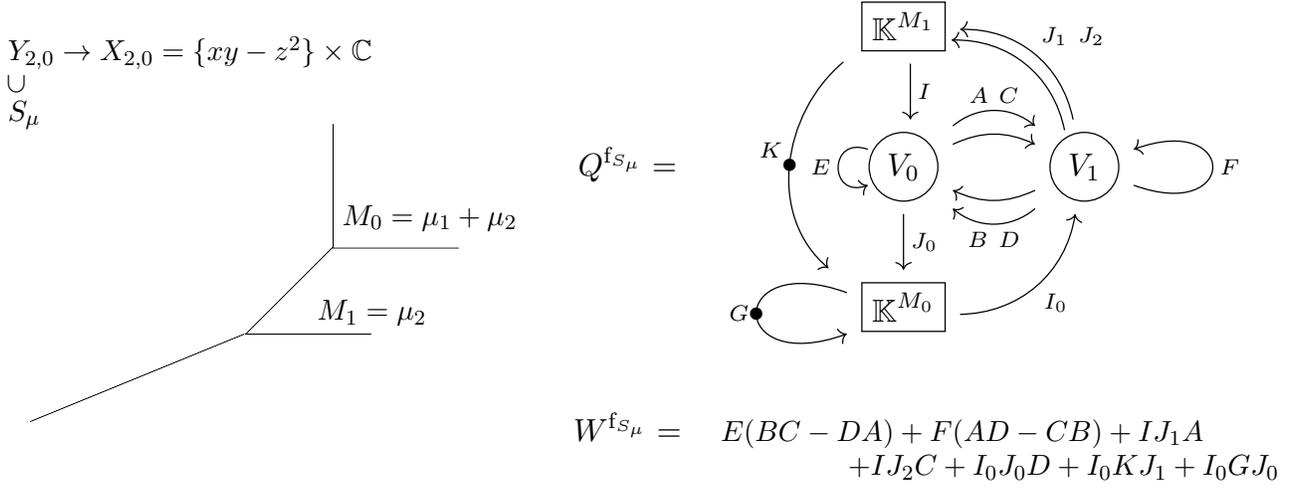

Inspired by the construction of \cite{CCDS}, in \cite{BR1} we observed that for each divisor $S$ there is a natural framing structure $\f_S$ of rank $\rr_S$, determined by the object $\mc O_S[1]$ itself, considered as an iterated extension of $r_d$ copies of each of the objects $\mc O_{S_d}[1]$, or equivalently the summands of $\mc O_{S^\red}^\sss[1]=\oplus_{d \in \mf D_S} \mc O_{S_d}[1]$, where we recall that $\rr_S=(r_d)_{d\in\mf D_S}$ is determined by the coefficients of the decomposition $[S]=\sum_{d\in \mf D_S} r_d[ S_d]$. We denote the corresponding representation by simply 
\[\V_S=\V_S^{\f_S}=\V^{\f_S,\zeta_\VW}(Y,\mc O_{S^\red}^\sss[1]) \ , \]
and we have the following variant of the preceding conjecture, which gives a geometric construction of the vacuum module of $\V(Y,S)$:

\begin{conj}\label{vacuumconj} The representation of Conjecture \ref{voamodconj} above,
\[ \rho_S: \mc U(\V(Y,S)) \to \End_F(\V_S)  \ , \]
identifies $\V_S$ with the vacuum module for the vertex algebra $\V(Y,S)$.
\end{conj}

In the example $Y=Y_{2,0}=\widetilde{A}_1\times \bb A^1 = |\mc O_{\bb P^1}(-2)\oplus\mc O_{\bb P^1}|$ with divisor $S_\mu$ determined by the labelling of the faces of the moment polytope as on the left side of Figure \ref{fig:Walgquivfig}, the corresponding quiver with potential is that given the right side of \emph{loc. cit.}. This is a special case of the family of examples considered in Section \ref{Walgsec}, for which the corresponding vertex algebra is given by $W_{f_{\mu}}^\kappa(\gl_{M_0})$, the affine $W$-algebra determined by $f_\mu \in \mc N_{\gl_{M_0}}$ with two Jordan blocks of sizes $\mu_1$ and $\mu_2$.

More generally, in the setting of Proposition \ref{factprop}, given an expression for $\mc O_S$ as an extension
\[ 0\to \mc O_{R} \to \mc O_S \to  \mc O_{T} \to 0 \]
we obtain an alternative framing structure $\f_{R,T}$ of rank $\rr_S$ corresponding to the iterated extension of the summands of $\mc O_{S^\red}^\sss[1]$ given by the partial semisimplification $\mc O_R[1]\oplus\mc O_S[1]$, and we conjecture:

\begin{conj}\label{FFRconj} The representation of Conjecture \ref{voamodconj} above,
	\[ \rho_S^{\f_{R,T}}: \mc U(\V(Y,S)) \to \End_F(\V_S^{\f_{R,T}})  \ , \]
	identifies $\V_S^{\f_{R,T}}$ with the restriction of the vacuum module for the vertex algebra $\V(Y,R) \otimes_K \V(Y,T)$.
\end{conj}

The inductive application of the preceding conjecture implies the existence of compatible geometric constructions of the vacuum modules of all the relative free field realizations of the algebras $\V(Y,S)$ defined by the analogous inductive application of Proposition \ref{factprop}. In particular, in the limiting case in which we consider the full semisimplification, the corresponding framing structure $\f$ is given by the trivial framing structure $0_S$ of rank $\rr_S$, that is,
\[ \mc O_S^\sss[1]=\oplus_{d\in\mf D_S} \mc O_{S_d}^{\oplus r_d}[1] \quad\quad \textup{determines the framing structure} \quad\quad \f=0_S \ ,  \]
 and the analogue of Conjecture \ref{FFRconj} is simply the original statement of the Gaiotto-Rapcak conjecture given in Conjecture \ref{GRconj} above; in the example of Figure \ref{fig:Walgframingsfig}, this is the bottom right quiver.

In the special case that $Y=Y_{2,0}$ and $S=S_\mu$ as in Figure \ref{fig:Walgquivfig} for $\mu=(\mu_1,\mu_2)=(1,1)$, for which we prove in Theorem \ref{affinesl2theo} that the corresponding vertex algebra is given by the affine algebra for $\spl_2$, that is $\V(Y_{2,0},S_{2,1,0,0})=V^\kappa(\gl_2)$, the quiver with potential corresponding to each of the partial free field realizations summarized in Figure \ref{fig:sl2factfig} are given in Figure \ref{fig:Walgframingsfig} below.

In the remainder of this section, we outline a general approach to these conjectures, and prove a few special cases thereof as well as several partial results towards the general statement.

\begin{figure}[b]
	\caption{The framings $\f$ and quivers $(Q^\f,W^\f)$ for the representations of Figure \ref{fig:sl2factfig}}
	\label{fig:Walgframingsfig}
	
	\vspace*{.25cm}
	\hspace*{-.5cm}
	\begin{overpic}[width=.5\textwidth]{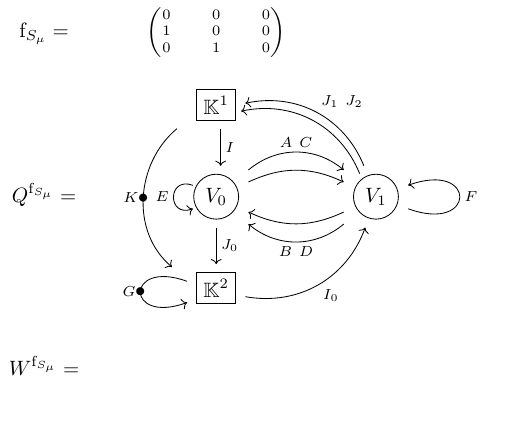}
		\put(16,10) {   $E(BC-DA)+ F(AD-CB) + IJ_1A $} 
		\put(30,4) {$\scriptsize {+ IJ_2C + I_0J_0D+I_0 K J_1 + I_0 G J_0 }$}
	\end{overpic}
	\begin{overpic}[width=.5\textwidth]{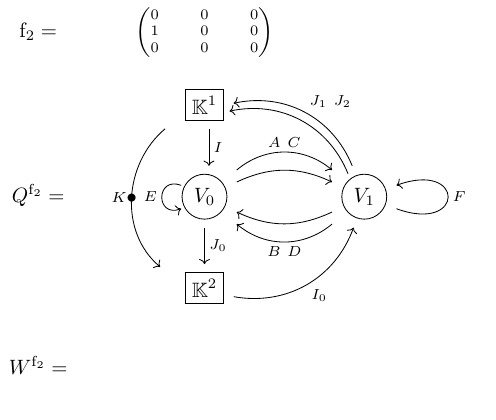}
		\put(16,10) {   $E(BC-DA)+ F(AD-CB) + IJ_1A $} 
	\put(30,4) {$\scriptsize {+ IJ_2C + I_0J_0D+I_0 K J_1 }$}
	\end{overpic}

\vspace*{.4cm}
\hspace*{-.5cm}
	\begin{overpic}[width=.5\textwidth]{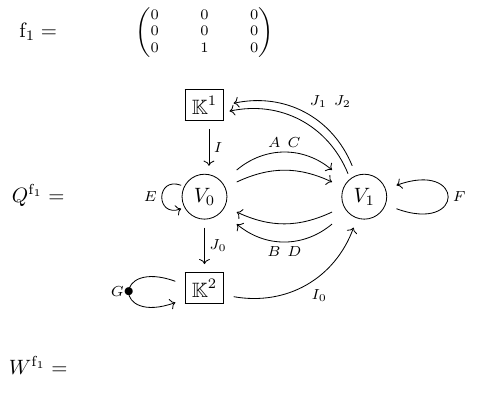}
			\put(16,10) {   $E(BC-DA)+ F(AD-CB) + IJ_1A $} 
		\put(30,4) {$\scriptsize {+ IJ_2C + I_0J_0D+ I_0 G J_0 }$}
	\end{overpic}
	\begin{overpic}[width=.5\textwidth]{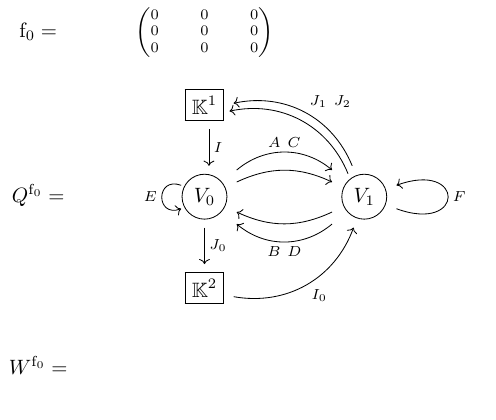}
		\put(16,10) {   $E(BC-DA)+ F(AD-CB) + IJ_1A $} 
	\put(30,4) {$\scriptsize {+ IJ_2C + I_0J_0D }$}
	\end{overpic}
\vspace*{-1.7cm}
\end{figure}

\subsection{Lattice vertex algebras and Hecke modifications along curves for rank 1 sheaves}\label{latticegeosec}
In this section, we outline a proof of Theorem \ref{abGRthm}, which establishes Conjectures \ref{GRconj} and \ref{vacuumconj} in the case that the divisor $S\subset Y$ consists of a single reduced, irreducible component $S=S_0$, and correspondingly the vertex algebra $\V(Y,S_0)=\Pi(Y,S_0)$ is given by a lattice-type vertex algebra generated by $H_\bullet(S_0;\Z)$. Note that there is a unique framing structure $\f_0:=\f_{S_0}=0_{S_0}$ of rank $1$ on any quiver, and correspondingly the vacuum and free field module for $\V(Y,S_0)=\Pi(Y,S_0)$ are by definition the same, so that the two aforementioned conjectures coincide in this case. Recall that
 \begin{equation}
	\mc M^{0}(Y,S_0)=\mf M^{\zeta_\VW,0}(Y,\mc O_{S_0}[1]) \quad\quad \textup{and} \quad\quad \V_{S_0}=\V^{0,\zeta_\VW}(Y,\mc O_{S_0}[1])=H_\bullet^T(\mc M^0(Y,S_0),\varphi_{W^{0}_{S_0}})
\end{equation}
denote the corresponding moduli spaces of $\zeta_\VW$-stable, rank 1, trivially- (or $\f_0$-) framed perverse coherent extensions of $\mc O_{S_0}[1]$, and their $T$-equivariant Borel-Moore homology groups with coefficients in the sheaf of vanishing cycles $\varphi_{W_{S_0}^{\f_0}}$ for the potential $W_{S_0}^{\f_0}$. We begin with some generalities:

\begin{rmk}\label{localityrmk}
We now explain a crucial principle that informs several properties of the geometric constructions of the preceding section, which will lead to the geometric analogues of the factorization and locality properties of $\V(Y,S)$. The basic idea, which is somewhat tautological, is the following:

\bigskip

\noindent \emph{Algebraic properties of the algebras $\mc H(Y)$ and $\mc Y_S(Y)$ and their representations $\V_S^\f$ which depend on the geometry of the threefold $Y$ and divisor $S$ only in some open set $U\subset Y$ can by computed in terms of the algebras $\mc H(U)$ and $\mc Y_{S\times_Y U}(U)$ and their representations $\V_{S\times_Y U}^{\f_U}$.}
\end{rmk}

A systematic explanation of the precise statement and implications of this general principle is beyond the scope of the present work, but we will use several statements of this form throughout the proofs of various partial results towards establishing Conjecture \ref{voamodconj} and its variants from the preceding section. The first example of this is the following:

Let $Y\to X$ be a toric Calabi-Yau threefold resolution satisfying the hypotheses of the companion paper \cite{BR1}, and let $U\subset Y$ be an open, toric subvariety such that $U$ itself satisfies the same hypotheses. Note that there is a canonical inclusion
\[ \mf M(U) \xrightarrow{\cong} \mf M_U(Y) \subset  \mf M(Y) \quad\quad \text{induced by} \quad\quad  \Perv_\cs(U) \xrightarrow{\cong } \Perv_{\cs}(Y)_U \subset \Perv_\cs(Y)  \ , \]
the equivalence of $\Perv_\cs(U)$ with the full subcategory $\Perv_\cs(Y)_U$ of compactly supported perverse coherent sheaves on $Y$ with cohomology sheaves supported on $U\subset Y$.

Further, suppose $S_0$ is an irreducible, reduced, toric Cartier divisor on $Y$, let
\begin{equation}\label{modspcmapeqn}
	 S_U = S_0 \times_Y U \quad\quad \text{and note there is} \quad\quad \mf M^{\f_0}(U,S_U) \xrightarrow{\cong} \mf M_U^{\f_0}(Y,S_0) \subset  \mf M^{\f_0}(Y,S_0)  
\end{equation}
a natural inclusion with image $\mf M^{\f_0}_U(Y,S_0)$ the substack of $\mf M^{\f_0}(Y,S_0)$ parameterizing the iterated extensions of $\mc O_{S_0}[1]$ with compactly supported perverse coherent sheaves supported on $U\subset Y$.

Under these hypotheses, we assume that our choice of stability conditions $\zeta_\VW$ for each of framed quivers with potential $(Q_{S_0}^{0_{S_0}},W_{S_0}^{\f_0})$ satisfy the following natural property:

\begin{defn}\label{stablocdefn} A choice of stability conditions $\zeta$ and $\zeta_U$ for the quivers with potential $(Q_{S_0}^{\f_0},W_{S_0}^{\f_0})$ and $(Q_{S_U}^{\f_0},W_{S_U}^{\f_0})$ is called locally compatible if the map of Equation \ref{modspcmapeqn} induces an isomorphism
\begin{equation}\label{stablemodspcmapeqn}
	 \mf M^{\f_0,\zeta_U}(U,S_U)  \xrightarrow{\cong} \mf M^{\f_0,\zeta}_U(Y,S_0) \subset \mf M^{\f_0,\zeta}(Y,S_0) \ , 
\end{equation}
and additionally if $S_0\subset U$ so that $S_U\cong S_0$ we require $\mf M^{\f_0,\zeta}_U(Y,S_0) = \mf M^{\f_0,\zeta}(Y,S_0)$.
\end{defn}

If we choose stability conditions for which the corresponding moduli spaces admit a common geometric description intrinsic to $S_0$ that is manifestly local, one can hope to deduce such identifications geometrically. Indeed, by definition our putative family of stability conditions $\zeta_\VW$ are chosen such that the spaces $\mc M^0(Y,S_0)=\mf M^{\f_0,\zeta_\VW}(Y,S_0) $ are equivalent to moduli spaces of rank 1 torsion free sheaves $E$ on $S_0$ of arbitrary first and second Chern class, with an isomorphism $\varphi:{\tilde E}^{\vee\vee} \xrightarrow{\cong} \mc O_{S_0}$.

The hypotheses of Definition \ref{stablocdefn} imply the following key locality result, which we use in the proof of Theorem \ref{abGRthm} below, establishing Conjectures \ref{GRconj} and \ref{vacuumconj} for reduced, irreducible divisors $S_0$:

\begin{prop}\label{localsubrepprop} The cohomological Hall algebra $\mc H(U)$ is a naturally a subalgebra of $\mc H(Y)$, and the map of Equation \ref{stablemodspcmapeqn} induces an isomorphism
\begin{equation}\label{localsubrepeqn}
 \V_{S_U} = H_\bullet^T(\mf M^{\f_0,\zeta_{\VW}}(U,S_U), \varphi_{W_{S_U}}) \xrightarrow{\cong} \bb V_{S_0, U}: =H_\bullet^T(\mf M^{\f_0,\zeta_\VW}_U(Y,S_0),\varphi_{W_{S_0}}) \subset \bb V_{S_0}
\end{equation}
of $\mc H(U)$-representations, where in particular $\V_{S_0,U}$ defines a subrepresentation of the restriction of the $\mc H(Y)$-representation $\V_{S_0}$ to $\mc H(U)$.
\end{prop}

A detailed proof is orthogonal to the goals of this paper and we defer it to future work, but the main idea is a straightforward base change argument relating the correspondences defining the representations of $\mc H(U)$ and $\mc H(Y)$. Together with the latter assumption of Definition \ref{stablocdefn}, this implies:

\begin{corollary} Suppose $S_0$ is supported on $U\subset Y$ so that $S_U=S_0 \times_Y U = S_0$. Then the map of Equation \ref{localsubrepeqn} defines an isomorphism of $\mc H(U)$-representations
\[ \V_{S_U} = H_\bullet^T(\mf M^{\f_0,\zeta_{\VW}}(U,S_U), \varphi_{W_{S_U}})\xrightarrow{\cong} \V_{S_0}  =H_\bullet^T(\mf M^{\f_0,\zeta_\VW}(Y,S_0),\varphi_{W_{S_0}}) \ ,\]
where the latter is the restriction of the $\mc H(Y)$-representation to $\mc H(U)$.
\end{corollary}

This reduces the proof of the Gaitto-Rapcak Conjecture \ref{GRconj} for divisors $S=S_0$ with a single reduced, irreducible component to computations in the following examples of $(Y,S_0)$:
\begin{equation}\label{localmodeleqn}
	(\C^3,\C^2),\  (Y_{1,1}, |\mc O_{\bb P^1}(-1)|),\  (Y_{2,0}, |\mc O_{\bb P^1}|),\  (Y_{2,0}, |\mc O_{\bb P^1}(-2)|) , \ (Y_{m,n}, \widetilde{A}_{m-1,n})  \ , (Y^*,S^*) \ ,
\end{equation}
where $\widetilde{A}_{m-1}=\widetilde{A}_{m-1,0}$ denotes the resolved type $A_{m-1}$ singularity, $\widetilde{A}_{m-1,n}$ denotes its iterated blowup at at $n$ $T$-fixed points, $Y_{m,n}\to X_{m,n}$ denotes a toric, Calabi-Yau resolution of the threefold singularity $X_{m,n}=\{ xy-z^m w^n\}$, and $S^*\subset Y^*=\widetilde{\C^3/ (\bb Z_2\times \bb Z_2 )}$ is one additional isolated example, which we do not consider in the present work; these examples exhaust the local models for reduced, irreducible divisors $S_0\subset Y$ occurring in the class of toric Calabi-Yau threefolds we consider.

In fact, we will use Proposition \ref{localsubrepprop} above to outline a general proof of Theorem \ref{abGRthm}, which does not depend on explicit calculations exhausting the above list of examples, but we begin by outlining the explicit computation in the first two examples in the list, to introduce the main structure of the underlying examples.

\begin{eg}\label{geometricC2eg}
	
In the simplest example $(Y,S_0)=(\C^3,\C^2)$, the framed quiver with potential $(Q_{\C^2}^{0},W_{\C^2}^{0})$ is given by that of Equation \ref{C3quiveqn} in the case $r=1$, for which the dimensional reduction isomorphism recalled in Equation  \ref{dimredeqn}, as the $r=1$ special case of that of Equation \ref{dimredrC2eqn}, induces an isomorphism
\begin{equation} \label{dimredC2eqn}
 \V_{\C^2} =  \bigoplus_{n\in \bb N}  H_\bullet^T( \mc M^0_n(\C^3,\C^2), \varphi_{W_{[\C^2]}^0}) \otimes_{H_T^\bullet(\pt)} F  \xrightarrow{\cong } \bigoplus_{n\in \bb N} H_\bullet^T(\Hilb_n(\C^2)) \otimes_{H_T^\bullet(\pt)} F   \ ,
\end{equation} 
where we choose the stability condition $\zeta_\VW$ defining $\mc M^0_n(\C^3,\C^2)$ to be that corresponding to the standard stability condition for the ADHM quiver identifying the resulting moduli space of stable presentations of dimension $n$ with $\Hilb_n(\C^2)$.

The set of $T$ fixed points in $\Hilb_n(\C^2)$ can be identified with the set $\Fp_n$ of partitions of length $n$, which correspond to subschemes supported set-theoretically at the origin $0\in (\C^2)^T$, whose structure sheaves are given by irreducible, graded iterated extensions of $\mc O_{0}$ of length $n$ corresponding to the given partition. Thus, we obtain an isomorphism
\begin{equation}
	\V_{\C^2} \xrightarrow{\cong } \bigoplus_{n\in \bb N} H_\bullet^T(\Hilb_n(\C^2)) \otimes_{H_T^\bullet(\pt)} F  \cong \bigoplus_{n\in \bb N} \bigoplus_{\lambda_\in \Fp_n} F_\lambda \cong F[b_i]_{i\leq -1} = \pi_F^k
\end{equation}
where $F_\lambda =  H_T^\bullet(\pt_\lambda)\otimes_{ H_T^\bullet(\pt)}F$, and $\pi_F^k$ denotes the vacuum module of the Heisenberg algebra over the base field $F$ at level $k\in F$, a polynomial algebra in variables $b_i$ for $i\in \bb Z_{\leq -1}$, in keeping with Equation \ref{fockmodeqn}. For example, for $n=7$ the following are each of the corresponding data for a particular partition $\lambda \in \Fp_7$, where we omit the map from the framing vertex in the representation:
\begin{equation}\label{ADHMrepeqn}
	 \scriptsize{\vcenter{\xymatrixcolsep{1pc}
	\xymatrixrowsep{1pc}\xymatrix{ \C_{2\e_2} \\ \C_{\e_2} \ar[r]^{B_1} \ar[u]^{B_2} & \C_{\e_1+\e_2} \\ \C_0 \ar[r]^{B_1} \ar[u]^{B_2} & \C_{\e_1}\ar[r]^{B_1} \ar[u]^{B_2} & \C_{2\e_1} \ar[r]^{B_1}& \C_{3\e_1} }}} \quad \mapsto \quad (y^3, xy^2, x^2y ,x^4)  \quad \mapsto \quad \begin{ytableau}
	y^2 & 	\none[ ] & 	\none[ ]  &\none[]  \\
	y & xy 	 & 	\none[ ] &	\none[ ] \\
	1 & x & x^2 & x^3
\end{ytableau} \quad \mapsto \quad b_{-1}^3b_{-2}^2b_{-3} b_{-4}  \ .  
\end{equation}

 \begin{figure}[b]
 	\caption{Reduced quiver and vacuum character for $Y=\C^3$ and $S_0=\C^2$}
 	\hspace*{-6cm}
 	\begin{overpic}[width=.3\textwidth]{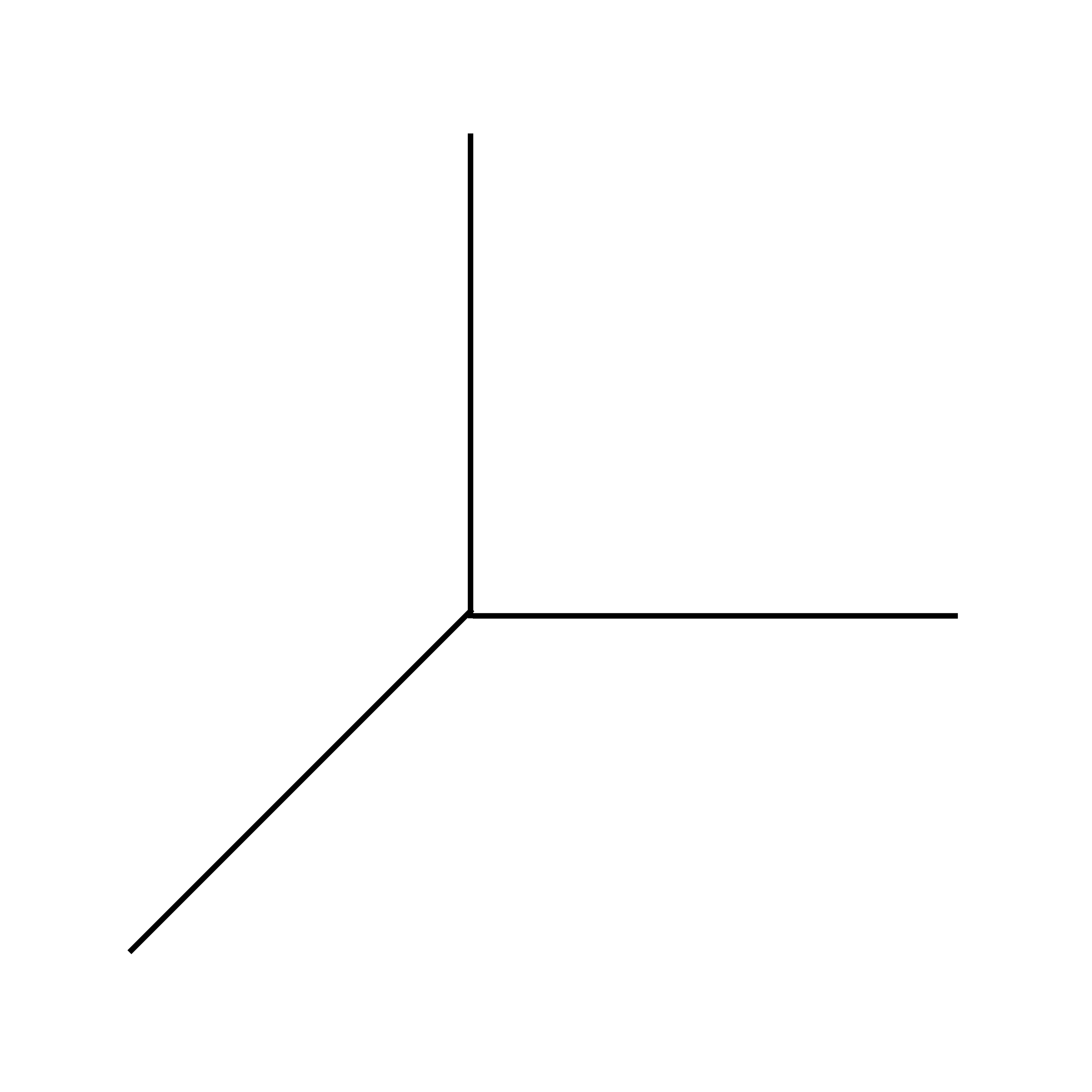}
 		\put(65,65) {$1$}
 		\put(46,46) {$J^0$}
 	\end{overpic}
 	\hspace*{.5cm}
 	\begin{overpic}[width=.25\textwidth]{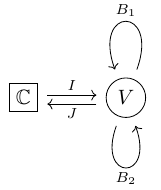}
 \put(5,-10) { $[B_1,B_2]+IJ=0 $ }
 \put(110,47) {$ \mc Z_{\C^2}^\VW(q) = \prod_{k = 1}^\infty \frac{1}{1-q^k} $}
 	\end{overpic}
 	\label{fig:C2redquivfig}
 	\vspace*{-1cm}
 \end{figure}

 The correspondences of Equation \ref{Nakcorrmainkeqn} and the induced endomorphisms of Equation \ref{Nakendmainkeqn} for $\V_{\C^2}$ are precisely those of the original arguments of Grojnowski \cite{Groj} and Nakajima \cite{Nak} recalled in the introduction in Equations \ref{Nakcorreqn} and \ref{Nakopeqn}, respectively, and thus by Theorem \ref{Nakthm} the level $k\in F$ is negative the equivariant norm of $[\C^2]\in H_\bullet^T(\C^2)$ with respect to the intersection pairing,
\begin{equation}\label{Heisleveleqn}
	  k = - \langle [\C^2] , [\C^2] \rangle = -\frac{1}{\e_i \e_j} \ ,
\end{equation}
 where $i,j \in \{1,2,3\}$ are determined by the choice of toric divisor $\C^2 \subset \C^3$. In fact, in this case one can explicitly calculate the commutators of the generators $b_k=\alpha_k([\C^2])$ of Equation \ref{Nakopeqn} and prove the result, using the the presentation in terms of the localization theorem
\begin{equation}\label{explicitC2geneqn}
	  b_{-k}^n= \sum_{(\lambda < \mu ) \in \Fp_{n,n+k}}  \frac{\e_T(T_\lambda \Hilb_n) }{ \e_T(T_{(\lambda,\mu )} \Hilb_{n,n+k})}\ : {H}_\bullet^T(\Hilb_n)\otimes_{H_T^\bullet(\pt)} F  \to  {H}_\bullet^T(\Hilb_{n+k})\otimes_{H_T^\bullet(\pt)} F  \ .
\end{equation}
\end{eg}

Our goal in the remainder of this section is to generalize the preceding calculations to directly check that the geometric construction in the preceding Section \ref{GRintrosec} defines the desired vertex algebra $\Pi(Y,S_0)$ from Section \ref{ablatticesec} in the remaining examples listed in Equation \ref{localmodeleqn}. We begin with an explicit check in one further example, then outline the analogous argument in the general case:

\begin{eg}\label{Om1GReg}
We now outline the analogous proof of Conjecture \ref{GRconj} in the case that the threefold $Y=Y_{1,1} = |\mc O_{\bb P^1}(-1)\oplus \mc O_{\bb P^1}(-1)| \to X_{1,1} = \{ xy-zw\}$ and $S_0=|\mc O_{\bb P^1}(-1)|$, for which the associated framed quiver with potential is that on the right hand side of Equation \ref{quiveregeqn} with $r=1$.

In this example, the dimensional reduction isomorphism recalled in Equation \ref{dimredeqn} above induces
\begin{equation}\label{Om1dimredeqn}
	 \hspace*{-1.5cm} \V_{|\mc O_{\bb P^1}(-1)|} =\bigoplus_{n_0,n_1\in \bb N}  H_\bullet^A( \mc M^0_{n_0,n_1}(Y_{1,1},|\mc O_{\bb P^1}(-1)|), \varphi_{W_{|\mc O_{\bb P^1}(-1)|}^0}) \otimes_{H_A^\bullet(\pt)} F  \xrightarrow{\cong } \bigoplus_{n_0,n_1\in \bb N} H_\bullet^A(\mc M_{(1,n_0,n_1)}(Q_\textup{NY})) \otimes_{H_A^\bullet(\pt)} F   \ , 
\end{equation}
where $\mc M_{(1,n_0,n_1)}(Q_\textup{NY})=\mf M_{(1,n_0,n_1)}^{^\infty \zeta}(Q_\textup{NY})$ denotes the moduli space of $^\infty\zeta$-stable representations of the framed quiver with relations $Q_\textup{NY}$ in Figure \ref{fig:conrotfig} below, such that $\dim V_0=n_0$, $\dim V_1=n_1$ and the framing vector space is of dimension $r=1$; precisely this moduli space was considered by Nakajima-Yoshioka \cite{NY1}, where they proved it is isomorphic to the moduli space $\widehat{\mc M}(1,\ell,n)$ of framed, rank 1 torsion free sheaves $E$ on $|\mc O_{\bb P^1}(-1)|$ with Chern classes $c_1(E)=\ell$ and $c_2(E)=n$:

\begin{figure}[b]
	\caption{Reduced quiver and vacuum character for $Y=Y_{1,1}$ and $S_0=|\mc O_{\bb P^1}(-1)|$}
	 	\hspace*{-7cm}
	\begin{overpic}[width=.5\textwidth]{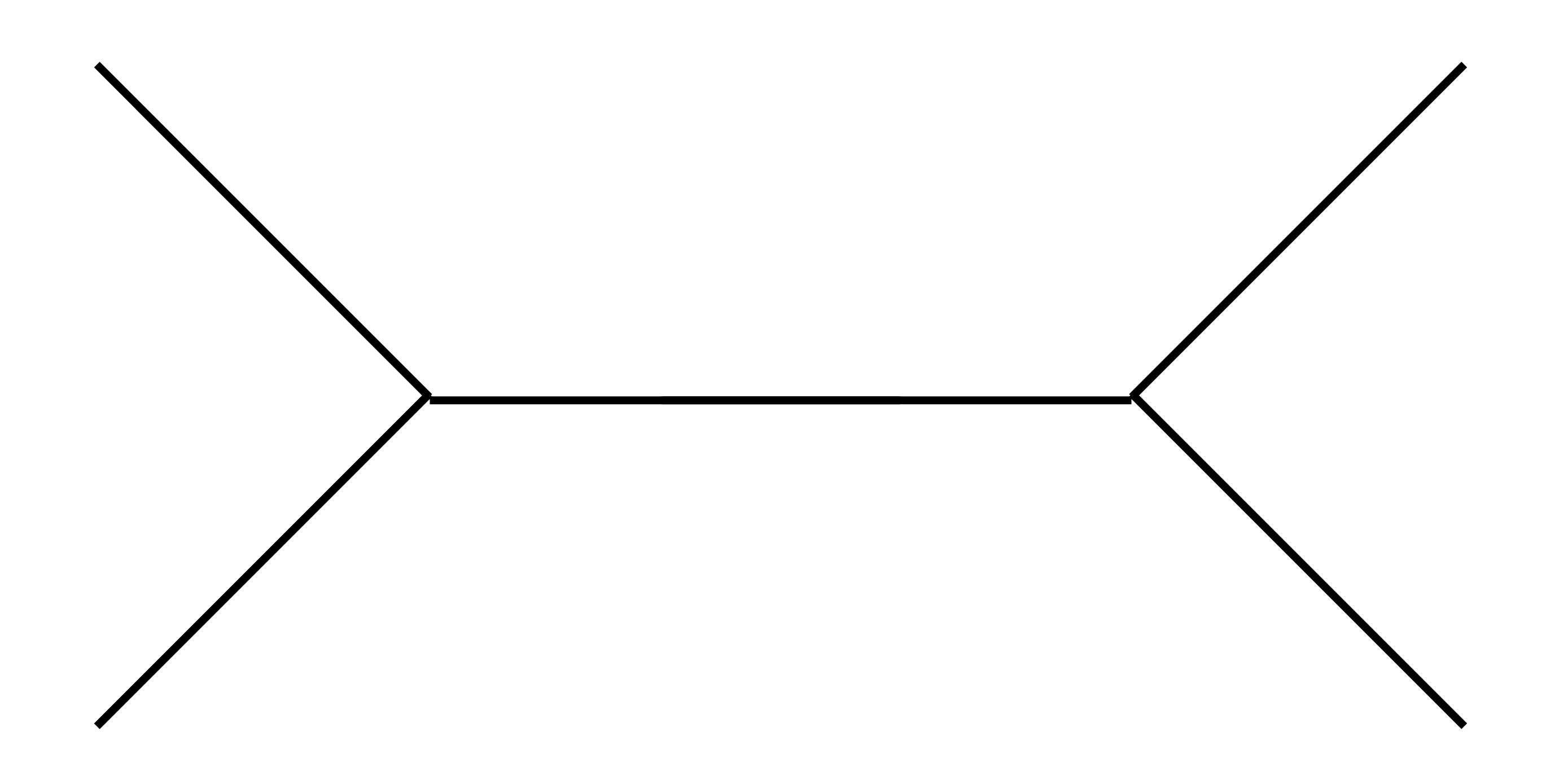}
		\put(27,26) {$ J^0$}
		\put(49,26) {$V$}
		\put(49,40) {$1$}
		\put(68,26) {$J^\infty$}
	\end{overpic}
	\hspace*{.5cm}
	\begin{overpic}[width=.25\textwidth]{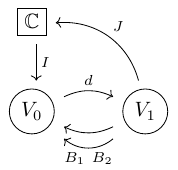}
		 \put(3,-15) { $ B_1 d B_2 -B_2 d B_1 + IJ=0 $ }
		\put(115,47) {$  \mc Z^\VW_{|\mc O_{\bb P^1}(-1)|} (q)= \sum_{\ell\in \bb Z} q^{\frac{\ell^2}{2}} \frac{1}{\prod_{k=1}^\infty (1-q^k)^2} \ .  $}
	\end{overpic}
	\label{fig:conrotfig}
\end{figure}
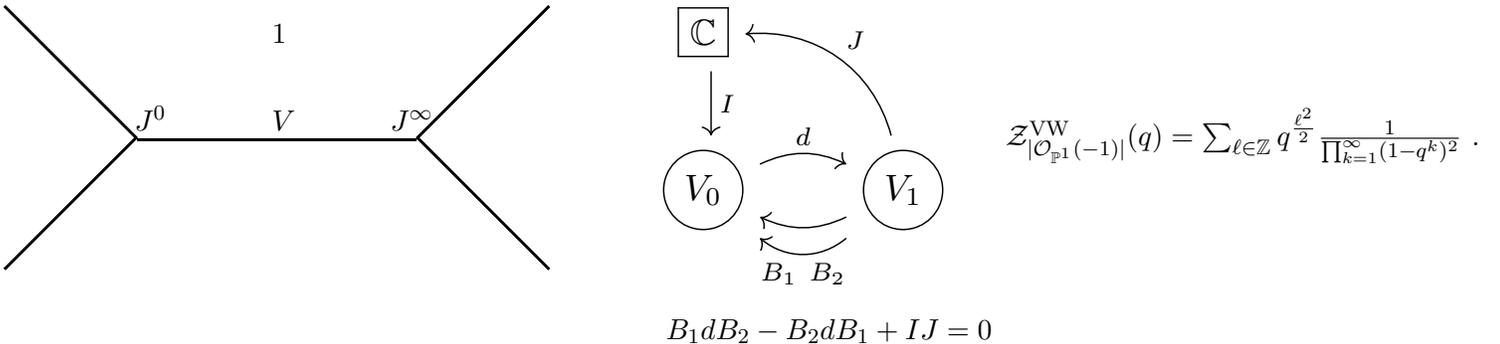

\begin{equation}\label{NYisoeqn}
	\mc M_{(1,n_0,n_1)}(Q_\textup{NY})\xrightarrow{\cong} \widehat{\mc M}(1,\ell,n) \quad\quad \text{for}\quad\quad \ell = n_1-n_0 \quad\quad\text{and}\quad\quad n=\frac{1}{2}( n_0+n_1 -\ell^2) \ ,
\end{equation}
where we have chosen the stability condition $\zeta_\VW$ defining $\mc M^0_{n_0,n_1}(Y_{1,1},|\mc O_{\bb P^1}(-1)|)$ to be that corresponding to the stability condition denoted by $^\infty \zeta$ in \emph{loc. cit.}.

By Proposition 3.2 of \cite{NY0}, the $T$ fixed points of $ \widehat{\mc M}(1,\ell,n)$ are of the form $I(\ell C)=I \otimes_{\mc O} \mc O(\ell C)$ for $I$ a $T$ fixed ideal sheaf of codimension $n$, defining a zero dimensional, $T$ fixed subscheme denoted $Z_I\subset S_0=|\mc O_{\bb P^1}(-1)|$. Any such subscheme is supported set theoretically on the union of the two distinct fixed points $0$ and $\infty$, and thus is canonically decomposed as a disjoint union $Z_I=Z_0\sqcup Z_\infty$ with each $T$-fixed $Z_i$ supported at $\{i\}$ corresponding to a partition $\lambda_i$ as in the preceding example.

In summary, we have that for each $\ell\in \bb Z$ and $n\in \bb N$, the set of $T$ fixed points $\Fp_{\ell,n} = \widehat{\mc M}(1,\ell,n)^T$ is given by the set of pairs of partitions $(\lambda_0,\lambda_\infty)$ such that $|\lambda_0|+|\lambda_\infty|:=n_0+n_\infty =n$. In particular
\begin{equation}\label{VOm1loceqn}
	\V_{|\mc O_{\bb P^1}(-1)|}  \xrightarrow{\cong }  \bigoplus_{n\in \bb N, \ell\in \bb Z}  \bigoplus_{(\lambda_0,\lambda_\infty)\in \Fp_{\ell,n}}  F_{(\ell,\lambda_0,\lambda_\infty)}  \cong  \bigoplus_{ \ell\in \bb Z} \bigg( \bigoplus_{n_0 \in \bb N} \bigoplus_{ \lambda_0 \in \Fp_{n_0}} F_{\lambda_0} \bigg)_{\ell} \otimes \bigg(  \bigoplus_{n_\infty \in \bb N} \bigoplus_{ \lambda_\infty \in \Fp_{n_\infty}} F_{\lambda_\infty}\bigg)_{\ell}   \ ,  
\end{equation}
which is evidently consistent with our algebraic definition of the algebra $\V(Y,S_0)$  for $S_0=|\mc O_{\bb P^1}(-1)|$. We now explain the concrete proof of Conjecture \ref{GRconj} in this case using localization, analogous to the proof in the case $S_0=\C^2$ recalled in Example \ref{geometricC2eg} above.

As in Equation \ref{ADHMrepeqn}, each representation of the quiver with relations in Figure \ref{fig:conrotfig} describes the presentation of the corresponding torsion free sheaf as an iterated extension of the compactly supported perverse coherent sheaves corresponding to the vertices of the quiver, together with a single copy of the structure sheaf $\mc O_{S_0}[1]$, as explained in the companion paper \cite{BR1}. For simplicity, we will describe these presentations in terms of sheaves on $S_0=|\mc O_{\bb P^1}(-1)|$ rather than $Y=Y_{1,1}$, since all such $\zeta_\VW$-stable extensions in this case are supported scheme-theoretically on $S_0$. We still follow the structural results described in \emph{loc. cit.} in the three dimensional setting:

Let $C=\bb P^1$, $\iota:C\to S_0=|\mc O_{\bb P^1}(-1)|$ denote the zero section, and to begin consider the representations of the unframed variant of the reduced quiver of Figure \ref{fig:conrotfig}, which corresponds to the dimensional reduction of the quiver on the left side of Equation \ref{quiveregeqn}. The one dimensional representations determined by the vertices of this quiver correspond to the simple, compactly supported, perverse coherent sheaves on $S_0\to S_0^\textup{aff}$ in the sense of Bridgeland \cite{Bdg1}, which are given by
\[ F_0 = \iota_* \mc O_C \quad\quad \text{and}\quad\quad F_1 = \iota_* \mc O_C(-1)[1] \ ,\]
for the one dimensional, unframed representations of dimension $(1,0)$ and $(0,1)$, respectively.

The linear maps $B_1$ and $B_2$ are the Koszul dual generators to the standard arrows in the Beilinson quiver describing coherent sheaves on $\bb P^1$ in terms of representations of the Kronecker quiver \cite{Bei2}. For example, the structure sheaves of zero dimensional subschemes of $C$ correspond to quiver representations of equal dimension $n_0=n_1$ such that $d=0$, and in particular the structure sheaf a single closed point $c\in C$ can be expressed as a quiver representation, and correspondingly an extension of the objects $F_0$ and $F_1$, given by 
\begin{equation}\label{Om1ptshveqn}
	  \xymatrix{ \bb C_{0}  & \ar@<5pt>[l]^{B_1} \ar@<-5pt>[l]_{B_2}  \C_{1} } \quad\quad \text{and} \quad\quad  F_0=\iota_*\mc O_C \to \iota_{c*} \mc O_{c} \to \iota_*\mc O_C(-1)[1]=F_1  \ ,  
\end{equation}
 respectively, where the values of the linear maps $B_1,B_2\in \C$ determine $c=B_1/B_2\in \bb P^1$; this was the basic example motivating the definition of the perverse coherent heart in \cite{Bdg1}.
 
 Similarly, the line bundles $\iota_*\mc O_C(m)$ and $ \iota_* \mc O_C(-m-1)[1]$ for $m\geq 0$ are determined by
 \begin{equation}\label{Beilinebuneqn}
 	  \xymatrix{ \bb C_{0}^{m+1}  & \ar@<5pt>[l]^{[\id \ 0]} \ar@<-5pt>[l]_{[0 \ \id ]}  \C_{1}^m }  \quad\quad \text{and}\quad\quad  \xymatrix{ \bb C_{0}^{m}  & \ar@<5pt>[l]^{[\frac{\id}{0}]} \ar@<-5pt>[l]_{[\frac{0}{\id}]}  \C_{1}^{m+1} } \ , 
 \end{equation}
 respectively, which correspond to expressions for the desired sheaves as iterated extensions of the simple objects $F_0$ and $F_1$: for example, for $m=0$ they are the one dimensional representations corresponding to $F_0$ and $F_1$ themselves as above, and for $m=1$ they correspond to the extensions
 \[ \mc O_C^{\oplus 2} \to \mc O_C(1) \to \mc O_C(-1)[1] \quad\quad \text{and}\quad\quad  \mc O_C \to \mc O_C(-2)[1] \to \mc O_C(-1)^{\oplus 2}[1] \ , \]
 determined by the Euler-type short exact sequences of vector bundles on $\bb P^1$ given by
 \[ \mc O_C(-1) \to \mc O_C^{\oplus 2} \to \mc O_C(1) \quad\quad \text{and}\quad\quad  \mc O_C(-2)[1] \to \mc O_C(-1)^{\oplus 2}[1] \to \mc O_C[1] \ . \]

We now describe the geometric interpretation of the stable, framed representations of the reduced quiver as torsion free sheaves. We focus on the graded representations, which correspond to the $T$-fixed points in the corresponding moduli space of sheaves, for the choice of compatible bigrading 
\[ \deg(B_1) = (1,0) \quad \quad \deg(B_2) = (-1,0) \quad \quad \deg(d)=(0,1) \quad\quad \deg(I)=(0,0) \quad\quad \deg(J)=(0,1)  \ .\]

To begin, recall that on the $\ell=0$ component a fixed point corresponds to the ideal sheaf $\mc I_Z$ of a zero dimensional subscheme $Z=Z_0\sqcup Z_\infty$, with each $Z_i$ supported set theoretically at the point $\{i\}\subset C=\bb P^1$ for $i=0,\infty$, which are determined by partitions $\lambda_0$ and $\lambda_\infty$. As usual, the ideal sheaves $\mc I_Z$, and their corresponding representations, can be expressed as an extensions
\[ \mc O_Z \to \mc I_Z[1] \to \mc O_{S_0}[1] \quad\quad \text{and correspondingly decompose as} \quad\quad  \C_\f \xrightarrow{I} R \]
where $\C_f$ denotes the framing vector space and $R$ denotes a representation of the unframed quiver with potential corresponding to the compactly supported (perverse) coherent sheaf $\mc O_Z$. Thus, we omit the map from the framing vertex when describing such quiver representations in order to simplify notation, as in Equation \ref{ADHMrepeqn} in the preceding example. The two fixed points for which $R$ is of the form in Equation \ref{Om1ptshveqn} correspond to $[\mc O_{S_0} \to \mc O_{\{0\}}]$ or $[ \mc O_{S_0} \to \mc O_{\{\infty\}}]$, with graded quiver representations and in turn coloured partitions
\begin{equation}\label{fixedptrepseqn}
	 \xymatrix{ \bb C_{0,0}  & \ar[l]_{B_2} \C_{1,\e_1} }\quad\quad\text{or}\quad\quad \xymatrix{ \C_{1,-\e_1} \ar[r]^{B_1} & \bb C_{0,0}  } \quad\quad\text{and}\quad\quad  \begin{ytableau}
	 	0  & *(lightgray) \ \end{ytableau} \quad\quad\text{or} \quad\quad \begin{ytableau}*(lightgray) \  & 0 \end{ytableau} \ .
\end{equation}

The analogous descriptions for fixed points labelled by $\lambda_0$ or $\lambda_\infty$ in the $\ell=0$ component are
\begin{equation}\label{conquiverrepegeqn}
\hspace*{-2.3cm} \scriptsize{\vcenter{\xymatrixcolsep{1pc}
		\xymatrixrowsep{1pc}\xymatrix{ 
	\C_{0,-2\e_1+2\e_2} &  \ar[l]_{B_2} \C_{1,-\e_1+2\e_2} & \\
 & \C_{0,-\e_1+\e_2} \ar[u]_{d} & \ar[l]_{B_2} \C_{1,\e_2}\ar[r]^{B_1}  & \C_{0,\e_1+\e_2} & \ar[l]_{B_2} \C_{1,2\e_1+\e_2}  \\
 && \C_{0,0} \ar[u]_d & \ar[l]_{B_2} \C_{1,\e_1} \ar[r]^{B_1} & \C_{0,2\e_1} \ar[u]_{d} & \ar[l]_{B_2} \C_{1,3\e_1}
 }}  \ \text{or} \  
\vcenter{\xymatrixcolsep{1pc}
		\xymatrixrowsep{1pc}\xymatrix{ & & & &  \C_{1,\e_1+2\e_2} \ar[r]^{B_1} & \C_{0,2\e_1+2\e_2} \\
		& & & \C_{1,\e_2} \ar[r]^{B_1} & \C_{0,\e_1+\e_2} \ar[u]_{d} \\ 
		\C_{1,-3\e_1} \ar[r]^{B_1} & \C_{0,-2\e_1}  & \ar[l]_{B_2} \C_{1,-\e_1} \ar[r]^{B_1} &  \C_{0,0} \ar[u]_{d}	}}}
\end{equation}
for example, which correspond to coloured partitions determining $\lambda_0$ or $\lambda_\infty$, as for example
\begin{equation}
\hspace*{-1cm}\begin{ytableau}
 \  & *(lightgray) \ \\
 \none[] &  \  & *(lightgray) \  &  \  & *(lightgray) \ \\
 \none[] & \none[] &  0  & *(lightgray) \ & \  & *(lightgray) \ 
\end{ytableau} 
\quad\quad \text{or}\quad\quad 
\begin{ytableau}
 \none[] &  \none[] &  \none[] &  \none[] &*(lightgray) \ & \  		\\
 \none[] & \none[] 	&   \none[] & *(lightgray) \ &  \  \\
	*(lightgray)	\ 	&  \  & 	*(lightgray) \ & 0 & \none[]
\end{ytableau} 
\quad\quad \implies \quad\quad 
\lambda_0= \begin{ytableau}
	x^2 & 	\none[ ]   \\
	x	& 	xz \\
1		&  z
\end{ytableau} 
\quad\quad \text{or} \quad\quad \lambda_\infty =
\begin{ytableau}
	\none &  \tilde{x}^2 \\
\none	&  \tilde{x} \\
\scriptsize{z^{-1}}	&  1 
\end{ytableau}
\end{equation}
for the representations of the preceding Equation \ref{conquiverrepegeqn}.

Similarly, for each $\ell \geq 0$ the line bundle $\mc O(l)$ can be described in terms of such quiver representations as follows: let $\ell C$ denote the degree $\ell$ Cartier divisor on $S_0=|\mc O_{\bb P^1}(-1)|$ with $[\ell C]=\ell[C]$, so that we have a short exact sequence
\begin{equation}\label{Olcompeqn}
	 \mc O_{S_0}(-\ell [C]) \to \mc O_{S_0} \to \mc O_{\ell C}  \quad\quad \text{and thus an exact triangle} \quad\quad  \mc O_{\ell C} \to \mc O(\ell)[1] \to \mc O_{S_0}[1] \ , 
\end{equation}
noting $\mc O_{S_0}([C])=\mc O_{S_0}(-1)$. Moreover, for each $\ell$ there is a short exact sequence
\[ \mc O_C(\ell-1) \to \mc O_{\ell C} \to \mc O_{(\ell-1) C} \quad\quad \text{inducing} \quad\quad \mc O_{\ell C} = \left[ \mc O_C(\ell-1) < \mc O_C(\ell-2) < \hdots < \mc O_C  \right] \]
a composition series in which each of the factors are given by the iterated extensions of $F_0$ and $F_1$ corresponding to the quiver representations on the left side of Equation \ref{Beilinebuneqn}. These composition series correspond to the sequence of quiver representations 
\begin{equation}\label{linebunrepseqn}
	\scriptsize{
\vcenter{\xymatrix{ \C_{0,0} }}  \quad \quad \vcenter{
\xymatrixcolsep{1pc}\xymatrixrowsep{1pc}\xymatrix{ 
		& \C_{0,-\e_1+\e_2} & \ar[l]_{B_2} \C_{1,\e_2}\ar[r]^{B_1}  & \C_{0,\e_1+\e_2}\\
		&& \C_{0,0} \ar[u]_d 
	}} \quad\quad
\vcenter{\xymatrixcolsep{1pc}\xymatrixrowsep{1pc}\xymatrix{ 
	\C_{0,-2\e_1+2\e_2} &  \ar[l]_{B_2} \C_{1,-\e_1+2\e_2} \ar[r]^{B_1} & \C_{0,2\e_2}  & \ar[l]_{B_2} \C_{1,\e_1+2\e_2} \ar[r]^{B_1} & \C_{0,2\e_1+2\e_2}  \\
	& \C_{0,-\e_1+\e_2} \ar[u]_{d} & \ar[l]_{B_2} \C_{1,\e_2}\ar[r]^{B_1}  & \C_{0,\e_1+\e_2} \ar[u]_{d} \\
	&& \C_{0,0} \ar[u]_d 
}
} \quad \hdots } 
\end{equation}
and in turn the sequence of coloured partitions
\[ \scriptsize{ \begin{ytableau}  \none[] \\ \none[] \\ 0
\end{ytableau}  \quad\quad \begin{ytableau}  \none[] \\
\ & *(lightgray) \  & \ \\ \none[] & 0
\end{ytableau}  \quad\quad
\begin{ytableau}  \ & *(lightgray) \  & \ & *(lightgray) \  & \ \\
  \none[]  & \ & *(lightgray)  &\   \\ \none[] &  \none[] & 0
\end{ytableau}  \quad\quad \hdots  \quad\quad }\text{and similarly} \scriptsize{ \quad\quad  \begin{ytableau} \none[0] \\ *(lightgray) \  \end{ytableau}
\quad\quad \begin{ytableau} \none[]  & \none[0] \\
 \none[] &  *(lightgray)  \\
 *(lightgray)\  & \ & *(lightgray) \ 
 \end{ytableau}
 \quad\quad \begin{ytableau}\none[]  & \none[] & \none[0] \\
 \none[] &  	\none[] &  *(lightgray)  \\
 \none[] &	*(lightgray)\  & \ & *(lightgray) \ \\
 	*(lightgray)\  & \ &	*(lightgray)\  & \  &	*(lightgray)\ 
 \end{ytableau} }  \quad\quad \hdots   \]
  are the analogous shorthand expressions for the quiver representations corresponding to the presentations of $\mc O(-l)$ for $\ell\geq 0$ as iterated extensions of the objects $F_0$ and $F_1$ with a single $\mc O_{S_0}$; the latter are determined similarly from the composition series of $\mc O(-\ell)$ induced by the exact triangles
  \[ \mc O_{S_0}[1] \to \mc O(-\ell)[1] \to \mc O_{\ell C}(-\ell) [1] \quad\quad \text{and}\quad\quad   \mc O_C (-1)[1]  \to \mc O_{\ell C}(-\ell)[1] \to \mc O_{(\ell-1)C} (-\ell +1)[1]  \ ,\]
  together with the expressions for $\mc O_C(\ell-1)[1]$ given on the right side of Equation \ref{Beilinebuneqn}.

 Recall that the conjectural generators for the geometric action of the algebra of modes of the vertex algebra are defined in Section \ref{GRintrosec}, by Theorem \ref{Bthm}, Corollary \ref{Nakopcoro}, and Propositions \ref{Nakopprop} and \ref{opactprop}. Consider the spherical subalgebras $\mc{SH}^y(Y)\subset \mc H(Y)$ generated by the equivariant fundamental classes $[\pt_y] \in \mc H(Y)_{(1,1)}$ for $y=0,\infty$ corresponding to the unframed variants of the representations pictured in Equation \ref{fixedptrepseqn}. The natural basis vectors in each component of the subalgebra
\[ [\pt_y]^k \in \mc{SH}_k^y(Y):= \mc{SH}^y(Y)\cap \mc{SH}_{2k}(Y)  \quad\quad \textup{define} \quad\quad \tilde{b_k^y} = \rho_{S_0}([\pt_y]^k) \in \End_F(\V_{S_0})  \]
 as their image under the representation $\rho_{S_0}:\mc H(Y) \to \End_F(\V_{S_0}) $ of Corollary \ref{Nakopcoro} for $k\leq 0$, and similarly their adjoints for $k>0$. Further, by Propositions \ref{Nakopprop} and \ref{opactprop}, these endomorphisms can be computed in terms of the correspondences of Equation \ref{Nakcorrmainkeqn}, and by localization we obtain an analogous formula to that of Equation \ref{explicitC2geneqn}, from which one can readily compute
\begin{equation} 
	 [\tilde{b}_n^y,\tilde{b}_m^{y'}] = -n \delta_{m-n} \delta_{y=y'} \e_T(T_y S_0)  \quad\quad \text{where}\quad\quad \begin{cases} \e_T(T_0 S_0) & = 2\e_1 (-\e_1+\e_2) \\ \e_T(T_\infty S_0) & = -2\e_1 (\e_1+\e_2) \end{cases}
\end{equation}
 are determined in the calculation by our grading conventions, and match the grading implicit the orientation of the toric diagram in Figure \ref{fig:conrotfig}. In particular, applying the rescaling automorphism of the Heisenberg algebra over $F$ we can define the equivalent generators
 \[b_n^y = \e_T(T_y S_0)^{-1} \tilde{b}_n^y \quad\quad\text{and let} \quad\quad b_0^y=\langle [\{y\}],c_1 \rangle = \ell \  \e_T(N_{C,y} S_0) =: \ell  \lambda_y \]
 so that the fields defined by
 \[ J^y(z) := \sum_{n\in \bb Z} b_n^y z^{-n-1} \ \in \End(\V_{S_0})\lP z^{\pm 1} \rP \quad\quad \text{satisfy} \quad\quad  J^y(z) J^{y'}(w) \sim  - \frac{\delta_{y=y'}}{\e_T(T_yS_d)} \frac{1}{(z-w)^2}  \] 
  as in Equation \ref{geoheisopeeqn}, for $y=0,\infty$. In summary, we obtain the desired representation of the Heisenberg subalgebra $\pi_{S_0}\cong \pi^{k_0} \otimes \pi^{k_\infty}$ defined in Equation \ref{HeisSeqn}, giving
\begin{equation}\label{Om1HeisSrepeqn}
	    \mc U(\pi^{k_0} \otimes \pi^{k_\infty} ) \to \End_F(\V_{S_0}) \quad\quad \text{for}\quad\quad k_0= -\frac{1}{2\e_1(-\e_1+\e_2)} \quad\quad k_\infty = \frac{1}{2\e_1(\e_1+\e_2)}   \ . 
\end{equation}
  Since the endomorphisms $b_n^y$ preserve the components corresponding to distinct first Chern class, as well as the factorization into pairs of partitions, the decomposition of Equation \ref{VOm1loceqn} induces
  \begin{equation}
  	\V_{S_0} \cong \bigoplus_{ \ell\in \bb Z} \pi^{k_0}_{\ell \lambda_0}\otimes \pi^{k_\infty}_{\ell \lambda_\infty}   \end{equation}
  an isomorphism of representations of $\pi_{S_0}\cong \pi^{k_0} \otimes \pi^{k_\infty}$, where $\lambda_y$ are as in Equation \ref{} above. This is evidently consistent with the vacuum module for the lattice extension of the Heisenberg algebra $\pi_{S_0}$ generated by the exponential vertex operators
\begin{equation}\label{Om1vertopeqn}
	   V_\ell(z) = \nol \exp \left( \ell ( \e_T(N_{C,0} S_0) \phi^{0}(z) + \e_T(N_{C,\infty}S_0) \phi^{\infty}(z) \right) \nor  \quad  \in \End_F(\V_{S_0}) \LFz \  
\end{equation}
  for $\ell \in \bb Z$, just as in Equation \ref{curveextopeqn} in the definition of $\V(Y,S_0)=\Pi(Y,S_0)$.
  
  It remains to show that this algebra is generated by the endomorphisms in the image of the geometric representation defined above. Each of the unframed variants $R_\ell$ of the representations pictured in Equation \ref{linebunrepseqn}, corresponding to the $\mc O_{\ell C}$ factors in the presentations of the line bundles $\mc O(\ell)$ as extensions in Equation \ref{Olcompeqn}, define classes
  \[ [\pt^0_\ell] \in \mc{SH}(Y)_{(\ell_0,\ell_1)} \quad\quad \text{and thus} \quad\quad e^{\ell,0} = \rho_{S_0}([\pt_\ell^0]) \in \End_F(\V_{S_0})  \]
   where $\ell_0= \ell(\ell+1)/2$ and $\ell_1= \ell(\ell-1)/2$ for $\ell\geq 0$ and opposite for $\ell < 0$.
  The only non-trivial components of the correspondence of Equation \ref{Nakcorrmainkeqn} inducing the endomorphism $e^\ell$ are given by
  \[  \widehat{\mc M}_{(1,0,n),(1,\ell,n)}^{[\pt^0_\ell]} \cong \{ (E,E')\in  \widehat{\mc M}_{(1,0,n)}\times  \widehat{\mc M}_{(1,\ell,n)} \ | \ E'=E(\ell):= E \otimes_{\mc O_{S_0}} \mc O(\ell) \  \}   \] 
  noting that for $E \in  \widehat{\mc M}_{(1,0,n)}$ there is a unique stable extension by $\mc O_{\ell C}$ given by
  \[ \mc O_{\ell C} \to  E(\ell)[1] \to  E[1] \]
  where the maps are induced by those of the exact triangle on the right side of Equation \ref{Olcompeqn}. Evidently the correspondence is isomorphic to $\widehat{\mc M}_{(1,0,n)}$ itself by the canonical projection, so that the endomorphism $e^\ell$ is simply induced by pushforward along the map of moduli spaces,
  \[ e^{\ell,0}  = E^{\ell,0}_* : \bb V_{S_0,(0,n)} \to \bb V_{S_0,(\ell,n)} \quad\quad \text{for}\quad\quad E^{\ell,0} :=(\cdot) \otimes_{\mc O_{S_0}} \mc O(\ell) :\widehat{\mc M}_{(1,0,n)} \to \widehat{\mc M}_{(1,\ell,n)}  \]
  which are all isomorphisms. The induced maps evidently commute with the endomorphisms $b_n^y$ for $n\neq 0$, and extend uniquely to a representation of the group algebra $F[\bb Z]\to \End_F(\V_{S_0})$ defined by
\begin{equation}\label{shiftopgeoeqn}
	   \ell \mapsto  e^\ell = \sum_{\ell'\in \bb Z} e^{\ell,\ell '} : \V_{S_0} \to \V_{S_0} \quad\quad \text{where} \quad\quad   e^{\ell,\ell '} : \V_{S_0,(\ell',n)} \to \V_{S_0,(\ell+\ell',n)} \ , 
\end{equation}
  are defined uniquely by compatibility with the endomorphisms $e^{\ell,0}$. In summary, the endomorphisms $e^\ell$ identify with the shift operators $S_\ell$ of Equation \ref{shiftopeqn} in this case, and together with the action of the Heisenberg subalgebra $\pi_{S_0}$ defined in Equation \ref{Om1HeisSrepeqn} above, generate the action of the vertex operators $V_\ell(z)$ of Equation \ref{Om1vertopeqn} above, and thus the desired free field vertex algebra $\V(Y,S_0)=\Pi(Y,S_0)$ in this example, as claimed.
\end{eg}

We now outline a proof of the main result of this section, establishing Conjectures \ref{GRconj} and \ref{vacuumconj} for irreducible, reduced divisors $S_0\subset Y$:

\begin{theo}\label{abGRthm} There exists a natural representation 
	\[\rho: \mc U(\V(Y,S_0)) \to \End_F(\V_{S_0} )\]
	of the algebra of modes $\mc U(\V(Y,S_0))$ of the vertex algebra $\V(Y,S_0)=\Pi(Y,S_0)$ on $\V_{S_0}$, such that $\V_{S_0}$ is identified with the vacuum module $\Pi(Y,S_0)_0$ for $\Pi(Y,S_0)$.
\end{theo}

\begin{proof} To begin, note that for the class of surfaces $S_0$ appearing in Equation \ref{localmodeleqn}, we have
\[ \Pic(S_0) \xrightarrow{\cong}  H^2(S_0; \bb Z) \cong H_2(S_0;\bb Z) \cong \bb Z^{\mf C_{S_0}} \ , \]
where we recall that $\mf C_{S_0}$ denotes the index set for the irreducible, compact toric curve classes $C_i\subset S_0$. Further, recall the notation for the set of fixed points $y\in \mf F_{S_0}=S_0^T$, and note that each fixed point is contained in an affine, toric Zariski chart $U_y \cong \C^3 \subset Y$ such that 
\[S_y=S_0\times_{Y}U_y \cong \C^2 \subset  \C^3 \ .\]
	
Analogously to Equation \ref{VOm1loceqn}, by the localization theorem we have the vector space isomorphism
\[ \V_{S_0} \xrightarrow{\cong } \bigoplus_{\ell \in \bb Z^{\mf C_{S_0}}}  \bigotimes_{y\in \mf F_{S_0}} \bigg( \bigoplus_{n_y \in \bb N} \bigoplus_{\lambda_{n_y}\in \Fp_{n_y}} F_{\lambda_{n_y}} \bigg)_{\ell}  \]
corresponding to the decomposition of $\mc M^0(Y,S_0)$ into components with first Chern class $\ell=c_1(E) \in H_2(S_0,\Z)^\vee \cong \bb Z^{\mf C_{S_0}}$, and the enumeration of the $T$-fixed points in each component in terms of tuples of partitions $(\lambda_{n_y} \in \Fp_{n_y})_{y\in \mf F_{S_0}}$ of lengths $n_y\in \bb N$, corresponding to the zero dimensional subschemes of $S_0$ supported set theoretically on the fixed points $S_0^T$. Moreover, applying Proposition \ref{localsubrepprop} to each of the open sets $U_y$, together with the proof of the desired result for $(Y,S)=(\C^3,\C_2)$ from \cite{SV} as recalled in the introduction, we obtain inclusions of subalgebras
\[ \mc Y(\glh_1)\cong \mc Y_{S_{y}}(U_y) \to \mc Y_{S_0}(Y) \]
under which the restriction of the representation $\V_{S_0}$ to $\mc Y_{S_{y}}(U_y)$ has subrepresentation $\V_{S_0,U_y}$ induced by Proposition \ref{localsubrepprop} is given by
\[ \V_{S_0,U_y} = \bigoplus_{n_y \in \bb N} \bigoplus_{\lambda_{n_y}\in \Fp_{n_y}} F_{\lambda_{n_y}} \quad\quad \text{admitting a factorization}\quad \quad  \mc Y(\glh_1) \to \mc U(\pi^y) \to \End_F(\pi_{S_0,y}) \]
such that the representation $\V_{S_0,U_y}$ of $\mc U(\pi^{y})$ identifies with the vacuum module for $\pi^{y}$, the Heisenberg vertex algebra at level $k_y$ as in Equation \ref{Heisleveleqn}.

The representations $\mc U(\pi^y) \to \End_F(\V_{S_0})$ naturally commute for distinct fixed points $y\in \mf F_{S_0}$, as the supports of the compactly supported perverse coherent sheaves generating the cohomological Hall algebras $\mc H(U_y)$ are disjoint for distinct $y$, so that we obtain the desired representation of the Heisenberg subalgebra
\begin{equation}\label{piS0acteqn}
	 \mc U(\pi_{S_0})= \widehat{\bigotimes_{y\in \mf F_{S_0}}} \mc U(\pi^y) \to \End(\V_{S_0}) \quad\quad \text{where}\quad\quad \pi_{S_0}=\bigotimes_{y\in \mf F_{S_0}} \pi^y 
\end{equation}
denotes the Heisenberg subalgebra $\pi_{S_0} \subset \Pi(Y,S_0)$, as in Equation \ref{Heiscompeqn}.

Similarly, for each compact toric curve class $C_i\in \mf C_{S_0}$, let $0_i,\infty_i\in \mf F_{S_0}$ denote the $T$-fixed points contained in $C_i$ and define
\[ U_i = U_{0_i} \cup U_{\infty_i} \subset Y \quad\quad \text{and}\quad\quad S_i = S\times_Y U_i  \ .\]
Given our assumptions on $(Y,S_0)$, we must have either
\[ U_i \quad\cong \quad Y_{1,1} \quad\text{or} \quad Y_{2,0} \quad\quad \text{and similarly} \quad\quad S_i\quad \cong  \quad  |\mc O_{\bb P^1}(-1)|  \quad , \quad |\mc O_{\bb P^1}|\quad\text{or} \quad   |\mc O_{\bb P^1}(-2)| \ ,\]
where we recall that $Y_{1,1}\cong |\mc O_{\bb P^1}(-1)\oplus\mc O_{\bb P^1}(-1) | $ and $Y_{2,0}\cong |\mc O_{\bb P^1}\oplus\mc O_{\bb P^1}(-2) | $. Again, applying Proposition \ref{localsubrepprop} we obtain an inclusion of algebras $\mc Y_{S_i}(U_i)\to \mc Y_{S_0}(U)$ such that the subrepresentation of the restriction of $\V_{S_0}$ to $\mc Y_{S_i}(U_i)$ is given by
\[ \V_{S_0,U_i}= \bigoplus_{\ell_i \in \bb Z} \bigg( \bigoplus_{n_{0_i} \in \bb N} \bigoplus_{\lambda_{n_{0_i}}\in \Fp_{n_{0_i}}} F_{\lambda_{n_{0_i}}} \bigg)_{\ell_i}  \otimes \bigg( \bigoplus_{n_{\infty_i} \in \bb N} \bigoplus_{\lambda_{n_{\infty_i}}\in \Fp_{n_{\infty_i}}} F_{\lambda_{n_{\infty_i}}} \bigg)_{\ell_i}  \ , \]
as in Equation \ref{VOm1loceqn}. In the case $(U_i,S_i)=(Y_{1,1},|\mc O_{\bb P^1}(-1)|)$, by the computation in Example \ref{Om1GReg} above, we have the desired factorization
\begin{equation}\label{localcurverepeqn}
	 \mc Y_{S_i}(U_i) \to \mc U(\Pi(U_i,S_i)) \to \End_F(\V_{S_0,U_i}) 
\end{equation}
under which $\V_{S_0,U_i}$ is identified with the vacuum module $\Pi(U_i,S_i)_0$. A similar computation in the two remaining cases $(U_i,S_i) = (Y_{2,0},|\mc O_{\bb P^1}|$ or $(Y_{2,0}, |\mc O_{\bb P^1}(-2)| )$ implies the analogous factorizations in those examples, where in both cases the resulting quiver with potential
\[ \hspace*{-1.5cm} \begin{cases}  Q_{|\mc O_{\bb P^1}|}^{\f_0} & = 
	\begin{tikzcd}
	\boxed{\C}\arrow[d,shift left=0.5ex, "I"] \\ 
	\arrow[out=160,in=200,loop,swap,"E"] \mathcircled{V_0} \arrow[r, bend left=25 ] \arrow[r, bend left=40 ,  "A\ C"]  & \arrow[l, bend left=25 ] \arrow[l, bend left=40 ,  "B\ D"] \mathcircled{V_1}\arrow[out=340,in=20,loop,swap,"F"] \ar[ul,shift left=1ex, bend right=40]\ar[ul, bend right=40, swap,"J_1 \ J_2"]    
\end{tikzcd} \\  W_{|\mc O_{\bb P^1}|}^{\f_0} & =E(BC-DA)+ F(AD-CB)  \\ & \quad\quad + IJ_1A + IJ_2C 
\end{cases}  \quad\quad\text{and}\quad\quad \begin{cases} Q_{|\mc O_{\bb P^1}(-2)|}^{\f_0} &  = \quad \begin{tikzcd}
	\boxed{\C}\arrow[d,shift left=0.5ex, "I"] \\ 
	\arrow[out=160,in=200,loop,swap,"E"] \ar[u,shift left=0.5ex, "J"] \mathcircled{V_0} \arrow[r, bend left=25 ] \arrow[r, bend left=40 ,  "A\ C"]  & \arrow[l, bend left=25 ] \arrow[l, bend left=40 ,  "B\ D"] \mathcircled{V_1}\arrow[out=340,in=20,loop,swap,"F"]
\end{tikzcd} \\ W_{|\mc O_{\bb P^1}(-2)|}^{\f_0} & = E(BC-DA)+ F(AD-CB) \\ & \quad\quad + EIJ - IG_\f J \end{cases}  \]
admits an analogous dimensional reduction isomorphism to that of Equation \ref{Om1dimredeqn}, where the reduced quiver with relations is given by
 \begin{equation}\label{dimredNakquiveqn}
 	\hspace*{-1cm} \begin{cases}  Q_{|\mc O_{\bb P^1}|} = &  \begin{tikzcd}
	\boxed{\C}\arrow[d,shift left=0.5ex, "I"] \\ 
	\arrow[out=160,in=200,loop,swap,"E"] \mathcircled{V_0}  & \arrow[l,swap, "B\ D"] \arrow[l,shift left=1ex] \mathcircled{V_1}\arrow[out=340,in=20,loop,swap,"F"] \ar[ul,shift left=1ex, bend right=40]\ar[ul, bend right=40, swap,"J_1 \ J_2"]  
\end{tikzcd} \\  R_{|\mc O_{\bb P^1}|}  = &  EB-BF + I J_2 =0 \\  & DF-ED + I J_1  = 0 \end{cases}\quad\quad\text{and}\quad\quad \begin{cases} Q_{|\mc O_{\bb P^1}(-2)| } & = \begin{tikzcd}
	\boxed{\C}\arrow[d,shift left=0.5ex, "I"] \\ 
	\ar[u,shift left=0.5ex, "J"] \mathcircled{V_0} \arrow[r, bend left=25 ] \arrow[r, bend left=40 ,  "A\ C"]  & \arrow[l, bend left=25 ] \arrow[l, bend left=40 ,  "B\ D"] \mathcircled{V_1}
\end{tikzcd}  \\   R_{|\mc O_{\bb P^1}|}  = & BC-DA + IJ  =0 \\    & AD-CB   = 0 \end {cases} \ , 
 \end{equation}
respectively. In particular, this implies that the endomorphisms induced as in Equation \ref{shiftopgeoeqn} by the representation of Equation \ref{localcurverepeqn} satisfy the desired relations of Equation \ref{geoheisopeeqn} with the Heisenberg generators induced by the representations of $\pi^{0_i}$ and $\pi^{\infty_i}$ defined in Equation \ref{piS0acteqn}, and that together these generate the action of $\mc Y_{S_i}(U_i)$ and in turn $\mc U(\Pi(U_i,S_i))$.

Finally, note that the algebra $\mc Y_{S_0}(Y)$ is generated by the collection of subalgebras $\mc Y_{S_i}(U_i)$,
\[ \mc Y_{S_0}(Y) = F\langle \mc Y_{S_i} (U_i) \rangle_{i\in \mf C_{S_0}}   \quad \quad \text{and thus the representation} \quad\quad \mc Y_{S_0}(Y) \to \End_F(\V_{S_0})\]
is determined by the action of the generators of each of the subalgebras $\mc Y_{S_i}(U_i)$, which by the preceding paragraph satisfy the desired relations to determine a factorization
\[ \mc Y_{S_0}(Y) \to  \mc U(\Pi(Y,S_0)) \to \End_F( \V_{S_0}) \]
inducing the desired representation of $\mc U(\Pi(Y,S_0))$ and identification of $\V_{S_0}$ with the vacuum module $\Pi(Y,S_0)$, as desired.

\end{proof}

\subsection{Towards the proof in higher rank: factorization and locality}\label{generalGRsec} In this section, we explain some partial results towards a proof of the main conjectures of Section \ref{GRintrosec}. The strategy is to use relations between the various representations $\V_S^\f$ to reduce to calculations which can be checked in local models and low rank. The first main result required for our approach is the following natural refinement of Conjecture \ref{FFRconj}: Suppose we can express $\mc O_S$ as an extension
\[ 0\to \mc O_{R} \to \mc O_S \to  \mc O_{T} \to 0 \]
for divisors $R,T\subset Y$  as in the hypotheses preceding Proposition \ref{factprop}, and let $\f_{R,T}$ be the framing structure on $\mc O_{S^\red}^\sss$ of rank $\rr_S$ determined by $\mc O_R\oplus \mc O_T$.

\begin{conj}\label{geofactconj} There exists a natural isomorphism
\begin{equation}\label{geomfacteqn} 	 \V_{S}^{\f_{R,T}} \xrightarrow{\cong} \V_R^{\f_R} \otimes \V_S^{\f_S}   \end{equation}
of representations of $\V(Y,S)$, where $\V_{S}^{\f_{R,T}}$ is as in Conjecture \ref{voamodconj} for $\f=\f_{R,T}$, and $\V_R^{\f_R} \otimes \V_S^{\f_S} $ is the restriction along the factorization map of Proposition \ref{factprop} of the representations of Conjecture \ref{vacuumconj}.
\end{conj}

In particular, if $R$ and $T$ are reduced, irreducible divisors, then $\f_{R,T}=0_S$ is given by the trivial framing structure, so that by inductive application of the preceding conjecture for a choice of composition series
\[ \mc O_S = \left[ \mc O_{S_{d_1}} < ... <\mc O_{S_{d_k}} < ... < \mc O_{S_{d_N}} \right]  \]
we obtain the following corollary, which analogously refines Conjecture \ref{GRconj}:
\begin{corollary} There exists a natural isomorphism
\[\V_S^0 \xrightarrow{\cong }\bigotimes_{d\in \mf D_S}  \V_{S_d}^{\otimes r_d} = \Pi(Y,S)_0 \]
of representations of $\V(Y,S)$, where $\V_S^0$ is as in Conjecture \ref{GRconj} and $\bigotimes_{d\in \mf D_S}  \V_{S_d}^{\otimes r_d}$ is the restriction along the defining free field realization of $\V(Y,S)$ of the tensor product of the representations defined in Theorem \ref{abGRthm}, which in particular implies the latter equality.
\end{corollary}

Note that often the existence of an isomorphism of underlying vector spaces in Equation \ref{geomfacteqn} can be directly checked. For example, in many cases where $[S]=r[S_0]$ for $S_0$ a reduced, irreducible divisor, and $\f=0_S$ is given by the trivial framing structure, there exists a dimensional reduction isomorphism analogous to that of Equation \ref{Om1dimredeqn},
\[ H_\bullet^A(\mc M_\nnn^{0_S}(Y,S),\varphi_{W^{0_S}_{S}}) \xrightarrow{\cong} H_\bullet^A(\mc M_{\nnn,r}(Q_S,R_S)) \quad\quad \text{and thus}\quad\quad \V_S^0\cong \bigoplus_{\nnn} H_\bullet^A(\mc M_{\nnn,r}(Q_S,R_S))\otimes_{ H_T^\bullet(\pt)}F  \ , \]
for $(Q_S,R_S)$ a Nakajima quiver with relations and $\mc M_{\nnn,r}(Q_S,R_S)$ the moduli space of stable, framed representations of dimension $(\nnn,r)$ with respect to an appropriate choice of stability condition; the higher rank analogues of the framed quivers with relations in Equation \ref{dimredNakquiveqn} provide examples of this phenomenon. In this case, the \emph{tensor product} structure on the $A$-fixed points of Nakajima quiver varieties, in the sense of Section 2.4 of \cite{MO} for example, implies the existence of an isomorphism between the underlying vector spaces induced by localization, that is,
\[ \mc M_{\nnn,r}(Q_S,R_S)^A = \bigsqcup_{\nnn_1+...+\nnn_r=\nnn} \mc M_{\nnn_1,1}(Q_S,R_S) \times ... \times \mc M_{\nnn_r,1}(Q_S,R_S) \quad\quad \text{implies}\quad\quad \V_S^0 \cong \V_{S_0}^{\otimes r} \ .  \]

Similarly, in the setting of Example \ref{Om1GReg}, by Proposition 3.2 of \cite{NY0} we have an analogous description of the fixed points in the moduli space of representations of the reduced quiver with relations in Figure \ref{fig:conrotfig} and thus by the isomorphism of Equation \ref{Om1dimredeqn} together with the analogous application of the localization again implies the desired factorization at the level of vector spaces. From this perspective, the content of Conjecture \ref{geofactconj} above is the existence of a canonical normalization of such isomorphisms which intertwines the geometric representation on $\V_S^0$ defined in Section \ref{GRintrosec} above, and the tensor product of those geometric representations for reduced, irreducible components $\V_{S_d}$, restricted along the factorization maps of Proposition \ref{factprop}.

Following \cite{SV} and in turn a suggestion of Nakajima cited therein, our expectation is that these factorization maps are related to generalized coproducts on the shifted affine Yangian-type quantum groups $\mc Y_S(Y)$, which essentially by construction admit geometric representations on the vector spaces $\V_{S}$, following the results of the companion paper \cite{BR1} and references therein:

\begin{conj}Suppose there exist divisors $S_l$ on $Y$ for $l=1,...,h$, such that $\mc O_S$ admits 
\[	\mc O_S = \left[ \mc O_{S_{1}} < ... <\mc O_{S_{l}} < ... < \mc O_{S_{h}} \right] \ ,\]
a composition series such that each of the induced extensions $\mc O_{S_{l}} \to E \to \mc O_{S_{l+1}}$ satisfies the hypotheses preceding Proposition \ref{factprop}.
Then there exist generalized coproduct maps
\[\hspace*{-1cm} \Delta_{S_1,...,S_h} :\mc Y_S(Y) \to \bigotimes_{l=1}^h \mc Y_{S_l}(Y) \quad\quad\text{such that}\quad\quad \vcenter{\xymatrix{\mc Y_S(Y)  \ar[r]\ar[d] &  \mc U(\V(Y,S))  \ar[d] \ar[r] & \End_F(\V_{S}^\f) \ar[d]  \\ \bigotimes_{l=1}^h \mc Y_{S_l}(Y)  \ar[r] &  \bigotimes_{l=1}^h  \mc U(\V(Y,S_l)) \ar[r] &  \bigotimes_{l=1}^h  \End_F(\V_{S_l}) }} \ ,\]
commutes, where the top and bottom horizontal maps are as induced by Conjecture \ref{GRconj} for $S$ and $S_l$, respectively, the middle vertical arrow is given by the vertex algebra embeddings of Proposition \ref{factprop}, and the right vertical arrow is induced analogously by the isomorphisms of Equation \ref{geomfacteqn}.
\end{conj}

Finally, we outline a geometric approach to the higher rank analogue of the locality principle, established in the algebraic setting in Proposition \ref{localityprop}. The main result required is the following higher rank analogue of Proposition \ref{localsubrepprop}. To begin, note that in the setting of \emph{loc. cit.} a framing structure $\f$ for a divisor $S$ induces a natural framing structure $\f_U$ on $S_U=S\times_Y U$ by restriction.

\begin{prop}\label{rlocalsubrepprop} The natural generalization of the map of Equation \ref{stablemodspcmapeqn} induces an isomorphism
	\begin{equation}\label{rlocalsubrepeqn}
		\V_{S_U}^{\f_U} = H_\bullet^T(\mf M^{\f_U,\zeta_{\VW}}(U,S_U), \varphi_{W_{S_U}^{\f_U}}) \xrightarrow{\cong} \bb V_{S, U}^\f: =H_\bullet^T(\mf M^{\f,\zeta_\VW}_U(Y,S),\varphi_{W_{S}^\f}) \subset \bb V_{S}
	\end{equation}
	of $\mc H(U)$-representations, where in particular $\V_{S,U}^\f$ defines a subrepresentation of the restriction of the $\mc H(Y)$-representation $\V_{S}$ to $\mc H(U)$.
\end{prop}

Again, the proof is orthogonal to the main goals of this paper and is deferred to future work. The main corollary of this result is a geometric analogue of Proposition \ref{localityprop}, which holds under the same hypotheses as \emph{loc. cit.}, and gives a geometric explanation for the property of Remark \ref{localityrmk}:

\begin{corollary} There is an isomorphism of $\mc H(Y_1)\otimes \mc H(Y_2)$-representations
\[ \bb V_S^\f \xrightarrow{\cong} \bigoplus_{l \in \bb Z^r} \bb V_{S_1,l}^{\f_1} \otimes \V_{S_2,l}^{\f_2} \ , \]
where the former is the restriction of the $\mc H(Y)$-representation to $\mc H(Y_1)\otimes \mc H(Y_2)$, for some geometric $\mc H(Y_i)$ representations $\V_{S_i,l}^{\f_i}$ such that $\V_{S_i,0}^{\f_i}=\V_{S_i}^{\f_i}$, defined for $i=1,2$ and each $l\in \bb Z^r$.
\end{corollary}

This establishes a key step in proving the desired geometric analogue of Proposition \ref{localityprop}. Our hope is that this will facilitate a local-to-global proof of the main conjectures of Section \ref{GRintrosec}.

\section{Examples and applications}\label{egsec} In this section, we explain several conjectures about the vertex algebras $\V(Y,S)$ and their geometric representations, and give proofs of some of the results in low rank examples.

\subsection{Examples of lattice vertex algebras associated to irreducible divisors}\label{egabsec} To begin, we simply recall the construction of Section \ref{screensec} in a few explicit examples in rank 1.

\begin{eg} Let $Y=|\mc O_{\bb P^1}\oplus \mc O_{\bb P^1}(-2)|$ and $S=|\mc O_{\bb P^1}|$. Then the vertex algebra $\Pi(Y,S)$ is generated two Heisenberg fields $J^1$ and $J^2$ satisfying
	\[ J^1(z) J^1(w) \sim -\frac{1}{\e_1\e_3} \frac{1}{(z-w)^2} \quad\quad\text{and}\quad\quad  J^2(z) J^2(w) \sim \frac{1}{\e_1\e_3} \frac{1}{(z-w)^2} \]
	as well as the vertex operator
	\[ V(z) = \nol \exp( \e_1(\phi_1+\phi_2) ) \nor  \quad\quad\text{and more generally} \quad\quad V_l(z) = \nol V(z)^l \nor \]
	for each $l\in \bb Z$, which satisfy the relations
	\[ J^i(z) V_l(w) \sim (-1)^i\frac{l}{\e_3} \frac{V_l(w)}{z-w} \quad\quad \text{and}\quad\quad V_l(z)V_m(w)= \nol V_l(z) V_m(w)\nor \ .  \]
\end{eg}

\begin{eg} Let $Y=|\mc O_{\bb P^1}(-1)\oplus \mc O_{\bb P^1}(-1)|$ and $S=|\mc O_{\bb P^1}(-1)|$. Then the vertex algebra $\Pi(Y,S)$ is generated two Heisenberg fields $J^1$ and $J^2$ satisfying
	\[ J^1(z) J^1(w) \sim -\frac{1}{\e_2\e_3} \frac{1}{(z-w)^2} \quad\quad\text{and}\quad\quad  J^2(z) J^2(w) \sim -\frac{1}{\e_3\e_1} \frac{1}{(z-w)^2} \]
	as well as the vertex operator
	\[ V(z) = \nol \exp( \e_2\phi_1(z)-\e_1\phi_2(z)) ) \nor  \quad\quad\text{and more generally} \quad\quad V_l(z) = \nol V(z)^l \nor \]
	for each $l\in \bb Z$, which satisfy the relations
	\[ J^i(z) V_l(w)  \sim (-1)^i\frac{l}{\e_3} \frac{V_l(w)}{z-w} \quad\quad \text{and}\quad\quad V_l(z)V_m(w)= (z-w)^{lm}\nol V_l(z) V_m(w)\nor \ .  \]
\end{eg}

\begin{figure}[b]
	\caption{Vertex algebra data for $\Pi(Y_{2,0},|\mc O_{\bb P^1}|)$ and $\Pi(Y_{1,1},|\mc O_{\bb P^1}(-1)|)$  }
	\begin{overpic}[width=1\textwidth]{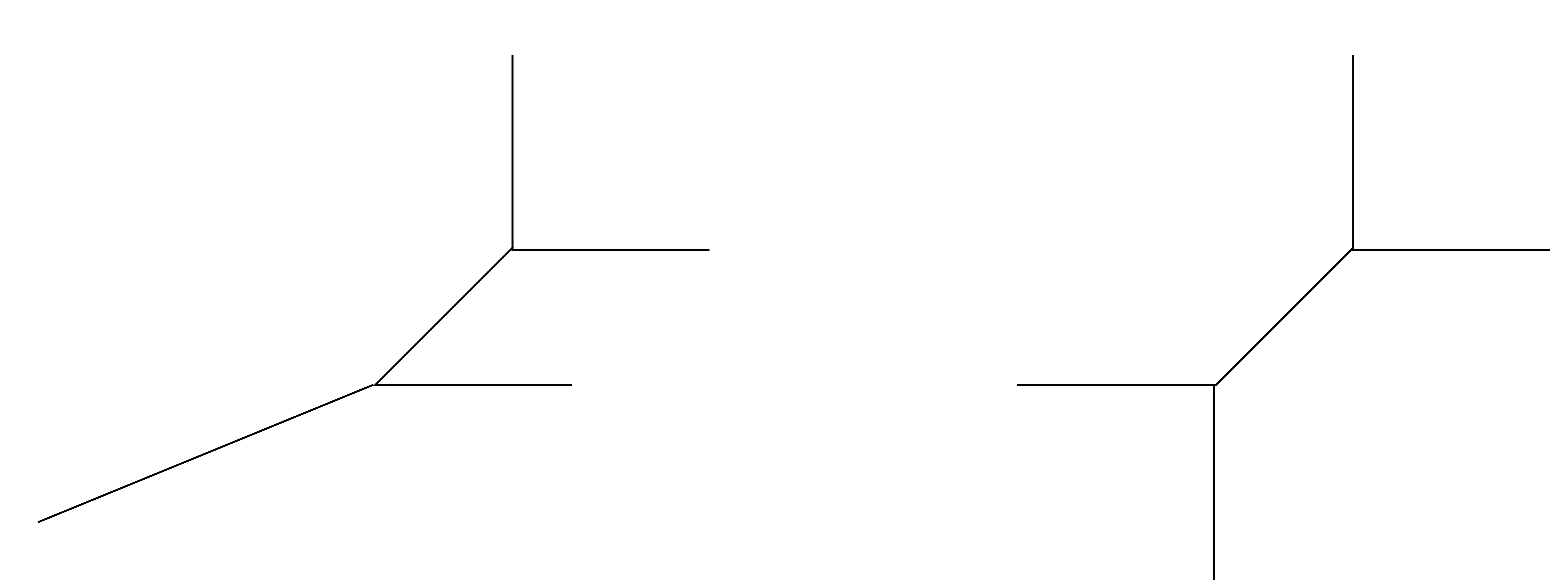}
	\put(32, 19.5) {\scriptsize{$J^1$}}
	\put(26, 13.5) {\scriptsize{$J^2$}}
	\put(29.5, 17) {\scriptsize{$V$}}
		\put(39.5, 16) {$1$}
	
	\put(84, 21) {\scriptsize{$J^1$}}
		\put(76, 13.5) {\scriptsize{$J^2$}}
	\put(81, 17.5) {\scriptsize{$V$}}
	\put(75,21) {$1$}
	\end{overpic}
	\label{fig:sl2config}
\end{figure}

\begin{eg} Let $Y=|\mc O_{\bb P^1}\oplus \mc O_{\bb P^1}(-2)|$ and $S=|\mc O_{\bb P^1}(-2)|$. Then the vertex algebra $\Pi(Y,S)$ is generated two Heisenberg fields $J^1$ and $J^2$ satisfying
	\[ J^1(z) J^1(w) \sim -\frac{1}{\e_2\e_3} \frac{1}{(z-w)^2} \quad\quad\text{and}\quad\quad  J^2(z) J^2(w) \sim - \frac{1}{\e_3(\e_1-\e_3)} \frac{1}{(z-w)^2} \]
	as well as the vertex operator
	\[ V(z) = \nol \exp( \e_2\phi_1(z)+(\e_3-\e_1)\phi_2(z)) ) \nor  \quad\quad\text{and more generally} \quad\quad V_l(z) = \nol V(z)^l \nor \]
for each $l\in \bb Z$, 	which satisfy the relations
	\[ J^i(z) V_l(w) \sim \sim (-1)^i\frac{l}{\e_3} \frac{V_l(w)}{z-w} \quad\quad \text{and}\quad\quad V_l(z)V_m(w)=(z-w)^{2lm} \nol V_l(z) V_m(w)\nor \ .  \]
\end{eg}

\begin{eg} Let $Y=\tilde {A_2} \times \bb A^1$ be the product of the resolved $A_2$ singularity with the affine line, and $S=\tilde {A_2}$. Then the vertex algebra $\Pi(Y,S)$ is generated by three Heisenberg fields $J^1$, $J^2$ and $J^3$ satisfying
	\[\hspace*{-2cm} J^1(z) J^1(w) \sim -\frac{1}{\e_2\e_3} \frac{1}{(z-w)^2} \quad\quad J^2(z) J^2(w) \sim -\frac{1}{\e_3(\e_1-\e_3)} \frac{1}{(z-w)^2} \quad\quad    J^3(z) J^3(w) \sim - \frac{1}{(\e_3-\e_1)(2\e_1-\e_3)} \frac{1}{(z-w)^2} \]
	as well as the vertex operators
	\[V_{1,0}(z) = \nol \exp( \e_2\phi_1(z)+(\e_3-\e_1)\phi_2(z)) ) \nor  \quad\quad\text{and} \quad\quad   V_{0,1}(z) = \nol \exp( -\e_3\phi_2(z)+(\e_3-2\e_1)\phi_3(z)) ) \nor \]
	as well as more generally
		\[V_{l}(z) = \nol V_{1,0}(z)^{l_1} V_{0,1}(z)^{l_2} \nor \quad\quad \text{for each $l=(l_1,l_2) \in \bb Z^2$ } \ .\]

		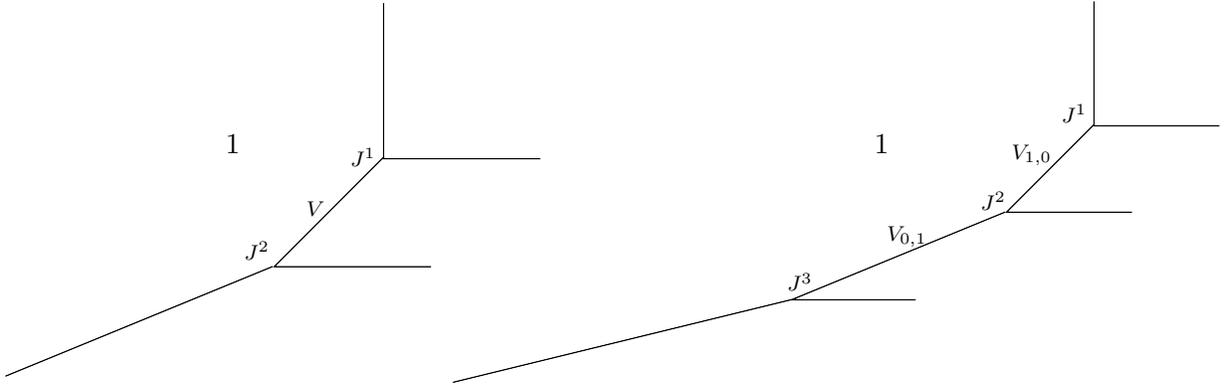
\begin{figure}[b]
			\caption{Vertex algebra data for the algebras $\Pi(Y_{2,0},\widetilde{A}_{1})$ and $\Pi(Y_{3,0},\widetilde{A}_{2})$}
			\begin{overpic}[width=1\textwidth]{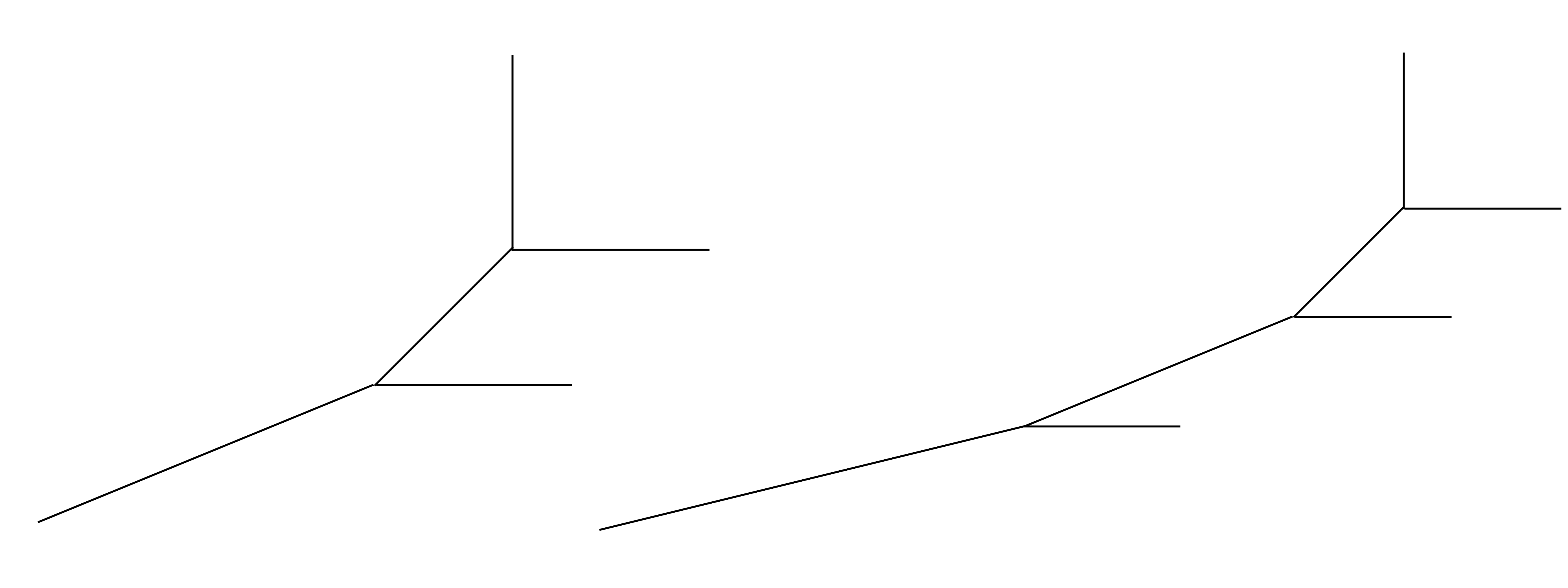}
\put(30, 21) {\scriptsize{$J^1$}}
\put(21.5, 13.5) {\scriptsize{$J^2$}}
\put(26.5, 17) {\scriptsize{$V$}}
\put(20, 22) {$1$}

\put(87, 24.5) {\scriptsize{$J^1$}}
\put(80.5, 17.5) {\scriptsize{$J^2$}}
\put(65, 11) {\scriptsize{$J^3$}}
\put(83, 21.5) {\scriptsize{$V_{1,0}$}}
\put(73, 15) {\scriptsize{$V_{0,1}$}}
\put(72,22) {$1$}
			\end{overpic}
			\label{fig:sl2sl3fig}
		\end{figure}
		which satisfy the relations
	\begin{align*}
		J^1(z) V_{1,0}(w)  & \sim -\frac{1}{\e_3} \frac{V_{1,0}(w)}{z-w} & J^2(z) V_{1,0}(w) & \sim \frac{1}{\e_3} \frac{V_{1,0}(w)}{z-w} \\
		 J^2(z) V_{0,1}(w) & \sim \frac{1}{(\e_1-\e_3)} \frac{V_{0,1}(w)}{z-w} & J^3(z) V_{0,1}(w) &\sim \frac{1}{(\e_3-\e_1)}\frac{V_{0,1}(w)}{z-w}
	\end{align*}
as well as the relations for each $l=(l_1,l_2),\ m=(m_1,m_2) \in \bb Z^2$ given by
\[ V_l(z)V_m(w)=(z-w)^{2(l_1 m_1 + l_ 2m_2) - l_1m_2 - l_2m_1} :V_l(z)V_m(w): \ .  \]
\end{eg}

\subsection{$W$-superalgebras $ W_{f_\mu,f_\nu}^\kappa(\gl_{M|N})$ from divisors $S_{\mu,\nu}$ in $Y_{m,n}$}\label{Walgsec}

We now describe the application of our results to examples of divisors $S$ in a resolution $Y_{m,n}$ of $X_{m,n}=\{xy-z^mw^n\}$. Let $\mu$ and $\nu$ be partitions of length $m$ and $n$, respectively
\[\mu=\{ \mu_1 \geq \hdots \geq \mu_m \geq 0 \} \quad \quad\text{and}\quad\quad \nu =\{\nu_1 \geq \hdots \geq \nu_n \geq 0 \} \]
and define the corresponding lists of integers
\[ \textbf{M} =(M_i)_{i=0}^{m-1} \quad \textbf{N} =(N_i)_{i=0}^{n-1}  \quad \quad \text{where} \quad \quad M_i=\sum_{k=i+1}^{m} \mu_k \quad N_i=\sum_{k=i+1}^{n} \nu_k \]
$i=0,...,m-1$, $j=0,...,n-1$; we also write just $M = M_0 = \sum_{k=1}^m \mu_k$ and $N=N_0=\sum_{k=1}^n \nu_k$. We define $S_{\mu,\nu}$ as the toric divisor corresponding to the labeling of the faces of the moment polytope of $Y_{m,n}$ by the integers $M_i$ and $N_i$ depicted in Figure \ref{fig:Smunufig} below in the case $m=3,n=2$. We also write simply $S_\mu$ or $S_\nu$ if $n=0$ or $m=0$, respectively.

\begin{figure}[b]
	\caption{The resolution $Y_{3,2}$, the toric divisor $S_{\mu,\nu}$, and the compact toric curve classes $C_i$ with their corresponding simple roots $\alpha_i$ of $\gl_{3|2}$}
	\label{fig:Smunufig}
	\begin{overpic}[width=1.0\textwidth]{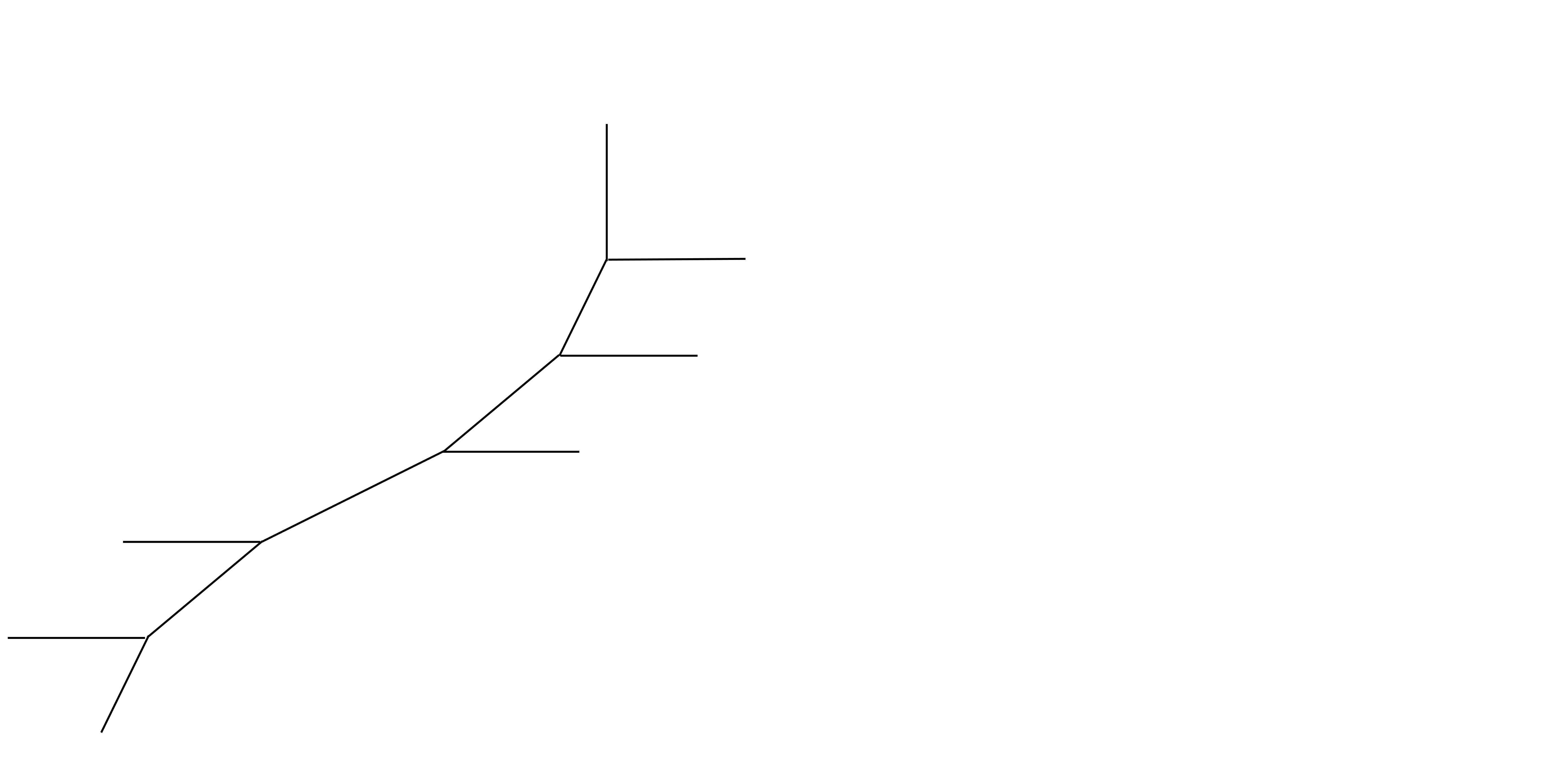}
		\put(0,43) {$ Y_{3,2}\to X_{3,2} = \{ xy-z^3w^2 \}$}
		\put (39.5,35) {$M_0= \mu_1+\mu_2+\mu_3$} 
		\put (37.5,28.5) {$M_1= \mu_2+\mu_3$} 
		\put(32, 22) {$M_2=\mu_3$} 
		\put (2,10.5) {$N_1=\nu_2$}
		\put (4, 16.5) {$N_0=\nu_1+\nu_2$}
		
		\put(34,30) {$C_1$}
		\put(29,25) {$C_2$}
		\put(23,17) {$C_3$}
		\put(14,11.5) {$C_4$}

		\put(70,22) {$ \mleft[ \begin{array}{c c c | c c}  &  &   &   &  \\  \alpha_1 & & & &  \\  &\alpha_2   &  &  &   \\
				\hline &  & \alpha_3 &  &  \\  & &  &   \alpha_4&  \end{array} \mright] $}
	\end{overpic}
	
\end{figure}
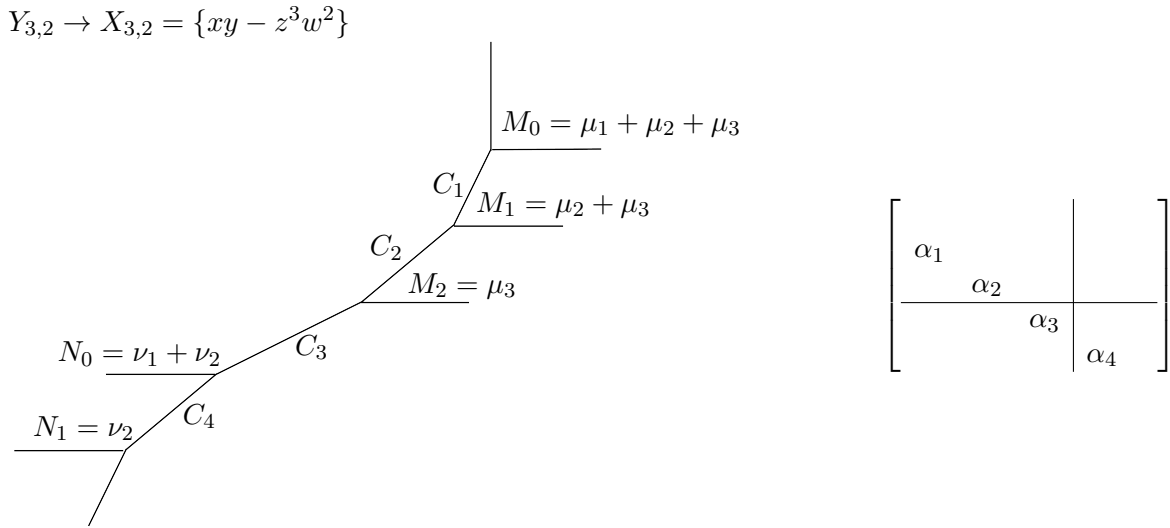

In this setting, following physical predictions from \cite{PrR} and \cite{Rap} we conjecture that the corresponding vertex algebra is given by the affine $W$-algebra $W_{f_\mu,f_\nu}^\kappa(\gl_{M|N})$ of $\gl_{M|N}$ corresponding to the nilpotents $f_\mu$ and $f_\nu$ in $\gl_M$ and $\gl_N$, respectively, determined by $\mu$ and $\nu$:

\begin{conj}\label{WLSconj}There is an isomorphism of vertex algebras
	\[ W_{f_\mu,f_\nu}^\kappa(\gl_{M|N}) \xrightarrow{\cong} \V(Y_{m,n},S_{\mu,\nu})  \ ,\]
	such that $\Pi(Y,S)$ identifies with a certain canonical free field realization of $W_{f_\mu,f_\nu}^\kappa(\gl_{M|N})$.
\end{conj}

\noindent In fact, we prove this conjecture in the case $n=0$, assuming some results in progress \cite{BBN} about the properties of free field realizations of $W$-algebras in type A, as stated in Theorem \ref{WLStheo} below.
	
The theorem follows from identifying the defining free field realization of $\V(Y_{m,n},S_{\mu,\nu})$ with a certain canonical bosonized presentation, in the sense of \cite{FMS}, of the generalization \cite{Gen1} of the Wakimoto resolution \cite{Wak} to $W$-algebras for general nilpotents, such that the embeddings of Equation \ref{voaembeqn} induce parabolic induction \cite{Gen2} and inverse reduction \cite{Sem} maps. We will give an explicit proof of the result for $\gl_2$ with any nilpotent, as well as a general argument conditional on the work in progress mentioned above on the properties of generalized Wakimoto realizations.

We begin by explaining the simplest example with a non-trivial $W$-algebra:
\begin{prop}\label{Virscreg} Let $Y=\bb C^3$ and $S=2[\bb C^2]$. There is an isomorphism of vertex algebras
\[ W^\kappa_{f_\prin}(\gl_2)  =  \pi^+\otimes \ W^\kappa_{f_\prin}(\spl_2)  \xrightarrow{\cong} \V(\C^3,r[\C^2])  \ , \]
such that the representation $\Pi(\bb C^3,2[\bb C^2])$ identifies with $ \pi_{\mf h_{\gl_2}}$, the Feigin-Frenkel realization \cite{FF1}.
\end{prop}
\begin{proof}
The free field vertex algebra $\Pi(Y,S)=\Pi(Y,\bb C^2)^{\otimes 2}$ is generated by two Heisenberg fields $J^1$, $J^2$ satisfying
\[ J^1(z) J^1(w) \sim -\frac{1}{\e_1\e_2} \frac{1}{(z-w)^2} \quad\quad \text{and}\quad\quad J^2(z) J^2(w) \sim -\frac{1}{\e_1\e_2}  \frac{1}{(z-w)^2}  \ ,\]
and the screening currents are given by
\[ Q_1(z)=\nol \exp( \e_1(\phi_1-\phi_2)) \nor \quad\quad \text{and}\quad\quad   Q_2(z)=\nol \exp( \e_2(\phi_1-\phi_2)) \nor \ .\]
After making the change of basis
\[ J^+(z)=J^1(z) + J^2(z) \quad\quad J^-(z)=J^1(z) - J^2(z) \]
we can identify $\Pi(Y,S) \cong \pi^+ \otimes \pi^-$ where $\pi^+$ denotes the Heisenberg subalgebra generated by $J^+$ at level $k_+=-\frac{2}{\e_1 \e_2}$ and similarly $\pi^-$ that generated by $J^-$ at level $k_-=-\frac{2}{\e_1 \e_2}$.

\begin{figure}[b]
	\caption{Vertex algebra data for $\V(\C^3,2[\C^2])$}
	\begin{overpic}[width=.3\textwidth]{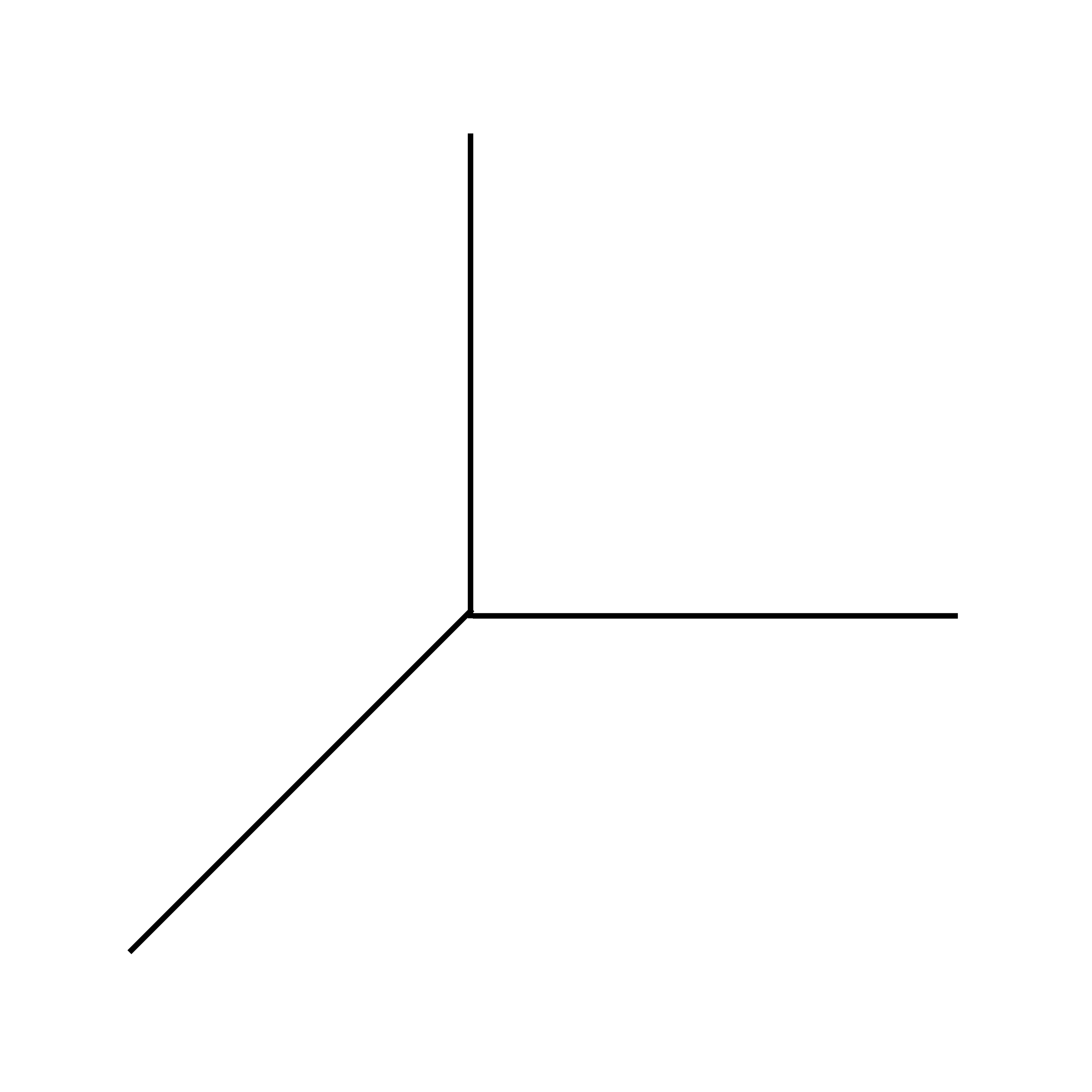}
		\put(45,45) {\scriptsize{$J^{1/2}$}}
		\put(40,90) {\scriptsize{$Q_1$}}
		\put(90,41) {\scriptsize{$Q_2$}}
		\put(70,70) {$2$}
	\end{overpic}
	\label{fig:C32C2fig}
\end{figure}

The screening current current is given by
\[ Q_1(z) = \nol \exp(\e_1(\phi_1(z)-\phi_2(z))) \nor \ ,\]
which has nonsingular OPE with $J^+$, so that $J^+$ generates a Heisenberg subalgebra $\pi^+\subset \V(Y,S)$ independent of the other fields, while on $\pi^-$ it induces
n \[ Q_1=\int Q_1(z)dz \ : \pi^- \to \pi^-_{-\e_1}   \ . \]

The vertex subalgebra $\ker(Q_1)\subset \pi^-$ is generated by the Virasoro current
\begin{equation}\label{Virgeneqn}
	T(z) = -\frac{\e_1\e_2}{4} \nol J^-(z) J^-(z) \nor -\frac{\e_1+\e_2}{2} \del_z J^-(z)  \quad\quad\text{of central charge}\quad\quad c=1+6\frac{(\e_1+\e_2)^2}{\e_1\e_2}  
\end{equation}
by the calculation recalled in Proposition \ref{FFsl2eg}, and identifies with the Feigin-Frenkel realization of the principal affine $W$-algebra $W_\kappa(\spl_2)$ at level
\[ \kappa +2 = - \frac{\e_2}{\e_1} \ .  \]
Moreover, the symmetry of $T(z)$ under interchanging $\e_1$ and $\e_2$, which identifies with the symmetry of $T^\beta$ interchanging $\beta$ and $-\frac{2}{k\beta}$ in \emph{loc. cit.}, is equivalent to interchanging $Q_2$ and $Q_1$, which implies that we have
\[ \V(Y,S) = \ker(Q_1) = \ker(Q_2) \cong  \pi^+\otimes \ W_\kappa(\spl_2) \ . \]
\end{proof}

In order to prove Theorem \ref{WLStheo} for $M=2$, it remains to consider the case of the trivial nilpotent $f_\mu=0\in\mc N_{\gl_2}$ corresponding to the $\spl_2$ coweight $\mu=(1,1)$. We begin with an intermediate case:

\begin{prop} \label{chantidomegprop}Let $Y=|\mc O_{\bb P^1}\oplus \mc O_{\bb P^1}(-2)|$ and $S=S_{1,1,0,0}= [\bb A^2_{xy}] +[|\mc O_{\bb P^1}|]$.  There is an isomorphism of vertex algebras
\[ \pi^+\otimes \mc D^\ch(\bb A^1)  \xrightarrow{\cong} \V(Y,S)  \ ,\]
such that the free field module is identified with $\Pi(Y,S)= \pi^+\otimes\Pi_0$ the bosonization of \cite{FMS}.
\end{prop}
\begin{proof}
	The Heisenberg subalgebra $\pi(Y,S)$ of the free field vertex algebra \[ \Pi(Y,S)=\Pi(Y,\bb C^2)\otimes \Pi(Y,|\mc O_{\bb P^1}|) \]
	is generated by three Heisenberg fields $J^2$, $J^3$ and $J^4$ (note we have so far omitted $J^1$; the notation is chosen to be consistent with the application in Theorem \ref{affinesl2theo} below). The field $J^2$ generating $\Pi(Y,\bb C^2)$ satisfies
	\[ J^2(z) J^2(w) \sim -\frac{1}{\e_1\e_2} \frac{1}{(z-w)^2}  \]
	and $J^3$ and $J^4$ from $\Pi(Y,|\mc O_{\bb P^1}|)$ satisfy
	\[ J^3(z) J^3(w) \sim -\frac{1}{\e_1\e_3} \frac{1}{(z-w)^2} \quad\quad\text{and}\quad\quad  J^4(z) J^4(w) \sim \frac{1}{\e_1\e_3} \frac{1}{(z-w)^2} \]
	which generate $\Pi(Y,|\mc O_{\bb P^1}|)$ together with the vertex operator
	\[ V(z) = \nol \exp( \e_1(\phi_3(z)+\phi_4(z)) ) \nor  \quad\quad\text{and more generally} \quad\quad V_l(z) = \nol V(z)^l \nor \]
	for each $l\in \bb Z$, and satisfy the relations
	\[ J^3(z) V_l(w) \sim - \frac{l}{\e_3} \frac{V_l(w)}{z-w} \quad\quad J^4(z) V_l(w) \sim \frac{l}{\e_3} \frac{V_l(w)}{z-w} \quad\quad \text{and}\quad\quad V_l(z)V_m(w)= \nol V_l(z) V_m(w)\nor \ .  \]
	The only screening current is given by
	\[  Q(z)=\nol \exp( \e_2\phi_2(z) - \e_3\phi_3(z))) \nor   \ . \ \]
	
	\begin{figure}[t]
		\begin{overpic}[width=.8\textwidth]{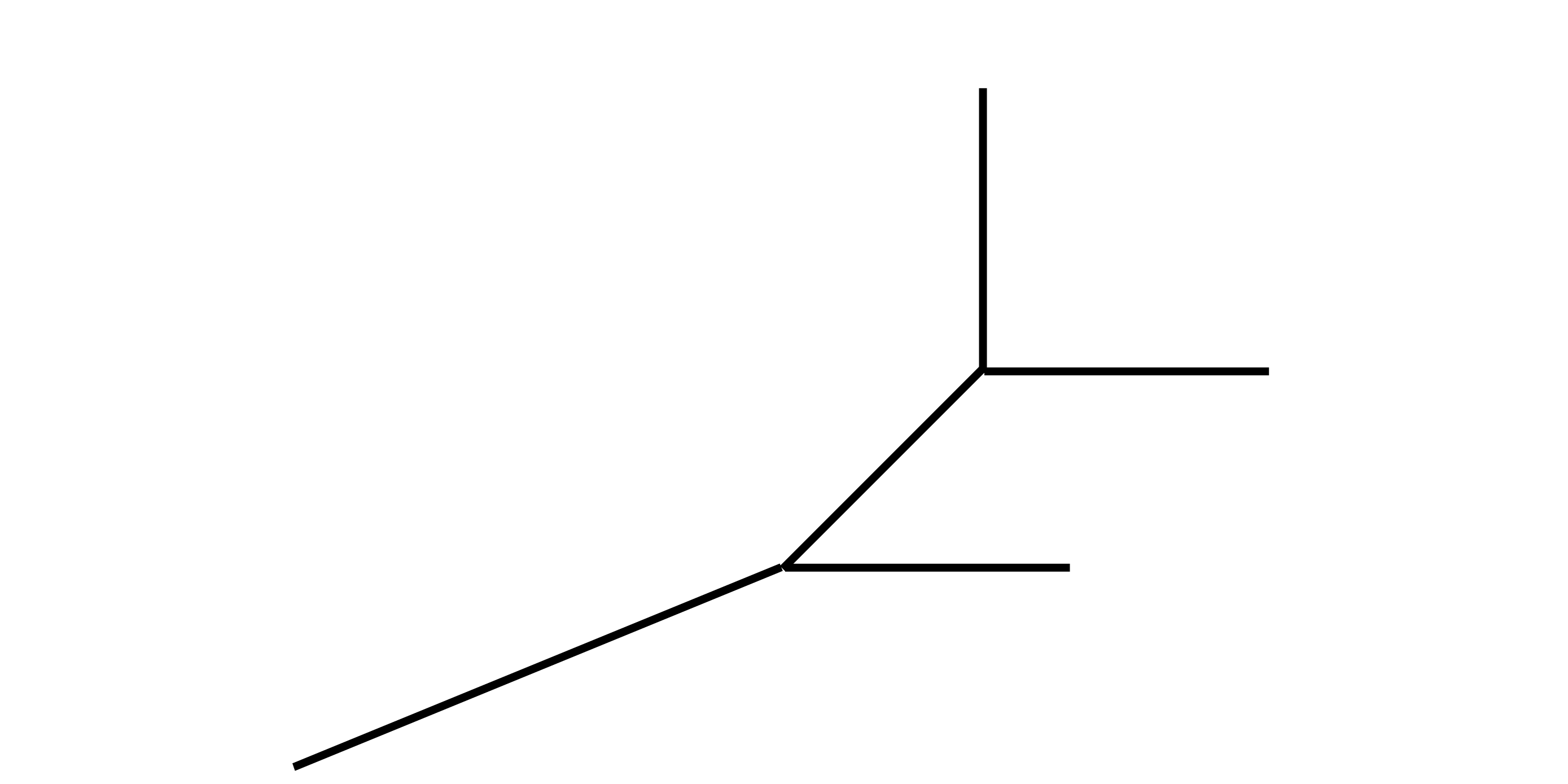}
			\put(64,27) {\scriptsize{$J^{2}$}}
			\put(62,23.5) {\scriptsize{$J^{3}$}}
			\put(53,15) {\scriptsize{$J^{4}$}}	
			\put(58,19.5) {\scriptsize{$V$}}
			\put(82,25.5) {\scriptsize{$Q$}}
			\put(74,35) {$1$}
			\put(69,19) {$1$}
			
		\end{overpic}
		\label{fig:S1100fig}
			\caption{Vertex algebra data for $\V(Y_{2,0},S_{1,1,0,0})$}
	\end{figure}
	
Consider the change of basis in the total Heisenberg subalgebra of $\Pi(Y,S)$ given by
\[ \hspace*{-1cm} J^+= J^2 + J^3 + J^4 \quad\quad \alpha = \e_2 J^2 - \e_3 J^3 \quad\quad  \beta = - \e_2(J^2+J^3) +  \e_1 J^4 \ . \]

Note that $J^+$ has nonsingular OPE with $\alpha$, $\beta$, and the screening current $S$, and it generates a Heisenberg subalgebra of $\pi^+\subset \V(Y,S)$ independent of the other fields, as in Proposition \ref{Virscreg}.

Similarly, note that $\alpha$ and $\beta$ define commuting Heisenberg subalgebras of levels $1$ and $-1$, respectively, that is, they satisfy
\[  \alpha(z) \alpha(w) \sim \frac{1}{(z-w)^2} \quad\quad \beta(z) \beta(w) \sim -\frac{1}{(z-w)^2} \quad\quad \alpha(z) \beta(w) \sim 0  \ , \]
and in terms of $\alpha$ and $\beta$, we have the identifications
\[ V(z) = \nol \exp( \phi_\alpha(z) + \phi_\beta(z) ) \nor  \quad\quad \text{and} \quad\quad  S(z) = \nol \exp( \phi_\alpha(z) ) \nor  \ , \]
where $\phi_\alpha=\e_2 \phi_2- \e_3 \phi_3$ and $\phi_\beta =- \e_2(\phi^2+\phi^3) +  \e_1 \phi^4 $ are the corresponding vertex operators. Thus, together the fields $\alpha$, $\beta$ and $V$ generate a canonically normalized half lattice vertex algebra $\Pi_0$, so that in summary we have
\[ \Pi(Y,S) = \pi^+ \otimes \Pi_0 \ . \]
Moreover, note that under this identification the screening current $S$ identifies with the standard screening current for the realization of a canonically normalized chiral Weyl vertex algebra $\mc D^\ch(\bb A^1)$ as a subalgebra of the half lattice vertex algebra, that is
\begin{equation}\label{betagammageneqn}
	 \ker(Q)\cong \mc D^\ch(\bb A^1) \subset \Pi_0 \quad\quad
\textup{generated by} \quad\quad 
 a(z) = \nol \exp( \phi_\alpha(z) + \phi_\beta(z) ) \nor  \quad\quad a^*(z) =  \nol  \alpha(z) \exp(- \phi_\alpha(z) - \phi_\beta(z) ) \nor  
\end{equation} 
which satisfy the defining relations thereof
 \[ a(z)a(w) \sim a^*(z) a^*(w) \sim 0 \quad\quad a(z) a^*(w) \sim \frac{1}{z-w} \ ; \]
 see for Example \cite{AlW}, though we recall the original physics reference is \cite{FMS}. In summary, we have shown that that the vertex algebra $V(Y,S)$ associated to the divisor is given by
\[ \V(Y,S) = \pi^+\otimes \mc D^\ch(\bb A^1)  \ .\]
\end{proof}

We now give a proof of Theorem \ref{WLStheo} for $M=2$ and trivial nilpotent $f_\mu=0\in\mc N_{\gl_2}$:

\begin{theo}\label{affinesl2theo}
	Let $Y=|\mc O_{\bb P^1}\oplus \mc O_{\bb P^1}(-2)|$ and $S=S_{2,1,0,0}=2 [\bb A^2_{xy}] +[|\mc O_{\bb P^1}|]$. There is an isomorphism of vertex algebras
	\[  V^\kappa(\gl_2) =\pi^+ \otimes  V^\kappa(\spl_2)\xrightarrow{\cong} \V(Y,S) \]
	such that the representation $\Pi(Y,S)$ is identified with the bosonization, in the sense of \cite{FMS}, of the Wakimoto realization \cite{Wak}, used in the proof of inverse reduction for $\spl_2$ in \cite{Sem}.
\end{theo}
\begin{proof}

	The free field vertex algebra \[ \Pi(Y,S)=\Pi(Y,\bb C^2)^{\otimes 2}\otimes \Pi(Y,|\mc O_{\bb P^1}|) \]
	is generated by four Heisenberg fields, $J^1$ and $J^2$ generating $\Pi(Y,\bb C^2)^{\otimes 2}$ satisfying
	\[ J^1(z) J^1(w) \sim -\frac{1}{\e_1\e_2} \frac{1}{(z-w)^2} \quad\quad \text{and}\quad\quad J^2(z) J^2(w) \sim -\frac{1}{\e_1\e_2}  \frac{1}{(z-w)^2}  \]
	and $J^3$ and $J^4$ from $\Pi(Y,|\mc O_{\bb P^1}|)$ satisfying
	\[ J^3(z) J^3(w) \sim -\frac{1}{\e_1\e_3} \frac{1}{(z-w)^2} \quad\quad\text{and}\quad\quad  J^4(z) J^4(w) \sim \frac{1}{\e_1\e_3} \frac{1}{(z-w)^2} \]
	which generate $\Pi(Y,|\mc O_{\bb P^1}|)$ together with the vertex operator
	\[ V(z) = \nol \exp( \e_1(\phi_3(z)+\phi_4(z)) ) \nor  \quad\quad\text{and more generally} \quad\quad V_l(z) = \nol V(z)^l \nor \]
	for each $l\in \bb Z$, which satisfy the relations
	\[ J^3(z) V_l(w) \sim - \frac{l}{\e_3} \frac{V_l(w)}{z-w} \quad\quad J^4(z) V_l(w) \sim \frac{l}{\e_3} \frac{V_l(w)}{z-w} \quad\quad \text{and}\quad\quad V_l(z)V_m(w)= \nol V_l(z) V_m(w)\nor \ .  \]
	
	The screening currents are given by
	\[ Q_1=\nol \exp( \e_1(\phi_1-\phi_2)) \nor \quad\quad \text{and}\quad\quad   Q_2=\nol \exp( \e_2\phi_2 - \e_3\phi_3)) \nor  \ . \ \]

	\begin{figure}[b]
	\caption{Vertex algebra data for $\V(Y_{2,0},S_{2,1,0,0})$}
\label{fig:S1100fig}
		\begin{overpic}[width=.8\textwidth]{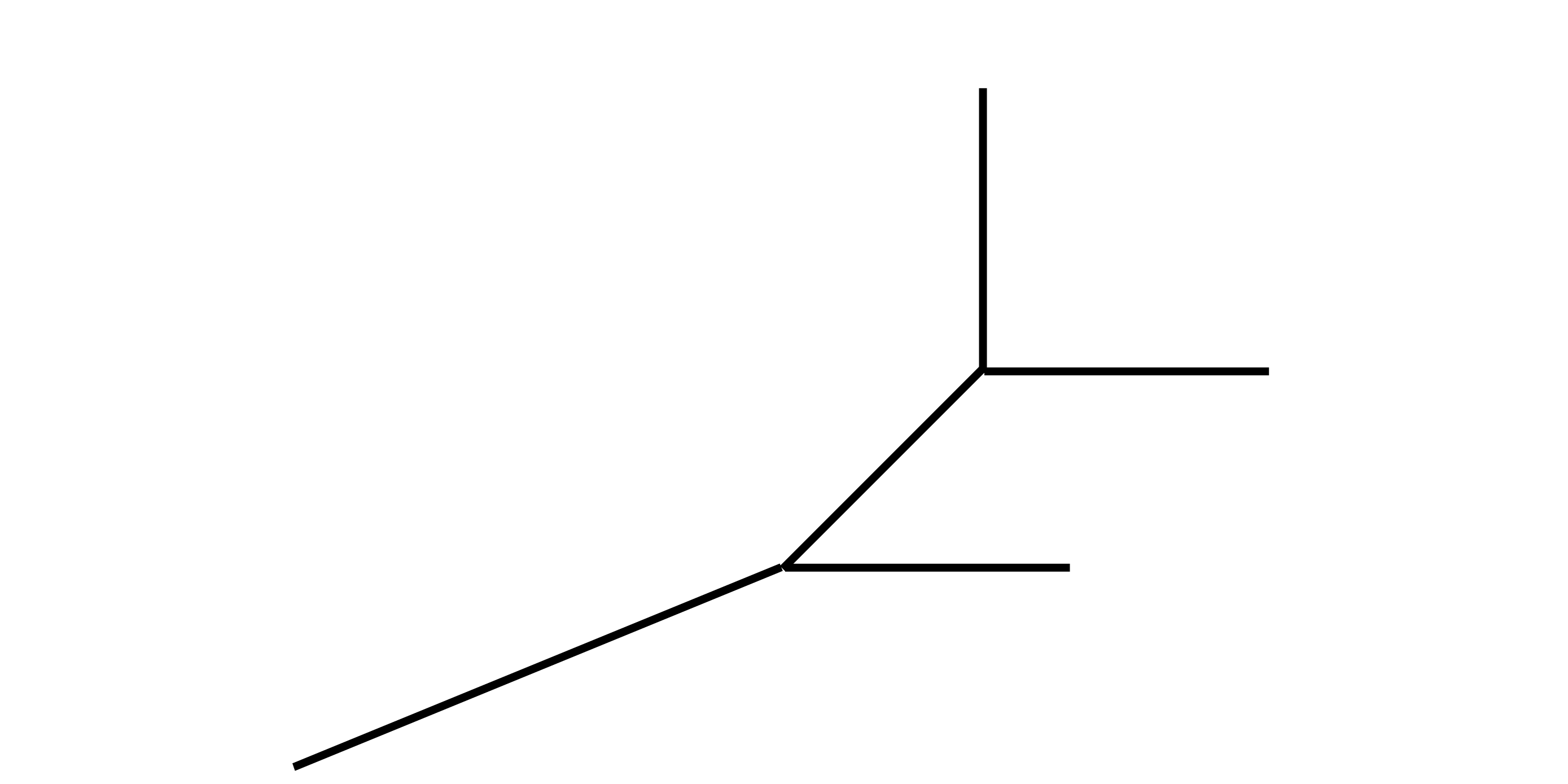}
			\put(64,27) {\scriptsize{$J^{1/2}$}}
			\put(62,23.5) {\scriptsize{$J^{3}$}}
			\put(53,15) {\scriptsize{$J^{4}$}}	
			\put(58,19.5) {\scriptsize{$V$}}
			\put(82,25.5) {\scriptsize{$Q_2$}}
			\put(61.5,46) {\scriptsize{$Q_1$}}
			\put(74,35) {$2$}
			\put(69,19) {$1$}
		\end{overpic}
	
	\end{figure}

We now explain the two descriptions of the vertex algebra $V(Y,S)$ induced by Proposition \ref{factprop} applied to the two screening currents above. First, partitioning the divisor along the curve class corresponding to the screening operator $Q_1$, we have a vertex algebra embedding
	\[ \V(Y,S) \subset \ker(Q_2)= \V(Y,S_{1,0,0,0}) \otimes \V(Y,S_{1,1,0,0} )\cong \pi\otimes \pi^+_{1,1,0,0} \otimes \mc D^\ch(\bb A^1)  \ ,\]
	where the tensor factors of the right hand side are generated by $J^1$, $J^2+J^3+J^4$, and $a$ and $a^*$ defined as in Equation \ref{betagammageneqn}, respectively, following Example \ref{chantidomegprop} above. The image of $V(Y,S)$ under this embedding is identified with the kernel of the screening operator $Q_1$ corresponding to the curve class along which we have partitioned the divisor.

	Now, consider the change of basis in the Heisenberg subalgebra $\pi_S$ of $\Pi(Y,S)$ given by
	\[  J^+=J^1 + J^2 + J^3 + J^4  \quad\quad \alpha  = \e_2 J^2 - \e_3 J^3 \quad\quad  \beta = - \e_2(J^2+J^3) +  \e_1 J^4  \quad\quad \delta =\e_2 (J^1 - J^2 - J^3-  J^4 )  \ . \]
	Note that $J^+$ still has nonsingular OPE with each of the other fields $J^-$, $\alpha$, $\beta$, and $V_l$, and both screening currents $Q_1$ and $Q_2$, so that it generates a Heisenberg subalgebra of $\pi^+\subset \V(Y,S)$ independent of the other fields, as in the previous examples. The field $\delta$ also has nonsingular OPE with each of the other fields and the screening current $Q_2$, so that it defines an independent Heisenberg subalgebra $\pi^\delta\subset \ker(Q_2)$ at level $k_\delta = -2\frac{\e2}{\e1}=2(\kappa + 2)  $. The fields $\alpha$ and $\beta$ and their relations with the vertex operators $V_l$ and the screening current $Q_2$ are evidently the same as in the previous example.
	
	In summary, the fields $J^+$ and $\delta$ together with $a$ and $a^*$ defined in terms of $\alpha$ and $\beta$ as in Equation \ref{betagammageneqn} give a new set of generators for $\ker(Q_2) $, so that we have
	\[ \ker(Q_2)  \cong \pi^+ \otimes \pi^\delta \otimes \mc D^\ch(\bb A^1)   \ ,\]
	where the residual screening operator $Q_1$ vanishes identically on $\pi^+$.
	Moreover, in terms of these generators the residual screening current $Q_1$ is given by
	\[ Q_1(z)=\nol \exp( \e_1(\phi_1(z)-\phi_2(z))) \nor = \nol \exp( \phi_\alpha(z)+\phi_\beta(z) + \frac{\e_1}{\e_2} \phi_\delta(z) ) \nor = \nol a(z) \exp(\frac{\e_1}{\e_2} \phi_\delta(z) )  \nor \]
	which identifies with the screening current for the Wakimoto realization \cite{Wak} of the affine algebra of $\spl_2$ at level $\kappa$, that is, we have
	\[ \ker(Q_1|_{\pi^\delta\otimes \mc D^\ch(\bb A^1) }) \cong V^\kappa(\spl_2) \subset \pi^\delta\otimes A  \quad\quad \text{generated by} \quad\quad \begin{cases} J^e(z) = & a(z)  \\ J^h(z)  =   & - 2 : a^*(z) a(z) \nor + \delta(z) \\
	J^f(z)= & - \nol a^*(z)^2 a(z) \nor + \kappa \del_z a^*(z) +  a^*(z) \delta(z) \end{cases}  \ ,  \]
	noting that $-\frac{\e_1}{\e_2}=\frac{1}{\kappa+2}$; see for Example \cite{FrG}. In summary, we obtain that the vertex algebra $V(Y,S)$ associated to the divisor is given by
	\[ \V(Y,S) = \ker(Q_1)\cap \ker(Q_2)  \cong \pi^+ \otimes  V^\kappa(\spl_2) \ , \]
	the affine Kac-Moody vertex algebra associated to $\gl_2$ at level $\kappa= -2 -\frac{\e_2}{\e_1}$.

	Next, we consider instead the embedding induced by Proposition \ref{factprop} by partitioning the divisor along the curve class corresponding to the screening current $Q_2$, so that we have
	\[ \V(Y,S) \subset \ker(Q_1) = \V(Y,S_{2,0,0,0} )\otimes \V(Y,S_{0,1,0,0} )\cong \pi^+_{2,0,0,0}\otimes \mc W_\kappa(\spl_2) \otimes \Pi_0\]
	where $\pi^+_{2,0,0,0}$ denotes the Heisenberg algebra generated by $J^1+J^2$, $\mc W_\kappa(\spl_2)$ the affine $W$ algebra generated by $T(z)$ as defined in Equation \ref{Virgeneqn}, and $\Pi_0=\Pi[Y,S_{0,1,0,0}]$ denotes the half lattice vertex algebra generated by $J^3$, $J^4$ and $V^l$ for $l\in \bb Z$.
	
	Now, consider the alternative change of basis in the Heisenberg subalgebra of $\Pi(Y,S)$ given by
	\[  J^+=J^1 + J^2 + J^3 + J^4  \quad\quad J^-  = J^1-J^2 \quad\quad \mu = \frac{\e_2}{2}(J^1+J^2 + J^3 - J^4) -  \e_1 J^4  \quad\quad \nu =\frac{\e_2}{2} (J^1 + J^2) - \e_3 J^3   \ . \]
	Note that $J^+$ and $J^-$ have non-singular OPE with each other as well as $\mu$, $\nu$ and $V^l$, $J^+$ has non-singular OPE with both screening currents so that it generates a Heisenberg subalgebra as before, and $J^-$ has the same relations with $Q_1$ as in Example \ref{Virscreg}. The fields $\mu$ and $\nu$ define commuting Heisenberg subalgebras of levels $k_\mu= \frac{1}{2}(-2-\frac{\e_2}{\e_1} )=\frac{\kappa}{2}$ and $k_\nu= \frac{1}{2}(2+\frac{\e_2}{\e_1} )=-\frac{\kappa}{2}$, that is, they satisfy
	\[  \mu(z) \mu(w) \sim \frac{\kappa}{2}\frac{1}{(z-w)^2}  \quad\quad  \nu(z) \nu(w) \sim -\frac{\kappa}{2}\frac{1}{(z-w)^2} \quad\quad \mu(z) \nu(w) \sim 0  \ . \]

	\vspace*{.5cm}
	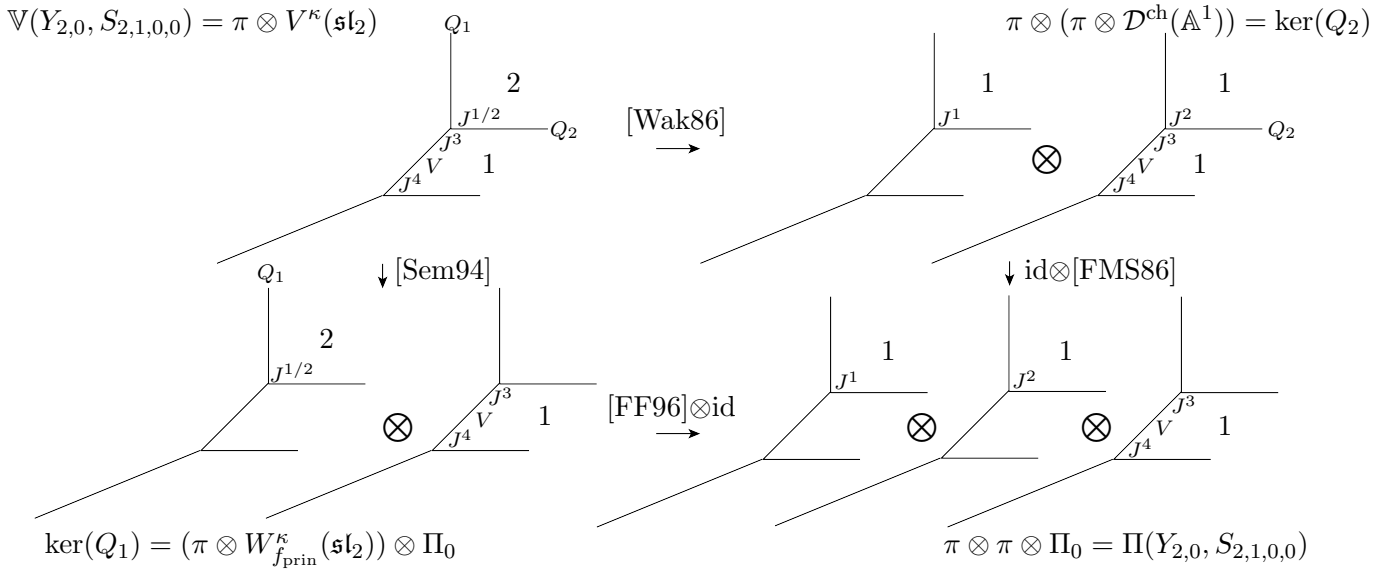
\begin{figure}[b]
					\vspace*{-1cm}
		\caption{Factorization structure on the vertex algebra $\V(Y_{2,0},S_{2,1,0,0})$}
		\label{fig:sl2factfig}
			\vspace*{1cm}
\hspace*{-1cm}	\begin{overpic}[width=1\textwidth]{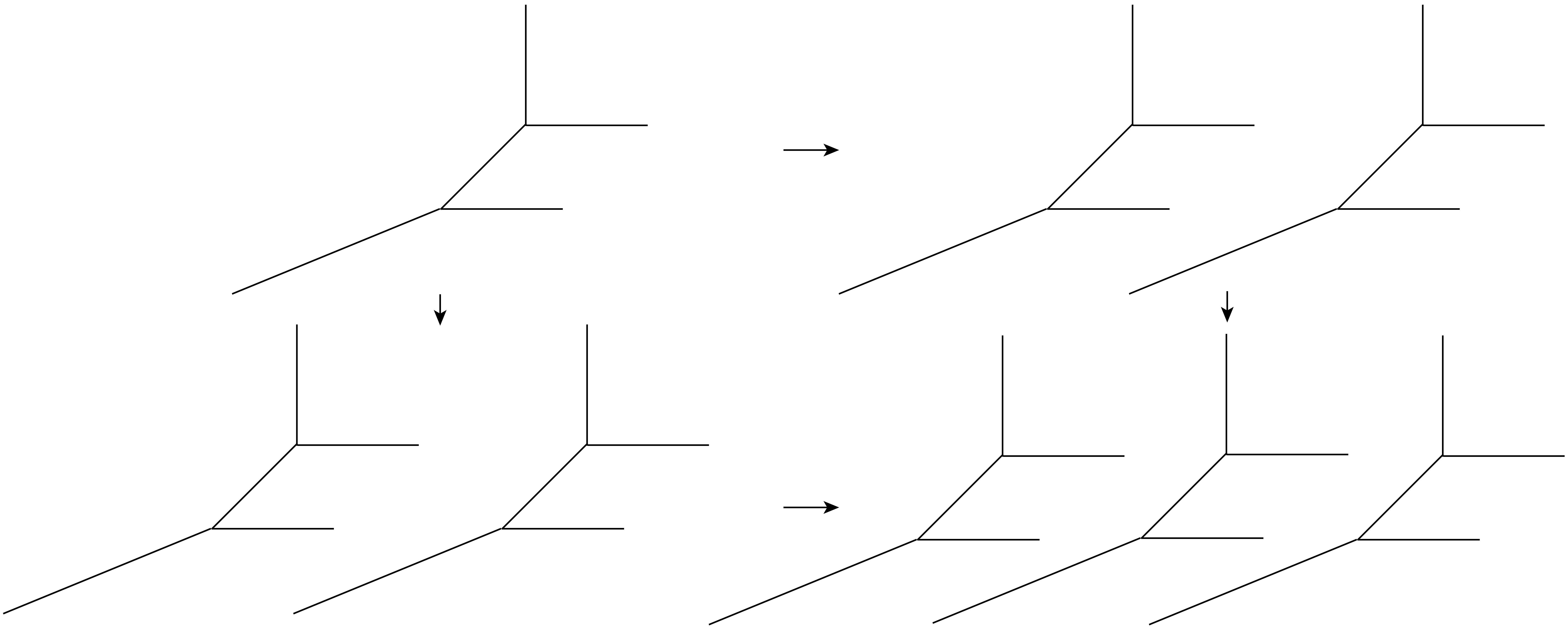}
		\put(34,32.25) {\scriptsize{$J^{1/2}$}}
		\put(32.5,30.25) {\scriptsize{$J^{3}$}}
		\put(29.25,27) {\scriptsize{$J^{4}$}}	
		\put(31.5,28.5) {\scriptsize{$V$}}
		\put(41.5,31.5) {\scriptsize{$Q_2$}}
		\put(33,40) {\scriptsize{$Q_1$}}
		\put(38,35) {$2$}
		\put(36,28.5) {$1$}

		\put(76,35) {$1$}
		\put(72.5,32.25) {\scriptsize{$J^{1}$}}
		
		\put(95,35) {$1$}
		\put(91,32.25) {\scriptsize{$J^{2}$}}
		\put(89.75,30.25) {\scriptsize{$J^{3}$}}
		\put(86.5,27) {\scriptsize{$J^{4}$}}	
		\put(88.5,28.5) {\scriptsize{$V$}}
		\put(99,31.5) {\scriptsize{$Q_2$}}
		\put(93,28.5) {$1$}

		\put(19,11.75) {\scriptsize{$J^{1/2}$}}
		\put(36.5,9.85) {\scriptsize{$J^{3}$}}
		\put(33.25,6.5) {\scriptsize{$J^{4}$}}	
		\put(35.5,8.1) {\scriptsize{$V$}}
		\put(18,20) {\scriptsize{$Q_1$}}
		\put(23,14.5) {$2$}
		\put(40.5,8.25) {$1$}

		\put(64.25,11.15) {\scriptsize{$J^{1}$}}
		\put(78.5,11.25) {\scriptsize{$J^{2}$}}
		\put(68,13.5) {$1$}
		\put(82.25,13.5) {$1$}
		
			\put(91.25,9.25) {\scriptsize{$J^{3}$}}
		\put(87.75,5.75) {\scriptsize{$J^{4}$}}	
		\put(90,7.5) {\scriptsize{$V$}}
		
				\put(95,7.5) {$1$}

		\put(80,29) {$\bigotimes$}
		\put(28,7.5) {$\bigotimes$}
		\put(70,7.5) {$\bigotimes$}
		\put(84,7.5) {$\bigotimes$}

		\put(29,20) {\cite{Sem}}
		\put(47.5,32) {\cite{Wak}}
		\put(46, 9) {\cite{FF1}$\otimes\id$}
		\put(79.5,20) {$\id\otimes$\cite{FMS} }

		\put(78,40) {$\pi\otimes (\pi \otimes \mc D^\ch(\bb A^1))=\ker(Q_2)$}
		\put(-2,40) {$\V(Y_{2,0},S_{2,1,0,0})=\pi \otimes V^\kappa(\spl_2)$}
		\put(73,-2) {$\pi\otimes\pi\otimes\Pi_0=\Pi(Y_{2,0},S_{2,1,0,0})$}
		\put(1,-2) {$\ker(Q_1)=(\pi\otimes W^\kappa_{f_\prin}(\spl_2))\otimes \Pi_0$}

	\end{overpic}
			\vspace*{-2cm}
	\end{figure}

	Moreover, in terms of $\mu$ and $\nu$, we have the identifications
	\[ \hspace*{-1cm} V =\nol \exp( \e_1(\phi_3(z)+\phi_4(z)) ) \nor = \nol \exp(  \frac{2}{\kappa}(\phi_\mu - \phi_\nu ) ) \nor  \quad\quad \text{and}\quad\quad Q_2(z)=  \nol \exp( \e_2\phi_2 - \e_3\phi_3)) \nor =\nol \exp( \phi_\nu- \frac{\e_2}{2} \phi^- )) \nor  \ .\]
	Thus, the fields $\mu$, $\nu$ and $V$ again generate a half lattice vertex algebra $\Pi$ in the kernel of $Q_1$, and together with the field $T$ defined in terms of $J^-$ as in Equation \ref{Virgeneqn} give a new set of generators for $\ker(Q_1)$, so that we have
	\[ \ker(Q_1) \cong \pi^+ \otimes \mc W_\kappa(\spl_2) \otimes \Pi  \ ,\]
	where the residual screening operator $Q_2$ vanishes identically on $\pi^+$.
	Moreover, in terms of these generators the residual screening current $Q_2$ screens the corresponding subalgebra
	\begin{align}
		J^e(z)  & = \nol \exp( \e_1(\phi_3(z)+\phi_4(z)) ) \nor  & = \nol \exp(  \frac{2}{\kappa}(\phi_\mu - \phi_\nu ) ) \nor   \\
		J^h(z)  & =    - 2 \beta(z) +\delta(z)   & = 2 \mu(z) \\
		J^f(z) & = ...  & = (T(z) + T_\nu(z) )\nol \exp(  -\frac{2}{\kappa}(\phi_\mu - \phi_\nu ) ) \nor   \ .
	\end{align}
which precisely corresponds to the embedding of $V^\kappa(\spl_2)$ in $ \mc W_{f_\prin}^\kappa(\spl_2) \otimes \Pi$ constructed in \cite{Sem}, where $T_\nu(z)=\nol \nu(z) \nu(z) \nor + c_\kappa \del_z \nu(z)$ denotes the Segal-Sugawara field in $\pi_\nu$. In summary, we have
	\[ \ker(Q_2 |_{\mc W_\kappa(\spl_2) \otimes \Pi })\cong V^\kappa(\spl_2) \subset \mc W_{f_\prin}^\kappa(\spl_2) \otimes \Pi  \quad\quad \text{and again} \quad\quad  \V(Y,S) = \ker(Q_1)\cap \ker(Q_2) \cong \pi^+\otimes V^\kappa(\spl_2)  \ . \]
\end{proof}

We recall that the framing structures and induced quivers with potential that conjecturally correspond to each of the representations of $V^\kappa(\gl_2)$ in Figure \ref{fig:sl2factfig} above are given in Figure \ref{fig:Walgframingsfig}.

Next, we describe the generalization of these arguments to outline a proof of Theorem \ref{WLStheo} below. Recall that we let $\mu$ denote a partition of length $m$ given by
\[ \mu=\{ \mu_1 \geq \hdots \geq \mu_m \geq 0 \} \quad\quad\text{and let} \quad\quad  \textbf{M} =(M_i)_{i=0}^{m-1}\quad\quad\text{where}  \quad\quad M_i=\sum_{j=i+1}^{m}\mu_j  \]
and write simply $M = M_0 = |\mu|=\sum_{j=1}^m \mu_j$. We define the corresponding divisor $S_\mu$ in $Y_{m,0}$ by the labelling of the faces of the moment polytope of $Y_{m,0}$ pictured in Figure \ref{fig:Walgfig} below. We also introduce the shorthand expression for $S_\mu$ indicated in the top left of the figure. We have:

\begin{theo}\label{WLStheo} There is an isomorphism of vertex algebras
	\[ W^\kappa_{f_\mu}(\gl_M) \xrightarrow{\cong}\V(Y_{m,0}, S_\mu)  \ . \]
\end{theo}

The proof will consist of identifying the defining free field realization of $\V(Y_{m,0},S_{\mu})$ with a certain bosonized presentation, in the sense of \cite{FMS}, of the generalization \cite{Gen1} of the Wakimoto resolution \cite{Wak} to $W$-algebras for general nilpotents, such that the embeddings of Equation \ref{voaembeqn} induce parabolic induction \cite{Gen2} and inverse reduction \cite{Sem} maps, generalizing Theorem \ref{affinesl2theo}.

\begin{proof}
To begin, we note that the statement of the theorem is naturally compatible with the existence of parabolic induction maps between affine $W$-algebras for $\gl_M$: Given a pyramid $\pi_\sigma$ which is the sum of two pyramids $\pi_{\mu}$ and $\pi_{\nu}$ for a decomposition $\sigma=\mu+\nu$ as a sum of dominant coweights $\mu,\nu\geq 0$, it is proved in \cite{Gen2} that there is an embedding of vertex algebras
\begin{equation}\label{parindeqn}
	 W^\kappa_{f_\sigma}(\gl_O) \to  W^{\kappa_1}_{f_{\mu}}(\gl_{M}) \otimes  W^{\kappa_2}_{f_{\nu}}(\gl_{N})  
\end{equation}
which preserves the gradings from the respective pyramids, where $M=|\mu|$, $N=|\nu|$ and $O=|\sigma|$.

\begin{figure}[b]
	\caption{The divisor $S_\mu$ in $Y_{m,0}$ and its associated pyramid $\pi_\mu$ and shorthand}
	\label{fig:Walgfig}
	\hspace*{-3cm}
	\begin{overpic}[width=1\textwidth]{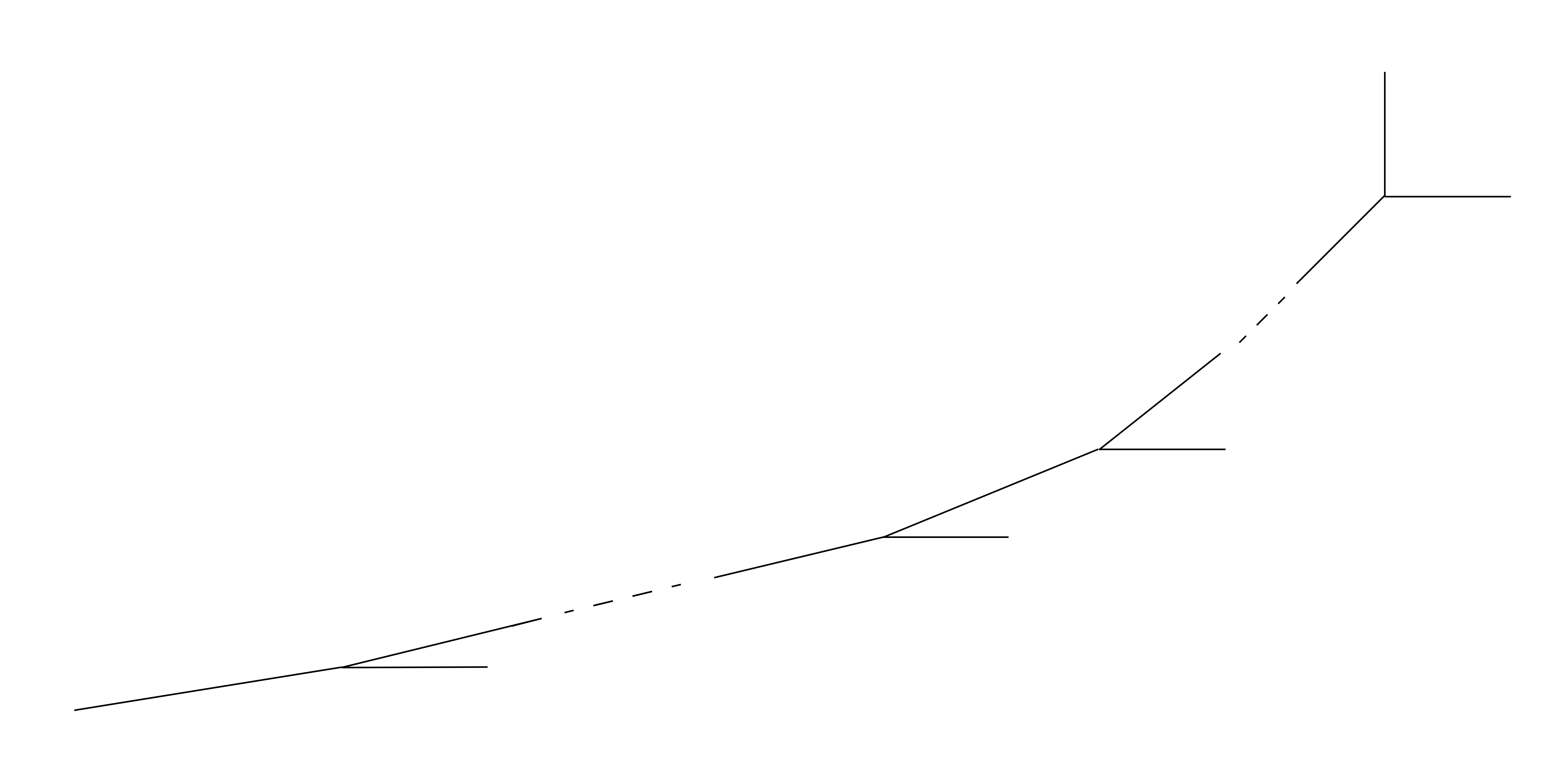}
			\put(5,43) {$ Y_{m,0}\to X_{m,0} = \{ xy-z^m \}\times \bb A^1$}
			\put(5,40)  {\rotatebox[origin=c]{90}{$\subset$}}
			\put(5,37) {$S_\mu  \ = $}
			\put(49,37) {$ M_0 $}
			\put(44,37) {$ \hdots $}
			\put(36.5,37) {$ M_{i-1} $}
			\put(31,37) {$ M_i $}
			\put(24,37) {$ \hdots $}
			\put(16.5,37) {$ M_{m-1}$}
			\put(11,32) {$\includegraphics{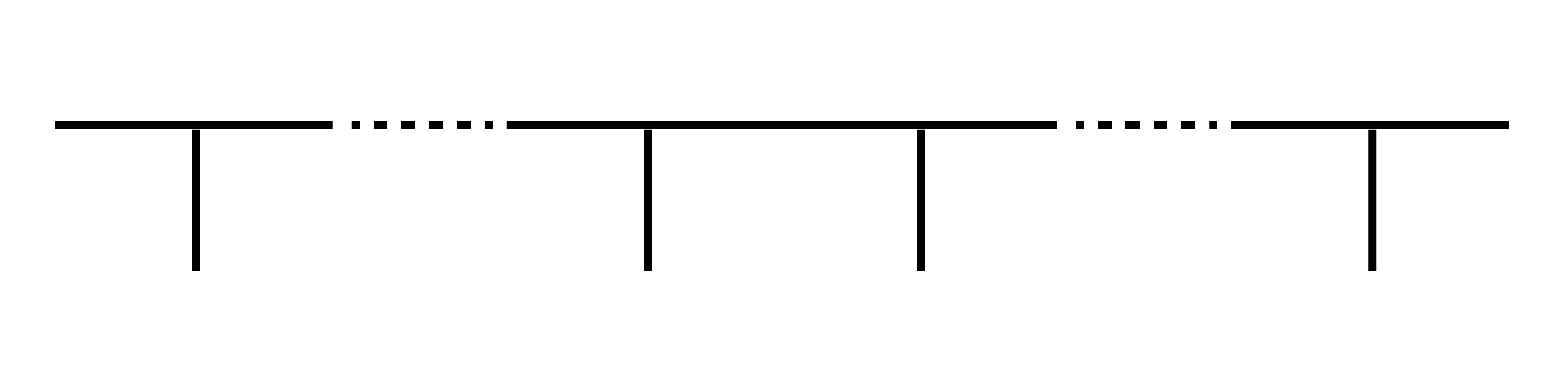}$}
			
			\put(47.5,34)  {$\mu_1$}
			\put(35,34)  {$\mu_i$}
			\put(28,34)  {$\mu_{i+1}$}
			\put(15,34)  {$\mu_m$}

		\put (90,39) {$M_0= \mu_1+\hdots+\mu_m$} 
		\put (88,34) {$M_1= \mu_2+\hdots+\mu_m$} 
		\put(75, 23) {$M_{i-1}=\mu_{i}+\hdots+\mu_m$}
		\put(67, 18) {$M_{i}=\mu_{i+1}+\hdots+\mu_m$} 
		\put(54, 13) {$M_{i+1}=\mu_{i+2}+\hdots+\mu_m$} 
		\put(31,8) {$M_{m-1}=\mu_m$}
		
		\put(69,2) {$\pi_{\mu}= $}
		\put(75, 17) {\begin{ytableau}
				\none[] &\none[ ]  & 	\none[ ] & 	\none[ ] & \none[ ] & 	\none[ ]  & \none[ ] & 	\none[ ]  & \mu_m & \none[\hdots] &  \\
				\none[] &\none[ ]  & 	\none[ ] & 	\none[ ] & \none[ ] & 	\none[ ]  &  \mu_{m-1} & \none[\hdots] & & & \\
				\none[] &\none[ ]  & 	\none[ ] & 	\none[ ] & \none[ ] & 	\none[ ]  & \none[\vdots] & \none[\vdots] & \none[\vdots] & \none[\vdots] &\none[\vdots]  \\
				\none[] &\none[ ]  & 	\none[ ] & 	\none[ ] & \mu_i & \none[\hdots] & & & & &  \\
				\none[] &\none[ ]  & 	\none[ ] & 	\none[ ] & \none[\vdots] & \none[\vdots] & \none[\vdots] & \none[\vdots] & \none[\vdots] & \none[\vdots] & \none[\vdots] \\
				\none & \none & \mu_2 & \none[\hdots]& & & & & & & \\
				\mu_1 & \none[\hdots] &  & & & & & & & & 
		\end{ytableau}}
	\end{overpic}
	\vspace*{-.5cm}
\end{figure}

\begin{figure}[t]
	
\begin{overpic}[width=1\textwidth]{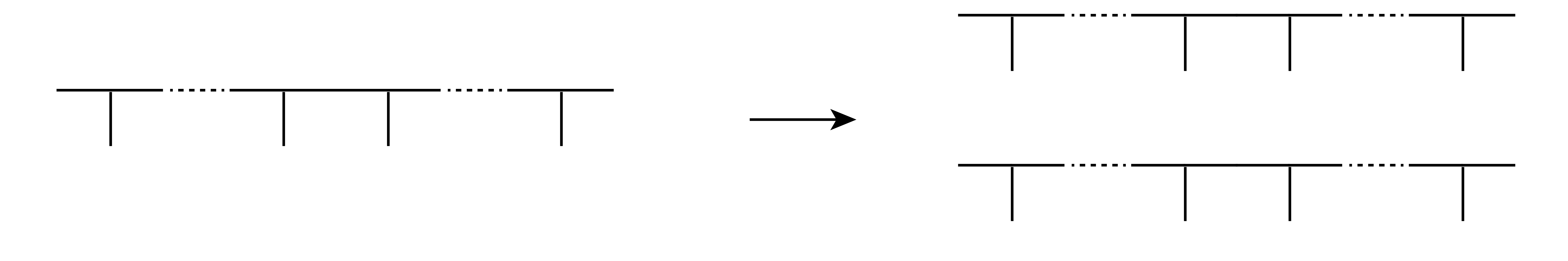}

		\put(36,8.5) {$ O_0 $}
		\put(32,8.5) {$ \hdots $}
		\put(25.5,8.5) {$ O_{i-1} $}
		\put(20,8.5) {$ O_i $}
		\put(14,8.5) {$ \hdots $}
		\put(8,8.5) {$ O_{m-1}$}
		
		\put(35,6)  {$\sigma_1$}
		\put(24,6)  {$\sigma_i$}
		\put(17,6)  {$\sigma_{i+1}$}
		\put(6,6)  {$\sigma_m$}

		\put(94,3.5) {$ M_0 $}
		\put(90,3.5) {$ \hdots $}
		\put(83.5,3.5) {$ M_{i-1} $}
		\put(78,3.5) {$ M_i $}
		\put(72,3.5) {$ \hdots $}
		\put(65,3.5) {$ M_{m-1}$}
		
		\put(93,1)  {$\mu_1$}
		\put(82,1)  {$\mu_i$}
		\put(75,1)  {$\mu_{i+1}$}
		\put(64,1)  {$\mu_m$}

				\put(78,7.5) {$\bigotimes$}

		\put(94,13.5) {$ N_0 $}
		\put(90,13.5) {$ \hdots $}
		\put(83.5,13.5) {$ N_{i-1} $}
 		\put(78,13.5) {$ N_i $}
		\put(72,13.5) {$ \hdots $}
		\put(65,13.5) {$ N_{m-1}$}
		
		\put(93,11)  {$\nu_1$}
		\put(82,11)  {$\nu_i$}
		\put(75,11)  {$\nu_{i+1}$}
		\put(65,11)  {$\nu_m$}

\end{overpic}

	\caption{The divisor $S_\mu$ in $Y_{m,0}$ and its associated pyramid $\pi_\mu$ and shorthand}
	\label{fig:shorthandfactpic}
	\hspace*{-3cm}

\end{figure}

Further, note that the corresponding divisors $S_\mu$, $S_\nu$ and $S_\sigma$ satisfy the natural relation
\[ 0\to \mc O_{S_\mu} \to \mc O_{S_\sigma} \to  \mc O_{S_\nu} \to 0 \]
so that by the factorization principle of Proposition \ref{factprop}, there is a natural embedding
\begin{equation}\label{VSmuparindeqn}
	 \V(Y_{m,0}, S_\sigma) \to \V(Y_{m,0}, S_\mu)\otimes \V(Y_{m,0}, S_\nu) 
\end{equation}
with image characterized by the kernel of a single screening operator, induced by filtrations
\[ \hspace*{-1cm}\begin{cases} \mc O_{S_\mu}  =  &  \left[ \mc O_{S_{d_1}} < ... <\mc O_{S_{d_{k_0}}} \right] \\   \mc O_{S_\nu} =  &  \left[ \mc O_{S_{d_{k_0+1}}} < ... < \mc O_{S_{d_N}} \right] \end{cases}  \quad\quad \textup{such that} \quad\quad  \mc O_{S_\sigma} = \left[ \mc O_{S_{d_1}} < ... <\mc O_{S_{d_{k_0}}} <\mc O_{S_{d_{k_0+1}}} <  ... < \mc O_{S_{d_N}} \right]  \ .  \]
Indeed, recall that such a decomposition induces an identification $\Pi(Y,S_\sigma)\cong \Pi(Y,S_\mu)\otimes_K \Pi(Y,S_\nu)$, such that the screening operator $Q_{s_k}:\Pi(Y,S_\sigma)\to \Pi(Y,S_\sigma)_{\lambda_k}$ is the only screening operator which does not vanish on either of the subalgebras $ \Pi(Y,S_\mu)\otimes_K \one$ or $ \one\otimes \Pi(Y,S_\nu)$ of $ \Pi(Y,S_\mu)\otimes_K \Pi(Y,S_\nu)$, and the remaining screening operators can be identified with those defining the tensor product $\V(Y_{m,0}, S_\mu)\otimes \V(Y_{m,0}, S_\nu) $; this is summarized in Figure \ref{fig:shorthandfactpic} above.

The proof of the existence of parabolic induction maps given in \cite{Gen2} uses essentially the same argument, comparing the generalized Wakimoto resolutions of the relevant $W$-algebras $ W^\kappa_{f_\mu}(\gl_M)$ introduced in \cite{Gen1}. Indeed, we will argue that the isomorphisms of Theorem \ref{WLStheo} identify the free field algebras $\Pi(Y_{m,0},S_\mu)$ with certain canonical bosonizations of the generalized Wakimoto resolutions, and identify the vertex algebra embeddings of Equations \ref{parindeqn} and \ref{VSmuparindeqn}. Inductive application of parabolic induction realizations gives an embedding of any $W$-algebra $ W^\kappa_{f_\mu}(\gl_M)$ into a tensor product of affine algebras determined by the columns of the corresponding pyramid $\pi_\mu$, so that the preceding claim reduces Theorem \ref{WLStheo} to the case of affine algebras. We begin by outlining the proof of the result for affine algebras, after which we explain the identification with generalized Wakimoto resolutions and compatibility with parabolic induction maps.

The $\gl_m$ coweight $\mu_{m}=(1,...,1)$ corresponds to the trivial nilpotent $f_{\mu_{m}}=0\in \gl_M=\gl_m$, for which the corresponding $W$-algebra is given by the affine Kac-Moody algebra $W_{0}^\kappa(\gl_m)=V^\kappa(\gl_m)$. The corresponding divisor $S_{\mu_{m}}$ is given by the labelling of the moment polytope as in Figure \ref{fig:affinealgfig} for $M_i=m-i$ for each $i=0,...,m-1$. We now proceed to analyze the defining free field realization of the vertex algebra $\V(Y_{m,0},S_{\mu_{m}})$ in this case, which will be identified with the Wakimoto realization \cite{Wak}, proving Theorem \ref{WLStheo} for coweights $\mu$ of the form ${\mu_{m}}$.

To begin, note that for any $\textbf{M}=(M_i)_{i=0}^{m-1}\in \bb N^{m}$ we have a corresponding divisor $S_\MM$ determined by the labelling of the moment polytope as above, but in general the corresponding integers $\mu_i=M_{i-1}-M_i$ need not be positive or ordered, and thus the corresponding coweight $\mu$ need not be dominant. Nonetheless, as we will explain, there are several such examples for which the corresponding vertex algebras admit natural interpretations in representation theory.

\begin{figure}[t]
	
	\begin{overpic}[width=1\textwidth]{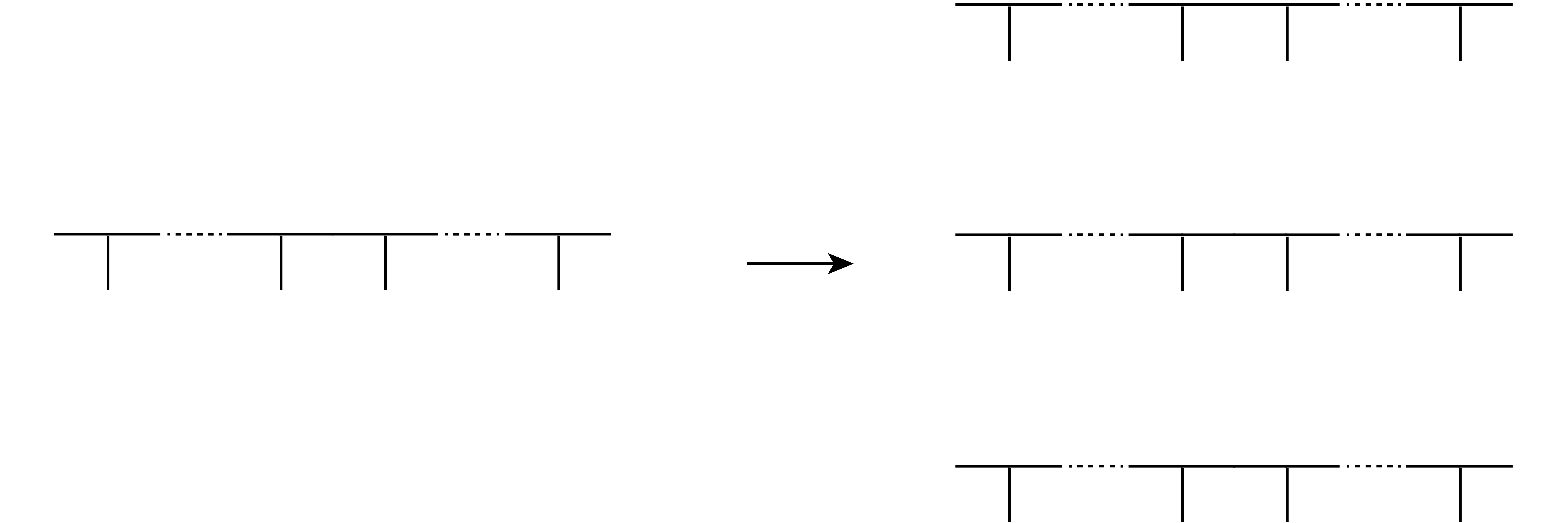}
		
		\put(94,1) {$1$}
		\put(88,1) {$ \hdots $}
		\put(83.5,1) {$1 $}
		\put(78,1) {$ 1 $}
		\put(73,1) {$1$}
		\put(68,1) {$ \hdots $}
		\put(65,1) {$1$}

		\put(92.5,-2)  {$0$}
		\put(82,-2)  {$0$}
		\put(75,-2)  {$0$}
		\put(64,-2)  {$1$}

		\put(78,5) {$\bigotimes$}
		\put(78,11) {$\bigotimes$}
		\put(79,7.5) {$\vdots$}

		\put(78,20) {$\bigotimes$}
		\put(78,26) {$\bigotimes$}
		\put(79,22.5) {$\vdots$}

		\put(94,16) {$1$}
		\put(88,16) {$ \hdots $}
		\put(83.5,16) {$1 $}
		\put(78,16) {$ 1 $}
		\put(73,16) {$0$}
		\put(68,16) {$ \hdots $}
		\put(65,16) {$0$}

		\put(92.5,13)  {$0$}
		\put(82,13)  {$0$}
		\put(75,13)  {$1$}
		\put(64,13)  {$0$}

		\put(94,31) {$1$}
		\put(88,31) {$ \hdots $}
		\put(83.5,31) {$0 $}
		\put(78,31) {$ 0 $}
		\put(73,31) {$0$}
		\put(68,31) {$ \hdots $}
		\put(65,31) {$0$}

		\put(92.5,27.5)  {$1$}
		\put(82,27.5)  {$0$}
		\put(75,27.5)  {$0$}
		\put(64,27.5)  {$0$}

		\put(36,16) {$m$}
		\put(31,16) {$ \hdots $}
		\put(25,16) {$i+1 $}
		\put(21,16) {$ i $}
		\put(13,16) {$i-1$}
		\put(9,16) {$ \hdots $}
		\put(7,16) {$1$}

		\put(35,13)  {$1$}
		\put(24,13)  {$1$}
		\put(17.5,13)  {$1$}
		\put(6,13)  {$1$}

	\end{overpic}

	\caption{Factorization for $S_{\mu_{m}}$ in $Y_{m,0}$ associated to the trivial nilpotent $f_{\mu_{m}}=0$}
	\label{fig:affinealgfig}
	\hspace*{-3cm}
	
\end{figure}
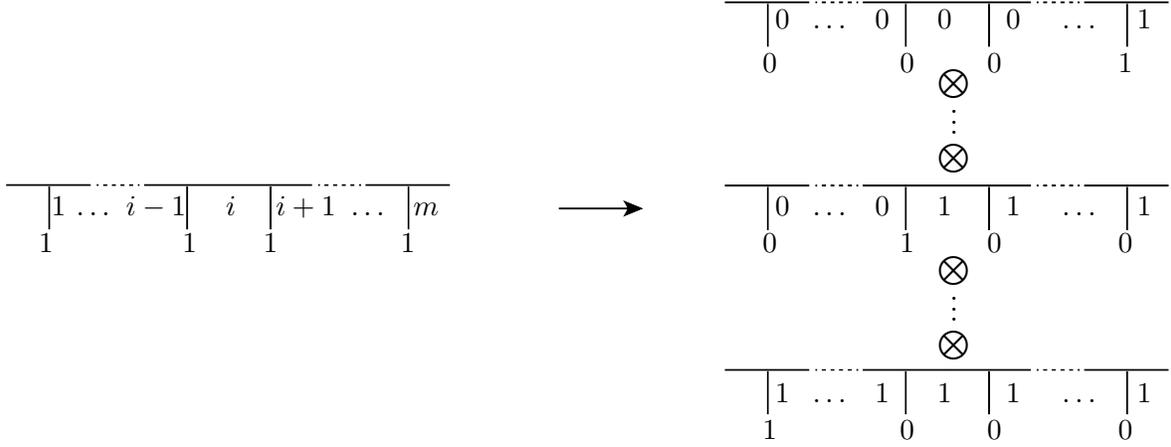

In particular, applying Proposition \ref{factprop} inductively to the above divisor $S_{\mu_{m}}$, as in Figure \ref{fig:affinealgfig} above, we see that the vertex algebra $\V(Y_{m,0},S_{\mu_{m}})$ admits an embedding
\begin{equation}\label{prewakieqn}
	\V(Y_{m,0},S_{\mu_{m}}) \to \bigotimes_{i=0}^{m-1} \V(Y_{m,0}, S_{\textbf{M}_j})
\end{equation}
characterized as the intersection of the kernels of screening operators $Q_j$ for $j=1,...,m-1$, where $S_{\textbf{M}_j}$ is the divisor corresponding to the labelling $\textbf{M}_i=(1,...,1,0,...,0)$, where the first zero entry is $M_{i+1}$. In fact, we have the following description of these vertex algebras:
\end{proof}

\begin{prop}\label{CDOprop} There is an isomorphism of vertex algebras
\[ \mc D^\ch(\bb A^j)\otimes \pi  \xrightarrow{\cong} \V(Y_{m,0},S_{\textbf{M}_j}) \ .\]
\end{prop}
\begin{proof}
	The argument is the same as that in the proof of Proposition \ref{chantidomegprop}, which is precisely the claimed result in the case $j=1$, \emph{mutatis mutandis}.
\end{proof}

\begin{proof}(of Theorem \ref{WLStheo}, continued) We now complete the proof of Theorem \ref{WLStheo} in the case that $f_\mu=0$, before proceeding with the generalization to arbitrary $\mu \geq 0$. By Proposition \ref{CDOprop}, the codomain of the embedding of vertex algebras in Equation \ref{prewakieqn} is isomorphic to the Wakimoto module,
\begin{equation}\label{affineWakeqn}
	 \bb W_{0}^\kappa(\gl_m)_0 := \mc D^\ch(N)\otimes \pi_{\mf h} \xrightarrow{\cong} \bigotimes_{i=0}^{m-1} \V(Y_{m,0}, S_{\textbf{M}_j})  \ , 
\end{equation}
 the tensor product of the chiral differential operators $\mc D^\ch(N)$ on the lower triangular unipotent matrices $N\subset \Gl_m$, with the Heisenberg algebra $\pi_\h^k$ on the diagonal subalgebra $\mf h\subset \gl_m$, as follows from the identifications $N\cong \bb A^{\frac{m(m-1)}{2}}$ and $\pi_\h^k \cong (\pi^k)^{\otimes m}$. Similarly, we have
 \[ \bigoplus_{l(w)=1}  \bb W_{0}^\kappa(\gl_m)_w  \xrightarrow{\cong}  \bigoplus_{j=1}^{m-1} \left( \bigotimes_{i=0}^{m-1} \V(Y_{m,0}, S_{\textbf{M}_i}) \right)_{\lambda_j}  \]
 so that it remains to show that there exists a coordinate system on $T^*N$ for which the direct sum of the Wakimoto screening operators
 $Q_w^\text{Wak}: \bb W_{0}^\kappa(\gl_m)_0 \to  \bb W_{0}^\kappa(\gl_m)_w$
is identified with that of the screening operators $Q_j$ for $j=1,..,m-1$ introduced above, which are all of the canonical linear form given in the definition of $\V(Y,S_{\mu})$. In the forthcoming work \cite{BBN} of the author together with Christopher Beem and Sujay Nair, we construct such coordinate systems in terms of multiplication maps on generalized slices between orbits in the affine Grassmannian of $\Gl_m$. This implies the desired result in the case $f_\mu=0$, and more generally as we now explain:

For a general nilpotent $f_\mu$ in $\gl_M$ corresponding to a dominant $\spl_m$ coweight $\mu$ with $|\mu|=M$ and a choice of pyramid $\pi_\mu$ for $f_\mu$ of width $h$ and necessarily of height $m$, there are two natural labellings of the pyramid, given by bibliographically with respect to columns and rows. These correspond to the Jordan normal form of $f_\mu$ (which by definition has $m$ Jordan blocks), and what we call the \emph{generalized hook-type form} of $f_\mu$: that for which $f_\mu$ is contained in the Lie algebra $\mf n_{\mu}$ of the lower triangular unipotent subgroup $N_{\mu}$, defined as the complement to the Levi factors $L_l=\Gl_{m_l}$ determined by the heights $m_l$ of the columns of $\pi_\mu$ for $l=1,...,h$.

In the example that $\mu=(3,2,1,1)$, the two labellings and corresponding permutations of coordinates are given by
\[ \hspace*{-1cm}   \pi_\mu=  \scriptsize{\begin{ytableau}
 \none & \none & 7 \\  \none[] & \none[] & 6 \\ \none & 4 & 5 \\ 1 & 2 & 3
\end{ytableau}}
\quad,\quad
\pi_\mu= \scriptsize{\begin{ytableau}
 \none & \none &7 \\  \none[] & \none[] & 6 \\ \none & 3 & 5 \\ 1 & 2 & 4
\end{ytableau}} \quad  \implies  \quad  f_\mu = \mleft[ \scriptsize{ \begin{array}{c  c  c| c c| c |c}  0& 0 &  0 & 0  & 0  &0 & 0  \\ 
1 & 0 &  0 & 0  & 0  &0 & 0  \\ 
0 & 1 & 0 &  0 & 0  & 0  &0   \\ \hline
0& 0 &  0 & 0  & 0  &0 & 0  \\ 
0  & 0  & 0 & 1 & 0 &  0 & 0   \\  \hline
  0& 0 &  0 & 0  & 0  &0 & 0  \\ \hline
   0& 0 &  0 & 0  & 0  &0 & 0  \end{array}} \mright]  \quad ,\quad
  f_\mu = \mleft[  \scriptsize{\begin{array}{c  |  c  c|  cc c  c}   
  	0& 0 &  0 & 0  & 0  &0 & 0  \\ \hline 
  	\boxed{ {\color{blue}1}} &  0 & 0  & 0  &0 & 0  \\
  	{\color{blue}0} & {\color{red} 0 } & 0 & 0  & 0  &0 & 0 \\ \hline 
  	{\color{blue}0} &  {\color{blue}1} & \boxed{ {\color{blue}0} }  &0  & 0  &0 & 0 \\
  	{\color{blue}0}& {\color{blue}0} &  {\color{blue}1} &{\color{red}0}  & 0  &0 & 0 \\
  	{\color{blue}0}& {\color{blue}0} &  {\color{blue}0} &{\color{red}0} &{\color{red}0} & 0  & 0 \\
  	{\color{blue}0}& {\color{blue}0} &  {\color{blue}0} &{\color{red}0} &{\color{red}0} & {\color{red}0}   & 0 \\
  \end{array}} \mright] \ , 
\]
where the blue entries indicate the Lie algebra $\mf n_\mu$, and the red denote the Lie algebras $\mf n_l$ of the lower triangular unipotent subgroups $N_l$ of the aforementioned Levi factors $L_l$.

The latter presentation of the nilpotent $f_\mu$ and the good grading induced by $\pi_\mu$ implies that
\begin{equation}\label{Nmuzeroeqn}
	 N_\mu^0 := N/N_\mu  \xrightarrow{\cong} \prod_{l=1}^h N_l \ , 
\end{equation}
where $N$ denotes the lower triangular unipotent subgroup, with respect to the basis inducing the generalized hook-type form of $f_\mu$. Using this identification and applying the factorization maps of Proposition \ref{factprop} corresponding to the decomposition of $\pi_\mu$ into columns, as pictured in Figure \ref{fig:affinefactfig} below in the example $\mu=(3,2,1,1)$, we obtain a vertex algebra embedding
\begin{equation}\label{affineparindeqn}
	 \V(Y,S_\mu) \to  \bigotimes_{l=1}^h\V(Y_{m,0},S_{\mu_{m_l}}) \cong \bigotimes_{l=1}^h V^\kappa(\gl_{m_l}) \ ,
\end{equation}

\begin{figure}[b]	
	\caption{Factorization for $\mu=(3,2,1,1)=(1,1,1,1)+(1,1,0,0)+(1,0,0,0)$}
	\label{fig:affinefactfig}
	
	\begin{overpic}[width=1\textwidth]{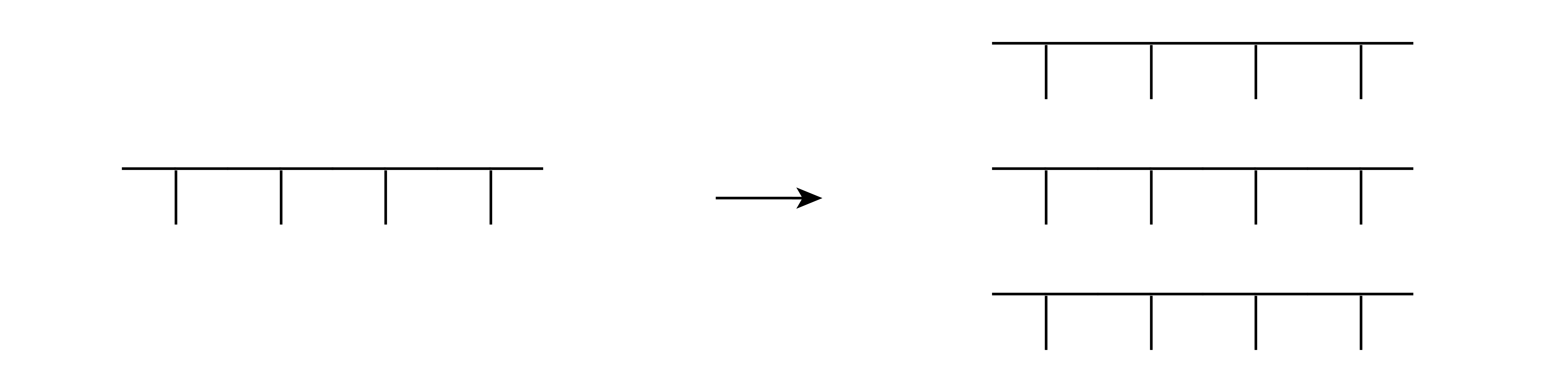}

		\put(-6,15) {\begin{ytableau}
				\none & \none & \  \\  \none[] & \none[] & \ \\ \none & \ &\ \\ \ & \ & \
		\end{ytableau}}
	
\put(95,15) {\begin{ytableau}
		\none  \\  \none[]  \\ \none  \\  \
\end{ytableau}}

\put(100,15) {\begin{ytableau}
		\none  \\  \none[]  \\  \  \\  \
\end{ytableau}}

\put(105,15) {\begin{ytableau}
		\  \\  \  \\ \ \\  \
\end{ytableau}}

\put(75.8, 15.3) {$\bigotimes$}
\put(75.8, 7.3) {$\bigotimes$}	
	
		\put(69.5,3) {$0$}
		\put(76.5,3) {$0 $}
		\put(83,3) {$ 0 $}
		\put(88.5,3) {$1$}
		
		\put(66.5,0.5)  {$0$}
		\put(73,0.5)  {$0$}
		\put(79.5,0.5)  {$0$}
		\put(86.5,0.5)  {$1$}
		
		\put(69.5,11) {$0$}
		\put(76.5,11) {$0 $}
		\put(83,11) {$ 1 $}
		\put(88.5,11) {$2$}
		
		\put(66.5,8.5)  {$0$}
		\put(73,8.5)  {$0$}
		\put(79.5,8.5)  {$1$}
		\put(86.5,8.5)  {$1$}
		
		\put(69.5,19) {$1$}
		\put(76.5,19) {$2 $}
		\put(83,19) {$ 3 $}
		\put(88.5,19) {$4$}
		
		\put(66.5,16.5)  {$1$}
		\put(73,16.5)  {$1$}
		\put(79.5,16.5)  {$1$}
		\put(86.5,16.5)  {$1$}

		\put(14,11) {$1$}
		\put(20.5,11) {$2 $}
		\put(27.5,11) {$ 4 $}
		\put(33,11) {$7$}
		
		\put(10.5,8.5)  {$1$}
		\put(17.5,8.5)  {$1$}
		\put(24,8.5)  {$2$}
		\put(31,8.5)  {$3$}
		
\end{overpic}

\vspace*{-1.5cm}
\end{figure}

\noindent such that composing with the tensor product over $l$ of the factorization maps of Equation \ref{prewakieqn} for $m=m_l$ identifies the resulting free field realization with the generalized Wakimoto realization of \cite{Gen1}, as we now explain; this implies the compatibility of the factorization structure maps with the parabolic inductions maps as claimed in the beginning of the proof.

First, note that composing the factorization map of Equation \ref{affineparindeqn} with the tensor product of those of Equation \ref{prewakieqn}, as described, gives a vertex algebra embedding
\[ \V(Y,S_\mu) \to \bigotimes_{l=1}^h \bigotimes_{i=0}^{m_l-1} \V(Y_{m,0}, S_{\textbf{M}_j}) \cong  \bigotimes_{l=1}^h  \mc D^\ch(N_l) \otimes \pi_{\mf h_l}  \ ,\]
where $\mf h_l$ denotes the Cartan subalgebra of the Levi factor $L_l$, so that the identification of Equation \ref{Nmuzeroeqn} gives an identification of the codomain with the generalized Wakimoto module
\[   \bb W_{f_\mu}^\kappa(\gl_M)_0 := \mc D^\ch(N_\mu^0) \otimes \pi_{\mf h} \xrightarrow{\cong} \bigotimes_{l=1}^h \bigotimes_{i=0}^{m_l-1} \V(Y_{m,0}, S_{\textbf{M}_j})  \ ,\]
for nilpotent $f_\mu\in \gl_M$ at level $\kappa$ of \cite{Gen1}, generalizing the identification of Equation \ref{affineWakeqn}, where we let $\mf h\cong \bigoplus_{l=1}^h \mf h_{m_l}$ be the Cartan subalgebra of $\gl_M$.

Moreover, the screening operators characterizing the images of each of the factorization embeddings of Equation \ref{prewakieqn} are identified with those of the Wakimoto realizations for each of the Levi factors factors $L_l$, by the previous results in the affine case. Thus, we have identifications 
 \[ \bigoplus_{l(w)=1}  \bb W_{f_\mu}^\kappa(\gl_{m_l})_w  \xrightarrow{\cong}  \bigoplus_{j=1}^{m_l-1} \left( \bigotimes_{i=0}^{m_l-1} \V(Y_{m_l,0}, S_{\textbf{M}_i}) \right)_{\lambda_j}  \]
for each $l=1,...,h$, under which the geometric screening operators $Q_{s_j}$ for $j=1,...,m_l$ and $l=1,...,h$ correspond to the $m_l$ Wakimoto screening operators for each of the Levi factors $L_l$. In turn, acting on the larger algebra these can be identified with the generalized Wakimoto screening operators $Q^\textup{Wak}_w:\bb W_{f_\mu}^\kappa(\gl_M)_0 \to W_{f_\mu}^\kappa(\gl_M)_w $ for the length one Weyl group elements $w$ corresponding to only those simple roots $\alpha_i$ for $\gl_M$ which are contained in one of the corresponding lower triangular nilpotent Lie algebras $\mf n_l$. This accounts for $\sum_{l=1}^h m_l-1 = M-h$ of the $M-1$ screening operators in the generalized Wakimoto presentation for $W_{f_\mu}^\kappa(\gl_M)$.

Similarly, the image of the vertex algebra embedding in Equation \ref{affineparindeqn} is given by the kernel of $h-1$ screening operators, which can be identified with precisely the remaining $h-1$ screening operators in the generalized Wakimoto realization, those corresponding to the simple roots with root space contained in the complementary subspace $\mf n_\mu$. These are precisely the screening operators such that if one of their kernels is omitted in the intersection defining $W_{f_\mu}^\kappa(\gl_M)$, the induced embedding into the remaining intersection defines the parabolic induction map given by splitting the pyramid $\pi_\mu$ above the $l^{th}$ row, for $l=1,...,h-1$. These are the simple roots corresponding to the boxed entries of the matrix in the above example of the expression for the generalized hook-type form in the case $\mu=(3,2,1,1)$.

\end{proof}

We now explain the application of the results of this section to the construction of inverse quantum Hamiltonian reduction maps, generalizing those in \cite{Sem}, \cite{ACG}, \cite{Feh1}, and \cite{Feh2} to arbitrary W-algebras in type A, following work in progress proving this result joint with Christopher Beem and Sujay Nair \cite{BBN}. The statement of the main Theorem (in progress) is the following:

\begin{theo}\label{IRthm}\cite{BBN} Let $f_1\leq f_2$ be nilpotents in $\gl_M$. There is an embedding of vertex algebras
	\[ W^\kappa_{f_1}(\gl_M) \to  W^\kappa_{f_2}(\gl_M) \otimes \mc D^\ch(\bb G_m^{\times a} \times \bb A^b) \ .  \]
\end{theo}

Let $\mu_1$ and $\mu_2$ be dominant coweights corresponding to the partitions determined by the Jordan normal form of $f_1$ and $f_2$, respectively, and let $m_2\leq m_1\in \bb N$ be the lengths of the corresponding partitions. Then by Theorem \ref{WLStheo} above (which we recall relies on the results of \emph{loc. cit.}) we have divisors $S_{\mu^1}$ and $S_{\mu^2}$ in $Y_{m_1,0}$ such that
\begin{equation}\label{IR1eqn}
	W^\kappa_{f_1}(\gl_M) \xrightarrow{\cong }\V(Y_{m_1,0}, S_{\mu_1})  \quad\quad\text{and}\quad\quad  W^\kappa_{f_2}(\gl_M) \xrightarrow{\cong }\V(Y_{m_1,0}, S_{\mu_2})  \ .
\end{equation}

The condition that $f_1\leq f_2$ implies that there exists a divisor $S_\sigma$ on $Y_{m_1,0}$ such that
\begin{equation}\label{IRfacteqn}
	 \mc O_{S_\sigma} \to \mc O_{S_{\mu_1}} \to \mc O_{S_{\mu_2}}  
\end{equation}
defines a short exact sequence as in the hypotheses of Proposition \ref{factprop}, so that we have
\begin{equation}\label{IR2eqn}
	\V(Y_{m_1,0}, S_{\mu_1})  \to \V(Y_{m_1,0}, S_{\mu_2})  \otimes \V(Y_{m_1,0},S_\sigma)  
\end{equation}
an embedding of vertex algebras. Similarly, there exists a further filtration on $\mc O_{S_\sigma}$ such that \emph{loc. cit.} induces an embedding of vertex algebras
\begin{equation}\label{IR3eqn}
	 \V(Y_{m_1,0},S_\sigma) \to  \mc D^\ch(\bb G_m^{\times a} \times \bb A^b) \ ,
\end{equation}
where $a$ is the minimum number of boxes that must be moved to change the partition corresponding to $\mu_1$ into that corresponding to $\mu_2$, and $b$ is determined by $2(a+b)=\dim \mc S_{f_1}-\dim  \mc S_{f_2}$ where $\mc S_{f_i}$ denotes a slice to the nilpotent orbit $\bb O_{f_i}$ in $\gl_M$ for $i=1,2$.

In summary, composing the embedding of Equation \ref{IR2eqn} with that of Equation \ref{IR3eqn} (tensored with the identity on $\V(Y_{m_1,0}, S_{\mu_2}) $), implies Theorem \ref{IRthm}, by the identifications of Equation \ref{IR1eqn} induced by Theorem \ref{WLStheo}. Thus, we see that the presentation of affine W-algebras as vertex algebras $\V(Y_{m,0},S_\mu)$ in \emph{loc. cit.} is naturally adapted to the proof of inverse Hamiltonian reduction.

 Finally, we provide a simple example: Let $\mu_1=(3,2,1,1)$, as in the example discussed in the proof of \emph{loc. cit.} outlined above, and let $\mu_2=(4,1,1,1)$ so that we can take $\sigma=(-1,1,0,0)$ and we have the desired short exact sequence of Equation \ref{IRfacteqn}. Then we have
\[ S_\sigma = |\mc O_{\bb P^1}| \quad \quad \text{and thus} \quad\quad \V(Y_{m_1,0},S_\sigma) = \mc D^\ch(\bb G_m)  \ ,\]
so that we obtain the desired inverse quantum Hamiltonian reduction map
\[ W^\kappa_{f_1}(\gl_7) \to  W^\kappa_{f_2}(\gl_7) \otimes \mc D^\ch(\bb G_m)  \ ,\]
as summarized in Figure \ref{fig:IRfig} below. 

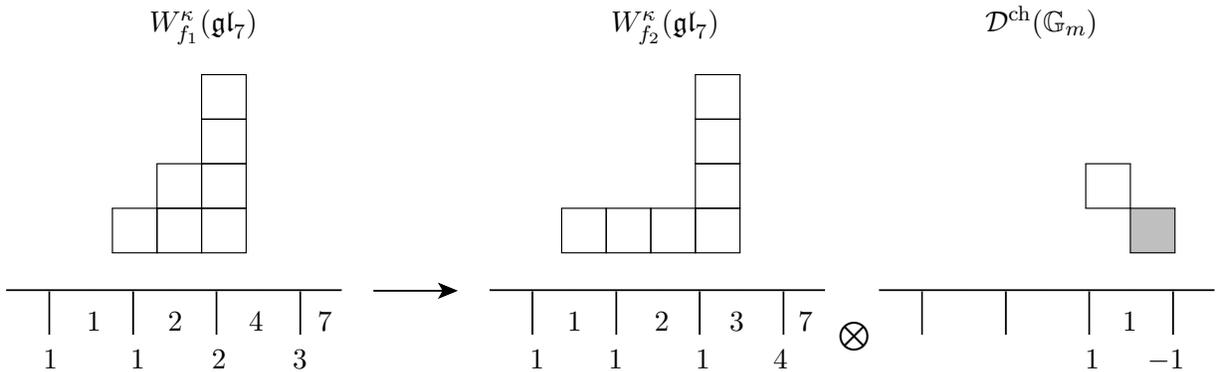
\begin{figure}[b]	
	\caption{Factorization for inverse quantum Hamiltonian reduction}
	\label{fig:IRfig}
	\vspace*{1.5cm}
	\begin{overpic}[width=1\textwidth]{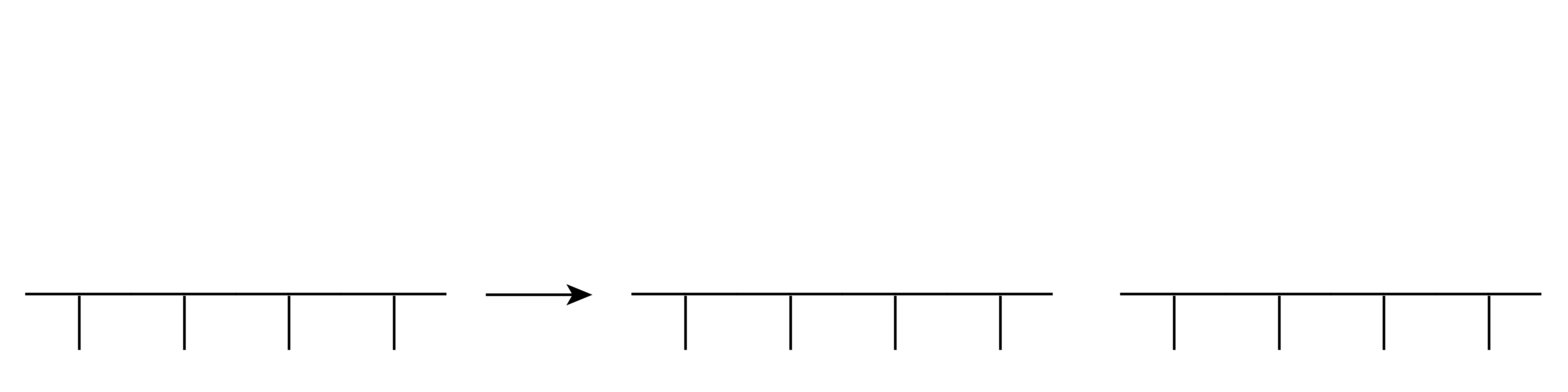}

		\put(13,27) {$W^\kappa_{f_1}(\gl_7)$}
		
		\put(50,27) {$ W^\kappa_{f_2}(\gl_7) $}
		
		\put(80 ,27) {$\mc D^\ch(\bb G_m)$}

		\put(10,20) {\begin{ytableau}
				\none & \none & \ \\  \none[] & \none[] & \ \\ \none & \ & \ \\ \ & \ & \
		\end{ytableau}}

	\put(46,20) {\begin{ytableau}
		\none &	\none & \none & \ \\  \none & \none[] & \none[] & \ \\ \none & \none & \none & \  \\  \ & \ & \ & \
	\end{ytableau}}
		
		\put(88, 20) {\begin{ytableau}
				\none \\ \none \\  \ \\ \none & *(lightgray) \ 
		\end{ytableau}}
		
		\put(8,3) {$1$}
		\put(14.5,3) {$2 $}
		\put(21,3) {$ 4 $}
		\put(26.5,3) {$7$}
		
		\put(4.5,0)  {$1$}
		\put(11.5,0)  {$1$}
		\put(18,0)  {$2$}
		\put(24.5,0)  {$3$}

		\put(46.5,3) {$1$}
		\put(53.5,3) {$2 $}
		\put(59.5,3) {$ 3 $}
		\put(65,3) {$7$}

		\put(43.5,0)  {$1$}
		\put(49.8,0)  {$1$}
		\put(56.8,0)  {$1$}
		\put(63,0)  {$4$}
		
		\put(68.2,2) {$\bigotimes$}
		
		\put(91,3) {$1$}
		\put(88,0)  {$1$}
		\put(93,0)  {$-1$}
		
	\end{overpic}
\end{figure}

\subsection{Genus zero class $\mc S$ chiral algebras}\label{classSsec}
In this section, we explain the conjectural application of our results to the genus zero, class $\mc S$ chiral algebras, which were proposed in \cite{BeemS} and defined mathematically in the case of genus zero with regular singularities in \cite{Ar}. The main result of \emph{loc. cit.} is that there exists a family of vertex algebras $\bb V^{\mc S}_{G;f_1,...,f_k}$, for any simply connected, semisimple algebraic group $G$ and collection of nilpotent elements $f_i\in \g$ for $i=1,...,k$, which satisfy several natural compatibilities as the number of points $k$ and nilpotents $f_i$ vary.

In our geometric setting, we must take $G=\Gl_M$, and to begin we restrict to the case that $k\leq 2$. Then we must fix nilpotents $f$ and $\tilde f$ in $\mc N_{M}\subset \gl_M$, which correspond to a pair of partitions of $M$
\[ \mu=\{ \mu_1 \geq \hdots \geq \mu_m \geq 0 \} \quad \quad\text{and}\quad\quad \tilde \mu=\{\tilde \mu_1 \geq \hdots \geq \tilde\mu_{\tilde m } \geq 0  \}  \ ,\]
of some lengths $m$ and $\tilde m$. As in the preceding Section \ref{Walgsec}, we also define the lists of integers
\[ \textbf{M} =(M_i)_{i=0}^{m-1} \quad \tilde{\textbf{M}} =(\tilde M_i)_{i=0}^{\tilde m-1}   \quad \quad \text{where} \quad \quad M_i=\sum_{k=i+1}^{m} \mu_k \quad \tilde M_i=\sum_{k=i+1}^{\tilde m} \tilde \mu_k \]
$i=0,...,m-1$, $j=0,...,\tilde m-1$, noting that by hypothesis we have
\[ |\mu| = M_0 = M = \tilde M_0 = |\tilde \mu| \ , \]
and define a divisor $S_{\mu,\tilde \mu}\subset Y_{m+\tilde m,0}$ by labelling the faces of the moment polytope as pictured in Figure \ref{fig:classSfig} below, and we have the following conjecture:

\begin{conj}\label{classSconj} There is an isomorphism of vertex algebras
\[ \V_{\Gl_M;f,\tilde f}^{\mc S,\kappa} \xrightarrow{\cong } \V(Y_{m+\tilde m,0},S_{\mu,\tilde{\mu}}) \ .  \]
\end{conj}

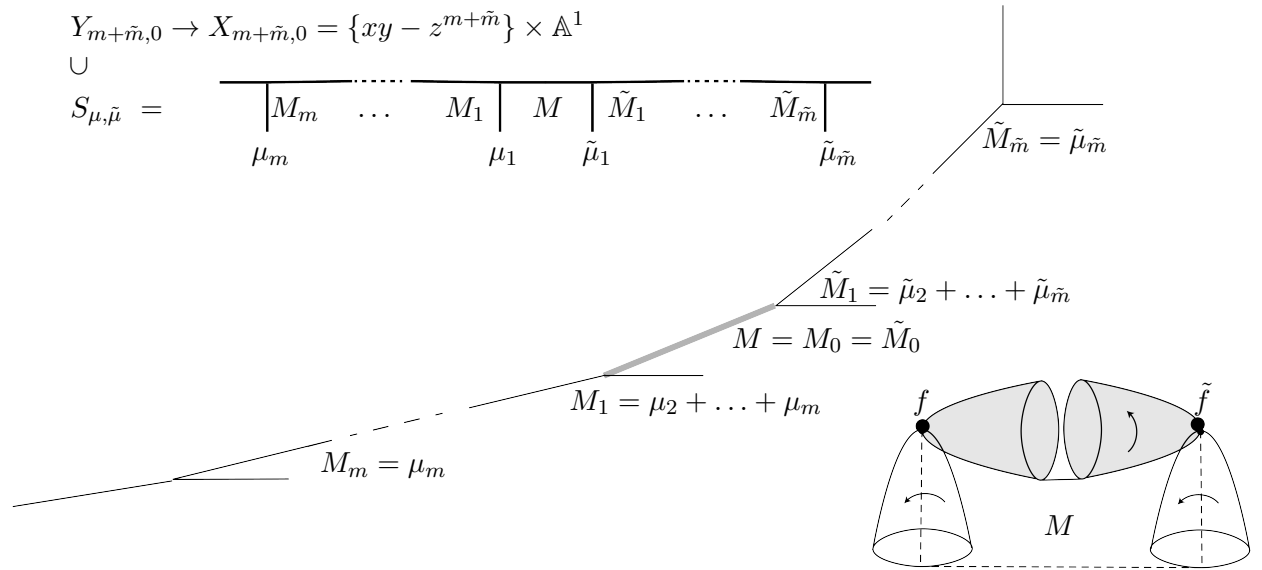
\begin{figure}[b]
	\caption{The divisor $S_{\mu,\tilde{\mu}}$ in $Y_{m+\tilde m,0}$ associated to partitions $(\mu,\tilde \mu)$ of $M$}
	\label{fig:classSfig}
	\begin{overpic}[width=1\textwidth]{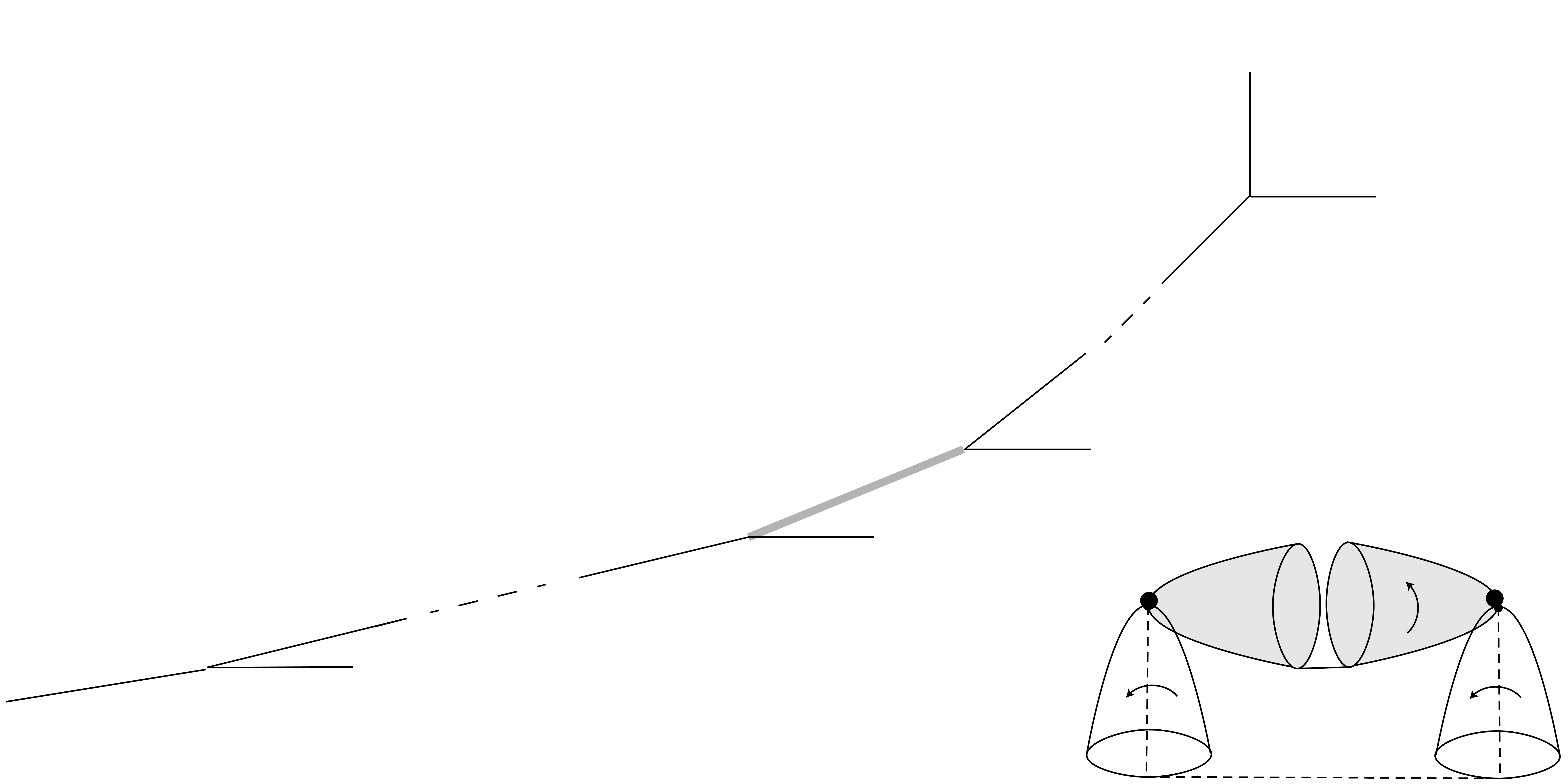}
		\put(5,43) {$ Y_{m+\tilde m ,0}\to X_{m+\tilde m,0} = \{ xy-z^{m+\tilde m} \}\times \bb A^1$}
		\put(5,40)  {\rotatebox[origin=c]{90}{$\subset$}}
		\put(5,36.5) {$S_{\mu,\tilde{\mu}}  \ = $}
		\put(61,36.5) {$\tilde M_{\tilde m} $}
		\put(55,36.5) {$ \hdots $}
		\put(48,36.5) {$\tilde M_1 $}
		\put(42,36.5) {$ M $}
		\put(35,36.5) {$  M_1 $}
		\put(28,36.5) {$ \hdots $}
		\put(21,36.5) {$ M_m$}
		\put(11,32) {$\includegraphics{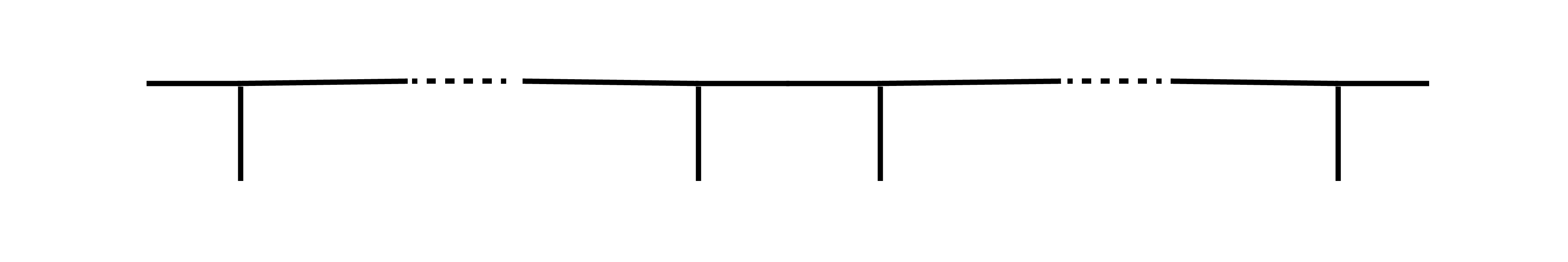}$}
		
		\put(65,33)  {$\tilde\mu_{\tilde m}$}
		\put(46,33)  {$\tilde\mu_1$}
		\put(38.5,33)  {$ \mu_1$}
		\put(19.5,33)  {$ \mu_m$}

		\put (78,34) {$\tilde M_{\tilde m}= \tilde \mu_{\tilde m}$} 
		\put(65, 22) {$\tilde M_1=\tilde \mu_2+\hdots+\tilde \mu_{\tilde m}$}
		\put(58, 18) {$M=M_0=\tilde M_0$} 
		\put(45, 13) {$M_1= \mu_2+\hdots+\mu_m$} 
		\put(25,8) {$M_m= \mu_m$}

		\put(72.5,13) {$ f$}
		
		\put(83,3) {$M$}
		\put(95,13) {$\tilde f$}
\end{overpic}
	\vspace*{-2cm}
\end{figure}

Here we let $\V_{\Gl_M;f,\tilde f}^{\mc S,\kappa}=H^0_{\textup{DS};f,\tilde f}(\mc D^\ch_{\kappa}(G))$ denote the Drinfeld-Sokolov reduction with respect to $\mf n_f\times \mf n_{\tilde{f}} $ at $(f,\tilde f)$ of the $\kappa$-twisted chiral differential operators $\mc D^\ch_\kappa(G)$ on $G$, the natural deformation of $\V_{\Gl_M;f,\tilde f}^{\mc S}=H^0_{\textup{DS};f,\tilde f}(\mc D^\ch_{-h^\vee}(G))$ from $\kappa=-h^\vee$; this algebra was introduced under the notation $I^\kappa_{G,f,\tilde f}$ in Section 7 of \cite{Ar}, generalizing the example of $\textup{Sl}_2$ at generic $\kappa$ considered in \cite{FrSt}.

In particular, this conjecture implies the existence of free field realizations of the class $\mc S$ chiral algebras $\V_{\Gl_M;f,\tilde f}^{\mc S}$ with image characterized as the kernel of the explicit screening operators defined in Section \ref{geoscreensec}. Moreover, these various free field realizations for different rank $\Gl_M$ and choices of nilpotents $f$ and $\tilde f$ are manifestly compatible, so that the conjecture would also imply the existence of parabolic induction and inverse reduction relations between the algebras $\V_{\Gl_M;f,\tilde f}^{\mc S}$, generalizing those for the usual affine type A $W$-algebras $W_f^\kappa(\gl_M)$ described in the preceding Section \ref{Walgsec}.

In particular, consider the case $f=\tilde f = f_\prin \in  \mc N_M \subset \gl_M$, for which we have
\[ \V_{\Gl_M;f_\prin,f_\prin}^{\mc S} \cong   \V_{\Gl_M}^{\mc S} = I_G^\ch  \ , \]
that is, $\V_{\Gl_M;f_\prin,f_\prin}^{\mc S}$ is equivalent to the genus zero class $\mc S$ chiral algebra for $k=0$ points $\V_{\Gl_M}^{\mc S}$, which is by definition the chiral universal centralizer. In this case, we have $m=\tilde m=1$, so that the corresponding divisor is given by $S=M[|\mc O_{\bb P^1}(-1)]$ on $Y=Y_{2,0}$ and the free field module is given by
\[ \Pi(Y,S)= \Pi(Y_{2,0},M[|\mc O_{\bb P^1}(-1)|])= \Pi(Y_{2,0},|\mc O_{\bb P^1}(-1)|)^{\otimes M} \cong \mc D^\ch(\bb G_m)^{\otimes M} \cong \mc D^\ch(T) \ , \]
in keeping with the Beem-Nair conjecture \cite{BN2}, part of which was recently proved in \cite{Fur}. Thus, our approach gives a refinement of the Beem-Nair conjecture, providing an explicit description of the screening operators which conjecturally characterize the image of the desired free field realization, as well as a generalization thereof to generic $\kappa$.

\begin{eg} \label{MRRVireg}	Let $Y=|\mc O_{\bb P^1}\oplus \mc O_{\bb P^1}(-2)|$ and $S=S_{0,2,0,0}=2[|\mc O_{\bb P^1}|]$. Then the Heisenberg subalgebra $\pi_S$ of the free field vertex algebra \[ \Pi(Y,S)= \Pi(Y,|\mc O_{\bb P^1}|)^{\otimes 2} \]
	is generated by four Heisenberg fields, $J^1$ and $J^3$ satisfying
\[ J^1(z) J^1(w) \sim -\frac{1}{\e_1\e_2} \frac{1}{(z-w)^2} \quad\quad\text{and}\quad\quad  J^3(z) J^3(w) \sim \frac{1}{\e_1\e_2} \frac{1}{(z-w)^2} \]
	which generate the first copy of $\Pi(Y,|\mc O_{\bb P^1}|)$ together with the vertex operator
	\[ V_{1,0}(z) = \nol \exp( \e_1(\phi_1(z)+\phi_3(z)) ) \nor  \quad\quad\text{and more generally} \quad\quad V_{l,0}(z) = \nol V_{1,0}(z)^l \nor \]
	for each $l\in \bb Z$, that is, they satisfy the relations
	\[ J^1(z) V_{l,0}(w) \sim - \frac{l}{\e_2} \frac{V_{l,0}(w)}{z-w} \quad\quad J^3(z) V_{l,0}(w) \sim \frac{l}{\e_2} \frac{V_{l,0}(w)}{z-w} \quad\quad \text{and}\quad\quad V_{l,0}(z)V_{m,0}(w)= \nol V_{l,0}(z) V_{m,0}(w)\nor \ , \]
	and the independent fields $J^2$ and $J^4$ satisfying
	\[ J^2(z) J^2(w) \sim -\frac{1}{\e_1\e_2} \frac{1}{(z-w)^2} \quad\quad\text{and}\quad\quad  J^4(z) J^4(w) \sim \frac{1}{\e_1\e_2} \frac{1}{(z-w)^2} \]
	which generate the second copy of $\Pi(Y,|\mc O_{\bb P^1}|)$ together with the vertex operator
	\[ V_{0,1}(z) = \nol \exp( \e_1(\phi_2(z)+\phi_4(z)) ) \nor  \quad\quad\text{and more generally} \quad\quad V_{0,l}(z) = \nol V(z)^l \nor \]
	for each $l\in \bb Z$, which satisfy the relations
	\[ J^2(z) V_{0,l}(w) \sim - \frac{l}{\e_2} \frac{V_{0,l}(w)}{z-w} \quad\quad J^4(z) V_{0,l}(w) \sim \frac{l}{\e_2} \frac{V_{0,l}(w)}{z-w} \quad\quad \text{and}\quad\quad V_{0,l}(z)V_{0,m}(w)= \nol V_{0,l}(z) V_{0,m}(w)\nor \ .  \]
The screening currents are given by
\begin{equation}\label{MRRscreeneqn}
	 Q_1=\nol \exp( \e_2(\phi_1-\phi_2)) \nor \quad\quad \text{and}\quad\quad   Q_2=\nol \exp(-\e_2(\phi_3-\phi_4))  \nor  \ .
\end{equation}

\begin{figure}[t]
	\begin{overpic}[width=.8\textwidth]{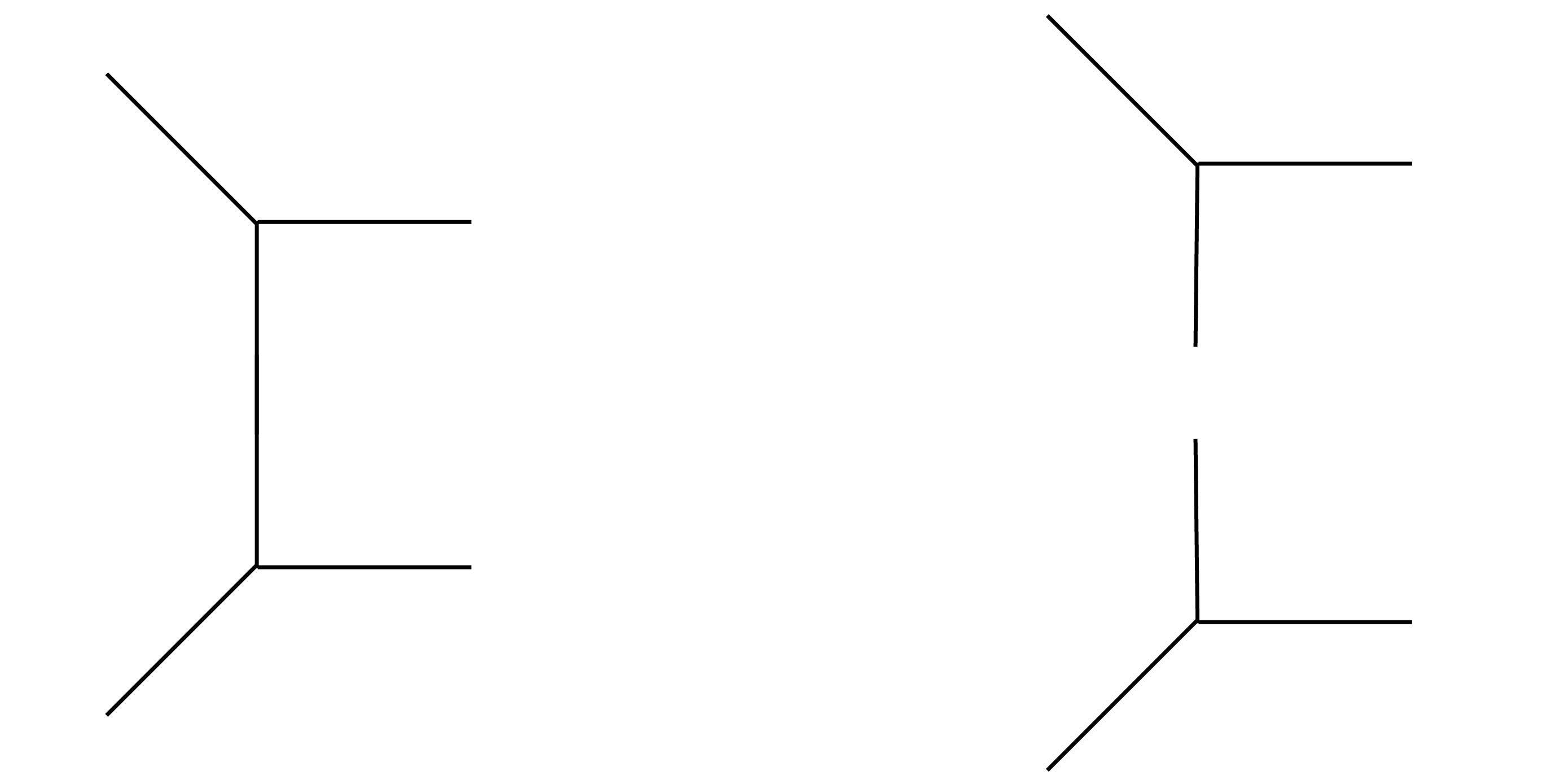}
		\put(17,14.5) {\scriptsize{$J^{1/2}$}}
		\put(17,33) {\scriptsize{$J^{3/4}$}}	
		\put(17,24) {\scriptsize{$V_{1,0/0,1}$}}
		\put(31,13) {\scriptsize{$Q_1$}}
		\put(31,35.5) {\scriptsize{$Q_2$}}
		\put(28,24) {$2$}
		
		\put(77.5,11) {\scriptsize{$J^{1/2}$}}
		\put(77.5,37) {\scriptsize{$J^{3/4}$}}	
		\put(91,9.5) {\scriptsize{$Q_1$}}
		\put(91,39) {\scriptsize{$Q_2$}}
		\put(85,18) {$2$}
		\put(85,30) {$2$}
		
	\end{overpic}
		\caption{Vertex algebra data and locality principle for $\V(Y_{2,0},S_{0,2,0,0})$}
	\label{fig:S1100fig}
\end{figure}
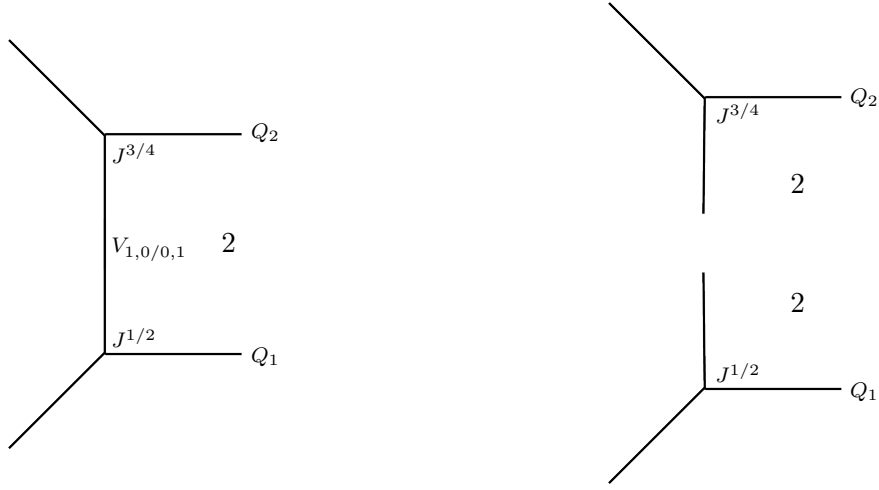

Consider the change of basis in the Heisenberg subalgebra $\pi_S$ of $\Pi(Y,S)$ given by
\[ \alpha = J^1+J^2 \quad\quad \beta = J^3+J^4 \quad\quad \alpha^- = J_1-J_2 \quad\quad \beta^- = J^3-J^4 \ , \]
as well as in the lattice direction by defining
\[ V^+ =  V_{1,1} = \nol \exp( \e_1 ( \phi_1+\phi_2+\phi_3+\phi_4)) \nor \quad\quad \text{and}\quad\quad V^- = V_{1,-1} =\nol \exp( \e_1 ( \phi_1-\phi_2+\phi_3-\phi_4))  \ .  \]
 Note that the fields $\alpha$, $\beta$, and $V_+$ all commute with the remaining fields $\alpha^-$, $\beta^-$ and $V_-$, and have non-singular OPE with both of the screening currents, so that they generate an independent half-lattice vertex algebra $\Pi_+\subset \V(Y,S)$. The fields $\alpha^-$, $\beta^-$ and $V^-$ generate another independent half-lattice vertex algebra $\Pi_- \subset \Pi(Y,S)$ so that we have
\[ \Pi(Y,S) \cong \Pi^+\otimes \Pi^- \ , \]
where both screening operators $Q_1$ and $Q_2$ vanish on $\Pi^+$.

Now, we can apply the locality property of Proposition \ref{localityprop} analyze the structure of $\V(Y,S)$. In particular, each field $\alpha^-$ and $\beta^-$ has non-singular OPE with the complimentary screening current $Q_2$ and $Q_1$, respectively. Thus, we can decompose
\[ \Pi^- = \bigoplus_{\lambda \in \bb Z} \pi^{\alpha^-}_{\lambda \e_1} \otimes \pi^{\beta^-}_{\lambda \e_1 } \]
and we have that
\[ Q_1|_{\pi^\beta} = Q_2|_{\pi^\alpha} = 0 \quad\quad \text{and moreover}\quad\quad \ker( Q_1|_{\pi^\alpha})\cong \mc W^\kappa_{f_\prin}(\spl_2) \quad\quad  \ker( Q_2|_{\pi^\beta})\cong \mc W^{\kappa^*}_{f_\prin}(\spl_2) \]
by the calculation of Example \ref{Virscreg}, where
\begin{equation}\label{classSleveleqn}
	 \kappa + 2 = - \frac{\e_2}{\e_1} \quad\quad \textup{and}\quad\quad \kappa^*+2 =  \frac{\e_2}{\e_1} \ , 
\end{equation}
so that we have
\[ \hspace*{-1cm} \ker(Q|_{\pi^{\alpha^-} \otimes \pi^{\beta^-}}) \cong  \mc W^\kappa_{f_\prin}(\spl_2) \otimes  \mc W^{\kappa^*}_{f_\prin}(\spl_2) \quad\quad\textup{and more generally}\quad\quad \ker(Q|_{\Pi^-}) \cong \bigoplus_{\lambda \in P_+} \mc W^\kappa_{f_\prin,\lambda}(\spl_2) \otimes  \mc W^{\kappa^*}_{f_\prin,\lambda^*}(\spl_2)  \ , \]
which identifies with the modified regular representation, or chiral universal centralizer $I_{\spl_2}^\kappa$ for $\g=\spl_2$, at level $\kappa=-h^\vee-\frac{\e_2}{\e_1}$, as defined in \cite{FrSt,Ar}. In summary, we obtain
\[ \V(Y,S) \cong \Pi^+ \otimes I_{\spl_2}^\kappa \ . \]
\end{eg}

Our expectation is that a similar strategy can be used to prove Conjecture \ref{classSconj} in general. Indeed, applying Proposition \ref{localityprop} together with Theorem \ref{WLStheo}, we obtain in general:

\begin{corollary} The vertex algebra $\V(Y_{m+\tilde m,0},S_{\mu,\tilde{\mu}}) $ is canonically a vertex algebra extension of $W_{f_\mu}^\kappa(\gl_M)\otimes W_{f_{\tilde \mu}}^{\kappa^*}(\gl_{M})$, where $\kappa$ and $\kappa^*$ are as in Equation \ref{classSleveleqn}.
\end{corollary}

For example, generalizing Example \ref{MRRVireg} we can consider again the case $M=2$ and choose non-principal (and thus necessarily trivial) nilpotents $f=0$ or $\tilde f=0$, for which the corresponding partitions are given by $\mu=(1,1)$ or $\tilde \mu=(1,1)$; the divisor $S_{\mu,\tilde \mu}$ in $Y_{4,0}=\widetilde{A}_{3}\times \bb A^1$ and corresponding vertex algebra data for the case $f=\tilde f=0\in\gl_2$ are pictured in Figure \ref{fig:S01210fig} below.

In addition to the generators $J^1,J^2,J^3,J^4$ from Example \ref{MRRVireg} above, we have independent Heisenberg generators $J^5,J^6,J^7$ and $J^8$ at levels
\[k_5 =\frac{1}{\e_1(\e_1+\e_2)}\quad\quad  k_6= - \frac{1}{\e_1(\e_1+\e_2)} \quad\quad  k_7= \frac{1}{\e_1(\e_1-\e_2)} \quad\quad\text{and}\quad\quad k_8= -\frac{1}{\e_1(\e_1-\e_2)}  \ , \]
as well as the additional lattice generators
\[ V^+(z) = \nol \exp( \e_1(\phi_5(z)+\phi_6(z)) ) \nor\quad\quad\text{and}\quad\quad  V^-(z) = \nol \exp( \e_1(\phi_7(z)+\phi_8(z)) ) \nor \]
and the additional screening currents between the Heisenberg generators in distinct components
\[  Q_+=\nol \exp( \e_2\phi_1 + (\e_1+\e_2)\phi_5 ) \nor \quad\quad \text{and}\quad\quad   Q_-=\nol \exp(\e_2\phi_4 + (\e_1-\e_2)\phi_7 )  \nor  \ . \]

\begin{figure}[b]
	\caption{Vertex algebra data and locality principle for $\V(Y_{4,0},S_{0,1,2,1,0})$}
	\label{fig:S01210fig}
	\begin{overpic}[width=1\textwidth]{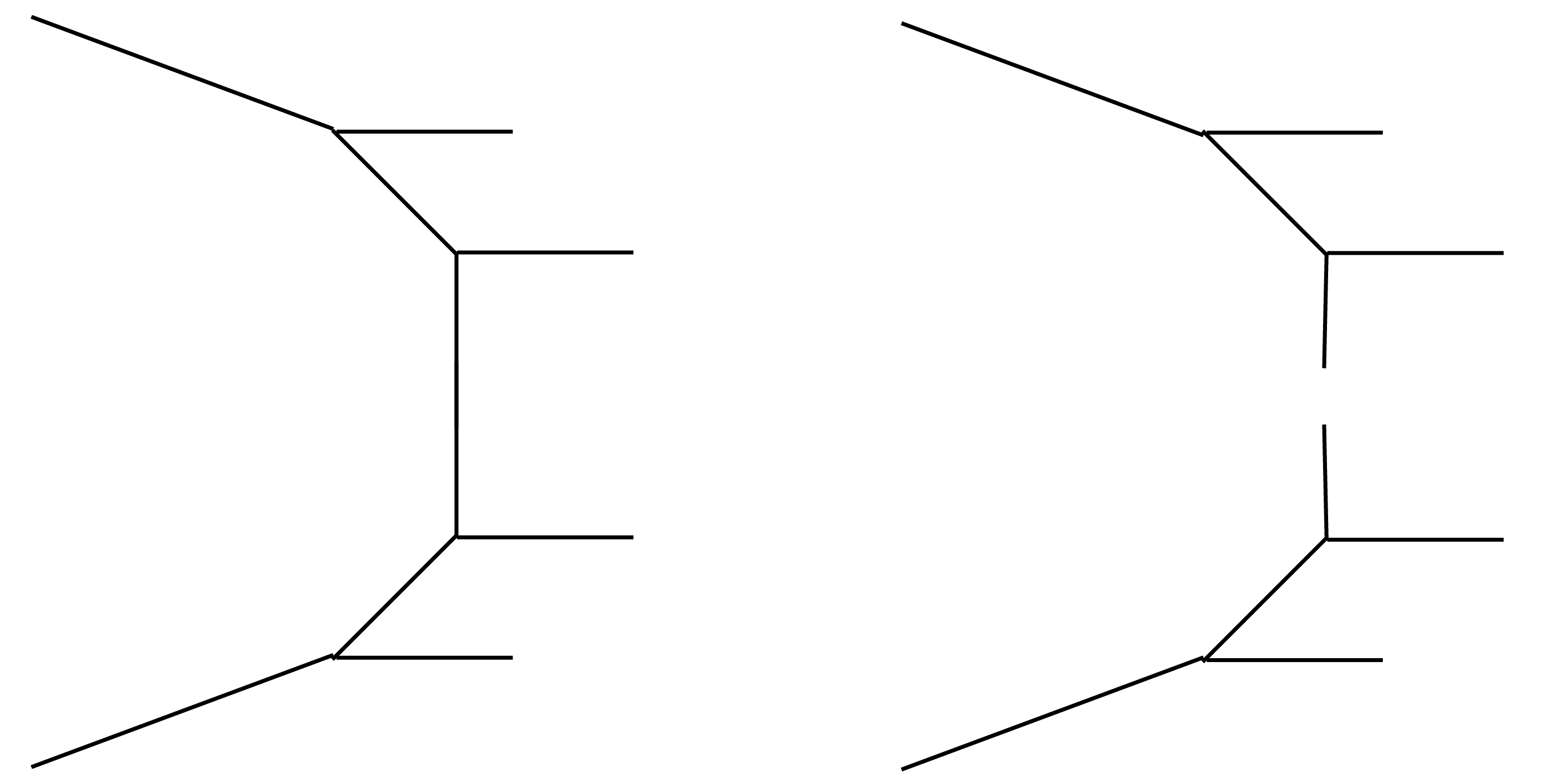}
		\put(29.5,16.5) {\scriptsize{$J^{1/2}$}}
		\put(29.5,31.5) {\scriptsize{$J^{3/4}$}}
		\put(29.5,24) {\scriptsize{$V_{1,0/0,1}$}}
		\put(41,16) {\scriptsize{$Q_1$}}
		\put(41,32) {\scriptsize{$Q_2$}}
		\put(41,14) {\scriptsize{$Q_+$}}
		\put(41,34) {\scriptsize{$Q_-$}}
		\put(38,24) {$2$}
		\put(32,11) {$1$}
		\put(32,37) {$1$}
		
		\put(28,13.5) {\scriptsize{$J^{5}$}}
		\put(28.5,34.5) {\scriptsize{$J^{7}$}}
		
		\put(23,8.5) {\scriptsize{$J^{6}$}}
		\put(23.5,39.8) {\scriptsize{$J^{8}$}}
		
		\put(26,11) {\scriptsize{$V^+$}}
		\put(26,37) {\scriptsize{$V^-$}}
		
		\put(85.5,16.5) {\scriptsize{$J^{1/2}$}}
		\put(85.5,31.6) {\scriptsize{$J^{3/4}$}}	
		\put(84,13.5) {\scriptsize{$J^{5}$}}
		\put(84,34.5) {\scriptsize{$J^{7}$}}
		\put(79,8.5) {\scriptsize{$J^{6}$}}
		\put(79.5,39.8) {\scriptsize{$J^{8}$}}

\put(81.5,11) {\scriptsize{$V^+$}}
\put(81.5,37) {\scriptsize{$V^-$}}

		\put(97,16) {\scriptsize{$Q_1$}}
		\put(97,14) {\scriptsize{$Q_+$}}
		\put(97,32.5) {\scriptsize{$Q_2$}}
				\put(97,34.5) {\scriptsize{$Q_-$}}
		\put(90,20) {$2$}
		\put(90,28) {$2$}
				\put(88,11) {$1$}
		\put(88,37) {$1$}
	\end{overpic}
\end{figure}
The subalgebra generated by $J^1,J^2,J^5,J^6$ and $V^+$ and screened by $Q_1$ and $Q_+$ is evidently identified with $V^\kappa(\gl_2)$ by Theorem \ref{affinesl2theo}, and similarly that for the complimentary generators $J^3,J^4, J^7,J^8$ and $V^-$ and screening operators $Q_2$ and $Q_-$ is identified with $V^{\kappa^*}(\gl_2)$, where $\kappa$ and $\kappa^*$ are again as in Equation \ref{classSleveleqn}, so that we obtain a vertex algebra extension of $V^\kappa(\gl_2)\otimes V^{\kappa^*}(\gl_2)$, conjecturally isomorphic to $\Pi_+\otimes \mc D^\ch_\kappa(\textup{SL}_2)$; this partition of a subset of the generators into subalgebras is pictured on the right of Figure \ref{fig:S01210fig}.

Finally, we describe a more speculative extension of the preceding proposal, which gives a conjectural alternative construction for genus zero class $\mc S$ chiral algebras $\bb V_{\Gl_M;f_1,...,f_k}^{\mc S}$ for $G=\Gl_M$ and $k>2$ marked points, and perhaps even the twisted variants constructed in \cite{BN1}.
The basic idea is that by judiciously forgetting the screening operators corresponding to certain compact curve classes in the definition of $\V(Y,S)$, one can formally separate the sheets of the divisor in order to include the data of multiple nilpotents within a single divisor in the same class of threefolds $Y_{m,0}$ used in the conjectural construction for $k\leq 2$ marked points described above.

To begin, we consider the following critical level limit $\kappa\to - h^\vee$ of the preceding constructions for $k\leq 2$. We introduce the rescaled Heisenberg generators $\tilde{J^y}=\e_2 J^y$ for $y=1,2,3,4$, which satisfy
\[  \tilde{J}^{1/2}(z)\tilde{J}^{1/2}(w) \sim -\frac{\e_2}{\e_1} \frac{1}{(z-w)^2} = \frac{\kappa+h^\vee}{(z-w)^2} \quad\quad\text{and} \quad\quad   \tilde{J}^{3/4}(z)\tilde{J}^{3/4}(w) \sim \frac{\e_2}{\e_1} \frac{1}{(z-w)^2} =- \frac{\kappa+h^\vee}{(z-w)^2}\]
and corresponding bosons $\tilde\phi_y=\e_2 \phi_y$ in terms of which we have
\[	 Q_1=\nol \exp( \tilde \phi_1-\tilde \phi_2) \nor \quad\quad \text{and}\quad\quad   Q_2=\nol \exp(\tilde \phi_3- \tilde \phi_4)  \nor  \ .\]
In particular, in the limit $\e_2\to 0$ or equivalently $\kappa \to -h^\vee$ we have that the $J^y$ generate commutative Heisenberg algebras for $y=1,2,3,4$, and the subalgebras screened by $Q_1$ and $Q_2$ in those generated by $J^1,J^2$ and $J^3,J^4$, respectively, are each isomorphic to the critical level $\kappa=\kappa^*=-h^\vee$ principal affine $W$ algebra, or equivalently the centre of the affine Kac-Moody algebra $\mf z(\glh_2) \subset V^{-h^\vee}(\gl_2)$ at critical level $\kappa=-h^\vee$, by Proposition 7.3.6 of \cite{Fr1}.

Further, note that in this limit the levels of the additional Heisenberg generators are given by
\[k_5 =k_7=\frac{1}{\e_1^2}\quad\quad\text{and}\quad\quad  k_6= k_8=  - \frac{1}{\e_1^2}  \ , \]
and the auxiliary screening currents are given by
\[  Q_+=\nol \exp(\tilde \phi_1 + \e_1\phi_5 ) \nor \quad\quad \text{and}\quad\quad   Q_-=\nol \exp(\tilde \phi_4 + \e_1\phi_7 )  \nor \ . \]
In particular, note that the subalgebra isomorphic to $\mc D^\ch(\bb G_m)$ generated by $J^5,J^6$ and $V^+$ is canonically identified with that generated by $J^7,J^8$ and $V^-$, under which the screening currents $Q_+$ and $Q_-$ differ only by the choice of boson $\tilde \phi_1$ or $\tilde \phi_4$. This additional symmetry naturally suggests a generalization to $k>2$ by simply tensoring with additional subalgebras $\mc D^\ch(\bb G_m)$ and introducing auxiliary screening operators of this form, for $\tilde \phi$ some linear combination of the bosons $\tilde \phi^1,\tilde \phi^2,\tilde \phi^3$ and $\tilde \phi^4$. Such algebras arise as modifications of vertex algebras $\V(Y_{m,0},S)$ for more general divisors $S$ in the same class of threefolds $Y_{m,0}$, as we now explain:

For concreteness, we begin by considering the case of $k=3$ non-trivial marked points and corresponding non-principal, and thus necessarily trivial, nilpotents $f_1=f_2=f_3=0$. Consider the divisor $S=S_{0,2,2,1,0}$ and corresponding vertex algebra $\V(Y_{4,0},S_{0,2,2,1,0})$, with generators and screening operators as indicated on the left in Figure \ref{fig:classSkfig} below. Our candidate for the class $\mc S$ chiral algebra $\V_{\Gl_2,b=3}^{\mc S}$ is defined by omitting, from the intersection defining $\V(Y_{4,0},S_{0,2,2,1,0})$, the kernel of the screening operator $Q_{s_1}$ corresponding to the non-compact curve class $C_{s_1}$.
We let $\tilde{\V}_{\Gl_2,b=3}^{\mc S}$ denote the resulting vertex algebra, defined by the intersection
\[ \tilde{\V}_{\Gl_2,b=3}^{\mc S} := \ker(Q_1) \cap \ker(Q_2) \cap \ker(Q_+) \cap \ker(Q_-) \cap \ker(Q_{s_2}) \subset  \Pi(Y_{4,0},S_{0,2,2,1,0}) \ . \]

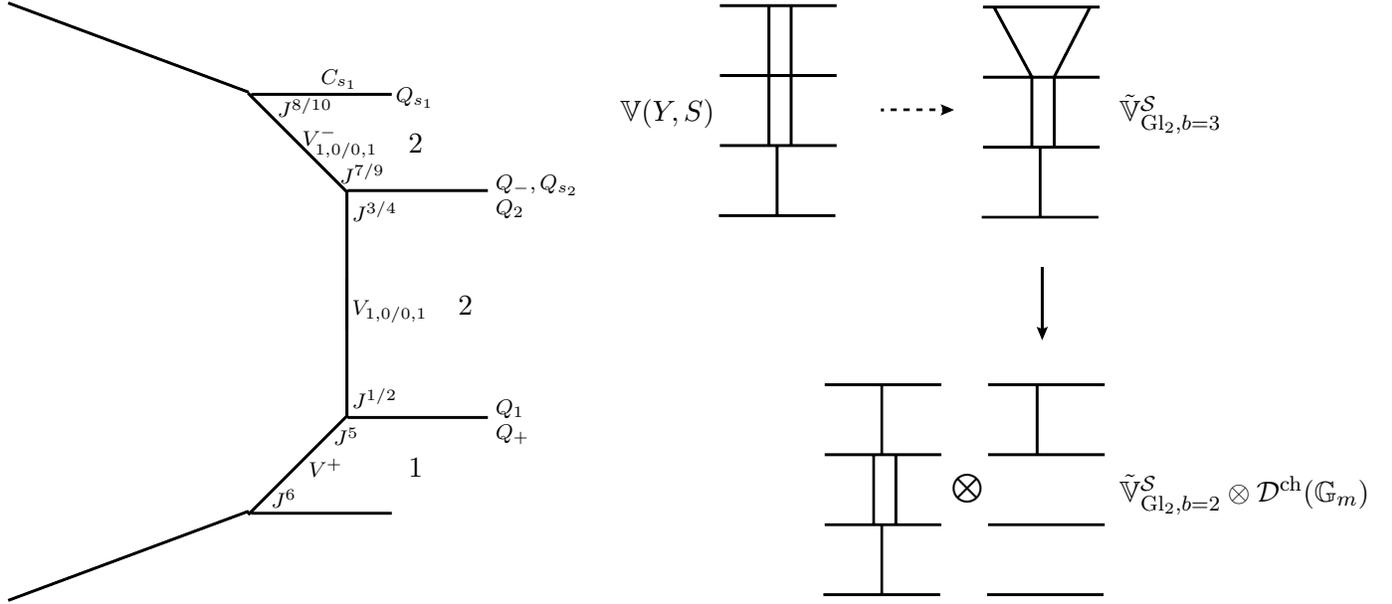
\begin{figure}[t]
	\begin{overpic}[width=1\textwidth]{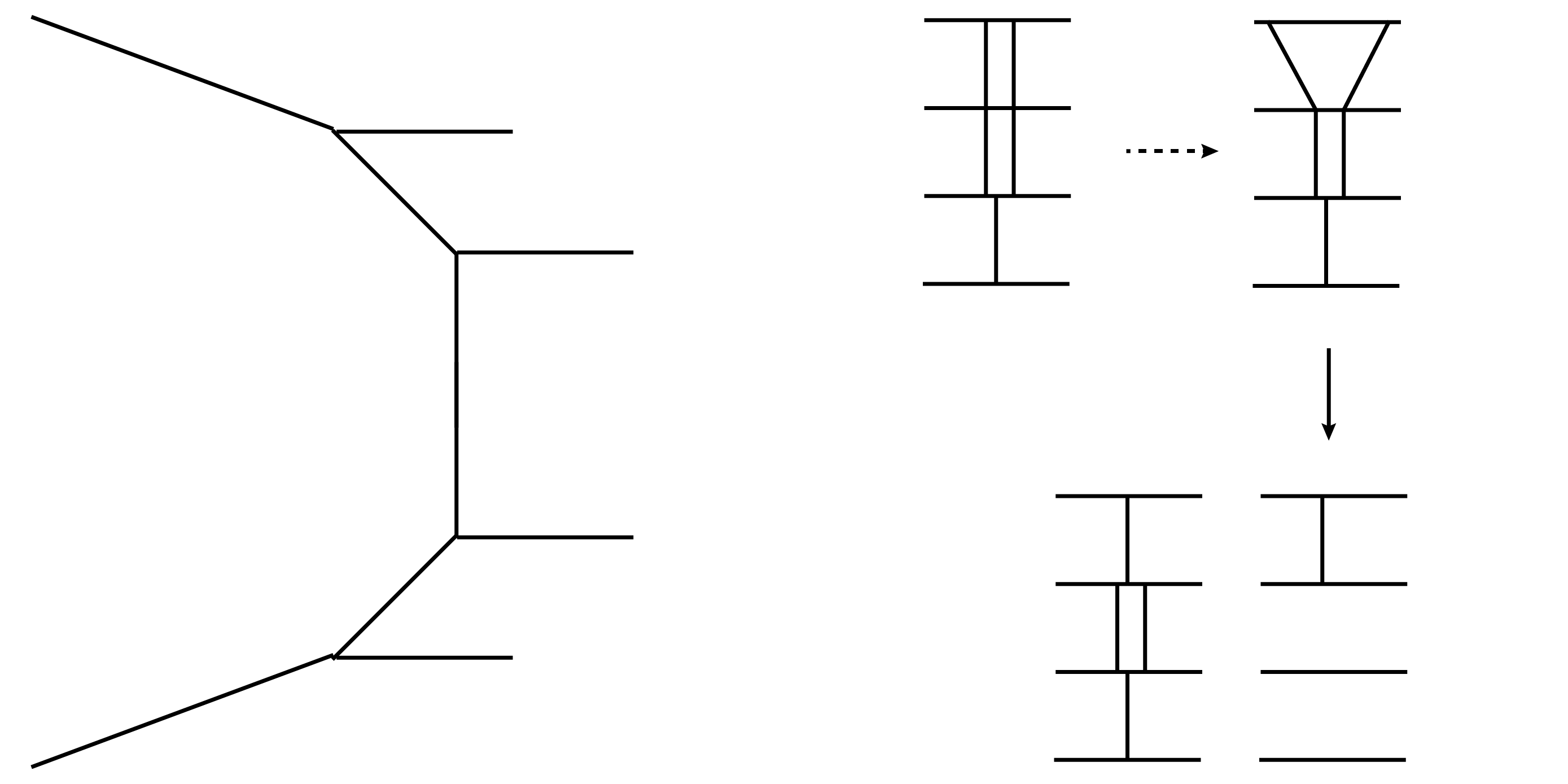}
		\put(29.5,16.5) {\scriptsize{$J^{1/2}$}}
		\put(29.5,31.5) {\scriptsize{$J^{3/4}$}}
		\put(29.5,24) {\scriptsize{$V_{1,0/0,1}$}}
		\put(41,16) {\scriptsize{$Q_1$}}
		\put(41,32) {\scriptsize{$Q_2$}}
		\put(41,14) {\scriptsize{$Q_+$}}
		\put(41,34) {\scriptsize{$Q_-,Q_{s_2}$}}
		\put(33,41) {\scriptsize{$Q_{s_1}$}}
		\put(38,24) {$2$}
		\put(34,11) {$1$}
		\put(34,37) {$2$}
		
		\put(28,13.5) {\scriptsize{$J^{5}$}}
		\put(28.5,34.5) {\scriptsize{$J^{7/9}$}}
		
		\put(23,8.5) {\scriptsize{$J^{6}$}}
		\put(23.5,39.8) {\scriptsize{$J^{8/10}$}}
		
		\put(26,11) {\scriptsize{$V^+$}}
		\put(25.5,37.5) {\scriptsize{$V^-_{1,0/0,1}$}}

\put(51,39.5) {$\V(Y,S)$} 
\put(91,39.5) {$\tilde{\V}_{\Gl_2,b=3}^{\mc S}$}
\put(91,9) {$\tilde{\V}_{\Gl_2,b=2}^{\mc S} \otimes \mc D^\ch(\bb G_m)$}

\put(77.5,9.5) {$\bigotimes$}

\put(27,42.5) {\scriptsize{$C_{s_1} $}}

	\end{overpic}
	
	\caption{The conjectural free field presentation $\tilde{\V}_{\Gl_2,b=3}^{\mc S}$ from $\V(Y_{4,0},S_{0,2,2,1,0})$}
	\label{fig:classSkfig}
	
\end{figure}

 We propose that one can understand this modification geometrically as deforming the two sheets of the divisor along that irreducible component in order to separate them at $C_{s_1}$, as pictured on the right of Figure \ref{fig:classSkfig}, and we conjecture that this vertex algebra (or a slight variant thereof) will be isomorphic to the desired class $\mc S$ chiral algebra $\V_{\Gl_2,b=3}^{\mc S}$. Indeed, as in Example \ref{MRRVireg} there is a subalgebra $\Pi_+\cong \mc D^\ch(\bb G_m) \subset \Pi(Y_{4,0},S_{0,2,2,1,0})$ generated by two rank 1 Heisenberg algebras and a single $\Z$-family of lattice vertex operators which commute with all the screening operators, and the complimentary free field algebra is generated by eight rank 1 Heisenberg algebras and four orthogonal families of lattice generators, which can together be identified with $\mc D^\ch(\bb G_m)^{\otimes 4}$. We conjecture that the subalgebra of the latter defined by the remaining screening operators is equivalent to four copies of the standard free field realization of \cite{FMS} given by
\[ \mc D^\ch(\bb A^4) \to \mc D^\ch(\bb G_m)^{\otimes 4}  \quad\quad \text{so that we have} \quad\quad  \tilde{\V}_{\Gl_2,b=3}^{\mc S} \cong \Pi_+ \otimes   \mc D^\ch(\bb A^4) =\Pi_+ \otimes   \bb V_{\textup{Sl}_2,b=3}^{\mc S}\ ,  \]
in agreement with the calculation of Theorem A.1 in \cite{Ar}. Further, note there is an embedding
\[ \tilde{\V}_{\Gl_2,b=3}^{\mc S} \to  \V(Y_{4,0},S_{0,1,2,1,0}) \otimes \V(Y_{4,0}, |\mc O_{\bb P^1}|) \cong \tilde{\V}_{\Gl_2,b=2}^{\mc S} \otimes \mc D^\ch(\bb G_m)   \]
with image characterized by the kernel of the single screening operator $Q_{s_2}$. More generally, applying the same procedure to define $\tilde{\V}_{\Gl_2,b}^{\mc S}$ for $b\geq 3$, we have analogous relative free field realizations
\[  \tilde{\V}_{\Gl_2,b}^{\mc S} \to \tilde{\V}_{\Gl_2,b-1}^{\mc S} \otimes \mc D^\ch(\bb G_m) \ , \]
which appear to be consistent with the system of partial free field realizations of the genus zero class $\mc S$ chiral algebras for $G=\textup{Sl}_2$ constructed in \cite{BN2}.

\bibliographystyle{alpha}

\bibliography{reCYsnewest}

\begin{thebibliography}{AHDM78}

\bibitem[ACG21]{ACG}
Drazen Adamovic, Thomas Creutzig, and Naoki Genra.
\newblock {Relaxed and logarithmic modules of $\widehat{\mathfrak{sl}_3}$}.
\newblock 10 2021.

\bibitem[AGT10]{AGT}
Luis~F. Alday, Davide Gaiotto, and Yuji Tachikawa.
\newblock Liouville correlation functions from four-dimensional gauge theories.
\newblock {\em Letters in Mathematical Physics}, 91(2):167–197, 2010.

\bibitem[AHDM78]{ADHM}
M.F. Atiyah, N.J. Hitchin, V.G. Drinfeld, and Yu.I. Manin.
\newblock Construction of instantons.
\newblock {\em Physics Letters A}, 65(3):185–187, 1978.

\bibitem[Ara18]{Ar}
Tomoyuki Arakawa.
\newblock Chiral algebras of class $\mathcal{S}$ and {M}oore-{T}achikawa
  symplectic varieties.
\newblock arXiv:1811.01577, 2018.

\bibitem[AW22]{AlW}
Robert Allen and Simon Wood.
\newblock Bosonic ghostbusting: The bosonic ghost vertex algebra admits a
  logarithmic module category with rigid fusion.
\newblock {\em Communications in Mathematical Physics}, 390(2):959–1015,
  2022.

\bibitem[BBN24]{BBN}
Christopher Beem, Dylan Butson, and Sujay Nair.
\newblock Inverse {H}amiltonian reduction in type {A} and generalized slices in
  the affine {G}rassmannian.
\newblock \emph{in preparation}, 2024.

\bibitem[Bei78]{Bei2}
A.~A. Beilinson.
\newblock Coherent sheaves on $\mathbb{P}^n$ and problems of linear algebra.
\newblock {\em Functional Analysis and Its Applications}, 12(3):214–216,
  1978.

\bibitem[BFN14]{BFN5}
Alexander Braverman, Michael Finkelberg, and Hiraku Nakajima.
\newblock Instanton moduli spaces and $\mathcal{W}$-algebras.
\newblock arXiv:1406.2381, 2014.

\bibitem[BK04]{BrKl}
Jonathan Brundan and Alexander~S. Kleshchev.
\newblock Shifted {Y}angians and finite {W}-algebras.
\newblock {\em Advances in Mathematics}, 200:136--195, 2004.

\bibitem[BN23a]{BN2}
Christopher Beem and Sujay Nair.
\newblock {Free Field Realisation of the Chiral Universal Centraliser}.
\newblock {\em Annales Henri Poincare}, 24(12):4343--4404, 2023.

\bibitem[BN23b]{BN1}
Christopher Beem and Sujay Nair.
\newblock {Twisted Chiral Algebras of Class ${\mathcal {S}}$ and Mixed
  Feigin\textendash{}Frenkel Gluing}.
\newblock {\em Commun. Math. Phys.}, 399(1):295--366, 2023.

\bibitem[BPRR15]{BeemS}
Christopher Beem, Wolfger Peelaers, Leonardo Rastelli, and Balt C.~Van Rees.
\newblock Chiral algebras of class $ \mathcal{S} $.
\newblock {\em Journal of High Energy Physics}, 2015(5), 2015.

\bibitem[BR23]{BR1}
D.~Butson and M.~Rap\v{c}\'{a}k.
\newblock Perverse coherent extensions on {C}alabi-{Y}au threefolds and
  representations of cohomological {H}all algebras.
\newblock arXiv:, 2023.

\bibitem[Bri02]{Bdg1}
Tom Bridgeland.
\newblock Flops and derived categories.
\newblock {\em Inventiones Mathematicae}, 147(3):613–632, 2002.

\bibitem[CCDS21]{CCDS}
Wu-Yen Chuang, Thomas Creutzig, Duiliu-Emanuel Diaconescu, and Yan Soibelman.
\newblock Hilbert schemes of nonreduced divisors in {C}alabi–{Y}au threefolds
  and {W}-algebras.
\newblock {\em European Journal of Mathematics}, 2021.

\bibitem[CG20]{CrG}
Thomas Creutzig and Davide Gaiotto.
\newblock Vertex algebras for {S}-duality.
\newblock {\em Communications in Mathematical Physics}, 2020.

\bibitem[Cos18]{CosMSRI}
Kevin Costello.
\newblock The cohomological {H}all algebra and {M}-theory, talk at
  `{S}tructures in {E}numerative {G}eometry' conference at {MSRI}.
\newblock Recording available at:
  \href{https://www.msri.org/workshops/816/schedules/23868}{https://www.msri.org/workshops/816/schedules/23868},
  2018.

\bibitem[Dav17]{Dav1}
Ben Davison.
\newblock The critical coha of a quiver with potential.
\newblock {\em The Quarterly Journal of Mathematics}, 68(2):635–703, 2017.

\bibitem[Feh23a]{Feh2}
Zachary Fehily.
\newblock {Inverse reduction for hook-type W-algebras}.
\newblock 6 2023.

\bibitem[Feh23b]{Feh1}
Zachary Fehily.
\newblock {Subregular W-algebras of type A}.
\newblock {\em Commun. Contemp. Math.}, 25(09):2250049, 2023.

\bibitem[FF96]{FF1}
Boris Feigin and Edward Frenkel.
\newblock Integrals of motion and quantum groups.
\newblock {\em Lecture Notes in Mathematics}, page 349–418, 1996.

\bibitem[FG06]{FrG}
Edward Frenkel and Dennis Gaitsgory.
\newblock Local geometric {L}anglands correspondence and affine {K}ac-{M}oody
  algebras.
\newblock {\em Algebraic Geometry and Number Theory Progress in Mathematics},
  page 69–260, 2006.

\bibitem[FG20]{FeiG}
Boris Feigin and Sergei Gukov.
\newblock {VOA}[{$M_4$}].
\newblock {\em Journal of Mathematical Physics}, 61(1):012302, 2020.

\bibitem[FMS86]{FMS}
Daniel Friedan, Emil~J. Martinec, and Stephen~H. Shenker.
\newblock {Conformal Invariance, Supersymmetry and String Theory}.
\newblock {\em Nucl. Phys. B}, 271:93--165, 1986.

\bibitem[Fre10]{Fr1}
Edward Frenkel.
\newblock {\em Langlands correspondence for loop groups}.
\newblock Cambridge University Press, 2010.

\bibitem[FS06]{FrSt}
Igor~B. Frenkel and Konstantin Styrkas.
\newblock Modified regular representations of affine and virasoro algebras, voa
  structure and semi-infinite cohomology.
\newblock {\em Advances in Mathematics}, 206(1):57–111, 2006.

\bibitem[Fur23]{Fur}
Shun Furihata.
\newblock {On the Beem-Nair Conjecture}.
\newblock 6 2023.

\bibitem[Gen17]{Gen1}
Naoki Genra.
\newblock Screening operators for $\mathcal{W}$-algebras.
\newblock {\em Selecta Mathematica}, 23(3):2157–2202, 2017.

\bibitem[Gen20]{Gen2}
Naoki Genra.
\newblock Screening operators and parabolic inductions for affine
  $\mathcal{W}$-algebras.
\newblock {\em Advances in Mathematics}, 369:107179, 2020.

\bibitem[G{\"o}t90]{Gott}
Lothar G{\"o}ttsche.
\newblock The betti numbers of the hilbert scheme of points on a smooth
  projective surface.
\newblock {\em Mathematische Annalen}, 286:193--207, 1990.

\bibitem[GR19]{GaiR}
Davide Gaiotto and Miroslav Rapčák.
\newblock Vertex algebras at the corner.
\newblock {\em Journal of High Energy Physics}, 2019(1), 2019.

\bibitem[GR22]{GaiR2}
Davide Gaiotto and Miroslav Rapcak.
\newblock {Miura operators, degenerate fields and the M2-M5 intersection}.
\newblock {\em JHEP}, 01:086, 2022.

\bibitem[Gro96]{Groj}
I.~Grojnowski.
\newblock Instantons and affine algebras {I}: The {H}ilbert scheme and vertex
  operators.
\newblock {\em Mathematical Research Letters}, 3(2):275–291, 1996.

\bibitem[GRZ]{GRZ}


\bibitem[MO19]{MO}
Davesh Maulik and Andrei Okounkov.
\newblock Quantum groups and quantum cohomology.
\newblock {\em Astérisque}, 408:1–212, 2019.

\bibitem[Nak94]{Nak1}
Hiraku Nakajima.
\newblock Instantons on {ALE} spaces, quiver varieties, and {K}ac-{M}oody
  algebras.
\newblock {\em Duke Mathematical Journal}, 76(2), 1994.

\bibitem[Nak97]{Nak}
Hiraku Nakajima.
\newblock Heisenberg algebra and {H}ilbert schemes of points on projective
  surfaces.
\newblock {\em The Annals of Mathematics}, 145(2):379, 1997.

\bibitem[Nak99]{NakLec}
Hiraku Nakajima.
\newblock Lectures on {H}ilbert schemes of points on surfaces.
\newblock {\em University Lecture Series}, 1999.

\bibitem[Nek03]{Nek1}
Nikita~A. Nekrasov.
\newblock Seiberg-{W}itten prepotential from instanton counting.
\newblock {\em Advances in Theoretical and Mathematical Physics},
  7(5):831–864, 2003.

\bibitem[NY05]{NY0}
Hiraku Nakajima and Kota Yoshioka.
\newblock Instanton counting on blowup. {I}. 4-dimensional pure gauge theory.
\newblock {\em Inventiones mathematicae}, 162(2):313–355, 2005.

\bibitem[NY08]{NY1}
Hiraku Nakajima and Kōta Yoshioka.
\newblock Perverse coherent sheaves on blow-up. {I}. {A} quiver description.
\newblock {\em Advanced Studies in Pure Mathematics}, 2008.

\bibitem[PR18]{PrR}
Tomáš Procházka and Miroslav Rapčák.
\newblock Webs of {W}-algebras.
\newblock {\em Journal of High Energy Physics}, 2018(11), 2018.

\bibitem[Rap19]{Rap}
Miroslav Rap\v{c}\'{a}k.
\newblock On extensions of $\gl (m|n)$ {K}ac-{M}oody algebras and
  {C}alabi-{Y}au singularities.
\newblock {\em Journal of High Energy Physics}, 2020:42, 2019.

\bibitem[RS15]{RenS}
Jie Ren and Yan~S. Soibelman.
\newblock Cohomological {H}all algebras, semicanonical bases and
  {D}onaldson-{T}homas invariants for 2-dimensional {C}alabi-{Y}au categories
  (with an appendix by {B}en {D}avison).
\newblock {\em arXiv: Quantum Algebra}, pages 261--293, 2015.

\bibitem[RSYZ19]{RSYZ}
Miroslav Rapčák, Yan Soibelman, Yaping Yang, and Gufang Zhao.
\newblock Cohomological {H}all algebras, vertex algebras and instantons.
\newblock {\em Communications in Mathematical Physics}, 376(3):1803–1873,
  2019.

\bibitem[Sem94]{Sem}
A.~M. Semikhatov.
\newblock {Inverting the Hamiltonian reduction in string theory}.
\newblock In {\em {28th International Symposium on Particle Theory}}, 8 1994.

\bibitem[Soi16]{Soi}
Yan Soibelman.
\newblock Remarks on cohomological {H}all algebras and their representations.
\newblock {\em Arbeitstagung Bonn 2013}, page 355–385, 2016.

\bibitem[SV13]{SV}
O.~Schiffmann and E.~Vasserot.
\newblock Cherednik algebras, {W}-algebras and the equivariant cohomology of
  the moduli space of instantons on $\mathbb{A}^2$.
\newblock {\em Publications mathématiques de lIHÉS}, 118(1):213–342, 2013.

\bibitem[SW94a]{SW1}
Nathan Seiberg and Edward Witten.
\newblock Electric - magnetic duality, monopole condensation, and confinement
  in {N}=2 supersymmetric {Y}ang-{M}ills theory.
\newblock {\em Nuclear Physics}, 426:19--52, 1994.

\bibitem[SW94b]{SW2}
Nathan Seiberg and Edward Witten.
\newblock Monopoles, duality and chiral symmetry breaking in {N}=2
  supersymmetric {QCD}.
\newblock {\em Nuclear Physics}, 431:484--550, 1994.

\bibitem[VW94]{VW}
Cumrun Vafa and Edward Witten.
\newblock A strong coupling test of s-duality.
\newblock {\em Nuclear Physics B}, 431(1-2):3–77, 1994.

\bibitem[Wak86]{Wak}
Minoru Wakimoto.
\newblock {Fock representations of the affine lie algebra A1(1)}.
\newblock {\em Commun. Math. Phys.}, 104:605--609, 1986.

\bibitem[YZ14]{YZ2}
Yaping Yang and Gufang Zhao.
\newblock The cohomological hall algebra of a preprojective algebra.
\newblock {\em Proceedings of the London Mathematical Society}, 116, 2014.

\bibitem[YZ16]{YZ1}
Yaping Yang and Gufang Zhao.
\newblock On two cohomological hall algebras.
\newblock {\em Proceedings of the Royal Society of Edinburgh: Section A
  Mathematics}, 150:1581 -- 1607, 2016.

\end{thebibliography}

\end{document}